\documentclass{amsbook}
\usepackage{amsmath}
\usepackage{amsfonts}
\usepackage{amssymb}
\usepackage{amscd}
\usepackage{amsthm}
\theoremstyle{plain}

\newtheorem{theorem}{Theorem}[chapter]
\theoremstyle{definition}
\newtheorem{defn}[theorem]{Definition}
\newtheorem{lem}[theorem]{Lemma}
\newtheorem{sublem}[theorem]{Sublemma}
\newtheorem{prop}[theorem]{Proposition}
\newtheorem{cor}[theorem]{Corollary}
\newtheorem{notation}[theorem]{Notation}

\newtheorem{rem}[theorem]{Remark}

\newtheorem{fact}[theorem]{Fact}
\newtheorem{thm}[theorem]{Theorem}
\newtheorem{problem}{Problem}
\newtheorem{examp}{Example}

\newcommand{\rr}{\mathbb{R}}

\newcommand{\ee}{\varepsilon}
\newcommand{\nn}{\mathbb{N}}
\newcommand{\ttt}{\mathcal{T}}

\newcommand{\ff}{\mathcal{F}}

\newcommand{\g}{\mathcal{G}}

\theoremstyle{plain}

\begin{document}
\frontmatter
\title[Higher Order Spreading Models in Banach Space Theory]{Higher Order Spreading Models in Banach Space Theory}
\author{Spiros A. Argyros}
\author{Vassilis Kanellopoulos}
\author{Konstantinos Tyros}
\address{National Technical University of Athens, Faculty of Applied Sciences, Department of Mathematics, Zografou Campus, 157 80, Athens,Greece} 
\email{sargyros@math.ntua.gr} 
\address{National Technical University of Athens, Faculty of Applied Sciences, Department of Mathematics, Zografou Campus, 157 80, Athens,Greece}
\email{bkanel@math.ntua.gr}
\address{Department of Mathematics, University of Toronto, Toronto, Canada M5S 2E4}
\email{ktyros@math.toronto.edu}

\maketitle

\tableofcontents

\begin{abstract} We extend
the classical Brunel-Sucheston definition of the spreading model
by introducing the   $\ff$-sequences $(x_s)_{s\in\ff}$ in a Banach
space and the plegma families in $\ff$ where $\ff$ is a regular
thin family. The new concept yields a transfinite increasing
hierarchy of classes of spreading sequences. We explore the
corresponding theory and we present examples establishing this
hierarchy and illustrating the limitation of the theory.
\end{abstract}


\footnotetext[2]{}

\footnotetext[3]{This research is partially supported by NTUA
Programme PEBE 2009.}

\footnotetext[4]{This work is part of the PhD Thesis of the third
named author}

 \mainmatter

\chapter*{Introduction}
Spreading models have been invented by A. Brunel and L. Sucheston
in the middle of 70's (c.f. \cite{BS}) and since then they have a
constant presence in the evolution of Banach space theory. Since
the goal of the present monograph is to extend and study that
notion, we  begin  by recalling the basics of their definition and
some of their consequences.

\section{The Brunel-Sucheston spreading models}
A spreading model of a Banach space $X$ is a spreading
sequence\footnote{A sequence $(e_n)_{n}$ in a seminormed space
$(E,\|\cdot\|_*)$ is called spreading if for every $n\in\nn$,
$k_1<\ldots<k_n$ in $\nn$ and $a_1,\ldots,a_n\in\rr$ we have that
$\|\sum_{j=1}^na_j e_j\|_*=\|\sum_{j=1}^n a_j e_{k_j}\|_*$}
$(e_n)_{n}$ in a seminormed space $(E,\|\cdot\|_*)$ connected to
$X$ through a \emph{Schreier almost isometry} which we next
describe. A sequence $(x_n)_{n}$ in a Banach space $(X,\|\cdot\|)$
is \emph{Schreier almost isometric} to a sequence $(e_n)_{n}$ in a
seminormed space $(E,\|\cdot\|_*)$, if there exists a null
sequence $(\delta_n)_{n}$ of positive reals such that for every
$F=\{n_1<\ldots<n_k\}$ with $|F|\leq n_1$ and every
$(a_i)_{i=1}^k\in[-1,1]^k$, we have that
 \[\Bigg|\Big\|\sum_{i=1}^ka_ix_{n_i}\Big\| -\Big\|\sum_{i=1}^ka_ie_i\Big\|_*\Bigg|<\delta_{n_1}\]
It is easy to see that any sequence $(e_n)_{n}$ as above is a
spreading one. A sequence $(e_n)_{n}$ is a spreading model of the
space $X$ if there exists a sequence $(x_n)_{n}$ in $X$ which is
Schreier almost isometric to $(e_n)_{n}$. In this case we say that
$(x_n)_{n}$ generates $(e_n)_{n}$ as a spreading model. By
applying Ramsey theorem (c.f. \cite{R}), Brunel and Sucheston
proved that every bounded sequence $(x_n)_{n}$ in a Banach space
$X$ has a subsequence $(x_{k_n})_{n}$ generating a spreading
model.

The spreading sequences possess regular structure. For instance if
$(e_n)_{n}$ is weakly null, then it is $1$-unconditional.
Furthermore if $(e_n)_{n}$ is unconditional then  either it is
equivalent to the standard basis of $\ell^1$ or it is norm
Ces\`aro summable to zero. The importance of the spreading models
arise from the fact that they connect in an asymptotic manner the
structure of an arbitrary Banach space $X$ to the corresponding
one of spaces generated by spreading sequences. For example every
Banach space admits unconditional sequences as a spreading model
and moreover every seminormalized weakly null sequence in a Banach
space $X$ contains a subsequence which either generates $\ell^1$
as a spreading model or it is norm Ces\`aro summable to zero. We
should also add that recent discoveries (c.f. \cite{G1},
\cite{GM}) have shown that similar regular structure is not
expected inside a generic Banach space.

It is clear from the definition of the spreading model that it
describes a kind of finite representability\footnote{A Banach
space $Y$ is finitely representable in $X$ if for every finite
dimensional subspace $F$ of $Y$ and every $\ee>0$ there exists
$T:F\to Y$ bounded linear injection such that
$\|T\|\cdot\|T^{-1}\|<1+\ee$.} of the space generated by the
sequence $(e_n)_{n}$ into the space $(X,\|\cdot\|)$. However there
exists a significant difference between the two concepts. Indeed
in the frame of the finite representability there are two
classical achievements: Dvoretsky's theorem (c.f. \cite{D})
asserting that $\ell^2$ is finitely representable in every Banach
space $X$ and also Krivine's theorem (c.f. \cite{Kr}) asserting
that for every linearly independent sequence $(x_n)_{n}$ in $X$
there exists a $1\leq p\leq \infty$ such that $\ell^p$ is block
finitely representable in $X$. On the other hand E. Odell and Th.
Schlumprecht (c.f. \cite{O-S}) have shown  there exists a
reflexive space $X$ admitting no $\ell^p$ as a spreading model .
Thus the spreading models of a space lie strictly between the
finitely representable spaces in $X$ and the spaces that are
isomorphic to a subspace of $X$.

The spreading models associated to a Banach space $X$ can be
considered as a cloud of Banach spaces, including many members
with regular structure, surrounding the space $X$ and offering
information concerning the local structure of $X$ in an asymptotic
manner. Our aim is to enlarge that cloud and to fill in the gap
between spreading models and the spaces which are finitely
representable in $X$. More precisely we extend the
Brunel-Sucheston concept of a spreading model and we show that
under the new definition the spreading models associated to a
Banach space $X$ form a whole hierarchy of classes of spaces
indexed by the countable ordinals. The first class of this
hierarchy is the classical spreading models.

\section{The extended notion of the spreading models} The new definition heavily depends
on the following two ingredients. The first one is the
$\ff$-\emph{sequences} $(x_s)_{s\in\ff}$ with $\ff$ a
\emph{regular thin} family of finite subsets of $\nn$. The
$\ff$-sequences will replace the usual sequences. The second one
is the \emph{plegma} families of finite subsets of $\nn$. This is
a new concept introduced here which is very crucial for our
approach due to their Ramsey properties. Let us also note that the
plegma families are invisible in the classical definition of
spreading models. Next we shall describe in detail the
aforementioned concepts and the definition of the spreading
models. In the sequel for an infinite subset $M$ of $\nn$ by
$[M]^{<\infty}$ (resp. $[M]^{\infty}$) we denote the set of all
finite (resp. infinite) subsets of $M$.

Thin families introduced by Nash-Williams in \cite{N-W} and were
studied by P. Pudlak and V. Rodl in \cite{Pd-R}. A detailed
presentation can be found in the second part of \cite{AT}. Here we
consider a special class of thin families called \emph{regular
thin} families (see Definition \ref{defn regular thin}). An
important feature of a thin family $\ff$ is the order of $\ff$,
denoted as $o(\ff)$ and which is defined to be the height of the
tree $\widehat{\ff}=\{t\in[\nn]^{<\infty}:\exists s\in\ff\text{
with }t\sqsubseteq s\}$ associated to the family $\ff$ (see
\cite{K}, \cite{Pd-R}). Typical examples of regular thin families
are the families of $k$-subsets of $\nn$, $\ff_k=[\nn]^k$ with
$o(\ff_k)=k$, the maximal elements of the Schreier family,
$\ff_\omega=\{s\subset\nn:\min s=|s|\}$ with
$o(\ff_\omega)=\omega$ and a generic one is the family
$\ff_{\omega^\xi}$ of the maximal elements of the $\xi$-Schreier
family $\mathcal{S}_\xi$ (c.f. \cite{AA}), with
$o(\ff_{\omega^\xi})=\omega^\xi$. We shall consider
$\ff$-sequences $(x_s)_{s\in\ff}$, in some set, as well as
$\ff$-\emph{subsequences} $(x_s)_{s\in\ff\upharpoonright L}$,
where $L\in[\nn]^\infty$ and $\ff\upharpoonright
L=\{s\in\ff:s\subset L\}$.

The plegma families are some special finite sequences of subsets
of $\nn$ (see Definition \ref{defn plegma}). Roughly speaking a
plegma family is a family $(s_1,\ldots,s_l)$ of pairwise disjoint
finite subsets of $\nn$ satisfying the following property. The
first elements of $(s_i)_{i=1}^l$ are in increasing order and they
lie before their second elements which are also in increasing
order and so on. A plegma family does not necessarily include sets
of equal size. For each $l\in\nn$, let $\text{\emph{Plm}}_l(\ff)$
be the set of all plegma families $(s_1,\ldots,s_l)$ with each
$s_i\in\ff$. The plegma families satisfy the following Ramsey
property which is fundamental for this work.
\begin{thm}\label{thm01}
Let $M$ be an infinite subset of $\nn$, $l\in\nn$  and $\ff$ be a
regular thin family. Then for every finite coloring of
$\text{\emph{Plm}}_l(\ff\upharpoonright M)$ there exists
$L\in[M]^\infty$
 such that $\text{\emph{Plm}}_l(\ff\upharpoonright L)$ is monochromatic.
\end{thm}
As in the case of the classical spreading models, an iterated use
of the above theorem,  yields that for every bounded
$\ff$-sequence $(x_s)_{s\in\ff}$ in a Banach space $X$  there
exist an infinite subset $M$ of $\nn$ and a seminorm $\|\cdot\|_*$
on $c_{00}(\nn)$ under which the natural Hamel basis $(e_n)_n$ is
a spreading sequence such that the following is satisfied: For
every $l\in\nn$, $a_1,\ldots,a_l\in\rr$ and every sequence
$((s_i^n)_{i=1}^l)_{n}$ in $\text{\emph{Plm}}_l(\ff\upharpoonright
M)$ with $\min s_1^n\to\infty$, we have
\[\Big\|\sum_{i=1}^l a_ie_i\Big\|_*=\lim_{n\to\infty}\Big\|\sum_{i=1}^l a_ix_{s^n_i}\Big\|\]
 The sequence $(e_n)_n$  will be  called an $\ff$-\emph{spreading model}
 of $X$ which is \emph{generated by} the $\ff$-subsequence
 $(x_s)_{s\in\ff\upharpoonright M}$.

 Let us point out that the $\ff$-sequences with $o(\ff)=1$ coincide
 with the usual sequences and also the corresponding plegma
 families are the finite subsets of $\nn$. Thus the above definition when
 $o(\ff)=1$ recovers
 the classical Brunel-Sucheston spreading models.

There is evidence supporting that the above definition of the
spreading models is the appropriate extension of the classical
one. The first one is that an $\ff$-spreading model $(e_n)_{n}$
depends \emph{only} on the order of $\ff$. More precisely the
following holds.
\begin{prop}\label{02}
  Let $X$ be a Banach space and $\ff,\g$ be regular thin families. If
$o(\ff)=o(\g)$ then
  $(e_n)_{n}$ is an $\ff$-spreading model of $X$ if and only if  $(e_n)_{n}$
  is a $\g$-spreading model of $X$. More generally, if $o(\ff)\leq o(\g)$ and  $(e_n)_{n}$ is an
$\ff$-spreading model of $X$ then $(e_n)_{n}$
   is a  $\g$-spreading model of $X$.
\end{prop}
The above allow us to classify the spreading models of a Banach
space $X$ in a transfinite hierarchy as follows.
\begin{defn}
  Let $X$ be a Banach space and $1\leq\xi<\omega_1$
  be a countable ordinal. We will say that $(e_n)_{n}$ is a $\xi$-spreading model  of $X$ if there
  exists a regular thin family $\ff$ with $o(\ff)=\xi$ such that
  $(e_n)_{n}$ is an $\ff$-spreading model of $X$. The set of all $\xi$-spreading models of $X$ will be denoted by
  $\mathcal{SM}_\xi(X)$.
\end{defn}
Notice that Proposition \ref{02} yields that the above defined
transfinite hierarchy of spreading models is an increasing one
i.e. for every Banach space $X$ and $1\leq\zeta<\xi<\omega_1$ we
have that $\mathcal{SM}_\zeta(X)\subseteq \mathcal{SM}_\xi(X)$. It
is an open problem whether this hierarchy is stabilized, i.e. does
for every separable Banach space $X$ exist a countable ordinal
$\xi$ such that for every $\zeta>\xi$, $\mathcal{SM}_\zeta(X)=
\mathcal{SM}_\xi(X)$?

This problem is open for certain classical spaces like
$L^p(\lambda)$, $1\leq p<\infty$. In the case of $\ell^p$, $1\leq
p<\infty$, the classes $\mathcal{SM}_\xi(\ell^p)$ are stabilized
for $\xi=1$. The case of $c_0$ is more interesting. Every
spreading model of order one of $c_0$ generates $c_0$, while every
Schauder basic spreading sequence is equivalent to a spreading
model of  order two  of $c_0$. This result indicates that the
higher order spreading models of a Banach space could be
considerably different from the classical ones. It remains open if
the class $\mathcal{SM}_\xi(c_0)$ is stabilized at some countable
ordinal $\xi$.

Let us point out that the $\xi$-spreading models of $X$ have a
weaker asymptotic relation to the space $X$ as $\xi$ increases to
$\omega_1$. A natural question arising from the above discussion
is whether the $\xi$-spreading models of $X$, $\xi<\omega_1$,
could recapture Krivine's theorem. As we will see this is not
always true.
\section{Other approaches}\label{intro other appraches section}
There are two other concepts in the literature sharing similar
features with the aforementioned extended definition. The first
one is due to L. Halbeisen and E. Odell appeared in \cite{H-O} and
concerns the so called asymptotic models which are associated to
$[\nn]^2$-bounded sequences in a Banach space. The asymptotic
models are not necessarily spreading sequences.

The second one is explicitly stated in \cite{O-S} although was
known to the experts of the Banach space theory. This concerns
what we call $k$-\emph{iterated} spreading models, that are
inductively  defined as follows.   First we need some notation
from \cite{O-S}. Let $X,E$ be Banach spaces. We write $X\to E$ if
$E$ has a Schauder basis which is a spreading model of some
seminormalized Schauder basic sequence in $X$ and
$X\stackrel{k}{\to}E$ if $X\to E_1\to\ldots\to E_{k-1}\to E$ for
some finite sequence $E_1,\ldots, E_{k-1}$. Note that for every
$k\in\nn$ if $X\stackrel{k}{\to}E$ then $E$ has a spreading
Schauder basis. A spreading Schauder basic sequence $(e_n)_{n}$ is
said to be a $k$-\textit{iterated} spreading model of a Banach space
$X$ if setting $E=\overline{<(e_n)_{n}>}$ then
$X\stackrel{k}{\to}E$. If in each inductive step we consider block
sequences instead of Schauder basic ones, we define in a similar
manner the \emph{block} $k$-iterated spreading models ($X\substack{k
\\\longrightarrow\\\text{bl}} E$). It is easy to see that the $k$-iterated spreading models define a countable hierarchy and a
problem posed by E. Odell and Th. Schlumprecht in \cite{O-S} is
whether there exists a Banach space $X$ such that no
$k$-iterated spreading model contains some $\ell^p$, $1\leq p<\infty$, or
$c_0$.

It is interesting that the class of $k$-spreading models includes
the most of the $k$-iterated ones as the following
describes.
\begin{prop} Let $X$ be a Banach space and $k\in\nn$. Then every
block $k$-iterated spreading model is also a $k$-spreading model.
Moreover if for every $1\leq l<k$, all $l$-iterated spreading models
are reflexive then all $k$-iterated spreading models are
$k$-spreading models.
\end{prop}
This proposition enable us to answer the aforementioned problem by
showing the following more general result. There exists a
reflexive Banach space $X$ such that no $\ell^p$, $1\leq
p<\infty$, or $c_0$ is embedded into $E=\overline{<(e_n)_{n}>}$,
where $(e_n)_{n}$ is a spreading model of any order of $X$.

Let us point out that the $k$-iterated spreading models of $c_0$
generate $c_0$. On the other hand, as we have mentioned before,
the class of spreading models of order two  of $c_0$ includes a
large variety of spreading sequences. Therefore the
$2$-iterated spreading models of $c_0$ is a proper subclass of
$\mathcal{SM}_2(c_2)$. Furthermore for every $k\geq2$, we shall
provide  an example of a reflexive space such that the
$l$-spreading models properly include the $l$-iterated spreading
models, for all $l\geq k$.

\section{An overview of the results}
The present work is organized into 14 chapters naturally divided
in two parts. The first part (Chapters 1-10) includes the
definition of the higher order spreading models and their theory.
In the second one we provide examples establishing the hierarchy
and also illustrating the boundaries of the theory. Next we
briefly present the content of the first part.
\subsection*{Chapter 1}
It is devoted to the regular thin and plegma families which as we
have mentioned possess a dominant role in our approach. The main
results of this chapter is the  aforementioned Theorem \ref{thm01}
concerning the Ramsey properties of the plegma families and the
following two theorems which further explain their behavior.

For a family $\ff$ of finite subsets of $\nn$, by
$\text{\emph{Plm}}(\ff)$ we denote  the set of all plegma families
in $\ff$, i.e. $\text{\emph{Plm}}(\ff)
=\cup_{l=1}^\infty\text{\emph{Plm}}_l(\ff)$.
\begin{thm} \label{04} Let $\ff,\g$ be regular thin families. If  $o(\ff)\leq o(\g)$
then for every $M\in[\nn]^\infty$ there exist $N\in[\nn]^\infty$
and a  map $\varphi:\g\upharpoonright N\to\ff\upharpoonright M$
such that for every $(s_i)_{i=1}^l\in
\text{\emph{Plm}}(\g\upharpoonright N)$, we have that
$(\varphi(s_i))_{i=1}^l\in \text{\emph{Plm}}(\ff\upharpoonright
M)$.
\end{thm}
A map $\varphi$ satisfying the above property is called a plegma
preserving map. We could say that the above result is the set
theoretic analogue of Proposition \ref{02} stated above and in
fact it is the basic ingredient of its proof. In the other
direction we prove the following which explicitly forbids the
existence of such maps from lower to higher order regular thin
families.
\begin{thm} \label{05} Let $\ff,\g$ be regular thin families.  If $o(\ff)<o(\g)$ then  for every
$\varphi:\ff\to\g$ and $M\in[\nn]^\infty$ there exists
$L\in[M]^\infty$ such that for every plegma pair $(s_1,s_2)$ in
$\ff\upharpoonright L$ neither $(\phi(s_1),\phi(s_2))$ nor
$(\phi(s_2),\phi(s_1))$ is a plegma pair.
\end{thm}
The proof of Theorem \ref{05} relies on the  \emph{plegma paths}
which are finite sequences $(s_i)_{i=1}^d$ such that
$(s_i,s_{i+1})$ is a plegma pair.  Let us point out that the
relation ``$(s_1,s_2)$ is a plegma pair" is neither symmetric nor
transitive. In the context of graph theory the above theorem
describes the following property. Viewing $\ff$ and $\g$ as graphs
with edges the plegma pairs in $\ff$ and $\g$ respectively, we
have that for every $\varphi:\ff\to\g$ and $M\in[\nn]^\infty$
there exists $L\in[M]^\infty$ such that $\varphi$ embeds the
associated to $\ff\upharpoonright L$ graph into the complement of
the corresponding one of $\g$.
\subsection*{Chapter 2}It includes the extended definition of the
spreading models and their classification into a transfinite
increasing hierarchy that we have already discussed. Additionally
we present a general method yielding $\xi$-spreading models which
goes as follows.

Let $1\leq \xi<\omega_1$ and $\ff$ be a regular thin family of
order $\xi$ and  $(e_n)_{n}$ be a spreading and $1$-unconditional
sequence in a Banach space $(E,\|\cdot\|)$.

We denote by $(e_s)_{s\in\ff}$ the natural Hamel basis of
$c_{00}(\ff)$. For $x\in c_{00}(\ff)$ we set
\[\|x\|_\ff=\sup\Big{\{}\Big\|\sum_{i=1}^lx(s_i)e_{s_i}\Big\|:l\in\nn,(s_i)_{i=1}^l\in\text{\emph{Plm}}_l(\ff)\text{ and }l\leq s_1(1)\Big{\}}\]
and let  $X_\ff=\overline{(c_{00}(\ff),\|\cdot\|_\ff)}$ be the
completion of $c_{00}(\ff)$ under the above norm.

It is easy to see that  $(e_n)_{n}$ is an $\ff$-spreading model of
$X_\ff$ generated by the $\ff$-sequence $(e_s)_{s\in\ff}$. Thus
setting $A=\{e_s:s\in\ff\}$, we have that $(e_n)_{n}$ belongs to
$\mathcal{SM}_\xi(A)$.

Moreover, if in addition  $(e_n)_{n}$ is  not equivalent to the
usual basis of $c_0$ then, using Theorem \ref{05}, one can verify
that for every $\zeta<\xi$ and every regular thin family $\g$ with
$o(\g)=\zeta$,  $(e_n)_{n}$ does not occur as a spreading model
generated by a $\g$-subsequence in $A$. However, it is not so
immediate how to construct a space $X$ admitting the sequence
$(e_n)_{n}$ as a $\xi$-spreading model such that for every
$\zeta<\xi$ the whole space $X$ does not admit a $\zeta$-spreading
model equivalent to $(e_n)_{n}$. We provide such examples in
Chapter 11.
\subsection*{Chapter 3}In this chapter we present a detailed study
of the spreading sequences. They are categorized as follows.

A spreading sequence $(e_n)_n$ in a seminormed space
$(E,\|\cdot\|_*)$ is called \emph{trivial} if the seminorm
$\|\cdot\|_*$ restricted on $Z=<(e_n)_n>$ is not a norm. It is
shown that in this case the subspace $N=\{x\in Z:\|x\|_*=0\}$ is
of codimension one in $Z$ and therefore $Z/_N\cong \rr$.

For the unconditional spreading sequences we present a well known
dichotomy. Namely an unconditional spreading sequence  either is
equivalent to the usual basis of $\ell^1$ or it is norm Cez\'aro
summable to zero.

A spreading sequence is called \emph{singular} if it is not
trivial and not Schauder basic. A singular spreading sequence
admits a decomposition as follows.
\begin{prop}\label{intro singular decomposition}
Let $(e_n)_{n}$ be a singular spreading sequence and let $E$ be
the Banach space generated by $(e_n)_{n}$. Then there is $e\in
E\setminus \{0\}$ such that $(e_n)_{n}$ is weakly convergent to
$e$. Moreover setting $e'_n=e_n-e$, we have that $(e'_n)_{n}$ is
nontrivial, spreading, $1$-unconditional and Ces\`aro summable to
zero.
\end{prop}

The last class consists of the Schauder basic spreading sequences
which are not unconditional. These are nontrivial weak-Cauchy
(i.e. they $w^*$-converge to an element of $X^{**}\setminus X$)
and also dominate the summing basis of $c_0$.

\subsection*{Chapter 4}We study $\ff$-sequences in topological spaces. We
start by giving the definition  of  the convergence of
$\ff$-subsequences which is the natural one. Namely, an
$\ff$-subsequence $(x_s)_{s\in\ff\upharpoonright M}$ in a
topological space $(X,\ttt)$ converges to some $x\in X$ if for
every $U\in\ttt$ with $x\in U$  there exists $m_0\in M$ such that
 $x_s\in U$, for every $s\in\ff\upharpoonright M$ with $\min s\geq m_0$.

The aim of Chapter 4 is to define and study the subordinated
$\ff$-subsequences.
\begin{defn}
    Let $(X,\ttt)$ be a topological space, $\ff$ be a regular thin family, $M\in[\nn]^\infty$ and
  $(x_s)_{s\in\ff}$ be an $\ff$-sequence in $X$. We say
    that $(x_s)_{s\in\ff\upharpoonright M}$ is \emph{subordinated} (with respect to $(X,\mathcal{T})$)
     if there exists a continuous map $\widehat{\varphi}:\widehat{\ff}\upharpoonright M\to (X,\mathcal{T})$
with  $\widehat{\varphi}(s)=x_s$, for all $s\in\ff\upharpoonright
M$.
\end{defn}
We recall that $\widehat{\ff}$ denotes the closure of $\ff$ with
respect to the initial segment ordering. The family
$\widehat{\ff}$ endowed with the order topology is a compact
metrizable space. It is easy to see that if an $\ff$-subsequence
$(x_s)_{s\in\ff\upharpoonright M}$ in a topological space
$(X,\ttt)$ is subordinated and
$\widehat{\varphi}:\widehat{\ff}\upharpoonright M\to X$ is the
continuous map witnessing that, then the $\ff$-subsequence
$(x_s)_{s\in\ff\upharpoonright M}$ converges to
$\widehat{\varphi}(\emptyset)$ and the closure of its range is a
compact metrizable space. It is worth pointing out that if
$o(\ff)\geq2$ and $(x_s)_{s\in\ff\upharpoonright M}$ converges to
some $x_0$ then the set $\{x_s:s\in\ff\upharpoonright M\}$ is not
necessarily relatively compact.

In the case of $\ff$-sequences in compact metrizable spaces the
following holds.
\begin{prop}\label{subsub}
   Every $\ff$-sequence in a compact metrizable space contains a
   subordinated $\ff$-subsequence.
  \end{prop}
  Among others this yields that in a compact metrizable space, every $\ff$-sequence
  has a convergent $\ff$-subsequence.

\subsection*{Chapter 5}We classify the spreading models into four types, namely, trivial, unconditional,
singular and Schauder basic according to where the corresponding
spreading sequence belongs. Further, we study their properties.
Some of the results of this chapter are new even for the classical
spreading model theory, where with few exceptions (see for example
\cite{Maurey} pp. 1310) the spreading models are generated by
Schauder basic sequences. The first result of this chapter
concerning trivial spreading models is the following.
\begin{thm}\label{into Theorem equivalent forms for having norm
on the spreading model}
  Let $X$ be a Banach space and $(x_s)_{s\in\ff\upharpoonright M}$
  be an $\ff$-subsequence in $X$ generating $(e_n)_n$ as a  spreading
  model. Then $(e_n)_n$ is trivial if and only if $(x_s)_{s\in\ff\upharpoonright
  M}$ contains a norm convergent  $\ff$-subsequence.
\end{thm}

Concerning the unconditional spreading models we have the
following, which generalizes a well known fact for classical
spreading models. Namely, seminormalized weakly null sequences
generate unconditional spreading models. In the sequel whenever we
refer to subordinated $\ff$-subsequences in a Banach space $X$, we
will always mean that they are subordinated with respect to the
weak topology of $X$.
\begin{thm}
  Every spreading model $(e_n)_{n\in\nn}$ generated by a
  seminormalized, subordinated and weakly null $\ff$-subsequence
  $(x_s)_{s\in\ff\upharpoonright L}$ is unconditional.
\end{thm}
Let us point out that spreading models generated by seminormalized
and weakly null $\ff$-subsequences are not necessary
unconditional. We present an example of a spreading model of order
two of $c_0$ yielding the above. Thus, the additional assumption
that the $\ff$-subsequence is  subordinated is necessary.

Regarding $\ff$-subsequences generating a singular spreading
model, it is shown that they admit a decomposition similar to the
one described by Proposition \ref{intro singular decomposition} as
follows.
\begin{thm}
  Let $(e_n)_{n\in\nn}$ be a singular spreading model generated by
  the $\ff$-subsequence $(x_s)_{s\in\ff\upharpoonright M}$ and
  $e_n=e'_n+e$ the decomposition corresponding to the singular
  spreading sequence $(e_n)_{n\in\nn}$.
  Then there
  exist $x\in X$ and $L\in[M]^\infty$ such that $\|x\|=\|e\|$ and setting
  $x'_s=x_s-x$, for all $s\in\ff\upharpoonright L$, we have that
  $(x'_s)_{s\in\ff\upharpoonright L}$ generates $(e'_n)_{n\in\nn}$ as
  a unique $\ff$-spreading model.
\end{thm}
As follows from Proposition \ref{intro singular decomposition} the
$\ff$-subsequence $(x'_s)_{s\in\ff\upharpoonright L}$ generates an
unconditional and Ces\'aro summable spreading model.

Finally we provide a sufficient condition for $\ff$-sequences to
admit Schauder basic spreading model. This is related to the
notion of Skipped Schauder Decomposition  that we next define.
\begin{defn}
    Let $A$ be a countable seminormalized subset of a Banach space $X$.
    We say that $A$ admits a Skipped Schauder Decomposition (SSD) if there exist $K>0$ and
    a sequence $(F_n)_{n\in\nn}$ of finite subsets of $A$ such that
    \begin{enumerate}
      \item[(i)] $\cup_{n\in\nn}F_n=A$
      \item[(ii)] For every $L\in[\nn]^\infty$ not containing two successive integers and
      every sequence  $(x_l)_{l\in L}$ with  $x_l\in F_l$, for all $l\in L$, $(x_l)_{l\in L}$
      is a
      Schauder basic sequence of constant $K$.
    \end{enumerate}
  \end{defn}

  \begin{thm} Let $X$ be a Banach space and
 $(x_s)_{s\in\ff}$ be an $\ff$-sequence in $X$.  If $\{x_s: s\in\ff\}$ admits a SSD then every nontrivial  $\ff$-spreading model generated by an
$\ff$-subsequence of  $(x_s)_{s\in\ff}$ is Schauder basic.
\end{thm}

  \subsection*{Chapter 6}As is well known if $(e_n)_n$ is a spreading model generated by a weakly convergent
  sequence $(x_n)_n$  in a
  Banach space $X$ with a Schauder basis then $(e_n)_n$ is also generated by a sequence of the form $\widetilde{x}_n=x+x'_n$
  where $x$ is the weak-limit of $(x_n)_n$ and $(x'_n)_n$ is a block sequence in $X$ weakly convergent to zero.
One of the main goals of Chapter 6 is to establish analogous
phenomena for weakly relatively compact $\ff$-sequences. An
$\ff$-sequence $(x_s)_{s\in\ff}$ in a Banach space $X$ is said to
be weakly relatively compact if the set
$\overline{\{x_s:s\in\ff\}}^w$ is weakly compact. Proposition
\ref{subsub} yields the first basic property of  such sequences
namely  that every weakly relatively compact $\ff$-sequence has a
subordinated $\ff$-subsequence. The following definition is a
necessary ingredient for the results of this chapter.
\begin{defn}
  Let  $X$ be a Banach space with a Schauder basis. Let $\ff$ be a regular thin family, $M\in[\nn]^\infty$ and
  $(x_s)_{s\in\ff}$ an $\ff$-sequence in $X$ consisting of finitely supported
  vectors. We will say that the $\ff$-subsequence $(x_s)_{s\in\ff\upharpoonright
  M}$ is \emph{plegma disjointly supported} (resp. \emph{plegma block}) if for
  every plegma pair $(s_1,s_2)$ in $\ff\upharpoonright M$ we have
  that $\text{supp}(x_{s_1})\cap\text{supp}(x_{s_2})=\emptyset$ (resp.
  $\text{supp}(x_{s_1})<\text{supp}(x_{s_2})$).
\end{defn}
Let us recall that when $o(\ff)=1$ the plegma families are the
finite subsets of $\nn$. Hence for classical sequences the above
definition concerns disjointly supported (resp. block) sequences.
Furthermore when $o(\ff)=1$, the above two concepts (i.e. plegma
disjointly supported and plegma block) coincide by passing, if
necessary, to a further subsequence. Under the above terminology
we have the following.
\begin{thm}
 Let $X$ be a Banach space with a Schauder basis
 and $(e_n)_{n}$ be an $\ff$-spreading model generated by a
 subordinated $\ff$-subsequence $(x_s)_{s\in\ff\upharpoonright M}$ in $X$ weakly convergent to $x$.
Then $(e_n)_n$  is also generated by an $\ff$-subsequence
$(\widetilde{x}_s)_{s\in\ff\upharpoonright L}$ of the form
$\widetilde{x}_s=x+x'_s$ where $(x'_s)_{s\in\ff\upharpoonright L}$
is subordinated,  weakly null and plegma disjointly supported.
\end{thm}
This theorem generalizes the result regarding classical sequences
mentioned at the beginning of our discussion for Chapter 6. We
should notice that in general we could not expect that the
$\ff$-subsequence $(x'_s)_{s\in\ff\upharpoonright L}$ is either
fully disjointly supported or plegma block. As we will explain
next if the Banach space $X$ admits $c_0$ as a $\xi$-spreading
model then $c_0$ is also generated by a plegma block
$\ff$-sequence. A similar result also holds for $\ell^1$ under
some additional assumptions.

\subsection*{Chapter 7}We deal with the natural problem, posed to us
by  Schlumprecht, of determining the spreading models of the
classical sequence spaces. As we have already mentioned the
spreading models of $\ell^p$, $1\leq p<\infty$, are as expected.
More precisely the following holds.
\begin{thm}
(i) Let $1<p<\infty$ and $(e_n)_n$ be a  nontrivial spreading model
 of order $\xi$ of $\ell^p$, for some $\xi<\omega_1$. Then the
 following are satisfied.
\begin{enumerate} \item[(a)]  If $(e_n)_n$  is normalized
Schauder basic then it is isometric to the usual basis of
$\ell^p$.\item[(b)] If $(e_n)_n$  is singular then it generates
an isometric to $\ell^p$ space.
\end{enumerate}
\item[(ii)] Every nontrivial spreading model of any order of
$\ell^1$ is Schauder basic and equivalent to the usual basis of
$\ell^1$. \item[(iii)] For all $1\leq p<\infty$ we have that for
every $(e_n)_{n\in\nn}\in\mathcal{SM}(\ell^p)$ there exists a
sequence $(u_n)_{n\in\nn}$ in $\ell^p$ isometric to
$(e_n)_{n\in\nn}$.
\end{thm}
As it is shown in Chapter 9, the Tsirelson spaces $T_\alpha$,
$1\leq \alpha <\omega_1$, also admit $\ell^1$ as  a unique up to
equivalence spreading model.

The spreading models of $c_0$ are described by the following.
\begin{thm} (i) Any bimonotone Schauder basic spreading sequence belongs to
the class of spreading models  of $c_0$ of order two.\\
(ii) Every singular spreading model of $c_0$ of any order is
contained in the class of  spreading models  of $c_0$ of order two.\\
(iii) For every nontrivial spreading sequence $(e_n)_n$ there
exists a spreading model $(e'_n)_n$ of $c_0$ of  order two which
is equivalent to $(e_n)_n$.
\end{thm}
Part (iii) of the theorem shows that $\mathcal{SM}_2(c_0)$ is
isomorphically universal for all spreading sequences, which yields
that the hierarchy $(\mathcal{SM}_\xi(c_0))_{\xi<\omega_1}$ is
isomorphically stabilized for $\xi=2$. It is not clear if it is
also isometrically stabilized for $\xi=2$.

  \subsection*{Chapter 8}We present   some composition properties of
spreading models. More precisely we have the following result.
\begin{thm}
  Let  $X$ be a Banach space and $(e_n)_{n}$ be a Schauder basic sequence in $\mathcal{SM}_\xi(X)$, for some  $\xi<\omega_1$. Let
  $E$ be the space generated by $(e_n)_n$  and for some  $k\in\nn$, let $(\overline{e}_n)_{n}\in\mathcal{SM}_k(E)$ be
  a plegma block generated spreading model of $E$.
  Then $$(\overline{e}_n)_{n}\in\mathcal{SM}_{\xi+k}(X)$$
\end{thm}
The above yields the following concerning $\ell^p$ spreading models.
\begin{prop}
  Let $X$ be a Banach space, $(e_n)_{n}$ be a $\xi$-spreading model of $X$, for some $\xi<\omega_1$ and let
   $E$ be the space generated by $(e_n)_n$.
   If for some $1<
  p<\infty$, $E$ contains an isomorphic copy of
  $\ell^p$ then $X$ admits a $(\xi+1)$-spreading model
  equivalent to the usual basis of $\ell^p$.
\end{prop}
Using additionally the non distortion property of $\ell^1$ and $c_0$
(c.f. \cite{J2}) we obtain the following stronger result.
\begin{prop}
  Let $X$ be a Banach space, $(e_n)_{n}$ be a $\xi$-spreading model of $X$, for some
  $\xi<\omega_1$ and let
  $E$ be the space generated by $(e_n)_n$. If  $E$ contains an isomorphic copy of
  $\ell^1$ (resp. $c_0$) then $X$ admits the usual basis of  $\ell^1$ (resp. $c_0$) as a $(\xi+1)$-spreading model.
\end{prop}
We also have the following trichotomy.
\begin{thm}\label{021}
     Let $X$ be a reflexive Banach space and $\xi<\omega_1$. Then one of the following
     holds.
     \begin{enumerate}
       \item[(i)] The space $X$ admits the usual basis of $\ell^1$
       as a $(\xi+1)$-spreading model.
       \item[(ii)] The space $X$ admits the usual basis of  $c_0$
       as a $(\xi+1)$-spreading model.
       \item[(iii)] All $\xi$-spreading models of $X$ generate reflexive spaces.
     \end{enumerate}
     Moreover, every Schauder basic spreading model of $X$ is
     unconditional.
\end{thm}

\subsection*{Chapter 9} We study spreading models equivalent to the usual basis
of $\ell^1$. Among others we provide  sufficient conditions
ensuring that a Banach space admits plegma block generated
$\ell^1$ spreading models. Let us say that a  Banach space $X$
with a Schauder basis satisfies the property $\mathcal{P}$ if for
every $\delta>0$ there exists a $k\in\nn$ such that for every
finite block sequence $(x_i)_{i=1}^k$ in $X$ with
$\|x_i\|\geq\delta$ for all $1\leq i\leq k$ we have that
$\|\sum_{i=1}^k x_i\|>1$.
\begin{thm}\label{015}
  Let $X$ be a Banach space with a Schauder basis  satisfying   the
  property $\mathcal{P}$.  If $X$ admits $\ell^1$ as a spreading model
generated by a weakly relatively compact $\ff$-subsequence then
$X$ admits $\ell^1$ as a plegma block generated spreading model.
\end{thm}
The above theorem is a key ingredient for showing the existence of
a reflexive space admitting no $\ell_p$ as a spreading model (see
Chapter 14). Also, as we will  see in Chapter 13, the additional
assumption concerning property $\mathcal{P}$ is a necessary one
for the conclusion of the above theorem.

The next result concerns Ces\`aro summability for
$[\nn^k]$-sequences. We first define the $k$-Ces\`aro summability.
\begin{defn}
Let $X$ be a Banach space, $x_0\in X$, $k\in\nn$, $(x_s)_{s\in
[\nn]^k}$ be a $[\nn]^k$-sequence in $X$ and $M\in [\nn]^\infty$. We
will say that the $[\nn]^k$-subsequence $(x_s)_{s\in [M]^k}$ is
$k$-Ces\`aro summable to $x_0$ if
\[ \Big(\substack{n\\ \\k}\Big)^{-1} \sum_{s\in [M|n]^k} x_s \;\;\substack{\|\cdot\| \\ \longrightarrow \\ n\to\infty}\;\; x_0\]
where $M|n=\{M(1),...,M(n)\}$.
\end{defn}
We prove the following extension of a well known result of H. P.
Rosenthal which corresponds to the case $k=1$.
\begin{thm}
Let $X$ be a Banach space, $k\in\nn$ and  $(x_s)_{s\in[\nn]^k}$ be a
weakly relatively compact
    $[\nn]^k$-sequence in $X$.
   Then there exists $M\in [\nn]^\infty$ such that
    at least one of the following holds:
\begin{enumerate}
\item[(i)] The subsequence  $(x_s)_{s\in[M]^k}$ generates an
$[\nn]^k$-spreading model equivalent to the standard basis of
$\ell^1$. \item[(ii)] There exists $x_0\in X$ such that for
every $L\in [M]^\infty$ the subsequence $(x_s)_{s\in [L]^k}$ is
$k$-Ces\`aro summable to $x_0$.
\end{enumerate}
\end{thm}
There are significant differences between the cases $k=1$ and
$k\geq2$. First for $k=1$ the two alternatives are exclusive but
this does not remain valid for $k\geq 2$. Second the proof for the
case $k\geq 2$ uses the following density result concerning plegma
families which is a consequence of the multidimensional
Szemeredi's theorem due to H. Furstenberg and Y. Katznelson (c.f.
\cite{FK}).
\begin{lem}
    Let  $\delta>0$ and $k, l\in\nn$. Then  there exists  $n_0\in \nn$ such that
  for every $n\geq n_0$ and   every subset $A$ of the set of all $k$-subsets of $\{1,\ldots, n\}$
  of size at least $\delta
    (\substack{n\\ k})$, there exists  a plegma $l$-tuple $(s_j)_{j=1}^l$ in $ A$.
  \end{lem}

\subsection*{Chapter 10}In this chapter we focus on $c_0$-spreading
models. We start by proving a combinatorial result concerning
partial unconditionality of tree basic sequences. Using this we
obtain the following domination property of spreading models.
\begin{thm}
  Let $\ff,\g$ be regular thin families and $N\in[\nn]^\infty$ such
  that every $t\in\g\upharpoonright N$ is an initial segment
  of some $s\in\ff$. Let
  $(x_s)_{s\in\ff}$ be a bounded $\ff$-sequence in a Banach space
  $X$ such that the $\ff$-subsequence
  $(x_s)_{s\in\ff\upharpoonright N}$ is subordinated and let
  $\widehat{\varphi}:\widehat{F}\upharpoonright N\to (X,w)$ be the
  continuous map witnessing this. For every $v\in
  \g\upharpoonright N$, let $z_v=\widehat{\varphi}(v)$. Suppose that
  $(x_s)_{s\in\ff\upharpoonright N}$ and $(z_v)_{v\in\g\upharpoonright
  N}$ generate $(e^1_n)_{n}$ and
  $(e^2_n)_{n}$ as spreading models respectively. Then
  for every $k\in\nn$ and $a_1,\ldots,a_k\in\rr$ we have that
  \[\Big\|\sum_{j=1}^ka_je_j^2\Big\|\leq\Big\|\sum_{j=1}^ka_je_j^1\Big\|\]
\end{thm}
The above is the key ingredient for proving the next.
\begin{thm}\label{013}
  Let $X$ be a Banach space with a Schauder basis. If $X$ admits $c_0$ as a spreading
model generated by a weakly relatively compact $\ff$-subsequence
then
  $X$ also admits $c_0$ as a plegma block generated
spreading model.
\end{thm}
Theorem \ref{013} permits us to extend the duality property of
$c_0$ and $\ell^1$ spreading model which is well known in the
classical case.
\begin{cor} Let  $X$ be a Banach space
with a Schauder basis. If $X$ admits $c_0$ as a spreading model
generated by a weakly relatively compact $\ff$-subsequence then
$X^*$ admits $\ell^1$ as a  plegma block generated spreading
model.
\end{cor}
It is well known that the above duality result does not hold in
the inverse direction. Namely there are reflexive spaces admitting
$\ell^1$ as a classical spreading model and their duals do not
admit $c_0$ as a spreading model. We also present  a similar
example for higher order spreading models.

\section{An overview of the examples}
The remained chapters (11-14) are devoted to several examples
which answer natural questions raised from the definition and the
results presented above.

\subsection*{Chapter 11}
A natural problem is to establish the hierarchy
$(\mathcal{SM}_\xi(X))_{\xi<\omega_1}$, namely the existence, for
arbitrarily large $\xi<\omega_1$, of Banach spaces $X_\xi$ such
that $\mathcal{SM}_\xi(X_\xi)$ properly includes
$\cup_{\zeta<\xi}\mathcal{SM}_\zeta(X_\xi)$. As we have mentioned
the space $c_0$ has this property for $\xi=2$. In this chapter we
provide more examples of reflexive spaces described by the
following two theorems.
\begin{thm}
  For every $k\in\nn$ there exists a reflexive space
  $\mathfrak{X}_{k+1}$ with an unconditional basis $(e_s)_{s\in[\nn]^{k+1}}$ satisfying the following properties. The
  basis $(e_s)_{s\in[\nn]^{k+1}}$
  generates $\ell^1$ as a $(k+1)$-spreading model and is
not $(k+1)$-Ces\`aro summable to any $x_0$ in
$\mathfrak{X}_{k+1}$. Furthermore the space $\mathfrak{X}_{k+1}$
does not admit an $\ell^1$ spreading model of order $k$.
\end{thm}
 The space $\mathfrak{X}_{k+1}$
shows that the non $(k+1)$-Ces\`aro summability of a
$[\nn]^{k+1}$-sequence does not yield any further information
concerning $\ell^1$ spreading models of lower order. For
arbitrarily large countable ordinals we have the following.
\begin{thm}\label{intro the hierarch does not colaps}
  For every countable ordinal $\xi$ there exists a reflexive space
$\mathfrak{X}_\xi$ with an unconditional basis satisfying the
following properties:
\begin{enumerate}
  \item[(i)] The space $\mathfrak{X}_\xi$ admits $\ell^1$ as a $\xi$-spreading model.
  \item[(ii)] For every ordinal $\zeta$ such that $\zeta+2<\xi$,
  the space $\mathfrak{X}_\xi$ does not admit $\ell^1$ as a $\zeta$-spreading model.
\end{enumerate}
\end{thm}
In particular, if $\xi$ is a limit countable ordinal, then the
space $\mathfrak{X}_\xi$ does not admit $\ell^1$ as a
$\zeta$-spreading model for every $\zeta<\xi$.

\subsection*{Chapter 12}The aim of the next example is to separate for $k>1$ the
class of the $k$-iterated spreading models from corresponding of the $k$-spreading
models.
\begin{thm}
  For every $1<p<q<\infty$ and $k>1$, there exists a reflexive Banach space $X$
  with an unconditional basis such that $X$ admits $\ell^q$ as a spreading model of order
  $k$ while for every $l\in\nn$, every $l$-iterated spreading model of $X$ is
  equivalent either to the usual basis of $\ell^1$ or $\ell^p$.
\end{thm}

\subsection*{Chapter 13}The next example concerns plegma block generated $\ell^1$
spreading models. Among others it shows that property
$\mathcal{P}$ appeared in Theorem \ref{015} is indeed necessary.
\begin{thm}
  There exists a reflexive space $X$ with an unconditional basis
  admitting $\ell^1$ as $\omega$-spreading model and not
  admitting $\ell^1$ as a plegma block generated spreading model of any
  order.
\end{thm}

\subsection*{Chapter 14}The last example shows that the hierarchy of spreading models
introduced in this work does not provide a Krivine's type result
(c.f. \cite{Kr}). More precisely we have the following
\begin{thm}
  There exists a reflexive space $X$ with an unconditional basis such
  that for every $\xi<\omega_1$  and every $(e_n)_{n}\in\mathcal{SM}_\xi(X)$, the space $E=\overline{<(e_n)_{n}>}$
  does not contain any isomorphic copy of  $c_0$
  or $\ell^p$, for all $1\leq p<\infty$.
\end{thm}
The latter answers in the affirmative the aforementioned problem
posed in \cite{O-S}. Moreover it is an interesting question if the
hierarchy of the spreading models of the space $X$ is stabilized
at some $\xi<\omega_1$. Furthermore by Corollary \ref{021} we have
that every spreading model of $X$ generates a reflexive space not
containing $\ell^p$, for all $1< p<\infty$. It is open whether all
these spaces are related to reflexive spaces generated by
saturation methods  like Tsirelson,  mixed Tsirelson and their
variants.
\section{Comments on the plegma families} In the last part of the
introduction we will discuss the significant role of the plegma
families which have a strong presence in the entire work. The
plegma families have three basic properties mentioned before. The
first one is the Ramsey property (Theorem \ref{thm01}) which is
the fundamental ingredient yielding the existence of higher order
spreading models. The second and the third ones concern plegma
preserving maps between regular thin families. We recall (Theorem
\ref{04}) that for $\ff$, $\g$ regular thin families with
$o(\ff)\leq o(\g)$ there exist $N,M\in[\nn]^\infty$ and a plegma
preserving map $\varphi:\g\upharpoonright N\to \ff\upharpoonright
M$. This result has two fundamental consequences, namely it
permits to pass from $\ff$-spreading models to $\xi$-spreading
models and furthermore makes the hierarchy
$(\mathcal{SM}_\xi(X))_{\xi<\omega_1}$ an increasing one. The
third one (Theorem \ref{05}) explains that the previous property
never occurs whenever $o(\ff)>o(\g)$. The latter and more
precisely the methods developed for its proof are the key tools in
showing that the hierarchy $(\mathcal{SM}_\xi(X))_{\xi<\omega_1}$
does not collapse (Theorem \ref{intro the hierarch does not
colaps}). We should also add that the above properties of the
plegma families also indicate their geometric nature. Indeed, the
plegma preserving maps distinguish the regular thin families
according to their order and they reveal a transfinite dimension
of them. It also seems that the plegma families, which are a pure
combinatorial concept, is a new one. As S. Todorcevic pointed out
to us, E. Specker in \cite{Sp} had used for pairs in $[\nn]^3$ a
concept which shares some common features with the plegma pairs.

 From Banach
space point of view the plegma families are also well hidden.
Indeed, in the definition of the classical spreading models they
are not visible as in this case they  coincide to the finite
subsets of $\nn$. The concept of the plegma families in $[\nn]^k$
emerged in our attempt to solve Odell-Schlumprecht's problem
concerning $k$-iterated spreading models (see Section \ref{intro
other appraches section} above). Using induction we proved that
spaces like Odell-Schlumprecht's  original one answer
affirmatively their problem. Namely for $k\in\nn$ every
$k$-iterated spreading model does not contain isomorphs of $\ell^p$, $1\leq
p<\infty$, or $c_0$. Illustrating the proof we understood that
every $k$-iterated spreading model, which by its own definition for
$k\geq2$ is not directly connected to the space, is actually
generated by a family $(x_s)_{s\in[\nn]^k}$ as is described by our
definition of the $k$-spreading models. This key observation was
the basis for the general definitions of $\ff$-sequences, plegma
families and higher order spreading models.

We expect that the plegma families and the related concepts to
have further applications to combinatorics and Banach space
theory.

\section{Preliminary notation and definitions}
We start with some notation related to the subsets of $\nn$. As
usual, we denote the set of the natural numbers by
$\nn=\{1,2,...\}$. Throughout the paper we shall identify strictly
increasing sequences in $\nn$ with their corresponding range i.e.
we view every strictly increasing sequence in $\nn$ as a subset of
$\nn$ and conversely every subset of $\nn$ as the sequence
resulting from the increasing ordering of its elements. We will
use capital letters as $L,M,N,...$ to denote infinite subsets  and
lower case letters as $s,t,u,...$ to denote finite subsets of
$\nn$.

For every infinite subset $L$ of $\nn$, $[L]^{<\infty}$ (resp.
$[L]^\infty$) stands for the set of all finite (resp. infinite)
subsets of $L$. For an $L=\{l_1<l_2<...\}\in [\nn]^\infty$ and a
positive integer $k\in\nn$, we set $L(k)=l_k$. Similarly, for a
finite subset $s=\{n_1<..<n_m\}$ of $\nn$ and  for $1\leq k\leq m$
we set $s(k)=n_k$. Also for every nonempty $s\in[\nn]^{<\infty}$
and $1\leq k\leq |s|$ we set $s|k=\{s(1),\ldots,s(k)\}$ and
$s|0=\emptyset$. Moreover, for $s,t\in[\nn]^{<\infty}$, we write
$t\sqsubseteq s$ (resp. $t\sqsubset s$) to denote that $t$ is an
initial (resp. proper initial) segment of $s$.

For an $L=\{l_1<l_2<...\}\in [N]^\infty$ and a finite subset
$s=\{n_1<..<n_k\}$ (resp. for an infinite subset
$N=\{n_1<n_2<...\}$ of $\nn$), we set
$L(s)=\{l_{n_1},...,l_{n_k}\}$ (resp.
$L(N)=\{l_{n_1},l_{n_2},...\}$).

For $s\in[\nn]^{<\infty}$ by $|s|$ we denote the cardinality of
$s$. For $L\in[\nn]^\infty$ and $k\in\nn$, we denote by $[L]^k$
the set of all $s\in[L]^{<\infty}$ with $|s|=k$. For every
$s,t\in[\nn]^{<\infty}$ we write $t<s$ if either at least one of
them is the empty set, or $\max t<\min s$.

We also recall some standard notation and definitions from Banach
space theory. Although the notation that we follow is the standard
one, as it can be found in textbooks like \cite{AK},
\cite{Lid-Tza}, we present for the sake of completeness some basic
concepts that are involved in this monograph.

By the term Banach space we shall always mean an infinite
dimensional one. Let $X$ be  a  Banach space. When we say that $Z$
is a subspace of $X$ we mean that $Z$ is a closed infinite
dimensional subspace of $X$. For a subspace $Z$ of $X$, $B_Z$
(resp. $S_Z$) stands for the unit ball $Z$, i.e. the set $\{x\in
Z:\|z\|\leq1\}$ (resp. the unit sphere of $Z$, i.e. the set
$\{x\in Z:\|x\|=1\}$). For a bounded linear operator $T:X\to Y$,
where $Y$ is a Banach space, we will say that $T$ is strictly
singular if there exists no subspace $Z$ of $X$ such that the
restriction $T|_Z$ of $T$ on $Z$ is an isomorphic embedding.

Let $(x_n)_{n}$ be a sequence in  $X$. We say that $(x_n)_{n}$ is
bounded (resp. seminormalized) if there exists $M>0$ (resp.
$C,c>0)$ such that $\|x_n\|\leq M$ (resp. $c\leq \|x_n\|\leq C$)
for all $n\in\nn$. We say that $(x_n)_{n}$ is normalized if
$\|x_n\|=1$ for all $n\in\nn$.

Two sequences $(x_n)_{n}$ and $(y_n)_{n}$, not necessarily
belonging to the same Banach space, are called equivalent if there
exist $c,C>0$ such that for every $n\in\nn$ and
$a_1,\ldots,a_n\in\rr$
\[c\Big\|\sum_{j=1}^na_jy_j\Big\|\leq\Big\|\sum_{j=1}^na_jx_j\Big\|\leq C\Big\|\sum_{j=1}^na_jy_j\Big\|\]

A sequence $(e_n)_{n}$ is a Schauder basis of  $X$ if for every
$x\in X$ there exists a unique sequence $(a_n)_{n}$ of reals such
that $x=\sum_{n=1}^\infty a_nx_n$. As  is well known the
associated projections $(P_n)_{n}$ to $(e_n)_{n}$, defined by
$P_n(\sum_{j=1}^\infty a_je_j)=\sum_{j=1}^na_je_j$  for every
$x=\sum_{j=1}^\infty a_je_j\in X$ and $n\in\nn$, are uniformly
bounded and the quantity $\sup_{n}\|P_n\|$ is  the basis constant
of $(e_n)_{n}$. A sequence $(x_n)_{n}$ in $X$ is called (Schauder)
basic if $(x_n)_{n}$ forms a Schauder basis for the subspace
$\overline{<(x_n)_{n}>}$ of $X$. A sequence $(x_n)_{n}$ in $X$ is
called $C$-unconditional, where $C>0$, if for every
$F\subseteq\nn$ and every $(a_n)_{n}$ sequence of reals such that
the series $\sum_{n=1}^\infty a_nx_n$ converges we have that
\[\Big\|\sum_{n\in F} a_nx_n\Big\|\leq C\Big\|\sum_{n=1}^\infty
a_nx_n\Big\|\] For a sequence $(x_n)_{n}$ in $X$ we say that a
sequence $(y_n)_{n}$ is a block subsequence of $(x_n)_{n}$ if
there exist a strictly increasing sequence $(p_n)_{n}$ of natural
numbers and a sequence $(a_n)_{n}$ of reals such that for every
$n\in\nn$ we have that
\[y_n=\sum_{k=p_n}^{p_{n+1}-1}a_kx_k\]

Suppose that  $X$ has a Schauder basis $(e_n)_{n}$. For every
$x=\sum_{n=1}^\infty a_n e_n\in X$ we define the support of $x$ to
be the set $\text{supp} (x)=\{n\in\nn:a_n\neq0\}$ and we say that
$x$ is finitely supported if $\text{supp}(x)$ is finite. For
$x_1,x_2\in X$ finitely supported we write $x_1<x_2$ if
$\text{supp}(x_1)<\text{supp}(x_2)$. The biorthogonal functionals
$(e_n^*)_{n}$ of $(e_n)_{n}$ are  defined  by $e_n^*(x)=a_n$ for
every $x=\sum_{i=1}^\infty a_ie_i\in X$ and $n\in\nn$.  We recall
that the basis $(e_n)_{n}$ is called shrinking if
$X^*=\overline{<(e_n^*)_{n}>}^{\|\cdot\|}$. The basis $(e_n)_{n}$
is called boundedly complete if for every sequence $(a_n)_{n}$ of
reals such that the sequence $(\sum_{k=1}^na_ke_k)_{n}$ is
bounded, the series $\sum_{n=1}^\infty a_ne_n$ converges.

\chapter{Regular thin and plegma families}
This chapter is devoted to the definitions and the initial study
of the regular thin and plegma families of finite subsets of
$\nn$. As we have mentioned in the introduction, these concepts
play a critical role in the definition of the higher order
spreading models. The results included in the present chapter are
of combinatorial nature and could be of independent interest.
\section{Regular thin families of finite subsets of $\nn$}
In this section we define the regular thin families and we recall
some standard facts concerning families of finite subsets of
$\nn$. The definition of regular thin families is based on two
well known concepts, namely the regular families, traced back to
\cite{AA}, and the thin families, defined in \cite{N-W} and
extensively studied in \cite{Pd-R}. The importance of regular thin
families arises from the fact that they posses strong Ramsey
properties, in particular concerning plegma families which are
later in this chapter.
\subsection{Basic definitions}
Recall that a family $\ff$ of finite subsets of $\nn$ is said to
be \textit{hereditary} if for every $s\in\ff$ and $t\subseteq s$
we have that  $t\in\ff$
  and  \textit{spreading} if for every $n_1<...<n_k$ and  $m_1<...<m_k$ with  $\{n_1,...,n_k\}\in\ff$
  and $n_1\leq m_1$, ..., $n_k\leq m_k$ we have that  $\{m_1,...,m_k\}$ belongs to $ \ff$ too.
   Also $\ff$ is called  \textit{compact} if the set of characteristic functions of the members of $\ff$,
   $\{\chi_s\in\{0,1\}^\nn:\;s\in \ff\}$, is a closed subspace of $\{0,1\}^\nn$.
   A family $\ff$ of finite subsets of $\nn$ will be called \textit{regular} if it is compact, hereditary and spreading.
   For a family
$\ff\subseteq [\nn]^{<\infty}$ of finite subsets of $\nn$ and an
infinite subset $L\in [\nn]^\infty$ we set \[\ff\upharpoonright
L=\{s\in \ff:\;s\subseteq L\}=\ff\cap [L]^{<\infty}\]
  The \textit{order}
  of a family $\ff\subseteq [\nn]^{<\infty}$ is defined as follows (see also \cite{Pd-R}).
  We assign to $\ff$ its ($\sqsubseteq-$)\textit{closure}
  \[\widehat{\ff}=\{t\in[\nn]^{<\infty}:\exists  s\in\ff\text{ with }t\sqsubseteq s\}\]
  which under the initial segment ordering is a tree.
  If $\widehat{\ff}$ is ill-founded
  (i.e. there exists an infinite sequence $(s_n)_{n\in\nn}$ in $\widehat{\ff}$ such that $s_n\sqsubset s_{n+1}$)
   then we set $o(\ff)=\omega_1$.
  Otherwise, for every maximal element $s$ of $\widehat{\ff}$ we set
  $o_{\widehat{\ff}}(s)=0$ and recursively for every $s$ in $\widehat{\ff}$ we define
  \[o_{\widehat{\ff}}(s)=\sup\{o_{\widehat{\ff}}(t)+1:t\in\widehat{\ff}\text{ and }s\sqsubset t\}\]
The order of $\ff$ denoted by $o(\ff)$ is defined to be the
ordinal $o_{\widehat{\ff}}(\emptyset)$.
   By convention  for an empty family  $\ff=\emptyset$ we
  set $o(\ff)=-1$. Also notice that $o(\ff)=o(\widehat{\ff})$.

For every $n\in \nn$ we define
\[\ff_{(n)}=\{s\in[\nn]^{<\infty}:n<\min s\text{ and }\{n\}\cup
s\in\ff\}\] where $n<s$ means that either $s=\emptyset$ or $n<\min
(s)$.
\begin{rem}\label{eq1}
  Notice that for every nonempty family $\ff$ we have that
\[
  o(\ff)=\sup\{o(\ff_{(n)})+1:n\in\nn\}
\]
\end{rem}

  A family
$\ff$ of finite subsets of $\nn$ is called \textit{thin} if there
do not exist  $s, t$ in $\ff$ such that $s$ is a proper initial
segment of $t$.
\begin{rem}\label{thin}
Let $\ff$ be a family of finite subsets of $\nn$. Then $\ff$ is
thin if and only if $\ff$ coincides  with the set of all
$\sqsubseteq$-maximal elements of $\widehat{\ff}$.
\end{rem}
\begin{defn}\label{defn regular thin} We will call  a family
$\ff$ of finite subsets of $\nn$ \emph{regular thin} if $\ff$ is
thin and in addition its closure $\widehat{\ff}$ is a regular
family.
\end{defn}
Let's point out that if $\ff$ is a regular family, then the
Cantor-Bendixson index of $\widehat{\ff}$ as a compact subset of
$\{0,1\}^\nn$ is equal to $o(\ff)+1$. The next lemma allow us to
construct regular thin families from regular ones. We will use the
following notation. For a regular family $\mathcal{R}$ we set
$\mathcal{M}(\mathcal{R})=\{s\in\mathcal{R}:\;s\;\text{is}\;\subseteq\text{-maximal
    in}\;\mathcal{R}\}$.
\begin{lem}\label{maximal}
  Let $\mathcal{R}$ be a regular family. Then the following hold.
  \begin{enumerate}
    \item[(i)]
    $\mathcal{M}(\mathcal{R})=\{s\in\mathcal{R}:\;s\;\text{is}\;\sqsubseteq\text{-maximal
    in}\;\mathcal{R}\}$.
    \item[(ii)] $\widehat{\mathcal{M}(\mathcal{R})}=\mathcal{R}$
    and therefore $\mathcal{M}(\mathcal{R})$ is regular thin with $o(\widehat{\mathcal{M}(\mathcal{R})})=o(\mathcal{R})$.
  \end{enumerate}
\end{lem}
\begin{proof}
 (i) Let $\mathcal{M}_1=\{s\in\mathcal{R}:s\;\text{is}\;\sqsubseteq\text{-maximal
    in}\;\mathcal{R}\}$. It is immediate that
    $\mathcal{M}(\mathcal{R})\subseteq\mathcal{M}_1$. Hence it remains to
    show that $\mathcal{M}_1\subseteq\mathcal{M}(\mathcal{R})$. Suppose on
    the contrary that there exists
    $s\in\mathcal{M}_1\setminus\mathcal{M}(\mathcal{R})$. Then there exists
    $t\in\mathcal{R}$ such that $s$ is a proper subset of $t$. By
    the spreading property of the family $\mathcal{R}$, we may
    find $t'\in\mathcal{R}$ such that $s\sqsubset t'$. Since $s\in\mathcal{M}_1$, this
    leads to a contradiction.

    (ii) Since $\mathcal{M}(\mathcal{R})\subseteq \mathcal{R}$ and $\mathcal{R}$ is hereditary, we
    readily have that $\widehat{\mathcal{M}(\mathcal{R})}\subseteq
    \mathcal{R}$. To show that
    $\mathcal{R}\subseteq \widehat{\mathcal{M}(\mathcal{R})}$ notice that
    for every $s\in\mathcal{R}$
     there exists a $t\in \mathcal{M}(\mathcal{R})$ such that $s\sqsubseteq
     t$ (otherwise $\mathcal{R}$ would not be compact).
\end{proof}
 Schreier families introduced in \cite{AA} provide an hierarchy
of regular thin families of order $\omega^\xi$, $\xi<\omega_1$.
Next we recall their definition.
\begin{defn}\label{Schreier families defn}
The Schreier families $(\mathcal{S}_\xi)_{\xi<\omega_1}$ are
inductively defined as follows. We set
\[\mathcal{S}_0=\big\{\{n\}:n\in\nn \big\}\cup\{\emptyset\}\]
Suppose that for some $\xi<\omega_1$ the families
$(\mathcal{S}_\zeta)_{\zeta<\xi}$ have been defined. If $\xi$ is a
successor, i.e. there exists countable ordinal $\zeta$ such that
$\xi=\zeta+1$, then we set
\[\mathcal{S}_\xi=\big\{\bigcup_{j=1}^n F_j:n\in\nn, F_1,\ldots,F_n\in\mathcal{S}_\zeta, F_1<\ldots<F_n \;\text{and}\;\min F_1\geq n\big\}\]
If $\xi$ is a limit countable ordinal, then we set
\[\mathcal{S}_\xi=\big\{F\in[\nn]^{<\infty}:\exists n\in\nn\;\text{such that}\;\min F\geq n\;\text{and}\;F\in\mathcal{S}_{\zeta_n}\big\}\]
with $(\zeta_n)_{n\in\nn}$ being a strictly increasing sequence of
countable ordinals convergent to $\xi$.
\end{defn}
 By induction to
$\xi<\omega_1$, one can verify that each $\mathcal{S}_\xi$ is a
regular family of order $\omega^\xi$. Hence by Lemma \ref{maximal}
the family $\mathcal{M}(\mathcal{S}_\xi)$ is regular thin of order
$\xi$, for all $\xi<\omega_1$.

In \cite{Pd-R}, for every $\xi<\omega_1$, a thin family of order
$\xi$ has been defined, which generally is not a regular thin one.
However, as it is well known (see \cite{AT}, \cite{LT}), a slight
modification of their definition yields regular families of order
$\xi$ and therefore by Lemma \ref{maximal} the corresponding
maximal elements provide regular thin families of order $\xi$. For
the sake of completeness we briefly describe their definition.

 \begin{prop}
  For every $\xi<\omega_1$ there exists a regular thin family $\ff_\xi$ with $o(\ff_\xi)=\xi$. \end{prop}
 \begin{proof} By Lemma \ref{maximal}, it suffices to built an hierarchy $\{\mathcal{R}_\xi:\xi<\omega_1\}$
 of regular families with $o(\mathcal{R}_\xi)=\xi$ and then set $\ff_\xi=\mathcal{M}(\mathcal{R}_\xi)$ for all $\xi<\omega_1$.
  We proceed by induction on $\xi<\omega_1$. For
 $\xi=0$ we set $\mathcal{R}_0=\{\emptyset\}$. Assume that for some
 $\xi<\omega_1$ and for each $\zeta<\xi$ we have defined a regular
 family $\mathcal{R}_\zeta$ with $o(\mathcal{R}_\zeta)=\zeta$. If
 $\xi$ is a successor ordinal, i.e. $\xi=\zeta+1$, then we set
 \[\mathcal{R}_{\xi}=\Big\{\{n\}\cup
 s:\;n\in\nn,\;s\in\mathcal{R}_\zeta\;\text{and}\;n<\min
 s\Big\}\]

 If $\xi$ is a limit ordinal, then we choose an increasing
 sequence $(\zeta_n)_{n\in\nn}$ such that $\zeta_n\to\xi$ and we
 set
 \[\mathcal{R}_\xi=\bigcup_{n\in\nn}\Big\{s\in\mathcal{R}_{\zeta_n}:\min
 s\geq
 n\Big\}=\bigcup_{n\in\nn}\mathcal{R}_{\zeta_n}\upharpoonright[n,+\infty)\]

It is easy to check that for every $\xi<\omega_1$ we have that
$\mathcal{R}_\xi$ is regular with
 $o(\mathcal{R}_\xi)=\xi$.
 \end{proof}
 The next lemma is a kind of converse of Lemma \ref{maximal}.
\begin{lem}\label{regular thin onto regular}
  Let $\ff$ be a regular thin family. Then
  $\mathcal{M}(\widehat{\ff})=\ff$.
\end{lem}
\begin{proof}
Let $\mathcal{M}_1$ be  the set of all $\sqsubseteq$-maximal
elements of $\widehat{\ff}$. Since $\ff$ is thin, by Remark
\ref{thin} we get that $\ff=\mathcal{M}_1$. Since $\ff$ is regular
thin we get that $\widehat{\ff}$ is regular and therefore by
assertion (i) of Lemma \ref{maximal} we get that
$\mathcal{M}(\widehat{\ff})=\mathcal{M}_1$. Hence
$\ff=\mathcal{M}(\widehat{\ff})$.
\end{proof}
The next proposition describes the relation between the regular
thin families and the regular ones.
\begin{prop}
  The map which sends $\ff$ to $\widehat{\ff}$ is a bijection
  between the set of all regular thin families and the set of all
  regular ones. Moreover, the inverse map sends each regular
  family $\mathcal{R}$ to $\mathcal{M}(\mathcal{R})$.
\end{prop}
\begin{proof}
  By the definition of regular thin families, the map $\ff\to\widehat{\ff}$ sends each
  regular thin family to a regular one. By Lemma
  \ref{regular thin onto regular} we get that the map is 1-1.
  Finally, by assertion (ii) of Lemma
  \ref{maximal} we conclude that the map is onto and the inverse
  map sends each regular
  family $\mathcal{R}$ to $\mathcal{M}(\mathcal{R})$.
\end{proof}

\begin{rem}
  If $\ff$ is a regular thin family with $o(\ff)=k<\omega$, then
  it is easy to see that there exists $n_0$ such that $\ff\upharpoonright
  [n_0,\infty)=\{s\in[\nn]^k:\min s\geq n_0\}$. Therefore, for
  each $k<\omega$, the family $[\nn]^k$ is essentially the unique
  regular thin family of order $k$. However this does not remain valid for regular thin
  families of order $\xi\geq \omega$.  For instance the families $\ff_f=\{s\in[\nn]^{<\infty}: |s|=f(\min s)\}$,
where $f:\nn\to\nn$ is an unbounded increasing map are all regular
thin families of order $\omega$.
\end{rem}

 It is also easy to see that for every regular thin family $\ff$ and every
 $L\in[\nn]^\infty$ the set of all $\sqsubseteq-$maximal
 elements of $\widehat{F}\upharpoonright L$ coincides with $\ff\upharpoonright L$.
This in particular yields the following fact.
 \begin{fact}\label{fact every regular is large}
   Let $\ff$ be a regular thin family. Then the following hold
   \begin{enumerate}
   \item[(i)] For every $L\in[\nn]^\infty$ we have that $\widehat{\ff\upharpoonright L}=\widehat{\ff}\upharpoonright L$.
   \item[(ii)] For every $n\in\nn$, $\ff_{(n)}$ is regular  thin  and $\widehat{\ff_{(n)}}=\widehat{\ff}_{(n)}$.
   \item[(iii)] For every $L\in[\nn]^\infty$ we have that $o(\ff)=o(\widehat{\ff}\upharpoonright L)=o(\ff\upharpoonright L)$.
 \end{enumerate}
 \end{fact}

\subsection{Ramsey properties of families of finite subsets of $\nn$}
  We will use the following terminology from \cite{G3}.
  Let $\ff\subseteq[\nn]^{<\infty}$ and $M\in[\nn]^\infty$.
  We say that $\ff$ is \textit{large} in $M$ if for every $L\in[M]^\infty$ there exists $s\in\ff$ such that $s\subseteq L$.
  We say that $\ff$ is \textit{very large} in $M$ if for every $L\in[M]^\infty$ there exists $s\in\ff$ such
  that $s\sqsubseteq L$. The following is a restatement  (see also \cite{G3}) of a well known theorem due to
  Nash-Williams \cite{N-W} and F. Galvin and K. Prikry \cite{G-P}.
  \begin{thm} \label{Galvin prikry}
    Let $\ff\subseteq[\nn]^{<\infty}$ and $M\in[\nn]^\infty$. If   $\ff$ is large in $M$
    then there exists $L\in[M]^\infty$ such that $\ff$ is very large in $L$.
  \end{thm}
  \begin{rem}\label{Galvin Pricley for regular thin}
  If $\ff$ is regular thin then it is easy to see that  $\ff$ is large in $\nn$. Hence by Theorem \ref{Galvin prikry} for every
   $M\in[\nn]^\infty$ there exists $L\in[M]^\infty$ such that
   $\ff\upharpoonright L$ is very large in $L$.
\end{rem}

 The following theorem is also due to Nash-Williams \cite{N-W}. Since it plays a crucial role in the sequel, for
 the sake of completeness we present its proof.
\begin{thm}\label{th2}
  Let $\ff\subseteq[\nn]^{<\infty}$ be a thin family. Then for every finite partition
  $\ff=\cup_{i=1}^k \ff_i$, ($k\geq 2$) of $\ff$ and every $M\in[\nn]^\infty$ there exist $L\in[M]^\infty$ and
  $1\leq i_0\leq k$ such that $\ff\upharpoonright L=\ff_{i_0}\upharpoonright L$.
  \end{thm}
\begin{proof} It suffices to  show the result only for  $k=2$ since the general case follows easily by induction.
 So let  $\ff=\ff_1\cup \ff_2$
and $M\in [\nn]^\infty$. Then either
there is $L\in [M]^\infty$ such that  $\ff_1\upharpoonright L=\emptyset$ or $\ff_1$ is large in $M$.
In the first case it is clear that  $\ff\upharpoonright L=\ff_2\upharpoonright L$. In the second case by Theorem \ref{Galvin prikry}
there is $L\in [M]^\infty$ such that $\ff_1$ is very large in $L$.
We claim  that $\ff\upharpoonright L=\ff_1\upharpoonright L$. Indeed, let $s\in\ff\upharpoonright L$.
We choose    $N\in [L]^\infty$ such that $s\sqsubseteq N$ and let $t\sqsubseteq N$ such that $t\in \ff_1$.
Then $s,t$ are $\sqsubseteq$-comparable
members of  $\ff$ and since $\ff$ is thin they must be equal. Therefore  $s=t\in\ff_1$ and   $\ff\upharpoonright L= \ff_1\upharpoonright L$.
\end{proof}

 For two families $\ff,\g$ of finite subsets of $\nn$, we write $\ff\sqsubseteq\g$
 (resp. $\ff\sqsubset\g$) if every element in $\ff$ has an extension  (resp. proper extension) in $\g$ and
 every element in $\g$ has an  initial (resp. proper initial) segment in $\ff$. The following proposition is a
 consequence of a more general result from  \cite{GI}.

  \begin{prop}\label{corollary by Gasparis}
    Let $\ff,\g\subseteq[\nn]^{<\infty}$ be regular thin families with $o(\ff)<o(\g)$.
    Then for every $M\in[\nn]^\infty$ there exists
    $L\in[M]^\infty$ such that $\ff\upharpoonright L\sqsubset\g\upharpoonright L$.
  \end{prop}
  \begin{proof}
    By Remark \ref{Galvin Pricley for regular thin} we have that there exists $L_1\in[M]^\infty$
    such that both $\ff,\g$ are very large in $L_1$. So for every $L\in[L_1]^\infty$ and
    every $t\in\g\upharpoonright L$ there exists $s\in\ff\upharpoonright L$ such that $s,t$ are comparable and
    conversely.

    Let $\g_1$ be the set of all elements of $\g$ which have a proper initial segment in $\ff$ and $\g_2=\g\setminus\g_1$.
    By Theorem \ref{th2} there exist $i_0\in\{1,2\}$ and $L\in[L_1]^\infty$
    such that $\g\upharpoonright L\subseteq \g_{i_0}$. It suffices to show
    that $i_0=1$. Indeed, if $i_0=2$ then for every $t\in G\upharpoonright L$ there is $s\in\ff$ such that $t\sqsubseteq s$.
    This in conjunction with assertion (iii) of Fact \ref{fact every regular is large} yields that $o(\g)=o(\g\upharpoonright L)\leq o(\ff)$ which is a contradiction.
  \end{proof}
\section{The notion of the plegma families}\label{section admissibility}
In this section we introduce  the concept of the \textit{plegma}
families. Roughly speaking a plegma family is a finite sequence of
finite nonempty subsets of $\nn$ having elements which alternate
each other in a very concrete way. This notion is the basic new
ingredient for the extension of the definition of the spreading
models. We also establish the Ramsey property of the plegma
families consisting of sets in a regular thin family, a property
which possesses a fundamental role for the following.
  \begin{defn}\label{defn plegma}
    Let $l\in\nn$ and $s_1,...,s_l$ be nonempty finite subsets of $\nn$.
    The $l-$tuple $(s_j)_{j=1}^l$ will be called \textit{plegma}
    if the following are satisfied.
    \begin{enumerate}
      \item[(i)] For every $i,j\in \{1,...,l\}$ and $k\in\nn$ with $i<j$ and $k\leq\min(|s_i|,|s_j|)$, we have that $s_i(k)<s_j(k)$.
      \item[(ii)] For every $i,j\in \{1,...,l\}$ and $k\in\nn$ such that $k\leq \min (|s_i|,|s_j| -1)$, we have that $s_i(k)<s_j(k+1)$.
    \end{enumerate}
  \end{defn}
  For instance a pair $(\{n_1\},\{n_2\})$ of singletons is plegma iff
  $n_1<n_2$ and
   a pair of doubletons $(\{n_1,m_1\},\{n_2,m_2\})$ is plegma iff $n_1<n_2<m_1<m_2$.
Also notice that if $(s_j)_{j=1}^l$ is plegma then for every
$1\leq k\leq l$ and $1\leq j_1<j_2<\ldots<j_k\leq l$ the $k-$tuple
$(s_{j_m})_{m=1}^k$ is plegma. This in particular implies that for
every $1\leq j_1<j_2\leq l$ we have that $s_{j_1}\cap
s_{j_2}=\emptyset$.
 Moreover it is easy to see that
$(s_j)_{j=1}^l$ is plegma iff $(s_{j_1},s_{j_2})$ is a plegma
pair, for all $1\leq j_1<j_2\leq l$. Finally, let us observe that
if $(s_j)_{j=1}^l$ is plegma then for every choice of nonempty
initial segments $t_j$ of $s_j$, $1\leq j\leq l$, the $l$-tuple
$(t_j)_{j=1}^l$ is also plegma.

  \begin{notation}
    Let $\ff$ be a family of finite subsets of $\nn$ and $l\in\nn$. We
    set
    \[\text{\emph{Plm}}_l(\ff)=\{ (s_j)_{j=1}^l:s_1,...,s_l\in\ff \text{ and } (s_j)_{j=1}^l\text{ plegma}\}\]
    Let also $\text{\emph{Plm}}(\ff)=\bigcup_{l=1}^\infty \text{\emph{Plm}}_l(\ff)$.
  \end{notation}

  \begin{lem}\label{abmissibility1-1union and union thin proposition}
    Let $\ff$ be a thin family of finite subsets of $\nn$. Let $(s_j)_{j=1}^l,(t_j)_{j=1}^l$ in
    $\text{\emph{Plm}}(\ff)$ with $|s_1|\leq\ldots\leq |s_l|$,
    $|t_1|\leq\ldots\leq|t_l|$ and $\cup_{j=1}^l s_j\sqsubseteq\cup_{j=1}^l t_j$.
    Then $(s_j)_{j=1}^l=(t_j)_{j=1}^l$ and hence $\cup_{j=1}^l s_j=\cup_{j=1}^l t_j$.
  \end{lem}

\begin{proof}
  Suppose that for some $1\leq m\leq l$ we have that $(s_i)_{i<m}=(t_i)_{i<m}$. We will show that $s_m=t_m$.
  Let $F=\cup_{j=m}^l s_j$ and $G=\cup_{j=m}^l t_j$. Since $\cup_{j=1}^l s_j\sqsubseteq\cup_{j=1}^l t_j$, we get that $F\sqsubseteq G$.
  Moreover since $|s_m|\leq\ldots\leq|s_l|$ and $|t_m|\leq\ldots\leq|t_l|$, we easily conclude that
  $s_m(j)=F((j-1)(l-m+1)+1)$,  for all $1\leq j\leq|s_m|$ and similarly $t_m(j)=G((j-1)(l-m+1)+1)$,
  for all $1\leq j\leq|t_m|$. Hence, as $F\sqsubseteq G$, we get that  for all $1\leq j\leq \min\{|t_m|, |s_m|\},$ $s_m(j)=t_m(j)$.
  Therefore $s_m$ and $t_m$ are $\sqsubseteq$-comparable which by the thiness of $\mathcal{F}$ implies that $s_m=t_m$.
  By induction we have that $s_j=t_j$ for all $1\leq j\leq l$.
\end{proof}
\begin{lem} \label{lemma increasing length of plegma}
   Let $\ff$ be a regular thin family. Then for every
$(s_j)_{j=1}^l\in \text{\emph{Plm}}(\ff)$ we have that
$|s_1|\leq\ldots\leq |s_l|$.
 \end{lem}
 \begin{proof}
   It suffices to prove it for $l=2$. Assume on the contrary that there exists a plegma pair $(s_1,s_2)$
   in $\ff$ with $|s_1|>|s_2|$. We pick $s\in[\nn]^{<\infty}$ such
   that $|s|=|s_1|$, $s_2\sqsubset s$ and $s(|s_2|+1)>\max s_1$.
   By the definition of the plegma family, we have that
   for every $1\leq k\leq|s_2|$, $s_1(k)<s_2(k)=s(k)$. Hence, for
   every $1\leq k\leq |s_1|$, we have that $s_1(k)\leq s(k)$. By
   the spreading property of $\widehat{\ff}$ we get that
   $s\in\widehat{\ff}$. Since $s_2$ is a proper initial segment of
   $s$ we get that $s_2\not\in\ff$, which is a contradiction.
 \end{proof}

\begin{lem}\label{ on the union of plegma}  Let $\ff$ be a family of finite subsets of $\nn$. For every  $l\in\nn$
 let \[\mathcal{U}_l(\ff)=\{\cup_{j=1}^l
   s_j:(s_j)_{j=1}^l\in\text{\emph{Plm}}_l(\ff)\}\]
If $\ff$ is regular thin for every $l\in\nn$ and
$M\in[\nn]^\infty$ the family $\mathcal{U}_l(\ff\upharpoonright
M)$ is thin and the map sending each plegma $l$-tuple to its union
is 1-1 and
  onto.
\end{lem}
\begin{proof}
Let $(s_j)_{j=1}^l\in \text{\emph{Plm}}_l(\ff\upharpoonright
M)\subseteq \text{\emph{Plm}}_l(\ff)$. By Lemma \ref{lemma
increasing length of plegma} we have that $|s_1|\leq\ldots\leq
|s_l|$. Hence the result follows readily by Lemma
\ref{abmissibility1-1union and union thin proposition}.
\end{proof}
\begin{thm} \label{ramseyforplegma}
Let $M$ be an infinite subset of $\nn$, $l\in\nn$  and $\ff$ be a
regular thin family. Then for every finite partition
$\text{\emph{Plm}}_l(\ff\upharpoonright M)=\cup_{j=1}^p A_j$,
there exist $L\in[M]^\infty$ and $1\leq j_0\leq p$ such that
$\text{\emph{Plm}}_l(\ff\upharpoonright L)\subseteq A_{j_0}$.
\end{thm}
\begin{proof}
Let $\mathcal{U}=\mathcal{U}_l(\ff\upharpoonright M)$ and for
$1\leq j\leq p$, let  $\mathcal{U}^{(j)}=\{\cup_{i=1}^l
s_i:(s_i)_{i=1}^l\in A_j\}$. Then $\mathcal{U}=\cup_{j=1}^p
\mathcal{U}^{(j)}$ and by Lemma \ref{ on the union of plegma} we
have that $\mathcal{U}$ is a thin family and
$\{\mathcal{U}^{(j)}\}_{j=1}^l$ is a partition of $\mathcal{U}$.
Hence by Theorem \ref{th2} there exist $j_0$ and $L\in[M]^\infty$
such that $\mathcal{U}\upharpoonright L\subseteq
\mathcal{U}^{(j_0)}$. Since $L\in [M]^\infty$ we have that
$\mathcal{U}\upharpoonright L=\mathcal{U}_l(\ff\upharpoonright
L)$. Hence $\text{\emph{Plm}}_l(\ff\upharpoonright L)\subseteq
A_{j_0}$.
\end{proof}
Let $\ff\subseteq [\nn]^{<\infty}$, $l\in\nn$, $M\in[\nn]^\infty$
and $A\subseteq \text{\emph{Plm}}_l(\ff)$. We will say that $A$ is
\textit{large} in $M$ if for every $L\in[M]^\infty$ we have that
$A\cap\text{\emph{Plm}}_l(\ff\upharpoonright L)\neq\emptyset$.
Under this terminology the following corollary is an immediate
consequence of Theorem \ref{ramseyforplegma}.
\begin{cor}\label{corollary ramseyforplegma}
  Let $\ff$ be a regular thin family , $M\in[\nn]^\infty$ and $l\in\nn$. Let $A\subseteq
  \text{\emph{Plm}}_l(\ff)$ be large in $M$. Then for every
  $M'\in[M]^\infty$, there exists $L\in[M']^\infty$ such that
  $\text{\emph{Plm}}_l(\ff\upharpoonright L)\subseteq A$.
\end{cor}
\section{The plegma paths of finite subsets of $\nn$}
In this section we introduce the definition of the \textit{plegma
paths} in finite subsets of $\nn$ and we present some of  their
properties. Such paths will be next used for the study of maps
from a regular thin family into the finite subsets of $\nn$ which
preserve plegma pairs.

\begin{defn}
  Let $k\in\nn$ and $s_0,...,s_k$ be nonempty finite subsets of $\nn$.
  We will say that $(s_j)_{j=0}^k$ is a \textit{plegma path} \textit{of length} $k$
  \textit{from} $s_0$ \textit{to} $s_k$, if for every $0\leq j\leq k-1$, the pair $(s_j,s_{j+1})$ is
  plegma. Similarly  a sequence $(s_j)_{j\in\nn}$
  of nonempty finite subsets of $\nn$
  will be called  an infinite plegma path if for every
  $j\in\nn$ the pair $(s_j,s_{j+1})$ is plegma.
\end{defn}
\begin{lem}\label{lemma conserning the length of the plegma path}
  Let $s_0,s$ be two nonempty finite subsets of $\nn$ such that $s_0<s$. Let $(s_0,\ldots,s_{k-1},s)$ be a plegma path
  of length $k$ from $s_0$ to $s$. Then \[k\geq \min\{|s_i|:0\leq i\leq
  k-1\}\]
\end{lem}
\begin{proof}
  Suppose that $k< \min\{|s_i|:0\leq i\leq k-1\}$. Then $s(1)<s_{k-1}(2)<s_{k-2}(3)<\ldots<s_1(k)<s_0(k+1)$, which contradicts that $s_0<s$.
\end{proof}
For a family $\ff\subseteq
  [\nn]^{<\infty}$ a \textit{plegma path in $\ff$} is a (finite or infinite) plegma
  path which consists of elements of $\ff$. It is easy to verify the existence of infinite plegma paths in $\ff$ whenever  $\ff$ is very large
   in an infinite subset $L$ of $\nn$.
   In particular let $s\in\ff\upharpoonright L$ satisfying
   the next property: for every $j=1,...,|s|-1$ there exists $l\in L$ such that $s(j)<l<s(j+1)$.
   Then it is straightforward that there exists $s'\in\ff\upharpoonright L$ such that the pair $(s,s')$ is
    plegma and moreover $s'$ shares the same property with $s$. Based on this one can built an infinite
     plegma path in $\ff$  of elements having the above property. These remarks motivate the following definition.
\begin{defn}
  Let $\ff$ be a family of finite subsets of $\nn$ and $L\in[\nn]^\infty$. We define the \textit{skipped restriction
   of} $\ff$ in $L$ to be the family
   \[\ff\upharpoonright\upharpoonright L=\Big{\{}s\in\ff\upharpoonright L:\forall j=1,...,|s|-1,\exists l\in L\text{ such that }s(j)<l<s(j+1)\Big{\}}\]
\end{defn}
Notice that if $\ff$ is a regular thin family and
$L\in[\nn]^\infty$ such that $\ff\upharpoonright L$ is very large
in $L$, then
\[\ff\upharpoonright\upharpoonright L=\Big{\{}s\in\ff\upharpoonright L:\exists s'\in\ff\upharpoonright L \text{ such that }
(s,s')\text{ is plegma}\Big{\}}\]

Also, by Lemma \ref{lemma increasing length of plegma}, we have
that the lengths of the elements of any plegma path in $\ff$ forms
an non decreasing sequence. This property is a key ingredient of
the proof of the next proposition.
\begin{prop}\label{accessing everything with plegma path of length |s_0|}
  Let $\ff$ be a regular thin family and $L\in[\nn]^\infty$ such that $\ff$ is very large in $L$.
  Then for every $s_0,s\in\ff\upharpoonright\upharpoonright L$ with $s_0<s$
   there exists a plegma path $(s_0,\ldots,s_{k-1},s)$ in $\ff\upharpoonright\upharpoonright L$
  of length $k=|s_0|$ from $s_0$ to $s$. Moreover  $k=|s_0|$ is
  the minimal length of a plegma path in $\ff\upharpoonright\upharpoonright L$ from $s_0$ to $s$.
\end{prop}
\begin{proof}
  We will actually prove the following stronger result. For every $t$
  in the closure $\widehat{\ff\upharpoonright\upharpoonright L}$ of $\ff\upharpoonright\upharpoonright L$ and
  $s\in\ff\upharpoonright\upharpoonright L$ with $t<s$ there exists a
  plegma path of length $|t|$ from $t$ to $s$ such that all its
  elements except $t$ belong to $\ff\upharpoonright\upharpoonright L$. The proof will be done by
  induction on the length of $t$. The case $|t|=1$ is trivial, since $(t,s)$ is
  already a plegma path of length 1 from $t$ to
  $s$. Suppose that for some $k\in\nn$ the above holds for all $t$ in
  $\widehat{\ff\upharpoonright\upharpoonright L}$ with $|t|=k$. Let $t\in \widehat{\ff\upharpoonright\upharpoonright L}$ with $|t|=k+1$
  and $s\in\ff\upharpoonright\upharpoonright L$ with $t<s$. Then there exist
  $n_1<n_2<\ldots<n_{k+1}$ in $\nn$ such that $t=\{L(n_j):1\leq j\leq
  k+1\}$. We set $t_0=\{L(n_j-1):2\leq j\leq k+1\}$. Notice that
  by the spreading property of $\widehat{\ff}$ the element $t_0$
  belongs to $\widehat{\ff}$. Since $\ff$ is very large in $L$, we
  actually have that $t_0\in\widehat{\ff\upharpoonright\upharpoonright L}$. Thus by the inductive
  hypothesis there exists a plegma path $(t_0,s_1,\ldots,s_{k-1},s)$
  of length $k$ from $t_0$ to $s$ with
  $s_1,\ldots,s_{k-1},s\in\ff\upharpoonright\upharpoonright L$. Let $l=|s_1|$ and
  $m_1<\ldots<m_l$ such that $s_1=\{L(m_j):1\leq j\leq l\}$. Again by
  the spreading property of $\ff$ it is easy to see the following
  \begin{enumerate}
    \item[(i)] $|s_1|\geq|t|$, that is $l\geq k+1$.
    \item[(ii)] $t_0\in \widehat{\ff\upharpoonright\upharpoonright L}\setminus \ff\upharpoonright\upharpoonright L$.
    \item[(iii)] There exists a (unique) proper extension $s_0$
  of $t_0$ such that $s_0\in\ff\upharpoonright\upharpoonright L$ and $s_0\sqsubseteq
  t_0\cup\{L(m_j-1):k+1\leq j\leq l\}$.
  \end{enumerate}
 It is
  easy to check that  $(t,s_0)$ and $(s_0,s_1)$ are plegma pairs.
  Hence the sequence $(t,s_0,\ldots,s_{k-1},s)$
  is a plegma path of length $k+1$ from $t$ to $s$ with
  $s_0,\ldots,s_{k-1},s\in\ff\upharpoonright\upharpoonright L$ and the proof is complete.

Finally by Lemmas \ref{lemma conserning the length of the plegma
path} and  \ref{lemma increasing length of plegma}, every plegma
path in $\ff$ from $s_0$ to $s$ is of length at least $|s_0|$.
Therefore the path $(s_0,\ldots,s_{k-1},s)$  is of minimal length.
\end{proof}
\begin{rem}\label{graphs} In terms of graph theory the
above proposition states that in the directed graph with vertices
the elements of  $\ff\upharpoonright\upharpoonright L$ and edges
the plegma pairs $(s,t)$ in $\ff\upharpoonright\upharpoonright L$,
the distance between two vertices $s_0$ and $s$ with $s_0<s$ is
equal to the cardinality of $s_0$.
\end{rem}
\section{Hereditarily nonconstant maps on regular thin families}
\begin{defn}
    Let $A$ be a set, $M\in[\nn]^\infty$,  $\ff\subseteq [\nn]^{<\infty}$ and $\varphi:\ff\to A$.
    We will say that $\varphi$ is \textit{hereditarily nonconstant} in $M$ if
    for every $L\in[M]^\infty$ the restriction of $\varphi$ on
    $\ff\upharpoonright L$ is nonconstant. In particular if $M=\nn$ then we will say that $\varphi$ is
    hereditarily nonconstant.
  \end{defn}
  \begin{prop}\label{lemma making a hereditary nonconstant function, nonconstant on plegma pairs}
    Let $\ff$ be a regular thin family, $A$ be a set  and
    $\varphi:\ff\to A$ be hereditarily nonconstant.
    Then for every $N\in[\nn]^\infty$ there exists $L\in [N]^\infty$ such that for every plegma pair
    $(s_1,s_2)$ in $\ff\upharpoonright L$,  $\varphi(s_1)\neq\varphi(s_2)$.
  \end{prop}
  \begin{proof}
    By Theorem \ref{ramseyforplegma} there exists an $L\in [N]^\infty$
    such that either $\varphi(s_1)\neq\varphi(s_2)$, for all plegma pairs
    $(s_1,s_2)$ in $\ff\upharpoonright L$, or $\varphi(s_1)=\varphi(s_2)$, for all plegma pairs
    $(s_1,s_2)$ in $\ff\upharpoonright L$. The second alternative is excluded.

    Indeed, suppose that  $\varphi(s_1)=\varphi(s_2)$, for every plegma pair
    $(s_1,s_2)$ in $\ff\upharpoonright L$.
    We may also assume that $\ff\upharpoonright L$ is very large in $L$.
    Let  $s_0$  be the unique   initial segment of  $L_0=\big{\{} L(2\rho):\rho\in\nn\big{\}}$ in
     $\ff\upharpoonright L$ and  let $k=|s_0|$.
     We set
    $L_0'=\big{\{} L(2\rho):\rho\in\nn\text{ and } \rho>k \big{\}}$. By Proposition
    \ref{accessing everything with plegma path of length |s_0|} for every
    $s\in\ff\upharpoonright L_0'$ there exist a plegma path $(s_0,s_1,\ldots,s_{k-1},s)$
    of length $k$ in $\ff\upharpoonright L$. Then  for every $s\in\ff\upharpoonright L_0'$
    we have that $\varphi(s)=\varphi(s_{k-1})=\ldots=\varphi(s_1)=\varphi(s_0)$, which
    contradicts that $\varphi$ is hereditarily nonconstant.
  \end{proof}
The main result of this section is the following.
\begin{thm}
  Let $\ff$ be a regular thin family, $M\in[\nn]^\infty$ and $\varphi:\ff\to \nn$
 be hereditarily non constant in $M$. Let also $g:\nn\to\nn$.
 Then there exists $N\in[M]^\infty$ such that for every plegma
pair $(s_1,s_2)$ in $\ff\upharpoonright N$,
$\varphi(s_2)-\varphi(s_1)>g(n)$, where $\min s_2=N(n)$.
\end{thm}
\begin{proof}
  By Theorem \ref{ramseyforplegma} there exists $L\in[M]^\infty$
  such that one of following holds.
  \begin{enumerate}
    \item[(i)] $\varphi(s_1)=\varphi(s_2)$, for all plegma pairs
    $(s_1,s_2)$ in $\ff\upharpoonright L$.
    \item[(ii)] $\varphi(s_1)>\varphi(s_2)$, for all plegma pairs
    $(s_1,s_2)$ in $\ff\upharpoonright L$.
    \item[(iii)] $\varphi(s_1)<\varphi(s_2)$, for all plegma pairs
    $(s_1,s_2)$ in $\ff\upharpoonright L$.
  \end{enumerate}
  First we will exclude cases (i) and (ii). Case (i) is easily
  excluded by Proposition \ref{lemma making a hereditary nonconstant function, nonconstant on plegma pairs}.
  Suppose that case (ii) holds. Let $(s_n)_{n\in\nn}$ be an
  infinite plegma path in $\ff\upharpoonright L$. Then
  $(\varphi(s_n))_{n\in\nn}$ forms a strictly decreasing sequence
  of natural numbers, which is impossible. Therefore, case (iii)
  holds. We choose $N\in[L]^\infty$ such that for every $n\geq2$,
  we have that \[\Big|\Big\{l\in L:N(n-1)<l<N(n)\Big\}\Big|\geq\max_{j\leq n}g(j)\]
  Let $(s_1,s_2)$ be a plegma pair in $\ff\upharpoonright N$ with $\min s_2=N(n)$. Notice for every $1\leq k\leq |s_1|$, we
  have that \[\Big|\Big\{l\in L:s_1(k)<l<s_2(k)\Big\}\Big|\geq
  g(n)\] and for every $|s_1|<k\leq |s_2|$ we have that
  \[\Big|\Big\{l\in L:s_2(k-1)<l<s_2(k)\Big\}\Big|\geq
  g(n)\]
  This easily yields that there exist
  $t_1,\ldots,t_{g(n)}\in\ff\upharpoonright L$ such that the
  $(g(n)+2)$-tuple $(s_1,t_1,\ldots,t_{g(n)},s_2)$ is plegma. By
  (iii) above we have that $\varphi(s_2)-\varphi(s_1)>g(n)$.
\end{proof}
  \begin{cor}\label{Proposition combinatorial for SSD}
    Let $\ff$ be a regular thin family, $M\in[\nn]^\infty$ and $\varphi:\ff\to \nn$
 be hereditarily non constant in $M$.
Then  there exists $N\in[M]^\infty$ such that for every plegma
pair $(s_1,s_2)$ in $\ff\upharpoonright N$,
$\varphi(s_2)-\varphi(s_1)>1$.
\end{cor}
\section{On the plegma preserving maps between thin families}
\subsection{The plegma preserving maps}
\begin{defn}
  Let $\ff\subseteq[\nn]^{<\infty}$ and $\varphi:\ff\to
  [\nn]^{<\infty}$. We will say that the map $\varphi$ is \textit{plegma
  preserving} if for every $l\in\nn$ and every plegma $l$-tuple $(s_j)_{j=1}^l$ in $\ff$
  the $l$-tuple $(\varphi(s_j))_{j=1}^l$ is also plegma.

  Moreover a plegma preserving map $\varphi:\ff\to
  [\nn]^{<\infty}$ will be called \textit{normal plegma preserving} if in addition it satisfies the
  following:
  \begin{enumerate}
    \item[(i)] For every plegma pair $(s_1,s_2)$ in $\ff$ we have that $|\varphi(s_1)|\leq|\varphi(s_2)|$.
    \item[(ii)] For every $s\in\ff$ we have that
    $|\varphi(s)|\leq|s|$.
  \end{enumerate}
\end{defn}
\begin{lem}\label{weakpp}  Let $\ff\subseteq[\nn]^{<\infty}$  and
  $\varphi:\ff\to[\nn]^{<\infty}$.
 If for every plegma  pair $(s_1,s_2)$ in $\ff$ the pair
$(\varphi(s_1),\varphi(s_2))$ is also plegma then
    the map $\varphi$ is plegma preserving.
\end{lem}
\begin{proof}
   Let $l\in\nn$ and $(s_j)_{j=1}^l$ be a plegma $l$-tuple in
  $\ff$. Then for every $1\leq i_1<i_2\leq l$ we
  have that $(s_{i_1},s_{i_2})$ is plegma and thus
  $(\varphi(s_{i_1}),\varphi(s_{i_2}))$ is plegma. Hence, by the remarks
  following Definition \ref{defn plegma},  we have  that
  $(\varphi(s_j))_{j=1}^l$ is a plegma $l$-tuple.
\end{proof}

\begin{lem}\label{Properties of plegma preserving map}
  Let $\ff$ be a regular thin family, $M\in[\nn]^\infty$ and $\varphi:\ff\upharpoonright M\to[\nn]^{<\infty}$
  be a plegma preserving map.
  Then there exists $L\in[M]^\infty$ such that the restriction $\varphi|_{\ff\upharpoonright L}$
  of $\varphi$ on $\ff\upharpoonright L$ is normal plegma preserving.
\end{lem}
\begin{proof}
 By  Theorem \ref{ramseyforplegma} there exists $N\in
  [M]^\infty$ such that either
  \begin{enumerate}\item[(a)] $|\varphi(s_1)|\leq|\varphi(s_2)|$, for all plegma pairs $(s_1,s_2)$ in
  $\ff\upharpoonright N$, or
  \item[(b)] $|\varphi(s_1)|>|\varphi(s_2)|$, for all plegma pairs $(s_1,s_2)$ in
  $\ff\upharpoonright N$.\end{enumerate}
  The second alternative cannot occur since otherwise for an infinite plegma path  $(s_n)_{n\in\nn}$
   in $\ff\upharpoonright N$ the sequence $(|\varphi(s_n)|)_{n\in\nn}$ would form  a strictly decreasing sequence of natural
   numbers.

Again by  Theorem \ref{ramseyforplegma} there exists
    $L\in[N]^\infty$ such that either
\begin{enumerate}
   \item[(c)] $|\varphi(s)|\leq|s|$, for all $s\in\ff\upharpoonright
    L$,  or \item[(d)] $|\varphi(s)|>|s|$, for all $s\in\ff\upharpoonright
    L$. \end{enumerate}
  The second  alternative  is excluded. Indeed, assume that (d) holds. Since $\varphi$ on $\ff\upharpoonright L$ is plegma preserving
   it is easy to choose  (using for example an infinite plegma path)
   $s_0, s$ in $\ff\upharpoonright\upharpoonright L$ such that
    $\min(s_0)<\min(s)$ and $\min(\varphi(s_0))<\min(\varphi(s_{j_0}))$.

     Let $k_0=|s_0|$. Then by
    Proposition \ref{accessing everything with plegma path of length
    |s_0|} there exists a plegma path $(s_i)_{i=0}^{k_0}$ in
    $\ff\upharpoonright\upharpoonright L$  from $s_0$ to
    $s=s_{k_0}$ of length $k_0$. Since  $(\varphi(s_i))_{i=0}^{k_0}$ is also a plegma path of
     length $k_0$ from $\varphi(s_0)$ to $\varphi(s_{k_0})$,
    by Lemma \ref{lemma conserning the length of the plegma path} we
have that
     \[\min\{|\varphi(s_i)|:0\leq i\leq k_0-1\}\leq k_0\]
But  assuming that (d) holds we should have that
     $|\varphi(s_i)|>|s_i|$ and  since $\ff$ is a regular thin family,
     $|s_i|\geq |s_0|=k_0$, for all $0\leq i\leq k_0-1$. Hence
      \[\min\{|\varphi(s_i)|:0\leq i\leq k_0-1\}> \min\{|s_i|:0\leq i\leq k_0-1\}\geq k_0\]
      which is a contradiction.

      By the above we have that $\varphi|_{\ff\upharpoonright L}$
 is a normal plegma preserving map
      and the proof is complete.
\end{proof}
\begin{prop}
  Let $\ff$ be a regular thin family and
  $\varphi:\ff\to[\nn]^\infty$. Then for every $M\in[\nn^\infty]$,
  there is $L\in[M]^\infty$ such that exactly one of the following holds.  \begin{enumerate}
    \item[(i)] The restriction $\varphi|_{\ff\upharpoonright L}$
    is normal plegma preserving.
    \item[(ii)] For every  $(s_1,s_2)\in\text{\emph{Plm}}_2(\ff\upharpoonright L)$, neither
    $(\varphi(s_1),\varphi(s_2))$ nor $(\varphi(s_2),\varphi(s_1))$ is a plegma pair.
  \end{enumerate}
\end{prop}
\begin{proof} By Theorem \ref{ramseyforplegma} there exists $N\in
  [M]^\infty$ such that one of the following is satisfied.
  \begin{enumerate}
  \item[(a)] The pair $(\varphi(s_1),\varphi(s_2))$ is plegma, for all $(s_1,s_2)\in\text{\emph{Plm}}_2(\ff\upharpoonright
  N)$.
\item[(b)] The pair $(\varphi(s_2),\varphi(s_1))$ is plegma, for
all
    $(s_1,s_2)\in\text{\emph{Plm}}_2(\ff\upharpoonright N)$.
    \item[(c)] For every $(s_1,s_2)\in\text{\emph{Plm}}_2(\ff\upharpoonright
  N)$,  neither the pair $(\varphi(s_1),\varphi(s_2))$ nor the
  pair $(\varphi(s_2),\varphi(s_1))$ is plegma.
    \end{enumerate}
 Notice that  the second alternative cannot occur.  Indeed, otherwise  for an infinite plegma path $(s_k)_{k\in\nn}$ in
    $\ff\upharpoonright N$, the sequence $(\min(s_k))_{k\in\nn}$ would be a strictly decreasing
    infinite sequence of natural numbers.

     Moreover if
    (a) holds then by Lemma \ref{weakpp} we have that $\varphi|_{\ff\upharpoonright
    N}$ is plegma preserving and by Lemma \ref{Properties of plegma preserving
map} there exists $L\in [N]^\infty$ such that
$\varphi|_{\ff\upharpoonright L}$
    is normal plegma preserving.
\end{proof}
\subsection{Forbidden plegma preserving maps}
The main aim of this subsection is to prove the following.
\begin{thm}\label{non plegma preserving maps}
    Let $\ff,\g$ be regular thin families. If $o(\ff)<o(\g)$
    then there is no plegma preserving map from $\ff$ to $\g$. More precisely for every
    $\varphi:\ff\to\g$ and $M\in[\nn]^\infty$ there exists $L\in[M]^\infty$ such that for every
    plegma pair $(s_1,s_2)$ in $\ff\upharpoonright L$ neither $(\phi(s_1),\phi(s_2))$ nor $(\phi(s_2),\phi(s_1))$ is plegma.
  \end{thm}
  For the proof of the above theorem we will need  the next  definition.
\begin{defn}\label{defn backwards shifting}
  Let $\ff\subseteq[\nn]^{<\infty}$ and $L\in[\nn]^\infty$. We define
  \[\ff(L^{-1})=\Big{\{} t\in[\nn]^{<\infty}:L(t)\in\ff\Big{\}}\]
\end{defn}
 It is easy to see that for every family $\ff$ of finite subsets of $\nn$ and $L\in[\nn]^\infty$ the following hold
  \begin{enumerate}
    \item[(a)] If $\ff$ is very large in $L$, then the family $\ff(L^{-1})$ is very large in $\nn$.
    \item[(b)] If $\ff$ is regular thin then so does the family $\ff(L^{-1})$.
    \item[(c)] $o(\ff(L^{-1}))=o(\ff\upharpoonright L)$. In
    particular if $\ff$ is regular thin then
    $o(\ff(L^{-1}))=o(\ff)$.
  \end{enumerate}

\begin{lem}\label{firsttheoreminadm}
  Let $\ff$ be a regular thin family and $\varphi:\ff\to[\nn]^{<\infty}$ be a plegma preserving map.
  Then for every
  $M\in[\nn]^\infty$ there exists $L\in[M]^\infty$ such that if $\psi:\ff(L^{-1})\to [\nn]^{<\infty}$ is defined by $\psi(t)=\varphi(L(t))$ for every $t\in\ff(L^{-1})$, then the following are satisfied:
  \begin{enumerate}
    \item[(a)] The map $\psi$ is normal plegma preserving.
    \item[(b)] For every $t\in\ff(L^{-1})$ and every $i\leq |\psi(t)|$, we have that $\psi(t)(i)> t(i)$.
  \end{enumerate}
\end{lem}
\begin{proof}
  By Lemma \ref{Properties of plegma preserving map} there exists
  $L_1\in[M]^\infty$ such that the restriction $\varphi|_{\ff\upharpoonright L_1}$ of $\varphi$ on $\ff\upharpoonright L_1$ is normal plegma preserving.
  We may also assume that $\ff$ is very large in $L_1$. Let $\psi_1=\varphi\circ L_1:\ff(L_1^{-1})\to [\nn]^{<\infty}$.
  It is easy to check that $\psi_1$ is normal plegma preserving.\\

  \textbf{Claim:} For every $u\in\ff(L_1^{-1})\upharpoonright\upharpoonright\nn$ and $1\leq i\leq|\psi_1(u)|$, we have $u(i)\leq\psi_1(u)(i)$.
  \begin{proof}[Proof of Claim]
     We will show that for every $i\in\nn$ the following holds: for  every
    $u\in\ff(L_1^{-1})\upharpoonright\upharpoonright\nn$ with $ i\leq |\psi_1(u)|$,
    $u(i)\leq\psi_1(u)(i)$.
    Indeed, let $i=1$ and let
    $u\in\ff(L_1^{-1})\upharpoonright\upharpoonright\nn$.
    If $u(1)=1$ then obviously $\psi_1(u)(1)\geq 1=u(1)$.
    Suppose that for some $k\in\nn$ and every
    $u'\in\ff(L_1^{-1})\upharpoonright\upharpoonright\nn$ with
    $u'(1)=k$ we have that $\psi_1(u')(1)\geq k$. Let
    $u\in\ff(L_1^{-1})\upharpoonright\upharpoonright\nn$ with $u(1)=k+1$. Since
    $\ff(L_1^{-1})$ is regular thin and very large in $\nn$ there exists a unique
    $u'\in\ff(L_1^{-1})$ with $u'\sqsubseteq u-1=\{u(i)-1: 1\leq i\leq |u|\}$. Then
    $(u',u)$ is a
    plegma pair and $u'(1)=k$. Since $\psi_1$ is normal plegma preserving we have that
    $(\psi_1(u'),\psi_1(u))$ is
    plegma. Hence   $\psi_1(u)(1)>\psi_1(u')(1)\geq u'(1)=k$, that is $\psi_1(u)(1)\geq      k+1=u(1)$.
    By induction on $k=u(1)$, we get that for all
    $u\in\ff(L_1^{-1})\upharpoonright\upharpoonright\nn$,
    $u(1)\leq\psi_1(u)(1)$.

    Suppose now that for some $i\in\nn$,   it holds that for every
    $u'\in\ff(L_1^{-1})\upharpoonright\upharpoonright\nn$ with $i\leq |\psi_1(u')|$,
    $u'(i)\leq\psi_1(u')(i)$.
    Let $u\in\ff(L_1^{-1})\upharpoonright\upharpoonright\nn$ with $i+1\leq |\psi_1(u)|$. Clearly there exists
    $u'\in\ff(L_1^{-1})\upharpoonright\upharpoonright\nn$ such that
    $\{u(\rho)-1: 2\leq \rho\leq |u|\}\sqsubseteq u'$. Then  $(u,u')$ is plegma,
    $|u|\leq |u'|$ and  $u(i+1)=u'(i)+1$. Since $\psi_1$ is normal plegma preserving, $(\psi_1(u),\psi_1(u'))$ is plegma
    and  $i+1\leq |\psi_1(u)|\leq|\psi_1(u')|$. Hence,
    $\psi_1(u)(i+1)>\psi_1(u')(i)\geq u'(i)=u(i+1)-1$, that is $\psi_1(u)(i+1)\geq u(i+1)$.
 By induction on $i\in\nn$ the proof of the claim is complete.
  \end{proof}
  Let $N_0=\{2\rho:\rho\in\nn\}$ and  $L=L_1(N_0)$. It is easy to
  check that for every $t\in \ff(L^{-1})$, $N_0(t)=2t=\{2t(i):1\leq i\leq |t|\}$ and $N_0(t)\in \ff(L_1^{-1})\upharpoonright\upharpoonright \nn$.
  Let $\psi=\varphi\circ L$. Then $\psi= \psi_1\circ N_0$. Indeed for every $t\in\ff(L^{-1})$
  \[\psi(t)=\varphi(L(t))=\varphi(L_1(N_0(t)))=\psi_1(N_0(t))\]
  Then for every $t\in\ff(L^{-1})$ and $1\leq i\leq|\psi(t)|$ we have that $\psi(t)(i)=\psi_1(N_0(t))(i)\geq N_0(t)(i)=2t(i)>t(i)$.
\end{proof}
  \begin{proof}[Proof of Theorem \ref{non plegma preserving maps}]
    Assume on the contrary that there exists $\varphi:\ff\to\g$ plegma preserving map.
    By Lemma \ref{firsttheoreminadm} there exists $L\in[\nn]^\infty$ such that if $\psi:\ff(L^{-1})\to \g$
    is defined by $\psi(t)=\varphi(L(t))$ for every $t\in\ff(L^{-1})$, then the following are satisfied:
  \begin{enumerate}
    \item[(a)] The map $\psi$ is normal plegma preserving.
    \item[(b)] For every $t\in\ff(L^{-1})$ and every $i\leq |\psi(t)|$, we have that $\psi(t)(i)> t(i)$.
  \end{enumerate}
    Since $o(\ff(L^{-1}))=o(\ff)<o(\g)$, by Proposition \ref{corollary by Gasparis} we have that there exists $N\in[\nn]^\infty$ such that
   \begin{equation}
   \ff(L^{-1})\upharpoonright N\sqsubset \g\upharpoonright N \label{eq3}
    \end{equation}
    We are now ready to derive a contradiction.
    Indeed, pick $t_0\in\ff(L^{-1})\upharpoonright N$. Then $t_0\in\widehat{\g}\setminus\g$. But then by (\ref{eq3}), assertion (b) above and the
    spreading property of $\widehat{\g}$ we should have $\psi(t_0)\in\widehat{\g}\setminus\g$ which is impossible.
  \end{proof}
  \subsection{Plegma preserving embeddings  between regular thin families}
  By Theorem \ref{non plegma preserving maps} we have the
  following.
\begin{cor}
  Let $\ff$, $\g$ be regular thin families. If there exist
  $M\in[\nn]^\infty$ and $\varphi:\g\to\ff$ such that
  $\varphi|_{\g\upharpoonright M}$ is plegma preserving then
  $o(\ff)\leq o(\g)$.
\end{cor}

  In the sequel we will show the converse of the above corollary, i.e. if $o(\ff)\leq o(\g)$ then there exist $N\in[\nn]^\infty$ and a plegma
  preserving
  map $\varphi:\g\upharpoonright N\to\ff$.
   First we will need  some combinatorial properties mainly concerning regular families.
\begin{defn}
Let $\ff\subseteq [\nn]^{<\infty}$ and  $L\in[\nn]^\infty$. We
define
 \[\ff(L)=\{L(s):\;s\in\ff\}\]
 \end{defn} Notice that $o(\ff)=o(\ff(L))$ and if $\ff$ is compact (resp. hereditary) then
 $\ff(L)$ is also compact (resp. hereditary).

  We will use the next easily verified lemma.
\begin{lem}\label{ls} Let $\ff$ be a spreading family of finite subsets of $\nn$. Then
\begin{enumerate} \item[(i)] For every $L_1\subseteq L_2$ in $[\nn]^\infty$, we have that $\ff(L_1)\subseteq \ff(L_2)$.
\item[(ii)] For every $k\in \nn$, $L_1,L_2\in[\nn]^\infty$ with
$\{L_1(j):j> k\}\subseteq \{L_2(j):j> k\}$, we have that
$\ff_{(k)}(L_1)\subseteq \ff_{(k)}(L_2)$ (where
$\ff_{(k)}(L)=\{L(s):s\in\ff_{(k)}\}$).
\end{enumerate}
\end{lem}

  \begin{prop}\label{embending a family F in G by a shifting where o(F)<o(G)}
    Let $\ff,\g$ be regular families of finite subsets of $\nn$ with $o(\ff)\leq o(\g)$.
    Then for every $M\in[\nn]^\infty$ there exists $L\in [M]^\infty$ such that $\ff(L)\subseteq\g$.
  \end{prop}
  \begin{proof}
  If  $o(\ff)= 0$, i.e. $\ff=\{\emptyset\}$, then the conclusion trivially holds.  Suppose that for some $\xi<\omega_1$ the proposition is true
   for every regular families
    $\ff',\g'\subseteq[\nn]^{<\infty}$ such that  $o(\ff')<\xi$ and
    $o(\ff')\leq o(\g')$.
    Let $\ff$, $\g$ be regular with $o(\ff)=\xi$ and let $M\in [\nn]^\infty$.
  By Remark \ref{eq1} we have that  $o(\ff_{(1)})<o(\ff)$. Hence $o(\ff_{(1)})<o(\g)$ and so
  there is some $l_1\in\nn$ such that $o(\ff_{(1)})\leq o(\g_{(l_1)})$. Since  $\g$ is spreading we have that
  $o(\g_{(l_1)})\leq o(\g_{(n)})$ for all $n\geq l_1$ and therefore we may suppose that $l_1\in M$.
 It is easy to see that   $\ff_{(1)}$ and $\g_{(l_1)}$ are regular families. Hence by  our inductive hypothesis there is  $L_1\in [M]^\infty$
 such that   $\ff_{(1)}(L_1)\subseteq \g_{(l_1)}$.

  Proceeding in the same way we construct a strictly increasing sequence
 $l_1<l_2<...$ in $M$ and a decreasing sequence $M=L_0\supset L_1\supset ...$ of infinite subsets of $M$ such that
  for all $j\geq 1$,
  the following properties are satisfied.
   \begin{enumerate}
 \item[(i)] $l_{j+1}\in  L_j$.
 \item[(ii)] $l_{j+1}> L_j(j)$.
\item[(iii)] $\ff_{(j)}(L_j)\subseteq \g_{(l_j)}$.
    \end{enumerate}
We set  $L=\{l_j\}_{j\in\nn}$. We claim that $\ff(L)\subseteq\g$.
Indeed, by the above construction we have that for every
$k\in\nn$, $\{L(j)\}_{j>k}\subseteq \{ L_k(j)\}_{j> k}$. Therefore
by the second part of Lemma \ref{ls} and the third condition of
the construction we get that
\begin{equation}\label{eq2}\ff_{(k)}(L)\subseteq
\ff_{(k)}(L_k)\subseteq \g_{(l_k)}\end{equation} It is easy to see
that $\ff_{(k)}(L)=\ff(L)_{(l_k)}$ and so by (\ref{eq2}) we have
that $\ff(L)_{(l_k)}\subseteq  \g_{(l_k)}$. Since this holds for
every $k\in\nn$, we conclude that $\ff(L)\subseteq \g$.
\end{proof}

\begin{cor}
  Let $\ff,\g$ be regular families of finite subsets of $\nn$ with $o(\ff)= o(\g)$.
  Then for every $M\in[\nn]^\infty$ there exists $L\in [M]^\infty$ such that $\ff(L)\subseteq\g$ and $\g(L)\subseteq\ff$.
\end{cor}
\begin{lem}\label{prop plegma preserving maps}
  Let $\ff,\g$ be regular thin families with $o(\ff)\leq o(\g)$.
  Then there exists $L_0\in[\nn]^\infty$ such that for every
  $M\in[\nn]^\infty$ there exist $N\in[\nn]^\infty$ and
  $\varphi:\g\upharpoonright N\to\ff\upharpoonright M$ such that
  $L_0(\varphi(t))\sqsubseteq t$, for all $t\in\g\upharpoonright N$.
\end{lem}
\begin{proof}
    By Proposition \ref{embending a family F in G by a shifting where o(F)<o(G)}
    there exists  $L_0\in[\nn]^\infty$ such that $\widehat{\ff}(L_0)\subseteq\widehat{\g}$.
 Let $M\in [\nn]^\infty$.
    Also notice that $\ff(L_0)$ and $\g$ are large in $L_0(M)$. Hence  by Theorem \ref{Galvin prikry} there exists $N\in [L_0(M)]^\infty$
    such that $\ff(L_0)$ and $\g$ are  very large in $N$.
    Since  $\widehat{\ff}(L_0)\subseteq\widehat{\g}$ we get  that $\ff(L_0)\upharpoonright N\sqsubseteq \g\upharpoonright
    N$. To define the map $\varphi:\g\upharpoonright N\to\ff\upharpoonright M$, let $t\in \g\upharpoonright N$. Then there exists a unique
$\widetilde{s}\in \ff(L_0)\upharpoonright N$ such that
$\widetilde{s}\sqsubseteq t$. We set $\varphi(t)=s$, where   $s$
is the unique element of $\ff$ with $L_0(s)=\widetilde{s}$. It is
easy to check that $\varphi$ is as desired.
  \end{proof}
  We are now ready for the main result of this subsection.
\begin{thm}\label{plegma preserving maps}
  Let $\ff,\g$ be regular thin families with $o(\ff)\leq o(\g)$.
  For every $M\in[\nn]^\infty$ there is $N\in[\nn]^\infty$ and
  a plegma preserving map
  $\varphi:\g\upharpoonright N\to\ff\upharpoonright M$.
\end{thm}
\begin{proof}
  Let $M\in[\nn]^\infty$.
  By Lemma \ref{prop plegma preserving maps} there exist $L_0\in[\nn]^\infty$, $N\in[\nn]^\infty$ and
  $\varphi:\g\upharpoonright N\to\ff\upharpoonright M$ such that
  $L_0(\varphi(t))\sqsubseteq t$, for all $t\in\g\upharpoonright
  N$. Let $l\in\nn$ and $(t_i)_{i=1}^l\in\text{\emph{Plm}}(\g\upharpoonright
  N)$. Then $L_0(\varphi(t_i))\sqsubseteq t_i$, for all $1\leq
  i\leq l$. Since $(t_i)_{i=1}^l$ is plegma, we have that
  $(L_0(\varphi(t_i))_{i=1}^l$ is also plegma. This easily yields
  that $(\varphi(t_i))_{i=1}^l\in\text{\emph{Plm}}(\ff\upharpoonright
  M)$.
\end{proof}

\chapter{The hierarchy of spreading models}\label{Chapter 2}
In this chapter we will introduce the notion of the higher order
spreading model, which is a generalization of the classical one
invented by A. Brunel and L. Sucheston in \cite{BS}. The
definition is based on the concepts of the $\ff$-sequences and
plegma families.
 It is shown  that if a spreading model is
generated by an $\ff$-sequence in a Banach space $X$, then it is
also generated by a $\g$-sequence in $X$, where $\g$ is any other
regular thin family with at least the same order as $\ff$. This
explains that a spreading model generated by an $\ff$-sequence is
completely determined by the order of the regular thin family
$\ff$ and furthermore  the $\xi$-spreading models define an
increasing hierarchy.
\section{Definition and existence of $\ff$-spreading models}
Let us start with the general definition of the $\ff$-sequences.
For a set $X$ and a regular thin family $\ff$, by the term
\emph{$\ff$-sequence} in $X$ we will understand a map
$\varphi:\ff\to X$. An $\ff$-sequence in $X$ will be usually
denoted by $(x_s)_{s\in\ff}$, where $x_s=\varphi(s)$ for all
$s\in\ff$. Moreover given $M\in[\nn]^\infty$, the map $\varphi:
\ff\upharpoonright M\to X$ will be called an
\emph{$\ff$-subsequence} of $(x_s)_{s\in\ff}$ and will be denoted
by $(x_s)_{s\in\ff\upharpoonright M}$.

For $\ff$-sequences in a Banach space $X$ we will use the
following terminology. An $\ff$-sequence $(x_s)_{s\in
\mathcal{F}}$ in $X$ will be called \textit{bounded} (resp.
\textit{seminormalized}) if there exists $C>0$ (resp. $0<c<C$)
such that $\|x_s\|\leq C$ (resp. $c\leq \|x_s\|\leq C$) for every
$s\in\ff$.

 We fix for the sequel a regular thin family $\ff$ and a bounded
$\ff$-sequence $(x_s)_{s\in\ff}$ in a space X.
\begin{lem}
  Let $l\in\nn$, $N\in[\nn]^\infty$ and $\delta>0$. Then
  there exists $L\in[N]^\infty$ such that
  \[\Bigg{|}\Big{\|}\sum_{j=1}^la_jx_{t_j}\Big{\|}-\Big{\|}\sum_{j=1}^la_jx_{s_j}\Big{\|}\Bigg{|}\leq\delta\]
  for every
  $(t_j)_{j=1}^l, (s_j)_{j=1}^l \in\text{\emph{Plm}}_l(\ff\upharpoonright L)$ and
  $a_1,...,a_l\in[-1,1]$.
\end{lem}
\begin{proof}
Let $(\textbf{a}_k)_{k=1}^{n_0}$ be a $\frac{\delta}{3l}-$net of
the unit ball of $\rr^l$ with $\|\cdot\|_\infty$. We set $N_0=N$.
By a finite induction on $1\leq k\leq n_0$, we construct a
decreasing sequence $N_0\supseteq N_1\supseteq\ldots\supseteq
N_{n_0}$ as follows. Suppose that $N_0,\ldots,N_{k-1}$ have been
constructed. Define $g_k:\text{\emph{Plm}}_l(\ff\upharpoonright
N_{k-1})\to[0,lC]$ such that $g_k((s_j)_{j=1}^l)=\|\sum_{j=1}^l
a_j^k x_{s_j}\|$, where $\textbf{a}_k=(a_j^k)_{j=1}^l$. By
dividing the interval $[0,lC]$ into disjoint intervals of length
$\frac{\delta}{3}$, Theorem \ref{ramseyforplegma} yields a
$N_k\in[N_{k-1}]^\infty$ such that for every
$(t_j)_{j=1}^l,(s_j)_{j=1}^l\in\text{\emph{Plm}}_l(\ff\upharpoonright
N_k)$, we have
$|g_k((t_j)_{j=1}^l)-g_k((s_j)_{j=1}^l)|<\frac{\delta}{3}$.
Proceeding in this way  we conclude that for every
  $(s_j)_{j=1}^l,(t_j)_{j=1}^l\in\text{\emph{Plm}}_l(\ff\upharpoonright N_{n_0})$ and
  $1\leq k\leq n_0$ we have that
  \[\Bigg{|}\Big\|\sum_{j=1}^la_j^kx_{t_j}\Big\|-\Big\|\sum_{j=1}^la_j^kx_{s_j}\Big\|\Bigg{|}\leq\frac{\delta}{3}\]
Taking into account that $(\textbf{a}_k)_{k=1}^{n_0}$ is a
$\frac{\delta}{3}-$net of the unit ball of $(\rr^l,
\|\cdot\|_\infty)$ it is easy to see that $L=N_{n_0}$ is as
desired.
\end{proof}

Using standard  diagonolization arguments we get the following.
\begin{prop} \label{B-Sp} Let  $\ff$ be a regular thin family, $(\delta_n)_n$ be a decreasing null sequence of positive real
numbers and $(x_s)_{s\in\ff}$ be a bounded $\ff-$sequence in a
Banach space $X$. Then for every $N\in[\nn]^\infty$ there exists
$M\in[N]^\infty$ such that for every $k\leq l$ in $\nn$,
$(t_j)_{j=1}^k, (s_j)_{j=1}^k
\in\text{\emph{Plm}}_k(\ff\upharpoonright M)$ with
$s_1(1),t_1(1)\geq M(l)$ and $a_1,...,a_k\in[-1,1]$, we have that
\[\Bigg{|}\Big\|\sum_{j=1}^k a_jx_{t_j}\Big\|-\Big\|\sum_{j=1}^k a_jx_{s_j}\Big\|\Bigg{|}\leq\delta_l\]
Hence for every $l\in\nn$ and $a_1,...,a_l\in\rr$ and every
sequence  $\big((s_j^n)_{j=1}^l\big)_{n}$ of plegma $l-$tuples in
$\ff\upharpoonright M$ with $s_1^n(1)\to \infty$ the sequence
  $(\|\sum_{j=1}^la_jx_{s_j}^n\|)_{n\in\nn}$ is Cauchy, with the limit
  independent from the choice of the sequence
  $((s_j^n)_{j=1}^l)_{n\in\nn}$.

  In particular there exists a seminorm $\|\cdot\|_*$ on  $c_{00}(\nn)$ under which the natural Hamel
basis $(e_n)_n$  is a spreading sequence and
\[\Bigg{|}\Big\|\sum_{j=1}^k a_jx_{s_j}\Big\|-\Big\|\sum_{j=1}^k
a_je_j\Big\|_*\Bigg{|}\leq\delta_l\] for all $k\leq l$ in $\nn$,
$a_1,\ldots, a_k\in[-1,1]$ and
$(s_j)_{j=1}^k\in\text{\emph{Plm}}_k(\ff\upharpoonright M)$ with
$s_1(1)\geq M(l)$.
\end{prop}

  Let us notice that there do exist bounded
  $\ff$-sequences in Banach spaces such that no seminorm resulting
  from Proposition \ref{B-Sp} is a norm. Also we should point out
  that even if the  $\|\cdot\|_* $ is a norm the sequence $(e_n)_{n\in\nn}$
  is not necessarily Schauder basic. In the sequel we shall give sufficient
  conditions for the semimorm to be a norm and later for the sequence $(e_n)_{n\in\nn}$
  to be a Schauder basic or even an unconditional one.
\begin{defn}\label{Definition of spreading model}
   Let   $X$ be a Banach space, $\ff$ be a regular thin family,
   $(x_s)_{s\in\ff}$ be an $\ff$-sequence in $X$ and $M\in[\nn]^\infty$. Let $(E,\|\cdot\|_*)$ be an infinite dimensional
   seminormed linear space with Hamel basis
   $(e_n)_{n\in\nn}$.

  We will say that the $\ff$-subsequence
  $(x_s)_{s\in\ff\upharpoonright M}$ generates  $(e_n)_{n\in\nn}$ as an
  $\ff$-spreading model if the following is satisfied. There exists a null sequence
   $(\delta_n)_{n\in\nn}$ of positive reals
   such that
  \[\Bigg{|}\Big{\|}\sum_{j=1}^k a_j x_{s_j}\Big{\|}-\Big{\|}\sum_{j=1}^k a_j
  e_j\Big{\|}_* \Bigg{|}\leq\delta_l\]
  for every $1\leq k\leq l$, every plegma $k$-tuple $(s_j)_{j=1}^k\in\text{\emph{Plm}}_k(\ff\upharpoonright M)$
  with $s_1(1)\geq M(l)$ and every choice of $a_1,...,a_k\in[-1,1]$.

  We will also say that $(x_s)_{s\in\ff}$ admits $(e_n)_{n\in\nn}$ as an $\ff$-spreading model
   if there exists  $M\in[\nn]^\infty$ such that the subsequence $(x_s)_{s\in\ff\upharpoonright M}$ generates  $(e_n)_{n\in\nn}$ as an
  $\ff$-spreading model.

  Finally, for a subset $A$ of $X$, we will say that $(e_n)_{n\in\nn}$ is
  an $\ff$-spreading model of $A$ if there exists an $\ff$-sequence $(x_s)_{s\in\ff}$ in $A$ which admits $(e_n)_{n\in\nn}$ as
  an $\ff$-spreading model.
\end{defn}
The next remark is straightforward by the above definition.
\begin{rem}\label{remark on the definition of spreading model}
  Let $\ff$ be a regular thin family and $(x_s)_{s\in\ff}$ an $\ff$-sequence in a Banach space $X$.
  Let $M\in[\nn]^\infty$ such that $(x_s)_{s\in\ff\upharpoonright M}$ generates $(e_n)_{n\in\nn}$ as an $\ff$-spreading model. Then the
  following hold.
  \begin{enumerate}
    \item[(i)] The sequence $(e_n)_{n\in\nn}$ is spreading, i.e. for every $n\in\nn$, $k_1<\ldots<k_n$ in
$\nn$ and $a_1,\ldots,a_n\in\rr$ we have that $\|\sum_{j=1}^na_j
e_j\|_*=\|\sum_{j=1}^n a_j e_{k_j}\|_*$.
    \item[(ii)] For every $M'\in[M]^\infty$ we have that $(x_s)_{s\in\ff\upharpoonright M'}$ generates $(e_n)_{n\in\nn}$ as an $\ff$-spreading model.
    \item[(iii)] For every $(\delta_n)_{n\in\nn}$ null sequence of positive reals there exists $M'\in[M]^\infty$ such that
    $(x_s)_{s\in\ff\upharpoonright M'}$ generates $(e_n)_{n\in\nn}$ as an $\ff$-spreading model with respect to $(\delta_n)_{n\in\nn}$.
    \item[(iv)] If $\|\cdot\|,|\|\cdot|\|$ are two equivalent
    norms on $X$, then every $\ff$-spreading model  admitted by
    $(X,\|\cdot\|)$ is equivalent an $\ff$-spreading model
    admitted by $(X,|\|\cdot\||)$.

  \end{enumerate}
\end{rem}
\begin{rem}\label{remark for k spr mod}
  Let $(x_s)_{s\in[\nn]^k}$ be an $[\nn]^k$-sequence in a Banach space $X$ and $M\in[\nn]$. We set $y_s=x_{M(s)}$, for all $s\in[\nn]^k$.
  Then it is straightforward that the $[\nn]^k$-subsequence $(x_s)_{s\in[M]^k}$ generates an $[\nn]^k$-spreading model $(e_n)_{n\in\nn}$ iff
  the $[\nn]^k$-sequence $(y_s)_{s\in[\nn]^k}$ generates $(e_n)_{n\in\nn}$ as an $[\nn]^k$-spreading model.
\end{rem}
\begin{rem}\label{old defn yields new}
  Let   $X$ be a Banach space, $\ff$ be a regular thin family,
   $(x_s)_{s\in\ff}$ be an $\ff$-sequence in $X$ and $M\in[\nn]^\infty$. Let $(E,\|\cdot\|_*)$ be an infinite dimensional
   seminormed linear space with Hamel basis
   $(e_n)_{n\in\nn}$.
   Assume that there exists a null sequence
   $(\delta_n)_{n\in\nn}$ of positive reals
   such that
  \[\Bigg{|}\Big{\|}\sum_{j=1}^l a_j x_{s_j}\Big{\|}-\Big{\|}\sum_{j=1}^l a_j
  e_j\Big{\|}_* \Bigg{|}\leq\delta_l\]
  for every $l\in\nn$, every plegma $l$-tuple $(s_j)_{j=1}^l\in\text{\emph{Plm}}_l(\ff\upharpoonright M)$
  with $s_1(1)\geq M(l)$ and every choice of $a_1,...,a_l\in[-1,1]$.
  Then we may pass to an $L\in[M]^\infty$ such that $(x_s)_{s\in\ff\upharpoonright L}$ generates $(e_n)_{n\in\nn}$ as an $\ff$-spreading model.
  It is enough to set $L=\{M(1+\sum_{j=1}^pj-1):p\in\nn\}$. This set has the property that for every $l\in\nn$ there exist $l-1$ elements
  of $M$ between $L(l)$ and $L(l+1)$. So every plegma $k$-tuple after $L(l)$, with $1\leq k\leq l$, in $\ff\upharpoonright L$ can be extended to
  a plegma $l$-tuple after $L(l)\geq M(l)$ in $\ff\upharpoonright M$.
\end{rem}
Under the Definition \ref{Definition of spreading model} we get the following reformulation of Proposition \ref{B-Sp}.
\begin{thm}
  Let $\ff$ be a regular thin family and $X$ a Banach space. Then every
  bounded $\ff$-sequence in $X$ admits an $\ff$-spreading model.

  In particular for every bounded $\ff$-sequence $(x_s)_{s\in\ff}$ in $X$ and every $N\in[\nn]^\infty$ there exists $M\in[N]^\infty$
  such that $(x_s)_{s\in\ff\upharpoonright M}$ generates $\ff$-spreading model.
\end{thm}

  \section{Spreading models of order $\xi$}
\begin{prop}\label{propxi}
  Let $\ff,\g$ be regular thin families with $o(\ff)\leq o(\g)$. Let $X$ be a
  Banach space and $(x_s)_{s\in\ff}$ be an
  $\ff$-sequence in $X$ which admits an
  $\ff$-spreading model $(e_n)_{n\in\nn}$. Then there exists a $\g$-sequence
  $(w_t)_{t\in\g}$ which admits the same sequence $(e_n)_{n\in\nn}$ as a $\g$-spreading model and
  such that $\{w_t:t\in\g\}\subseteq\{x_s:s\in\ff\}$.
\end{prop}
\begin{proof}
  Let $M\in[\nn]^\infty$ such that $(x_s)_{s\in\ff\upharpoonright M}$ generates $(e_n)_{n\in\nn}$.
  By Lemma \ref{prop plegma preserving maps} there exist $N\in[\nn]^\infty$ and
  $\varphi:\g\upharpoonright N\to\ff\upharpoonright M$ such that
  $L_0(\varphi(t))\sqsubseteq t$, for all $t\in\g\upharpoonright N$.
  Using the map $\varphi$ we define a $\g$-sequence $(w_t)_{t\in\g}$ as follows. For
  every $t\in\g\upharpoonright N$ we set $w_t=x_{\varphi(t)}$ and for every $t\in\g\setminus(\g\upharpoonright N)$
  let $x'_t$ be an arbitrary element of $\{x_s:s\in\ff\}$. We will show that $(w_t)_{t\in\g\upharpoonright N}$
  generates $(e_n)_{n\in\nn}$. Indeed, let
  $(\delta_n)\searrow0$ such that \[\Bigg{|}\Big{\|}\sum_{j=1}^k a_j x_{s_j}\Big{\|}-\Big{\|}\sum_{j=1}^k a_j
  e_j\Big{\|}_* \Bigg{|}\leq\delta_l\]
  for every $1\leq k\leq l$, every plegma $k$-tuple $(s_j)_{j=1}^k$ in    $\ff\upharpoonright M$
  with $s_1(1)\geq M(l)$ and every choice of $a_1,...,a_k\in[-1,1]$. Let $l\in\nn$, $1\leq k\leq l$, $(t_j)_{j=1}^k$
  be a plegma $k$-tuple in $\g\upharpoonright N$ with $t_1(1)\geq N(l)$ and $a_1,...,a_k\in[-1,1]$. Let
  $s_j=\varphi(t_j)\in\ff\upharpoonright M$, for all $1\leq j\leq k$. By Theorem \ref{plegma preserving maps}, we have
   that $(s_j)_{j=1}^k$ is plegma. Moreover  since $N\in[L_0(M)]^\infty$, $s_1(1)\geq M(l)$. Therefore,
  \[\Bigg{|}\Big{\|}\sum_{j=1}^k a_j w_{t_j}\Big{\|}-\Big{\|}\sum_{j=1}^k a_j
  e_j\Big{\|} \Bigg{|} =\Bigg{|}\Big{\|}\sum_{j=1}^k a_j x_{s_j}\Big{\|}-\Big{\|}\sum_{j=1}^k a_j
  e_j\Big{\|}_* \Bigg{|}\leq\delta_l\]
\end{proof}
\begin{cor}\label{equalspr}
  Let $X$ be a Banach space, $A\subseteq X$ and $\ff,\g$ be regular thin families with $o(\ff)=o(\g)$.
  Then
  $(e_n)_{n\in\nn}$ is an $\ff$-spreading model of $A$ iff $(e_n)_{n\in\nn}$
  is a $\g$-spreading model of $A$.
\end{cor}
The above permits us to give the following definition.
\begin{defn}
  Let $A$ be a subset of a Banach space $X$ and $1\leq\xi<\omega_1$
  be a countable ordinal. We will say that $(e_n)_{n\in\nn}$ is a $\xi$-spreading model of $A$ if there
  exists a regular thin family $\ff$ with $o(\ff)=\xi$ such that
  $(e_n)_{n\in\nn}$ is an $\ff$-spreading model of $A$. The set of all $\xi$-spreading models of $A$ will be denoted by
  $\mathcal{SM}_\xi(A)$.
\end{defn}
 Proposition \ref{propxi} also yields the following.
\begin{cor}\label{incresspr}
Let $X$ be a Banach space and $A\subseteq X$. Then
$\mathcal{SM}_\zeta(A)\subseteq \mathcal{SM}_\xi(A)$, for all
$1\leq\zeta<\xi<\omega_1$.
\end{cor}
The above reveals the following problems.
\begin{problem}\label{problem1}
  Let $\xi<\omega_1$.
  Does there exist a Banach space $X$ and
$(e_n)_{n\in\nn}\in\mathcal{SM}_\xi(X)$ such that
    $(e_n)_{n\in\nn}$ is not equivalent to any $(y_n)_{n\in\nn}$ in $\mathcal{SM}_\zeta(X)$ for all
    $\zeta<\xi$?
\end{problem}
In Chapter \ref{separating k l^1 spr mod space}, Problem
\ref{problem1} is  answered affirmatively for all finite  as well
as  all countable limit ordinals.
\begin{problem}\label{problem2}
  Does for every separable Banach space $X$ exist $\xi<\omega_1$ such that
  for every $\zeta>\xi$,
  $\mathcal{SM}_\zeta(X)=\mathcal{SM}_\xi(X)$?
\end{problem}
Problem \ref{problem2} can be also stated in an isomorphic
version, i.e. every sequence in $\mathcal{SM}_\zeta(X)$ is
equivalent to some sequence in $\mathcal{SM}_\xi(X)$  and vice
versa. Any version of Problem \ref{problem2} remains open.

We close this section by giving an example which shows that for
every $\xi<\omega_1$ there exists a Banach space $X_\xi$ such that
its base generates a spreading model $(e_n)_{n\in\nn}$ of order
$\xi$ while it does not admit $(e_n)_{n\in\nn}$ as a spreading
model of order $\zeta$, for all $\zeta<\xi$.

\begin{examp}
\label{example} Let $(e_n)_{n}$ be a normalized spreading and
$1$-unconditional sequence in a Banach space $(E,\|\cdot\|)$ not
equivalent to the usual basis of $c_0$, $1\leq \xi<\omega_1$ and
$\ff$ be a regular thin family of order $\xi$. We denote by
$(e_s)_{s\in\ff}$ the natural Hamel basis of $c_{00}(\ff)$. For
$x\in c_{00}(\ff)$ we set
\[\|x\|_\ff=\sup\Big{\{}\Big\|\sum_{i=1}^lx(s_i)e_i\Big\|:l\in\nn,(s_i)_{i=1}^l\in\text{\emph{Plm}}_l(\ff)\text{ and }l\leq s_1(1)\Big{\}}\]
We set $X_\ff=\overline{(c_{00}(\ff),\|\cdot\|_\ff)}$ and
$A=\{e_s:s\in\ff\}$. It is easy to see that the sequence
$(e_n)_{n\in\nn}$ belongs to $\mathcal{SM}_\xi(A)$. We shall show
that every $(x_n)_{n\in\nn}\in\mathcal{SM}_\zeta(A)$ with
$\zeta<\xi$, which defines a norm, is isometric to the usual basis
of $c_0$.

Indeed let $\zeta<\xi$ and
$(x_n)_{n\in\nn}\in\mathcal{SM}_\zeta(A)$. Then there exists a
regular thin family $\g$ of order $\zeta$, $N\in[\nn]^\infty$ and
a $\g$-subsequence $(y_s)_{s\in\g\upharpoonright N}$ in $A$ which
generates $(x_n)_{n\in\nn}$ as a $\g$-spreading model. Let
$\varphi:\g\upharpoonright N\to\ff$ such that
$y_t=e_{\varphi(t)}$, for all $t\in\g\upharpoonright N$. Since
$(x_n)_{n\in\nn}$ defines a norm, it is easy to see that $\varphi$
is hereditarily nonconstant. By Proposition \ref{lemma making a
hereditary nonconstant function, nonconstant on plegma pairs}
there exists $L\in [N]^\infty$ such that for every plegma pair
$(t_1,t_2)$ in $\ff\upharpoonright L$,
$\varphi(t_1)\neq\varphi(t_2)$. Theorem \ref{non plegma preserving
maps} yields that there exists $M\in[L]^\infty$ such that for
every plegma pair $(t_1,t_2)$ in $\g\upharpoonright M$ neither
$(\varphi(t_1),\varphi(t_2))$, nor $(\varphi(t_2),\varphi(t_1))$
is a plegma pair. Therefore, for every
$(t_j)_{j=1}^k\in\text{\emph{Plm}}(\g\upharpoonright M)$ and
$(s_j)_{j=1}^l\in\text{\emph{Plm}}(\ff)$ we have that
\[\big|\big\{j\in\{1,\ldots,k\}:t_j\in\{s_i:1\leq i\leq l\}\big\}\big|\leq1\]
This easily yields, by the definition of the norm $\|\cdot\|_\xi$,
that
\[\Big\|\sum_{j=1}^ka_jy_{t_j}\Big\|=\Big\|\sum_{j=1}^ka_je_{\varphi(t_j)}\Big\|=\max_{1\leq j\leq
k}|a_j|\] for all $k\in\nn$, $a_1,\ldots,a_k\in\rr$ and
$(t_j)_{j=1}^k\in\text{\emph{Plm}}(\g\upharpoonright M)$. Thus
$(x_n)_{n\in\nn}$ is isometric to the usual basis of $c_0$.
\end{examp}
\chapter{Spreading sequences}
As we have already mentioned every spreading model of any order of
a Banach space is a spreading sequence. In this chapter we will
present some basic results concerning such sequences that we will
need in the sequel. Most of them are well known (see for example
\cite{AK}, \cite{Lid-Tza}). We classify the spreading sequences
into four categories with respect to their norm properties. These
are the trivial, the unconditional, the singular and the non
unconditional Schauder basic spreading sequences. Finally, in the
last section we provide some results concerning spreading
sequences equivalent to the usual basis of $\ell^1$.

We start by recalling the definition of a spreading sequence in
the slightly more general setting of a seminormed space.
\begin{defn}
  Let $(E,\|\cdot\|_*)$ be a seminormed space and $(e_n)_{n\in\nn}$ be a sequence in $E$.
  Then $(e_n)_{n\in\nn}$ will be called spreading if for every $n\in\nn$, $k_1<\ldots<k_n$ in
$\nn$ and $a_1,\ldots,a_n\in\rr$ we have that $\|\sum_{j=1}^na_j
e_j\|_*=\|\sum_{j=1}^n a_j e_{k_j}\|_*$.
\end{defn}
Notice that if $(e_n)_{n\in\nn}$ is spreading  then for every
$n\in\nn$, $a_1,\ldots,a_n\in\rr$, $k_1<\ldots<k_n$ and
$l_1<\ldots<l_n$ in $\nn$, we have that $\|\sum_{j=1}^na_j
e_{k_j}\|_*=\|\sum_{j=1}^n a_j e_{l_j}\|_*$.
\section{Trivial spreading sequences}
\begin{defn}
  Let $(E,\|\cdot\|_*)$ be a seminormed space and $(e_n)_{n\in\nn}$ be a sequence in $E$.
  Then $(e_n)_{n\in\nn}$ will be called \emph{trivial} if
  \begin{equation}\label{eq14}\Big \|\sum_{i=1}^na_ie_i\Big\|_*=\big{|}\sum_{i=1}^na_i \big{|}\cdot
  \Big\|e_1\Big\|_*\end{equation}
  for all $n\in\nn$ and $a_1,\ldots,a_n\in\rr$.
\end{defn}
It is immediate that  every trivial sequence is spreading.
\begin{prop}\label{pdiff}
Let $(E,\|\cdot\|_*)$ be a seminormed space and $(e_n)_{n\in\nn}$
be a spreading  sequence in $E$. Then the following are
equivalent.
\begin{enumerate}
\item[(i)] The sequence $(e_n)_{n\in\nn}$ is trivial. \item[(ii)]
For every $n,m\in\nn$, $\|e_n-e_m\|_*=0$. \item[(iii)] There exist
$n\in\nn$ and $a_1,\ldots,a_n\in\rr$ not all zero, such that
\[\Big\|\sum_{i=1}^na_ie_i\Big\|_*=0\]
\end{enumerate}
\end{prop}
\begin{proof}
  The implication (i)$\Rightarrow$(ii) is immediate. To show the
  converse let $n\in\nn$ and $a_1,\ldots,a_n\in\rr$. Then we have
  that \[\Big\|\sum_{i=2}^na_i(e_i-e_1)\Big\|_*=0\] Therefore
  \[\begin{split}\Big|\sum_{i=1}^na_i\Big|\cdot\|e_1\|_*&=\Big|\sum_{i=1}^na_i\Big|\cdot\|e_1\|_*-\Big\|\sum_{i=2}^na_i(e_i-e_1)\Big\|_*\\
  &\leq \Big\|\sum_{i=1}^na_ie_1-\sum_{i=2}^na_i(e_i-e_1)\Big\|_*\\
  &\leq \Big|\sum_{i=1}^na_i\Big|\cdot\|e_1\|_*+\Big\|\sum_{i=2}^na_i(e_i-e_1)\Big\|_* =\Big|\sum_{i=1}^na_i\Big|\cdot\|e_1\|_*\end{split}\]
  Since $\sum_{i=1}^na_ie_1-\sum_{i=2}^na_i(e_i-e_1)=
  \sum_{i=1}^na_ie_i$, we get that \[\Big\|\sum_{i=1}^na_ie_i\Big\|_*=\Big|\sum_{i=1}^na_i\Big|\cdot\|e_1\|\]

   The implication  (ii)$\Rightarrow$(iii) is straightforward. To show the converse let   $x=\sum_{j=1}^na_je_j$,
   where $n\in\nn$ and $a_1,\ldots,a_n\in\rr$ not all zero, such that $\|x\|_*=0$. Since
   $(e_n)_{n\in\nn}$ is spreading  we may suppose that $a_j\neq 0$ for all $1\leq j\leq  n$.
  Moreover notice that
  \[\|x\|_*=\Big{\|}\sum_{j=1}^{n-1}a_je_j+a_ne_n\Big{\|}_*=\Big{\|}\sum_{j=1}^{n-1}a_je_j+a_ne_{n+1}\Big{\|}_*=0\]
  which yields that
  \[\Big{\|}e_n-e_{n+1}\Big{\|}_*\leq\frac{1}{|a_n|}
  \Bigg{(} \Big{\|}\sum_{j=1}^{n-1}a_je_j+a_ne_n\Big{\|}_*+\Big{\|}\sum_{j=1}^{n-1}a_je_j+a_ne_{n+1}\Big{\|}_* \Bigg{)}=0 \]
  Since $(e_n)_{n\in\nn}$ is spreading, the result follows.
\end{proof}

\begin{cor}\label{condition for a seminorm on a subsymmetric sequence to be a norm}
  Let $(E,\|\cdot\|_*)$ be a
  seminormed linear space and
  $(e_n)_{n\in\nn}$ be a spreading  linearly independent sequence in $E$. Then the following
  are equivalent:
  \begin{enumerate}
   \item[(i)] The sequence $(e_n)_{n\in\nn}$ is trivial.
   \item[(ii)] There exists $x\in <(e_n)_{n\in\nn}>$ such that $x\neq0$ and
   $\|x\|_*=0$. Therefore  the seminorm $\|\cdot\|_*$  is not a
   norm.
  \end{enumerate}
\end{cor}
\begin{cor}\label{nontrivial}
 Let $(E,\|\cdot\|_*)$ be a
  seminormed linear space and
  $(e_n)_{n\in\nn}$ be a spreading sequence in $E$. Then
  the following are equivalent.
   \begin{enumerate}
   \item[(i)] The sequence $(e_n)_{n\in\nn}$ is nontrivial.
   \item[(ii)]  The sequence $(e_n)_{n\in\nn}$ is linearly
 independent and  the seminorm $\|\cdot\|_*$  restricted on $<(e_n)_{n\in\nn}>$ is  a
   norm.
  \end{enumerate}
\end{cor}

\begin{cor}\label{cor trivial in normed spaces}
Let $(E,\|\cdot\|)$ be a normed space and $(e_n)_{n\in\nn}$ be a spreading
sequence in $E$. Then the following are equivalent.
\begin{enumerate}
  \item[(i)] The sequence $(e_n)_{n\in\nn}$ is trivial.
  \item[(ii)] The sequence $(e_n)_{n\in\nn}$ is constant.
  \item[(iii)] The sequence $(e_n)_{n\in\nn}$ is linearly
  dependent.
\end{enumerate}
\end{cor}

\section{Unconditional spreading sequences}
By Corollary \ref{nontrivial} we have that a nontrivial spreading
sequence $(e_n)_{n\in\nn}$  always generates an infinite
dimensional Banach space. In this section we deal with
unconditional spreading sequences.
 The following result is rather well
known but we include its proof for the sake of completeness.
\begin{prop}\label{l1 vs Cezaro summability}
  Let $(e_n)_{n\in\nn}$ be an unconditional and  spreading sequence. Then either $(e_n)_{n\in\nn}$ is equivalent
  to the usual basis of $\ell^1$ or $(e_n)_{n\in\nn}$ is norm Ces\`aro summable to
  0 (i.e.
$\lim \Big\|\frac{1}{n}\sum_{i=j}^n e_j\Big\|= 0$).
\end{prop}
\begin{proof}
  Since $(e_n)_{n\in\nn}$ is unconditional  there exist
  $c>0$
  such that for every $m\in\nn$ and $a_1,\ldots,a_m\in\rr$,
  \[\Big\|\sum_{i=1}^m\ee_ia_ie_i\Big\|\leq c \Big\|\sum_{i=1}^ma_ie_i\Big\|\]
  for all $\ee_1,\ldots,\ee_m\in\{-1,1\}$.
Also since it is spreading and nontrivial there exists $C>0$ such that
$\|e_n\|=C$, for all $n\in\nn$.
  Suppose that the sequence $(e_n)_{n\in\nn}$ is not Ces\`aro
  summable to zero. Then we will show that $(e_n)_{n\in\nn}$ is
  equivalent to the usual basis of $\ell^1$. Indeed, since
  $(e_n)_{n\in\nn}$ is not Ces\`aro summable to zero, there exist
  $\theta>0$ and a strictly increasing sequence of natural numbers
  $(p_n)_{n\in\nn}$ such that
  $\|\frac{1}{p_n}\sum_{i=1}^{p_n}e_i\|>\theta$, for all
  $i\in\{1,\ldots,p_n\}$. Hence for every $n\in\nn$ there exist
  $x^*_n$ of norm 1 such that
  $x^*_n(\frac{1}{p_n}\sum_{i=1}^{p_n}e_i)>\theta$. For every
  $n\in\nn$, we set $I_n=\{1,\ldots,p_n\}$ and $A_n=\{i\in I_n:\;
  x^*_n(e_i)>\frac{\theta}{2}\}$. Hence for every $n\in\nn$, we
  have that
  \[\begin{split}
    \theta&<x^*_n\Big(\frac{1}{p_n}\sum_{i\in I_n}e_i\Big)
    =\frac{1}{p_n}x^*_n\Big(\sum_{i\in A_n}e_i\Big)
    +\frac{1}{p_n}x^*_n\Big(\sum_{i\in I_n\setminus A_n}e_i\Big)\\
    &\leq\frac{1}{p_n}|A_n|C+\frac{\theta}{2}
  \end{split}\]
  Hence $|A_n|\geq \frac{\theta}{2C}p_n\to\infty$. Let
  $m\in\nn$ and $a_1,\ldots,a_m\in\rr$. Choose $n_0\in\nn$ such that
  $|A_{n_0}|\geq m$. Then
  \[\begin{split}
     C\sum_{i=1}^m|a_i|\geq\Big\|\sum_{i=1}^m a_i e_i\Big\|&
    \geq \frac{1}{c}\Big\|\sum_{i=1}^m|a_i|e_i\Big\|
    = \frac{1}{c}\Big\|\sum_{i=1}^m|a_i|e_{A_{n_0}(i)}\Big\|\\
    &\geq \frac{1}{c}\cdot
    x^*_n\Big(\sum_{i=1}^m|a_i|e_{A_{n_0}(i)}\Big)
    \geq \frac{\theta}{2c} \sum_{i=1}^m|a_i|
  \end{split}\]
\end{proof}
For the proof of the next proposition we refer the reader to
\cite{AK}.
\begin{prop}\label{propuncknown}
  Let $(e_n)_{n\in\nn}$ be a nontrivial spreading sequence.
  If $(e_n)_{n\in\nn}$ is weakly null, then
  $(e_n)_{n\in\nn}$ is 1-unconditional.
\end{prop}
\begin{cor}\label{equiv forms for 1-subsymmetric weakly null}
  Let $(e_n)_{n\in\nn}$ be a nontrivial spreading sequence. Then the following are equivalent.
  \begin{enumerate}
    \item[(i)] The sequence $(e_n)_{n\in\nn}$ is weakly null.
    \item[(ii)] The sequence $(e_n)_{n\in\nn}$ is Ces\'aro
  summable to zero.
  \item[(iii)] The sequence $(e_n)_{n\in\nn}$ is unconditional and
  not equivalent to the usual basis of $\ell^1$.
  \end{enumerate}
\end{cor}
\section{Singular spreading sequences}
A nontrivial spreading sequence which is not a Schauder basis of
the space that it generates will be called \emph {singular}. In
this section we will study the structure of singular sequences. To
this end, some well known results from Banach space theory will be
needed. The first one is  that a weakly convergent Schauder basic
sequence is  weakly null. Notice that this yields that a sequence
which is nontrivial, spreading and weakly convergent to a nonzero
element, is also singular. The second one is that every nontrivial
weak-Cauchy sequence
 (i.e. $x_n\stackrel{w^*}{\to}x^{**}$ with $x^{**}\in X^{**}\setminus X$) contains a Schauder basic subsequence
 (see the proof of Proposition 2.2 in \cite{Ro2}). Finally,
 we will need Rosenthal's $\ell^1$ theorem  \cite{Ro} stating that
 every bounded sequence in a Banach space contains a subsequence which is either equivalent to the usual basis of $\ell^1$ or  weak-Cauchy.

We start with the next lemma.
\begin{lem}\label{lem decomp of subsym weak convergent seq}
  Let $(e_n)_{n\in\nn}$ be a nontrivial and spreading sequence and $E$ be the Banach space generated by  $(e_n)_{n\in\nn}$.
  Suppose that  $(e_n)_{n\in\nn}$ is weakly  convergent to some element $e\in E$ and let $e'_n=e_n-e$, for all $n\in\nn$.
  Then the sequence $(e'_n)_{n\in\nn}$ is  nontrivial, spreading, $1$-unconditional and Ces\'aro summable to zero.
\end{lem}
\begin{proof}
  Since $\|e_n-e_m\|=\|e'_n-e'_m\|$, by Proposition \ref{pdiff} we have  that $(e'_n)_{n\in\nn}$ is nontrivial. To  show that  $(e'_n)_{n\in\nn}$ is spreading, let
$n\in\nn$, $\lambda_1,...,\lambda_n\in\rr$ and $k_1<...<k_n$
in $\nn$.

If $\sum_{i=1}^n\lambda_i=0$, then we have that
\[\sum_{i=1}^n\lambda_ie_i=\sum_{i=1}^n\lambda_ie'_i\;\text{and}\;\sum_{i=1}^n\lambda_ie_{k_i}=\sum_{i=1}^n\lambda_ie'_{k_i}\]
Therefore, since $(e_n)_{n\in\nn}$ is spreading, we have that
\[\Big\|\sum_{i=1}^n\lambda_ie'_i\Big\|=\Big\|\sum_{i=1}^n\lambda_ie_i\Big\|=
\Big\|\sum_{i=1}^n\lambda_ie_{k_i}\Big\|=\Big\|\sum_{i=1}^n\lambda_ie'_{k_i}\Big\|\]
Hence for every $n\in\nn$ and $\lambda_1,\ldots,\lambda_n$ with
$\sum_{i=1}^n\lambda_i=0$ we have that
\[\Big\|\sum_{i=1}^n\lambda_ie'_{k_i}\Big\|=\Big\|\sum_{i=1}^n\lambda_ie'_{l_i}\Big\|\]
for all $k_1<\ldots<k_n$ and $l_1<\ldots<l_n$.

Generally let $\sum_{i=1}^n\lambda_i=\lambda$. Since
$(e'_n)_{n\in\nn}$ is weakly null we may choose a convex block
subsequence $(w_m)_{m\in\nn}$ of $(e'_n)_{n\in\nn}$ which norm
converges to zero. Let $m_0\in\nn$ be such that
$k_n<\text{supp}(w_m)$ for all $m\geq m_0$. Then by the above case
we have that for all $m\geq m_0$,
\[\Big\|\sum_{i=1}^n\lambda_ie'_i-\lambda w_m\Big\|=\Big\|\sum_{i=1}^n\lambda_ie'_{k_i}-\lambda
w_m\Big\|\] and therefore by taking limits we get that
\[\Big\|\sum_{i=1}^n\lambda_ie'_i\Big\|=\Big\|\sum_{i=1}^n\lambda_ie'_{k_i}\Big\|\]
Hence the sequence   $(e'_n)_{n\in\nn}$ is spreading. Since $(e_n)_{n\in\nn}$ is weakly null it is  Ces\'aro summable to zero and   by Proposition \ref{propuncknown},
it is $1$-unconditional.
\end{proof}
\begin{prop}\label{decomp of singular spr mod}
Let $(e_n)_{n\in\nn}$ be a singular sequence and let $E$ be the
Banach space generated by $(e_n)_{n\in\nn}$. Then there is $e\in
E\setminus \{0\}$ such that $(e_n)_{n\in\nn}$ is weakly convergent
to $e$. Moreover setting $e'_n=e_n-e$, we have that
$(e'_n)_{n\in\nn}$ is nontrivial, spreading, $1$-unconditional and
Ces\`aro summable to zero.
\end{prop}
\begin{proof} Since $(e_n)_{n\in\nn}$ is equivalent to all its subsequences and
it is not Schauder basic, we have that it cannot contain a
Schauder basic subsequence. Hence  $(e_n)_{n\in\nn}$ cannot also
contain  a non trivial weak-Cauchy subsequence. Moreover, for the
same reasons, it cannot contain a weakly null subsequence. By
Rosenthal's $\ell^1$-theorem and taking into account the above
remarks, we conclude that $(e_n)_{n\in\nn}$ is weakly convergent
to a non zero element $e\in E$. Hence by Lemma \ref{lem decomp of
subsym weak convergent seq}, $(e'_n)_{n\in\nn}$ is nontrivial,
spreading $1$-unconditional and Ces\'aro summable to
zero.
\end{proof}
By the above we have the following.
\begin{cor}\label{thmsingular}
  Let $(e_n)_{n\in\nn}$ be a nontrivial and spreading sequence. Then the following are equivalent.
  \begin{enumerate}
    \item[(i)] The sequence $(e_n)_{n\in\nn}$ is singular.
    \item[(ii)] The sequence $(e_n)_{n\in\nn}$ is weakly convergent to a nonzero element.
  \end{enumerate}
\end{cor}
The above decomposition of a singular sequence $(e_n)_{n\in\nn}$ as $e_n=e'_n+e$, where $e$ is the weak limit of $(e_n)_{n\in\nn}$ will be called the \emph{natural decomposition} of $(e_n)_{n\in\nn}$. The next proposition describes the relation of the spaces generated by the sequences $(e_n)_{n\in\nn}$ and $(e'_n)_{n\in\nn}$.
\begin{prop}\label{singular splitted}
  Let $(e_n)_{n\in\nn}$ be a singular sequence  and $e_n=e'_n+e$ be  the natural decomposition of $(e_n)_{n\in\nn}$.
   Let $E$ (resp. $E'$) be the Banach space generated by $(e_n)_{n\in\nn}$ (resp. $(e'_n)_{n\in\nn}$). Then $E=E'\oplus<e>$.
   In particular there exists a norm one projection $P:E\to <e>$ such that for every $z=\sum_{j=1}^na_je_j\in <(e_n)_{n\in\nn}> $,
    $P(z)=\sum_{j=1}^na_je$.
\end{prop}
\begin{proof} By Corollary \ref{nontrivial}, the sequence $(e_n)_{n\in\nn}$ is linearly independent. Therefore we may
define the map $P':<(e_n)_{n\in\nn}>\to<e>$ by
$P'(\sum_{i=1}^na_ie_i)=\sum_{i=1}^na_ie$. Clearly $P'$ is linear.
We will show that $\|P'\|\leq1$. Indeed, let $z
\in<(e_n)_{n\in\nn}>$, $z=\sum_{i=1}^na_ie_i$. For every
$k\in\nn$, we set $z_k=\sum_{i=1}^na_ie_{k+j}$.
  Since $(e_n)_{n\in\nn}$ is spreading, we have that $\|z_k\|=\|z\|$, for all
  $k\in\nn$. Moreover, since $e_n\stackrel{w}{\to}e$, we have also
  that $z_k\stackrel{w}{\to}P'(z)$. Hence
  \[\|P'(z)\|\leq\liminf_{k\to\infty}\|z_k\|=\|z\|\] Hence $\|P'\|\leq 1$.

   Since $<(e_n)_{n\in\nn}>$ is dense in $E$, the operator $P'$ has a unique linear   extension
   $P:E\to<e>$  with $\|P\|\leq1$. Moreover, since $e_n\stackrel{w}{\to}e$, we have that $e=P'(e_n)=P(e_n)\stackrel{w}{\to}P{e}$
    and therefore $P(e)=e$. The latter yields that $P$ is a norm one projection.

To complete the proof we need to show that
     $\text{ker}P=E'$. First notice that
      for every $z=\sum_{i=1}^ka_ie'_i\in<(e'_n)_{n\in\nn}>$, we have that
  \[P(z)=P\Big(\sum_{i=1}^ka_ie'_i\Big)=P\Big(\sum_{i=1}^ka_ie_i\Big)-P\Big(\sum_{i=1}^ka_ie\Big)
  =\sum_{i=1}^ka_ie-\sum_{i=1}^ka_ie=0\]
  Hence  $E'\subseteq \text{ker}P$.

   To show the converse inclusion  let $z\in E$ with  $P(z)=0$. Choose  a sequence $(z_k)_{k\in\nn}$ in $<(e_n)_{n\in\nn}>$
   which norm converges  to $z$.  By the definition of $P'$, it is easy to see
  that  for every $w\in <(e_n)_{n\in\nn}>$,  $w-P(w)\in E'$.  Therefore  for all $k\in\nn$,  $z_k-P(z_k)\in E'$ and since $(z_k-P(z_k))_{k\in\nn}$
norm converges to $z$, we have that $z\in E'$. Hence $\text{ker}P\subseteq E'$, which completes the proof.
\end{proof}
\begin{cor}
  \label{isomorphy_between_E_and_E'}
  Let $(e_n)_{n\in\nn}$ be a singular sequence  and $e_n=e'_n+e$ be  the natural decomposition of $(e_n)_{n\in\nn}$.
   Let $E$ (resp. $E'$) be the Banach space generated by $(e_n)_{n\in\nn}$ (resp. $(e'_n)_{n\in\nn}$). Then $E$ and $E'$ are isomorphic.
\end{cor}

\section{Non unconditional Schauder basic spreading sequences}
\begin{prop} Let $(e_n)_{n\in\nn}$ be a spreading nontrivial
sequence and $E$ be the Banach space generated by $(e_n)_{n\in\nn}$.
Then the following are equivalent.
\begin{enumerate}
  \item[(i)] The sequence $(e_n)_{n\in\nn}$ is  a non
  unconditional Schauder basis of $E$.
  \item[(ii)] The sequence $(e_n)_{n\in\nn}$ is nontrivial weak-Cauchy.
\end{enumerate}
\end{prop}
\begin{proof}
$(i)\Rightarrow (ii)$: Since $(e_n)_{n\in\nn}$ is Schauder basic,
by Corollary \ref{thmsingular} we have that $(e_n)_{n\in\nn}$ cannot
contain a weakly convergent subsequence to a non zero element of
$E$. Moreover by Proposition \ref{propuncknown},  we have that
every subsequence of $(e_n)_{n\in\nn}$ is not weakly convergent to
zero. Hence by Rosenthal's theorem we have that $(e_n)_{n\in\nn}$
contains a non trivial weak-Cauchy subsequence. Since
$(e_n)_{n\in\nn}$ is equivalent to all its subsequences we get
that the whole sequence $(e_n)_{n\in\nn}$ is  non trivial weak-Cauchy.

$(ii)\Rightarrow (i)$: Since $(e_n)_{n\in\nn}$ is non trivial weak-Cauchy we have that it contains a Schauder basic subsequence and
hence  $(e_n)_{n\in\nn}$  is Schauder basic. Moreover, since it is
non trivial weak-Cauchy, it is not equivalent to the standard
basis of $\ell^1$ and  by Corollary \ref{equiv forms for
1-subsymmetric weakly null}, $(e_n)_{n\in\nn}$ cannot be
unconditional.
\end{proof}
\begin{rem}
  Since every non unconditional  Schauder basic spreading sequence is nontrivial
  weak-Cauchy, by Proposition 2.2 in \cite{Ro2}, we have that every such   sequence dominates the summing basis of $c_0$.
\end{rem}
\section{Splitting $\ell^1$ spreading sequences}
In this section we study some stability properties of spreading
sequences in seminormed spaces which are actually related to the
non distortion of $\ell^1$ (c.f. \cite{J2}).
\begin{prop}\label{propsplitl1}
  Let $(E,\|\cdot\|_\circ),(E_1,\|\cdot\|_*),(E_2,\|\cdot\|_{**})$ be seminormed spaces. Let   $(e_n)_{n\in\nn},(e_n^1)_{n\in\nn}$ and $(e_n^2)_{n\in\nn}$ be spreading sequences in $E, E_1$ and $E_2$ respectively. Suppose that
  $(e_n)_{n\in\nn}$ admits lower $\ell^1$-estimate with constant $c>0$ (i.e. for every $n\in\nn$ and $a_1,\ldots,a_n\in\rr$ we have that
  $c\sum_{i=1}^n |a_i|\leq \|\sum_{i=1}^n a_i e_i \|_\circ$).
  Also assume that
  \[\|\sum_{i=1}^na_ie_i\|_\circ\leq \|\sum_{i=1}^na_ie_i^1\|_*+\|\sum_{i=1}^na_ie_i^2\|_{**}\]
  for all $n\in\nn$ and $a_1,\ldots,a_n\in\rr$.
  If $(e_n^2)_{n\in\nn}$
  does not admit a lower $\ell^1$-estimate then $(e_n^1)_{n\in\nn}$ admits a lower $\ell^1$-estimate with constant $c$.
\end{prop}
\begin{proof}
  Suppose on the contrary that $(e_n^1)_{n\in\nn}$ does not admit a lower $\ell^1$-estimate with constant $c$. Hence
   there exist $n\in\nn$ and $a_1,\ldots,a_n\in\rr$ with $\sum_{i=1}^n|a_i|=1$ such that $\|\sum_{i=1}^n a_i e_i^1\|_*<c-\varepsilon$,
    for some $\varepsilon >0$. Since  $(e_n^1)_{n\in\nn}$ is spreading, we have that for every $k_1<\ldots<k_n$ in $\nn$
  \[\|\sum_{i=1}^n a_i e_{k_i}^1\|_*<c-\varepsilon\]
  Since $(e_n^2)_{n\in\nn}$ does not admit a lower $\ell^1$-estimate, there exist $m\in\nn$ and $b_1,\ldots,b_m\in\rr$ such that
   $\sum_{j=1}^m|b_j|=1$ and
  $\| \sum_{j=1}^mb_j e_j^2 \|_{**}<\frac{\varepsilon}{2}$. Similarly since  $(e_n^2)_{n\in\nn}$ is spreading we have that
  for every  $k_1<\ldots<k_m$ in $\nn$
  \[\| \sum_{j=1}^mb_j e_{k_j}^2 \|_{**}<\frac{\varepsilon}{2}\]
  So we have the following two inequalities
  \[\begin{split}
    \Big{\|}\sum_{i=1}^n\sum_{j=1}^m a_i\cdot b_j e^1_{(i-1)m +j}\Big{\|}_*& \leq \sum_{j=1}^m |b_j|\Big{\|}
     \sum_{i=1}^na_i e^1_{(i-1)m +j}\Big{\|}_*
    < \sum_{j=1}^m |b_j|\cdot(c-\varepsilon)=c-\varepsilon
  \end{split}\]
  and
  \[\begin{split}
    \Big{\|}\sum_{i=1}^n\sum_{j=1}^m a_i\cdot b_j e^2_{(i-1)m +j}\Big{\|}_{**}& \leq
     \sum_{i=1}^n|a_i|\Big{\|}\sum_{j=1}^m b_j  e^2_{(i-1)m +j}\Big{\|}_{**}
    < \sum_{j=1}^m |b_j|\frac{\varepsilon}{2}=\frac{\varepsilon}{2}
  \end{split}\]
  By the assumptions of the proposition we get that
  \[\begin{split}\Big{\|}\sum_{i=1}^n\sum_{j=1}^m a_i\cdot b_j e_{(i-1)m +j}\Big{\|}_\circ
  &\leq  \Big{\|}\sum_{i=1}^n\sum_{j=1}^m a_i\cdot b_j e^1_{(i-1)m +j}\Big{\|}_*\\
  &+ \Big{\|}\sum_{i=1}^n\sum_{j=1}^m a_i\cdot b_j e^2_{(i-1)m +j}\Big{\|}_{**}
  <c-\frac{\varepsilon}{2}\end{split}\]
 which since $\sum_{i=1}^n\sum_{j=1}^m |a_i|\cdot|b_j|=1$, contradicts that $(e_n)_{n\in\nn}$ admits
  a lower $\ell^1$-estimate with constant $c$.
\end{proof}
In connection with the spreading models we have the following.
\begin{cor}\label{ultrafilter property for ell^1 spreading models}
  Let $\ff$ be a regular thin family and $(x_s)_{s\in\ff},(x_s^1)_{s\in\ff},(x_s^2)_{s\in\ff}$
  be three $\ff$-sequences in a Banach space $X$ such that  for all $s\in\ff$,
  $x_s=x_s^1+x_s^2$.
   Let $M\in[\nn]^\infty$ and
  $(e_n)_{n\in\nn},(e_n^1)_{n\in\nn},(e_n^2)_{n\in\nn}$ generated by $(x_s)_{s\in\ff\upharpoonright M},(x_s^1)_{s\in\ff\upharpoonright
  M}$ and
  $(x_s^2)_{s\in\ff\upharpoonright M}$ respectively
  as $\ff$-spreading models. Suppose that
  $(e_n)_{n\in\nn}$ admits a lower $\ell^1$-estimate with constant $c>0$. If $(e_n^2)_{n\in\nn}$
  does not admit a lower $\ell^1$-estimate then $(e_n^1)_{n\in\nn}$ admits a lower $\ell^1$-estimate with constant $c$.
\end{cor}
\begin{proof} Let  $(E,\|\cdot\|_\circ),(E_1,\|\cdot\|_*),(E_2,\|\cdot\|_{**})$ be the seminormed spaces with Hamel bases  $(e_n)_{n\in\nn},(e_n^1)_{n\in\nn}$ and $(e_n^2)_{n\in\nn}$ respectively.
  Let $n\in\nn$ and  $a_1,\ldots,a_n\in\rr$. Then for every $(s_j)_{j=1}^n\text{\emph{Plm}}(\ff\upharpoonright M)$, we have that
  \[\Big\|\sum_{j=1}^n a_jx_{s_j}\Big\|\leq \Big\|\sum_{j=1}^n a_jx^1_{s_j}\Big\|+ \Big\|\sum_{j=1}^n a_jx^2_{s_j}\Big\|\] This easily yields that
  $\displaystyle\Big\|\sum_{j=1}^n a_je_j\Big\|_\circ\leq \Big\|\sum_{j=1}^n a_je^1_j\Big\|_*+\Big\|\sum_{j=1}^n a_je_j\Big\|_{**}$. The result now  follows by Proposition \ref{propsplitl1}.
\end{proof}
Corollary \ref{ultrafilter property for ell^1 spreading models}
also yields the following.
\begin{cor}\label{ultrafilter property for ell^1 spr mod simplified}
  Let $\ff$ be a regular thin family, $x\in X$ and $(x_s)_{s\in\ff},(x'_s)_{s\in\ff}$
  be two $\ff$-sequences in a Banach space $X$ such for all $s\in\ff$,
  $x_s=x'_s+x$.
   Let $M\in[\nn]^\infty$ and
  $(e_n)_{n\in\nn},(e'_n)_{n\in\nn}$ generated by $(x_s)_{s\in\ff\upharpoonright M}$ and
  $(x'_s)_{s\in\ff\upharpoonright M}$ respectively
  as $\ff$-spreading models. Then $(e_n)_{n\in\nn}$ is equivalent to the usual basis of $\ell^1$ if and only
   if $(e'_n)_{n\in\nn}$ is equivalent to the usual basis of $\ell^1$.
\end{cor}

\chapter{$\ff$-sequences in topological spaces}
In this chapter we deal with $\ff$-sequences in topological
spaces. In the first two sections we extend the classical
definition of convergent and Cauchy sequences and we present some
related results. In the third section we proceed to the study of
$\ff$-sequences with relatively compact metrizable range.
Specifically, the concept of the \emph{subordinated}
$\ff$-sequence is introduced in order to capture the following
fact: For every map $\varphi:\ff\to (X,\mathcal{T})$, with
$\overline{\varphi[\ff]}$ compact metrizable, and for every
$M\in[\nn]^\infty$ there exist $L\in[M]^\infty$ and a continuous
extension $\widehat{\varphi}:\widehat{\ff}\upharpoonright
L\to(X,\mathcal{T})$.

\section{Convergence of $\ff$-sequences}
\begin{defn}\label{defn convergence of f-sequences}
  Let $(X,\ttt)$ be a topological space, $\ff$ a regular thin family, $M\in[\nn]^\infty$, $x_0\in X$ and
  $(x_s)_{s\in\ff}$ an $\ff$-sequence in $X$. We will say that the
  $\ff$-subsequence $(x_s)_{s\in\ff\upharpoonright M}$ converges
  to $x_0$ if for every $U\in\ttt$ with $x_0\in U$ there exists
  $m\in \nn$ such that for every $s\in \ff\upharpoonright M$ with
  $\min s\geq M(m)$ we have that $x_s\in U$.
\end{defn}
\begin{rem}\label{rem on convergence in topological spaces}$\;$
\begin{enumerate}
  \item [(i)] It is immediate that if an $\ff$-subsequence
  $(x_s)_{s\in\ff\upharpoonright M}$ in a topological space is
  convergent to some $x_0$, then every further $\ff$-subsequence
  is also convergent to $x_0$.
  \item [(ii)] It is also easy to see that for families $\ff$ with
  $o(\ff)\geq 2$,  the convergence of an $\ff$-subsequence
  $(x_s)_{s\in\ff\upharpoonright M}$ does not in general imply the
  corresponding relative compactness of its range.
\end{enumerate}
\end{rem}
\begin{prop}\label{arxi metaforas gia ff sequences}
  Let $(X,\ttt_X)$, $(Y,\ttt_Y)$ be two topological spaces and
  $f:Y\to X$ be a continuous map. Let $\ff$ be a regular thin
  family, $M\in[\nn]^\infty$ and $(y_s)_{s\in\ff}$ an
  $\ff$-sequence in $Y$. Suppose that the $\ff$-subsequence $(y_s)_{s\in\ff\upharpoonright
  M}$ is convergent to some $y\in Y$. Then the $\ff$-subsequence $(f(y_s))_{s\in\ff\upharpoonright
  M}$ is convergent to $f(y)$.
\end{prop}
\begin{proof}
  Let $U_X\in\ttt_X$, with $f(y)\in U_X$. By the continuity of $f$
  there exists $U_Y\in\ttt_Y$ such that $y\in U_Y$ and $f[U_Y]\subseteq
  U_X$. Since $(y_s)_{s\in\ff\upharpoonright
  M}$ is convergent to $y$, there exists $m\in\nn$ such that for every $s\in \ff\upharpoonright M$ with
  $\min s\geq M(m)$ we have that $y_s\in U_Y$ and therefore $f(y_s)\in f[U_Y]\subseteq
  U_X$.
\end{proof}

\begin{defn}\label{defn Chauchy of f-sequences}
  Let $(X,\rho)$ be a metric space, $\ff$ a regular thin family, $M\in[\nn]^\infty$ and
  $(x_s)_{s\in\ff}$ an $\ff$-sequence in $Y$. We will say that the
  $\ff$-subsequence $(x_s)_{s\in\ff\upharpoonright M}$ is Cauchy
  if for every $\ee>0$ there exists $m\in\nn$ such that for every
  $s_1,s_2\in\ff\upharpoonright M$ with $\min s_1,\min s_2\geq
  M(m)$, we have that $\rho(x_{s_1},x_{s_2})<\ee$.
\end{defn}

\begin{prop}\label{remark on ff Cauchy-convergent sequences and their spreading model}
  Let $M\in[\nn]^\infty$,  $\ff$ be a regular thin family and  $(x_s)_{s\in\ff}$ be an $\ff$-sequence
  in a complete metric space $(X,\rho)$. Then  the $\ff$-subsequence $(x_s)_{s\in\ff\upharpoonright M}$ is Cauchy
  if and only if $(x_s)_{s\in\ff\upharpoonright M}$ is  convergent.
\end{prop}
\begin{proof}
  If the $\ff$-subsequence $(x_s)_{s\in\ff\upharpoonright M}$ is  convergent, then it is straightforward that $(x_s)_{s\in\ff\upharpoonright M}$
  is Cauchy. Concerning the converse we have the following. Suppose that the $\ff$-subsequence $(x_s)_{s\in\ff\upharpoonright M}$
  is  Cauchy. Let $(s_n)_{n\in\nn}$ be a sequence in $\ff\upharpoonright M$ such that $\min s_n\to \infty$.
  It is immediate that $(x_{s_n})_{n\in\nn}$ forms a  Cauchy sequence in $X$. Since $(X, \rho)$ is complete, there exists $x\in  X$ such that
  the sequence $(x_{s_n})_{n\in\nn}$ converges to $x$. We will show that the $\ff$-subsequence $(x_s)_{s\in\ff\upharpoonright M}$ converges to $x$.
  Indeed, let $\varepsilon>0$. Since $(x_s)_{s\in\ff\upharpoonright M}$
  is Cauchy, there exists $k_0\in\nn$ such that for every $t_1,t_2\in\ff \upharpoonright M$ with $\min t_1,\min t_2\geq M(k_0)$ we have that
  $\rho(x_{t_1},x_{t_2})<\frac{\varepsilon}{2}$. Since the sequence $(x_{s_n})_{n\in \nn}$ converges to $x$ and $\min s_n\to \infty$,
  there exists $n_0\in\nn$ such that $\min s_{n_0}\geq M(k_0)$ and $\rho(x,x_{s_{n_0}})<\varepsilon/2$. Hence for every $s\in\ff\upharpoonright M$
  such that $\min s\geq M(k_0)$, we have that
  $\rho(x,x_s)\leq \rho(x, x_{s_{n_0}})+\rho(x_{s_{n_0}},
  x_s)<\varepsilon$
  and the proof is completed.
\end{proof}
\section{Plegma $\varepsilon$-separated $\ff$-sequences}
\begin{lem}\label{Lemma epsilon near admissble yield Cauchy}
  Let $M\in[\nn]^\infty$,  $\ff$ be a regular thin family and $(x_s)_{s\in\ff}$ an $\ff$-sequence
  in a metric space $(X, \rho)$. Suppose that for every $\varepsilon>0$ and every $L\in[M]^\infty$ there exists a plegma pair
  $(s_1,s_2)$ in $\ff\upharpoonright L$ such that  $\rho(x_{s_1}, x_{s_2})<\varepsilon$. Then the  $\ff$-subsequence
   $(x_s)_{s\in\ff\upharpoonright M}$ has a  further Cauchy subsequence.
\end{lem}
\begin{proof}
  Let $(\varepsilon_n)_{n\in\nn}$ be a sequence of positive reals such that $\sum_{n=1}^\infty \varepsilon_n<\infty$. Using Theorem
  \ref{ramseyforplegma}, we inductively
  construct a decreasing sequence $(L_n)_{n\in\nn}$ in $[M]^\infty$, such that for every $n\in\nn$ and for every plegma
  pair $(s_1,s_2)$ in $\ff\upharpoonright L_n$ we have that $\rho(x_{s_1}, x_{s_2})<\varepsilon_n$. Let $L'$ be
  a diagonalization of $(L_n)_{n\in\nn}$, i.e. $L'(n)\in L_n$ for all $n\in\nn$, and $L=\{L'(2n):n\in\nn\}$.

  We claim that the $\ff$-subsequence
  $(x_s)_{s\in\ff\upharpoonright L}$ is Cauchy. Indeed let $\varepsilon>0$. There exists $n_0\in \nn$ such that
  $\sum_{n=n_0}^\infty\varepsilon_n<\frac{\varepsilon}{2}$. Let $s_0$ be the unique initial segment of $\{L(n):n\geq n_0\}$ in $\ff$.
  If $\max s_0=L(k)$ then we set $k_0=k+1$. Then for every $s_1,s_2\in\ff\upharpoonright L$ with $\min s_1, \min s_2\geq L(k_0)$,
  by Proposition \ref{accessing everything with plegma path of length |s_0|} there exist plegma paths $(s_j^1)_{j=1}^{|s_0|},
  (s_j^2)_{j=1}^{|s_0|}$ in $\ff\upharpoonright L'$ from $s_0$ to $s_1,s_2$ respectively. Then for $i=1,2$ we have that
  \[\rho(x_{s_0}, x_{s_i})\leq\sum_{j=0}^{|s_0|-1}\rho(x_{s_j^i},x_{s_{j+1}^i})<\sum_{j=0}^{|s_0|-1}\varepsilon_{n_0+j}<\frac{\varepsilon}{2}\]
  which implies that $\rho(x_{s_1}, x_{s_2})<\varepsilon$.
\end{proof}
\begin{defn} Let $\varepsilon>0$, $L\in[\nn]^\infty$,  $\ff$ be a regular thin family and $(x_s)_{s\in\ff}$ an $\ff$-sequence
   in a metric space $X$. We will say that the subsequence $(x_s)_{s\in\ff\upharpoonright L}$ is plegma $\varepsilon$-separated
    if for every plegma pair $(s_1,s_2)$ in $\ff\upharpoonright L$,  $\rho(x_{s_1}, x_{s_2})>\varepsilon$.
\end{defn}
The following proposition is actually a restatement of Lemma
\ref{Lemma epsilon near admissble yield Cauchy} under the above
definition.
\begin{prop}\label{Proposition equivalence of non Cauchy and separability}
  Let $M\in[\nn]^\infty$,  $\ff$ be a regular thin family and $(x_s)_{s\in\ff}$ an $\ff$-sequence
  in a metric space $X$. Then the  following are equivalent.
  \begin{enumerate}
    \item [(i)] The $\ff$-subsequence $(x_s)_{s\in\ff\upharpoonright M}$ has no  further Cauchy subsequence.
    \item [(ii)] For every $N\in [M]^\infty$ there exists $L\in [N]^\infty$ and $\varepsilon>0$ such that the
  subsequence $(x_s)_{s\in\ff\upharpoonright L}$ is plegma $\varepsilon$-separated.
  \end{enumerate}
\end{prop}

\section{Subordinated $\ff$-sequences and convergent
$\widehat{\ff}$-trees}
  \begin{defn}
    Let $(X,\ttt)$ be a topological space, $\ff$ be a regular thin family, $M\in[\nn]^\infty$ and
  $(x_s)_{s\in\ff}$ be an $\ff$-sequence in $X$. We say
    that $(x_s)_{s\in\ff\upharpoonright M}$ is subordinated (with respect to $(X,\mathcal{T})$)
     if there exists a continuous map $\widehat{\varphi}:\widehat{\ff}\upharpoonright M\to (X,\mathcal{T})$
with  $\widehat{\varphi}(s)=x_s$, for all $s\in\ff\upharpoonright
M$.

 \end{defn}
  \begin{rem}\label{remark on ff subordinated}
    If $(x_s)_{s\in\ff\upharpoonright M}$ is subordinated, then there exists a unique  continuous map
     $\widehat{\varphi}:\widehat{\ff}\upharpoonright M\to (X,\mathcal{T})$ witnessing this.
      This is a consequence of the fact that $\ff$ is dense in $\widehat{\ff}$.
    Also for the same reason we have that
    \[\overline{\{x_s:s\in\ff\upharpoonright M\}}=\widehat{\varphi}[\widehat{\ff}\upharpoonright
    M],\]
    where $\overline{\{x_s:s\in\ff\upharpoonright M\}}$ is the $\mathcal{T}$-closure of
    $\{x_s : s\in\ff\upharpoonright M\}$ in $X$. Therefore
    $\overline{\{x_s:s\in\ff\upharpoonright M\}}$ is a countable compact metrizable subspace of $(X,\ttt)$
     with Cantor-Bendixson index at most $o(\ff)+1$.
    Moreover notice that if $(x_s)_{s\in\ff\upharpoonright M}$ is subordinated then for every $L\in [M]^\infty$,
     we have that $(x_s)_{s\in\ff\upharpoonright L}$ is subordinated.
  \end{rem}
The following if an immediate consequence of Proposition \ref{arxi
metaforas gia ff sequences}.
\begin{cor}\label{subordinating yields convergence}
  Let $(X,\ttt)$ be a topological space, $\ff$ a regular thin family, $M\in[\nn]^\infty$ and
  $(x_s)_{s\in\ff}$ an $\ff$-sequence in $X$. Suppose that the
  $\ff$-subsequence $(x_s)_{s\in\ff\upharpoonright M}$ is
  subordinated. Let $\widehat{\varphi}:\widehat{\ff}\upharpoonright M\to
  (X,\mathcal{T})$ be the continuous map witnessing this. Then $(x_s)_{s\in\ff\upharpoonright
  M}$ is convergent to $\widehat{\varphi}(\emptyset)$.
\end{cor}
A related notion here is that of a \emph{convergent
$\widehat{\ff}$-tree}.
\begin{defn}
  Let  $(X,\ttt)$ be a topological space, $\ff$ be a regular thin
  family, $M\in[\nn]^\infty$ and
  $(x_t)_{t\in\widehat{\ff}}$ be a family of elements of $X$. Let
  for every $t\in\widehat{\ff}\setminus\ff$, $N_t=\{n\in M: t\cup\{n\}\in\widehat{\ff}\upharpoonright
  M\}$. We
  say that $(x_t)_{t\in\widehat{\ff}\upharpoonright M}$ is a convergent $\widehat{\ff}$-tree, if
  for every $t\in\widehat{\ff}\setminus\ff$, we have that the sequence $(x_{t\cup\{n\}})_{n\in N_t}$
  converges to $x_t$. Moreover if $(x_t)_{t\in\widehat{\ff}\upharpoonright M}$ is a convergent
  $\widehat{\ff}$-tree then we will say  that $(x_s)_{s\in\ff\upharpoonright
  M}$ admits a convergent $\widehat{\ff}$-tree.
  \end{defn}

It is easy to see that if an $\ff$-subsequence
$(x_s)_{s\in\ff\upharpoonright M}$ is subordinated then the family
$(x_t)_{t\in\widehat{\ff}\upharpoonright M}$ defined by
$x_t=\widehat{\varphi}(t)$, for all
$t\in\widehat{\ff}\upharpoonright M$ is a convergent $\ff$-tree.
 The next proposition yields the converse by passing to an infinite
subset of $M$.

\begin{prop}\label{prpconv}
   Let $(X,\mathcal{T})$ be a topological space, $\ff$ be  a  regular
    thin family, $M\in[\nn]^\infty$ and $(x_t)_{t\in\widehat{\ff}\upharpoonright M}$ be a convergent $\widehat{\ff}$-tree in
    $X$.
    Then  there exists $L\in[M]^\infty$
    such that the map  $\widehat{\varphi}:\widehat{\ff}\upharpoonright L\to X$ defined by  $\widehat{\varphi}(t)=x_t$, for all
    $t\in\widehat{\ff}\upharpoonright L$ is continuous.
    Consequently  there exists
    $L\in[M]^\infty$ such that $(x_s)_{s\in\ff\upharpoonright L}$ is
    subordinated.
\end{prop}
\begin{proof}
    We will use induction on the order of $\ff$. If $o(\ff)=1$, the result is immediate. Assume that for some $\xi<\omega_1$,
     the conclusion of the proposition holds for every family of order lower than $\xi$.
    Let $\ff$ be a regular
    thin family of order $\xi$ and $M\in[\nn]^\infty$.
    For every $n\in M$ we have that $o(\ff_{(n)})<o(\ff)=\xi$. Let $L_0=M$ and $l_1=\min L_0$.
    For every $t'\in \widehat{\ff_{(l_1)}}$, we set $x^1_{t'}= x_{\{l_1\}\cup t'}$.
    Notice that $(x^1_{t'})_{t'\in \widehat{\ff_{(l_1)}}\upharpoonright
    L_0}$ is a convergent  $\widehat{\ff}$-tree.
    By the inductive hypothesis there exists $L_1\in[L_0\setminus \{l_1\}]^\infty$ such that
     the map  $\widehat{\varphi}_1:\widehat{\ff_{(l_1)}}\upharpoonright
    L_1\to X$, with $\widehat{\varphi}_1(t')=x^1_{t'}$ for all
    $t'\in\widehat{\ff_{(l_1)}}\upharpoonright
    L_1$, is  continuous.
     Hence $\overline{\{x^1_s:\;s\in\ff_{(l_1)}\upharpoonright L_1\}}=\{x^1_{t'}: t'\in\widehat{\ff_{(l_1)}}\upharpoonright
    L_1\}$ is
     compact metrizable.  Using the compact coloring version of Theorem \ref{Galvin prikry} and passing to an infinite subset of
     $L_1$ if it is nessecary, we may also suppose that
     $\text{diam}(\{x^1_{t'}:\;t'\in\widehat{\ff_{(l_1)}}\upharpoonright L_1\})=\text{diam}(\{x^1_s:\;s\in\ff_{(l_1)}\upharpoonright L_1\})<1$.
      Proceeding in the same way,
      we construct by induction a strictly increasing sequence
    $(l_n)_{n\in\nn}$, a decreasing sequence $(L_n)_{n=0}^\infty$ of infinite subsets of $\nn$, such that setting
    $x^n_{t'}=x_{\{l_n\}\cup t'}$ for all $t'\in\widehat{\ff_{(l_n)}}\upharpoonright L_n$ the following are satisfied
    \begin{enumerate}
      \item[(i)] For every $n\in\nn$, $l_n=\min L_{n-1}$ and $L_n\in[L_{n-1}\setminus\{l_n\}]^\infty$.
      \item[(ii)] The  map $\widehat{\varphi}_n:\widehat{\ff_{(l_n)}}\upharpoonright
    L_n\to X$, with $\widehat{\varphi}_n(t')=x^n_{t'}$ for all
    $t'\in\widehat{\ff_{(l_n)}}\upharpoonright
    L_n$, is  continuous.
      \item[(iii)] $\text{diam}(\{x^n_{t'}:\;t'\in\widehat{\ff_{(l_n)}}\upharpoonright L_n\})<\frac{1}{n}$.
    \end{enumerate}
    Let $L=\{l_n:n\in\nn\}$. We define $\widehat{\varphi}:\widehat{\ff}\upharpoonright L\to X$
as follows. For $s=\emptyset$, we set
$\widehat{\varphi}(\emptyset)=x_\emptyset$ and for $t\neq
\emptyset$, we set
$\widehat{\varphi}(t)=\widehat{\varphi}_n(t\setminus\{l_n\})=x_t$,
where $l_n=\min t$. Then
 $\lim_{n\to\infty}\text{diam}\{\widehat{\varphi}(s):\;s\in\ff\upharpoonright L\text{ and }\min s=L(n)\}\to 0$
and  $x_{L(n)}\to x_\emptyset$. Therefore $\widehat{\varphi}$ is
continuous at $\emptyset$. For all other $t\in
\widehat{\ff}\upharpoonright L$ the continuity of
$\widehat{\varphi}$ follows easily by the continuity of
$\widehat{\varphi}_n$ and the proof is complete.
    \end{proof}
\begin{lem}\label{lemconv}
  Let $(X,\ttt)$ be a topological space, $\ff$ a regular thin
  family, $N\in[\nn]^\infty$ and $(x_s)_{s\in\ff}$ an $\ff$-sequence in
  $X$. If $\overline{\{x_s:s\in \ff\upharpoonright
  N\}}$ is compact metrizable, then there exists  $M\in[N]^\infty$,
  such that $(x_s)_{s\in\ff\upharpoonright
  M}$ admits a convergent $\widehat{\ff}$-tree.
\end{lem}
\begin{proof}
  We proceed by induction on the order of the regular thin
  family $\ff$. If $o(\ff)=1$, the result is immediate. Assume that for every $\xi<o(\ff)$ the
  lemma holds. Let $\ff$ be a regular thin
  family, $N\in\nn$ and $(x_s)_{s\in\ff}$ an $\ff$-sequence in
  $X$, with $Y=\overline{\{x_s:s\in \ff\upharpoonright
  N\}}$ compact metrizable.  For every $m\in\nn$ and
  every $s'\in \ff_{(m)}\upharpoonright N$, we set $z^m_{s'}=x_{\{m\}\cup
  s'}$. Since $o(\ff_{(m)})<o(\ff)$, for all $m\in N$, using our
  inductive hypothesis we may construct a
  strictly increasing sequence $(m_n)_{n\in\nn}$ in $N$ and a
  decreasing sequence $(M_n)_{n\in\nn}$ in $[N]^\infty$ such that
  for every $n\in\nn$
  the following are satisfied:
  \begin{enumerate}
    \item[(i)] $m_n<\min M_n$
    \item[(ii)] The $\ff_{(m_n)}$-sequence $(z^{m_n}_{s'})_{s'\in
    \ff_{(l_n)}\upharpoonright M_n}$ admits $(z^{m_n}_{t'})_{t'\in
    \widehat{\ff_{(m_n)}}\upharpoonright M_n}$ as a convergent
    $\widehat{\ff_{(m_n)}}$-tree in $Y$.
  \end{enumerate}
  Since $Y$ is a compact metrizable space, we have that the
  sequence $(z^{m_n}_{\emptyset})_{n\in\nn}$ contains a
  subsequence $(z^{l_{k_n}}_{\emptyset})_{n\in\nn}$ convergent to
  some $x\in Y$. Let $M=\{m_{k_n}:n\in\nn\}$ and $x_\emptyset=x$.
  Also let,
  for every $m\in M$ and $t'\in \widehat{\ff_{(m)}}\upharpoonright M$ (notice that $\widehat{\ff_{(m)}}\upharpoonright M\subseteq
  \widehat{\ff_{(m_{k_n})}}\upharpoonright M$, where $m=m_{k_n}$), $x_{\{m\}\cup
  t'}=z^m_{t'}$. It is easy to see that $(x_s)_{s\in\ff\upharpoonright M}$ admits $(x_t)_{t\in\widehat{\ff}\upharpoonright
  M}$ as a convergent $\widehat{\ff}$-tree.
\end{proof}

  \begin{prop}\label{Create subordinated}
    Let $(X,\mathcal{T})$ be a topological space, $\ff$ be  a  regular
    thin family and $(x_s)_{s\in\ff}$ be an $\ff$-sequence in $X$
such that $\overline{\{x_s:\;s\in\ff\}}$  is a compact metrizable
subspace of $(X,\ttt)$.
    Then for every $N\in[\nn]^\infty$ there exists $L\in[N]^\infty$ such that $(x_s)_{s\in\ff\upharpoonright L}$ is
    subordinated and if $\widehat{\varphi}:\widehat{\ff}\upharpoonright L\to X$ is the continuous map witnessing this,
    then $(x_s)_{s\in\ff\upharpoonright L}$ converges to $\widehat{\varphi}(\emptyset)$.
  \end{prop}
  \begin{proof}  By Lemma \ref{lemconv} there exists
 $M\in[N]^\infty$ such that $(x_s)_{s\in\ff\upharpoonright
  L}$ admits a convergent $\widehat{\ff}$-tree. By Proposition
  \ref{prpconv} there exists $L\in [M]^\infty$ such that $(x_s)_{s\in\ff\upharpoonright L}$ is
    subordinated. Finally by Corollary \ref{subordinating yields
    convergence}, we have that $(x_s)_{s\in\ff\upharpoonright L}$ converges to
    $\widehat{\varphi}(\emptyset)$ and the proof is complete.
  \end{proof}

\chapter{Norm properties of spreading models}\label{Chapter 5}
As we have seen the spreading sequences and therefore the
spreading models are classified into four categories. In this
chapter we provided conditions concerning $\ff$-sequences which
determine the class in which the generated spreading model
belongs.
\section{Trivial spreading models}
In this section we study the behavior of $\ff$-sequences which
generate trivial spreading models. Among others we prove that an
$\ff$-sequence admits a trivial spreading model iff it contains a
norm Cauchy subsequence. This generalizes the classical
Brunel-Sucheston condition for spreading models of any order.
\begin{thm}\label{Theorem equivalent forms for having norm on the spreading model}
  Let $M\in[\nn]^\infty$, $\ff$ be a regular thin family and $(x_s)_{s\in\ff}$ be an $\ff$-sequence in a Banach space $X$.
    Let $(E,\|\cdot\|_*)$ be an infinite dimensional
    seminormed linear space with Hamel basis $(e_n)_{n\in\nn}$ such that $(x_s)_{s\in\ff\upharpoonright M}$ generates
    $(e_n)_{n\in\nn}$ as an $\ff$-spreading model. Then the following are equivalent:
    \begin{enumerate}
    \item[(i)] The sequence $(e_n)_{n\in\nn}$ is trivial.
\item[(ii)] The seminorm $\|\cdot\|_*$ is not a norm on $E$.
            \item[(iii)] The sequence $(e_n)_{n\in\nn}$ is generated as a $\g$-spreading model by any constant
      $\g$-sequence $(y_t)_{t\in\g}$ in any Banach space $Y$, such that $\|y_t\|=\|e_1\|_*$ for all $t\in\g$ and for every regular thin family $\g$.
      \item[(iv)] The $\ff$-subsequence $(x_s)_{s\in\ff\upharpoonright M}$ contains a further norm Cauchy subsequence.
      \item[(v)] For every $\varepsilon>0$ and every $L\in[M]^\infty$, the $\ff$-subsequence
      $(x_s)_{s\in\ff\upharpoonright L}$ is not plegma $\varepsilon$-separated.
      \item[(vi)] There exists $x\in X$ such that every subsequence of $(x_s)_{s\in\ff\upharpoonright M}$ contains a further subsequence convergent to $x$.
    \end{enumerate}
\end{thm}
\begin{proof}
  The equivalence of (i) and  (ii) follows by Proposition  \ref{sing}.
  Also the equivalence of (i) and  (iii) is obvious.

  (i)$\Rightarrow$(v): Let $\varepsilon>0$ and  $L\in[M]^\infty$.
  Since  the $\ff$-subsequence
  $(x_s)_{s\in\ff\upharpoonright L}$ also generates the trivial sequence $(e_n)_{n\in\nn}$ as an $\ff$-spreading model,  there
  exists $n_0\in\nn$ such that for every plegma pair $(s_1,s_2)$ in $\ff\upharpoonright L$ with $\min s_1\geq L(n_0)$, we have that
  \[\Big\|x_{s_1}-x_{s_2} \Big\|=\Bigg| \Big\|x_{s_1}-x_{s_2} \Big\|-\Big\| e_1-e_2 \Big\|_* \Bigg|<\varepsilon\]
  and therefore   the $\ff$-subsequence
  $(x_s)_{s\in\ff\upharpoonright L}$ is not plegma $\varepsilon$-separated.

  (v)$\Rightarrow$(iv): This follows by Lemma \ref{Lemma epsilon near admissble yield Cauchy}.

  (iv)$\Rightarrow$(i): We can easily construct a sequence $((s_1^n,s_2^n))_{n\in\nn}$ of plegma pairs in $\ff\upharpoonright M$ such that
      $s_1^n(1)\to \infty$ and $\|x_{s_1^n}-x_{s_2^n}\|<\frac{1}{n}$. Hence
      \[\big\|e_1-e_2\|_*=\lim_{n\to\infty}\Big\|x_{s_1^n}-x_{s_2^n}\Big\|=0\]
    Thus by Proposition \ref{pdiff}
    we get  (i). By the above we have that (i)-(v) are equivalent.

      It remains to show that (i) and (vi) are equivalent. It is straightforward that (vi)$\Rightarrow$(iv)$\Rightarrow$(i).
      Conversely suppose that (i) holds. Then every subsequence of $(x_s)_{s\in\ff \upharpoonright M}$
      generates $(e_n)_{n\in\nn}$ as an $\ff$-spreading model. Hence, by the equivalence of (i) and (iv) and Proposition
      \ref{remark on ff Cauchy-convergent sequences and their spreading model}, every subsequence of $(x_s)_{s\in\ff\upharpoonright M}$
      contains a further convergent subsequence. It remains to show that
      all the convergent subsequences of $(x_s)_{s\in\ff \upharpoonright M}$ have a common limit. To this end, let $L_1,L_2\in[M]^\infty$ and $x_1,x_2\in X$ such
      that $(x_s)_{s\in\ff \upharpoonright L_1}$ converges to $x_1$ and $(x_s)_{s\in\ff \upharpoonright L_2}$ converges to $x_2$. We will show that
      $x_1=x_2$. Let $\varepsilon >0$. Then there exists $n_0\in\nn$ such that for every $s_1\in\ff\upharpoonright L_1$ with $\min s_1\geq L_1(n_0)$ and
      $s_2\in\ff\upharpoonright L_2$ with $\min s_2\geq L_2(n_0)$ we have that $\|x_1-x_{s_1}\|<\frac{\varepsilon}{3}$ and
      $\|x_2-x_{s_2}\|<\frac{\varepsilon}{3}$. Notice that $n_0$ can be chosen large enough such that for every plegma pair $(s_1,s_2)$
      in $\ff\upharpoonright M$ with $\min s_1\geq M(n_0)$,
      \[\Big\| x_{s_1}-x_{s_2} \Big\|= \Bigg|\Big\| x_{s_1}-x_{s_2} \Big\| -\Big\|e_1-e_2\Big\|_* \Bigg|<\frac{\varepsilon}{3}\]
      It is easy to see that we can choose $s_1\in\ff \upharpoonright L_1$ with $\min s_1\geq L_1(n_0)$ and
      $s_2\in\ff\upharpoonright L_2$ with $\min s_2\geq L_2(n_0)$ such that $(s_1,s_2)$ is a plegma pair. Then
      \[\Big\| x_1-x_2 \Big\|\leq \Big\|x_1-x_{s_1}\Big\|+\Big\|
x_{s_1}-x_{s_2} \Big\|+\Big\|x_2-x_{s_2}\Big\|<\varepsilon\]
      It follows that $x_1=x_2$ and the proof is complete.
\end{proof}
The next result follows by the equivalence of (i) and (vi) in the above theorem.
\begin{cor}\label{trivial ultrafilter weak property}
  Let $\ff$ be a regular thin family, $M\in[\nn]^\infty$, $(x_s)_{s\in\ff}$ and $(x'_s)_{s\in\ff}$ two $\ff$ sequences in a Banach space $X$. Suppose that the following are satisfied:
  \begin{enumerate}
    \item[(i)] There exists $x_0\in X$ such that $x_s=x'_s+x_0$, for all $s\in\ff\upharpoonright M$.
    \item[(ii)] The $\ff$-subsequence $(x_s)_{s\in\ff\upharpoonright M}$ generates an $\ff$-spreading model $(e_n)_{n\in\nn}$.
    \item[(iii)] The $\ff$-subsequence $(x'_s)_{s\in\ff\upharpoonright M}$ generates an $\ff$-spreading model $(e'_n)_{n\in\nn}$.
  \end{enumerate}
  Then $(e_n)_{n\in\nn}$ is trivial if and only if $(e'_n)_{n\in\nn}$ is trivial.
\end{cor}
\section{Unconditional spreading models}
As is well known every spreading model generated by a
seminormalized weakly null sequence is an unconditional spreading
sequence. In this section we provide an extension of this result
for subordinated seminormalized weakly null $\ff$-sequences. We
start with the following lemma.
\begin{lem}\label{Lemma finding convex  means}
  Let $X$ be a Banach space, $n\in\nn$, $\ff^1,\ldots,\ff^n$ be regular thin families,
  $L\in[\nn]^\infty$ and $\varepsilon>0$. For each $1\leq i\leq n$, let $\widehat{\varphi}_i:\widehat{\ff^i}\upharpoonright L\to (X,w)$ be a continuous map.
  Then there exist nonempty and finite $G^1\subseteq\ff^1\upharpoonright L,\ldots,G^n\subseteq\ff_n\upharpoonright L$ and sequences
  $(\mu^1_t)_{t\in G^1},\ldots,(\mu^n_t)_{t\in G^n}$ in $[0,1]$ such that
  \begin{enumerate}
    \item[(a)] For all $1\leq i\leq n$, $\sum_{t\in G^i}\mu_t^i=1$ and
    $\|\widehat{\varphi}_i(\emptyset)-\sum_{t\in G^i}\mu_t^i\widehat{\varphi}_i(t)\|<\varepsilon$.
    \item[(b)] For every choice of $t_i\in G^i$, the $n$-tuple $(t_1,\ldots,t_n)$ is plegma.
  \end{enumerate}
\end{lem}
\begin{proof}
  We use induction on $\xi_{\max}=\max\{o(\ff^i):1\leq i\leq n\}$. If
  $\xi_{\max}=0$, i.e. $\ff^i=\{\emptyset\}$ for all $1\leq i\leq n$ the result follows trivially.
  (Let $G^i=\ff^i$ and $\mu^i_\emptyset=1$.) Let $1\leq \xi\leq\omega_1$
  and suppose that the conclusion of the proposition holds provided that $\xi_{\max}<\xi$.
  We will show that it also holds for $\xi_{\max}=\xi$. To this end let $n\in\nn$, $L\in[\nn]^\infty$ and
  $\ff^1,\ldots,\ff^n$ be regular thin families such that $\xi=\max\{o(\ff^i):1\leq i\leq n\}$.
  Let also for each $1\leq i\leq n$, $\widehat{\varphi}_i:\widehat{\ff^i}\upharpoonright L\to (X,w)$ be a continuous map.
  We set $x^i_s=\widehat{\varphi}_i(s)$ , for all $1\leq i\leq n$ and $s\in\widehat{\ff^i}\upharpoonright L$.
  We may suppose that $\ff^i $ is very large in $L$, therefore $\{l\}\in\widehat{\ff^i}$ for all $l\in L$ and $1\leq i\leq n$.
  Since for every $1\leq i\leq n$  $w\text{-}\lim_{l\in L}x_{\{l\}}^i=x^i_{\emptyset}$, we can choose
  finite subsets $F^1<\ldots<F^n$ of $L$ and sequences $(\lambda_l^1)_{l\in F^1},\ldots,(\lambda_l^n)_{l\in F^n}$ in $[0,1]$
  such that for every $i=1,\ldots,n$
  \begin{enumerate}
    \item[(i)] $\sum_{l\in F^i}\lambda_l^i=1$ and $\|x^i_\emptyset -\sum_{l\in F^i}\lambda_l^i x^i_{\{l\}}\|<\frac{\varepsilon}{2}$
  \end{enumerate}
  Let $F=\cup_{i=1}^n F^i=\{l_1<\ldots<l_m\}$ and $L'=\{l\in L: l>l_m\}$. For every $1\leq j\leq m$ let $i_j$ be the unique $i\in \{1,\ldots,n\}$
  such that $l_j\in F^i$. For every $1\leq j\leq n$ we define $\g^j=\ff^{i_j}_{(l_j)}$ and
  $\widehat{\psi}_j:\widehat{\g^j\upharpoonright L'}\to (X,w)$, with $\widehat{\psi}_j(s)=\widehat{\varphi}_{i_j}(\{l_j\}\cup s)$.
  Since $o(\g^j)<o(\ff^{i_j})$ we have that $\max\{o(\g^j):1\leq j\leq m\}<\max\{o(\ff^i):1\leq i\leq n\}$
  and therefore we may use the inductive assumption. Hence there exist nonempty finite subsets
  $\widehat{G}^1\subseteq \g^1\upharpoonright L',\ldots, \widehat{G}^m\subseteq \g^m\upharpoonright L'$ and sequences
  $(\widetilde{\mu}^1_s)_{s\in\widehat{G}^1},\ldots,(\widetilde{\mu}^m_s)_{s\in\widehat{G}^m}$ in $[0,1]$
  such that for every $1\leq j\leq m$
  \begin{enumerate}
    \item[(ii)] $\sum_{s\in \widetilde{G}^j}\widetilde{\mu}^j_s=1$ and
    $\|\widehat{\psi}_j(\emptyset)- \sum_{s\in \widetilde{G}^j}\widetilde{\mu}^j_s \widehat{\psi}_j(s)\|<\frac{\varepsilon}{2}$.
    \item[(iii)] For every choice $t_j\in \widetilde{G}^j$ the $m$-tuple $(t_1,\ldots,t_m)$ is plegma.
  \end{enumerate}
  For every $1\leq i\leq n$ we set
  $G^i=\big{\{}\{l_j\}\cup s:l_j\in F_i\text{ and } s\in \widetilde{G}^j\big{\}}$ and for every $t\in G^i$ we set
  $\mu_t^i=\lambda^i_{l_j}\cdot \widetilde{\mu}_s^j$, where $t=\{l_j\}\cup s$.
  It is easy to see that
  \begin{enumerate}
    \item[(a)] For all $1\leq i\leq n$, $\sum_{t\in G^i}\mu_t^i=1$.
    \item[(b)] For every choice of $t_i\in G^i$, the $n$-tuple $(t_1,\ldots,t_n)$ is plegma.
  \end{enumerate}
  It remains to show that
    $\|\widehat{\varphi}_i(\emptyset)-\sum_{t\in G^i}\mu_t^i\widehat{\varphi}_i(t)\|<\varepsilon$, for all $1\leq i\leq n$. Indeed
    \[\begin{split}
      \|\widehat{\varphi}_i(\emptyset)-\sum_{t\in G^i}\mu_t^i\widehat{\varphi}_i(t)\|\leq&
      \|\widehat{\varphi}_i(\emptyset)-\sum_{\{j:i_j=i\}}\lambda^i_{l_j}\widehat{\varphi}_i(\{l_j\}) \|\\
      &+\|\sum_{\{j:i_j=i\}}\lambda^i_{l_j}\widehat{\varphi}_i(\{l_j\}) - \sum_{t\in G^i}\mu_t^i\widehat{\varphi}_i(t)\|
    \end{split}\]
    By (i) above we have that
    \[\|\widehat{\varphi}_i(\emptyset)-\sum_{\{j:i_j=i\}}\lambda^i_{l_j}\widehat{\varphi}_i(\{l_j\}) \|=
    \|x^i_\emptyset -\sum_{l\in F^i}\lambda_l^i x^i_{\{l\}}\|<\frac{\varepsilon}{2}\]
    Moreover
    \[\sum_{t\in G^i}\mu_t^i\widehat{\varphi}_i(t)
    = \sum_{\{j:i_j=i\}}\sum_{s\in \widetilde{G}^j} \lambda^i_{l_j}\widetilde{\mu}^j_s \widehat{\psi}_j(s)  \]
    Hence by (ii)
    \[\|\sum_{\{j:i_j=i\}}\lambda^i_{l_j}\widehat{\varphi}_i(\{l_j\}) - \sum_{t\in G^i}\mu_t^i\widehat{\varphi}_i(t)\|
    \leq \sum_{\{j:i_j=i\}}\lambda^i_{l_j} \|\widehat{\psi}_j(\emptyset)- \sum_{s\in \widetilde{G}^j}\widetilde{\mu}^j_s \widehat{\psi}_j(s)\|
    <\frac{\varepsilon}{2}\]
    And the proof is complete.
\end{proof}

\begin{thm}\label{unconditional spreading model}
  Let $\ff$ be a regular thin family and $L\in[\nn]^\infty$.
  Let $(x_s)_{s\in\ff\upharpoonright L}$ be an family of elements in a Banach space $X$ generating
  a spreading model $(e_n)_{n\in\nn}$. Suppose that $(x_s)_{s\in\ff\upharpoonright L}$ is
  subordinated with respect to the weak topology on $X$ and let $\widehat{\varphi}:\widehat{\ff\upharpoonright L}\to (X,w)$ be the continuous map
  witnessing it. If $\widehat{\varphi}(\emptyset)=0$, then either
  $(e_n)_{n\in\nn}$ is trivial with $\|e_n\|_*=0$ for all
  $n\in\nn$ or the sequence $(e_n)_{n\in\nn}$ is 1-unconditional
  (i.e. for every $n\in\nn$, $F\subseteq\{1,\ldots,n\}$ and $a_1,\ldots,a_n\in\rr$ we have that
  $\|\sum_{i\in F}a_ie_i\|\leq\|\sum_{i=1}^na_ie_i\|$).
\end{thm}
\begin{proof}
  Suppose that $(e_n)_{n\in\nn}$ is trivial. Then by Theorem
  \ref{Theorem equivalent forms for having norm on the spreading
  model} there exists $L_1\in[L]^\infty$ and $x_0\in X$ such that
  the $\ff$-subsequence $(x_s)_{s\in\ff\upharpoonright L_1}$
  converges to $x_0$. By Corollary \ref{subordinating yields convergence} we have that $x_0=\widehat{\varphi}(\emptyset)=0$. Since $(x_s)_{s\in\ff\upharpoonright
  L_1}$ also generates $(e_n)_{n\in\nn}$ we have that
  $\|e_1\|_*=0$. Since  $(e_n)_{n\in\nn}$ is spreading we get
  that $\|e_n\|_*=0$ for all $n\in\nn$.

  Assume that $(e_n)_{n\in\nn}$ is nontrivial. In this case it is enough to show that for every $n\in\nn$, $1\leq p\leq n$ and $a_1,\ldots,a_n\in [-1,1]$ we have
  that $$\Big{\|}\sum_{\substack{i=1\\i\neq p}}^na_ie_i\Big{\|}\leq\Big{\|}\sum_{i=1}^na_ie_i\Big{\|}$$
  To this end it suffices to prove that for every $\varepsilon >0$ we have that
  $$\Big{\|}\sum_{\substack{i=1\\i\neq p}}^na_ie_i\Big{\|}\leq\Big{\|}\sum_{i=1}^na_ie_i\Big{\|}+\varepsilon$$
  Indeed let $n\in\nn$, $1\leq p\leq n$, $a_1,\ldots,a_n\in [-1,1]$ and $\varepsilon >0$.
  We easily pass to $N\in[L]^\infty$ such that for every plegma $n$-tuple $(s_i)_{i=1}^n$ in $\ff\upharpoonright N$ the following hold
  \begin{equation}\label{eq8}
    \Bigg{|}\Big{\|}\sum_{i=1}^na_i x_{s_i}\Big{\|}-\Big{\|}\sum_{i=1}^na_ie_i\Big{\|} \Bigg{|}\leq\frac{\varepsilon}{3}\;\;\text{and}\;\; \Bigg{|} \Big{\|}\sum_{\substack{i=1\\i\neq p}}^na_ix_{s_i}\Big{\|} -\Big{\|}\sum_{\substack{i=1\\i\neq p}}^na_ie_i\Big{\|} \Bigg{|}\leq\frac{\varepsilon}{3}
  \end{equation}
  Let $\ff^1=\ldots=\ff^n=\ff$ and $\widehat{\varphi}_1=\ldots=\widehat{\varphi}_n=\widehat{\varphi}|_{\ff\upharpoonright N}$.
  By Lemma \ref{Lemma finding convex  means} there exist nonempty and finite $G^1,\ldots,G^n\subseteq\ff\upharpoonright N$ and sequences
  $(\mu^1_t)_{t\in G^1},\ldots,(\mu^n_t)_{t\in G^n}$ in $[0,1]$ such that
  \begin{enumerate}
    \item[(a)] For $1\leq i\leq n$, $\displaystyle\sum_{t\in G^i}\mu_t^i=1$ and
    $\displaystyle\|\sum_{t\in G^i}\mu_t^i x_t\|=\|\widehat{\varphi}(\emptyset)-
    \sum_{t\in G^i}\mu_t^i\widehat{\varphi}(t)\|<\frac{\varepsilon}{3}$.
    \item[(b)] For every choice of $t_i\in G^i$, the $n$-tuple $(t_1,\ldots,t_n)$ is a plegma family.
  \end{enumerate}
  We fix $s_i\in G^i$ for all $1\leq i\leq n$ with $i\neq p$ and we set $G=G^p$. So we have that $\sum_{t\in G}\mu_t^p=1$,
   $\|\sum_{t\in G}\mu_t^p x_t\|<\frac{\varepsilon}{3}$ and for every $t\in G$ the $n$-tuple $(s_1,\ldots,s_{p-1},t,s_{p+1},\ldots, s_n)$
    is a plegma family. Therefore by  (\ref{eq8}) we have that
  \[\begin{split}\Big{\|}\sum_{\substack{i=1\\i\neq p}}^na_ie_i\Big{\|} & \leq
  \Big{\|}\sum_{\substack{i=1\\i\neq p}}^na_ix_{s_i}\Big{\|}+\frac{\varepsilon}{3}
  \leq \Big{\|}\sum_{\substack{i=1\\i\neq p}}^na_ix_{s_i} +a_p\sum_{t\in G}\mu_t^p x_t\Big{\|}+|a_p|\frac{\varepsilon}{3}+\frac{\varepsilon}{3}\\
  &\leq\sum_{t\in G}\mu_t^p \Big{\|}\sum_{\substack{i=1\\i\neq p}}^na_ix_{s_i} +a_px_t\Big{\|}+\frac{2\varepsilon}{3}
  \leq \sum_{t\in G}\mu_t^p \Big{(}\Big{\|}\sum_{i=1}^n a_ie_i \Big{\|} +\frac{\varepsilon}{3} \Big{)}+\frac{2\varepsilon}{3} \\
  &=\Big{\|}\sum_{i=1}^n a_ie_i \Big{\|} +\varepsilon
    \end{split} \]
    and the proof is completed.
\end{proof}
By Corollary \ref{subordinating yields convergence} the following
is actually a restatement of the above theorem.
\begin{cor}\label{unconditional spreading model cor}
  Let $X$ be a Banach space, $\ff$ be a regular thin family,
  $L\in[\nn]^\infty$ and $(x_s)_{s\in\ff}$ an $\ff$-sequence in
  $X$ such that the $\ff$-subsequence
  $(x_s)_{s\in\ff\upharpoonright L}$ is seminormalized,
  subordinated and weakly null. Then every spreading model of $(x_s)_{s\in\ff\upharpoonright
  L}$ is 1-unconditional.
\end{cor}
\begin{rem}
  Since the weakly null sequences are identical with the
  subordinated ones, the above corollary is
  actually an extension of the classical result for unconditional
  spreading models generated by seminormalized weakly null
  sequences. Moreover, let us notice that if $o(\ff)\geq2$ then
  the additional assumption that
  requires the $\ff$-sequence to be subordinated is
  a necessary one. Indeed, consider the $[\nn]^2$-sequence in
  $c_0$,
  defined by \[x_s=\sum_{n=\min s}^{\max s}e_n\]
  for all $s\in[\nn]^2$, where $(e_n)_{n\in\nn}$ is the natural basis of $c_0$. It is easy to see that $(x_s)_{s\in[\nn]^2}$
  is normalized and weakly null. We will show that
  $(x_s)_{s\in[\nn]^2}$ generates a spreading model
  which is equivalent to the summing basis of $c_0$. Thus, although $(x_s)_{s\in[\nn]^2}$
  is normalized and weakly null, it does not generate unconditional spreading
  model. In order to show that $(x_s)_{s\in[\nn]^2}$ generates a
  spreading model equivalent to the summing basis of $c_0$, we
  will show that
  \begin{equation}\label{papae}
  \Big\|\sum_{j=1}^la_jx_{s_j}\Big\|=\max\Big(\max_{1\leq
  k\leq l}\Big|\sum_{j=1}^ka_j\Big|,\max_{1\leq k\leq
  l}\Big|\sum_{j=k}^la_j\Big|\Big)\end{equation}
  for all $l\in\nn$, $a_1,\ldots,a_l\in\rr$ and
  $(s_j)_{j=1}^l\in\text{\emph{Plm}}([\nn]^2)$. Indeed,
  let $l\in\nn$, $a_1,\ldots,a_l\in\rr$ and
  $(s_j)_{j=1}^l\in\text{\emph{Plm}}([\nn]^2)$. We set
  $z=\sum_{j=1}^la_jx_{s_j}$. Since
  $(s_j)_{j=1}^l$ is plegma, we have that \[\min s_1<\ldots<\min
  s_l<\max s_1<\ldots<\max s_l\]Therefore we have the following.
  \begin{enumerate}
    \item[(i)] If $1\leq n\leq \min s_1$ then $z(n)=0$.
    \item[(ii)] If $1\leq i<l$ and $\min s_i\leq n<\min s_{i+1}$
    then $z(n)=\sum_{j=1}^ia_j$.
    \item[(iii)] If $\min s_l\leq n\leq\max s_1$ then
    $z(n)=\sum_{j=1}^l a_j$.
    \item[(iv)] If $1<i\leq l$ and $\max s_{i-1}<n\leq\max s_i$
    then $z(n)=\sum_{j=i}^l a_j$.
    \item[(v)] If $n>\max s_l$ then $z(n)=0$.
  \end{enumerate}
  By (i)-(v) we have that equation (\ref{papae}) holds.
\end{rem}
\section{Singular spreading models}
In this section we characterize the  $\ff$-sequences  which
generate singular (i.e. nontrivial and  non Schauder basic)
spreading models. It turns out that the natural decomposition of
the singular spreading sequence is reflected back to the
$\ff$-sequence which generated it. More precisely for the case of
the  1-spreading models it is shown that every sequence which
generates a singular spreading model is  weakly convergent to a
nonzero element of norm  equal to that of the weak limit of the
singular sequence. The general case of an $\ff$-sequence is more
involved.

\subsection{Singular spreading models of order one}
\begin{lem}\label{isometric 1-subsym sequences}
  Let $(e_n)_{n\in\nn}$ and $(\widetilde{e}_n)_{n\in\nn}$ be two
  nontrivial sequences which are
  spreading and Ces\'aro summable to zero. Suppose that for
  every $n\in\nn$ and $\lambda_1,\ldots,\lambda_n$ with
  $\sum_{i=1}^n\lambda_i=0$, we have that
  \[\Big\|\sum_{i=1}^n\lambda_ie_i\Big\|=\Big\|\sum_{i=1}^n\lambda_i\widetilde{e}_i\Big\|\]
  Then they are isometric, i.e. for
  every $n\in\nn$ and $\lambda_1,\ldots,\lambda_n$, we have that
  \[\Big\|\sum_{i=1}^n\lambda_ie_i\Big\|=\Big\|\sum_{i=1}^n\lambda_i\widetilde{e}_i\Big\|\]
\end{lem}
\begin{proof}
  Let $n\in\nn$ and $\lambda_1,\ldots,\lambda_n$. Since $(e_n)_{n\in\nn}$ and
  $(\widetilde{e}_n)_{n\in\nn}$ are Ces\'aro summable to zero, we
  have that
\[\frac{1}{m}\sum_{j=1}^me_{n+j}\stackrel{m\to\infty}{\longrightarrow}0\;\text{and}
\;\frac{1}{m}\sum_{j=1}^m\widetilde{e}_{n+j}\stackrel{m\to\infty}{\longrightarrow}0\]
Since, for every $m\in\nn$,
$\sum_{i=1}^n\lambda_i-\sum_{j=1}^m\frac{\lambda}{m}=0$, we have
that
\[\begin{split}\Big\|\sum_{i=1}^n\lambda_ie_i\Big\|&=
\lim_{m\to\infty}\Big\|\sum_{i=1}^n\lambda_ie_i-\frac{\lambda}{m}\sum_{j=1}^me_{n+j}\Big\|\\
&=\lim_{m\to\infty}\Big\|\sum_{i=1}^n\lambda_i\widetilde{e}_i-\frac{\lambda}{m}\sum_{j=1}^m\widetilde{e}_{n+j}\Big\|
=\Big\|\sum_{i=1}^n\lambda_i\widetilde{e}_i\Big\|\end{split}\]
\end{proof}
\begin{lem}\label{lemmaforsingular}
  Let $X$ be a Banach space and $(x_n)_{n\in\nn}$ be a weakly
  convergent
  sequence in $X$ which generates a singular spreading model
  $(e_n)_{n\in\nn}$. Let $e_n=e'_n+e$ be the natural decomposition of
  $(e_n)_{n\in\nn}$ (see Proposition \ref{decomp of singular spr mod}).
  Let $x$ be the weak limit of $(x_n)_{n\in\nn}$ and let  $x'_n=x_n-x$,
  for all $n\in\nn$.

  Then $\|x\|=\|e\|$ and
  $(e'_n)_{n\in\nn}$ is the unique (up to isometry)
  spreading model of  $(x'_n)_{n\in\nn}$.
\end{lem}
\begin{proof}
  Let  $(\widetilde{e}_n)_{n\in\nn}$ be a spreading
  model generated by a subsequence  of
  $(x'_n)_{n\in\nn}$.
 Since
  $(x'_n)_{n\in\nn}$ is weakly null, we have that $(\widetilde{e}_n)_{n\in\nn}$  is
  $1$-unconditional. By Corollary \ref{equiv forms for 1-subsymmetric weakly
  null}, we have that either $(\widetilde{e}_n)_{n\in\nn}$  is
  equivalent to the usual basis of $\ell^1$, or it is is Ces\`aro summable to
  zero. The first alternative is excluded since otherwise by Corollary \ref{ultrafilter property for ell^1 spr mod simplified},
  we would have that the same holds for $(e_n)_{n\in\nn}$. Therefore   $(\widetilde{e}_n)_{n\in\nn}$ is Ces\`aro summable to
  zero. As $x_n=x_n'+x$ and $(x_n)_{n\in\nn}$ generates $(e_n)_{n\in\nn}$ as a spreading model we easily see that
  \[\lim_n\Big\|\frac{\sum_{i=1}^n e_i}{n}\Big\|=\|x\|\]
Since $e_n=e'_n+e$ and $(e'_n)_{n\in\nn}$ is also Ces\`aro
summable to zero, we get  that $\|x\|=\|e\|$.

It remains to show that $(\widetilde{e}_n)_{n\in\nn}$ and
$(e'_n)_{n\in\nn}$ are isometric. By Lemma \ref{isometric 1-subsym
sequences} it suffices to show that for every $n\in\nn$ and
$\lambda_1,\ldots\lambda_n\in\rr$ with $\sum_{i=1}^n\lambda_1=0$
we have that
\[\Big\|\sum_{i=1}^n\lambda_ie_i'\Big\|=\Big\|\sum_{i=1}^n\lambda_i\widetilde{e}_i\Big\|\]
Indeed, let $n\in\nn$ and $\lambda_1,\ldots\lambda_n\in\rr$ with
$\sum_{i=1}^n\lambda_1=0$. Indeed, let $(F_k)_{k\in\nn}$ be a
sequence of finite subsets of $\nn$ such that $|F_k|=n$, for all
$k\in\nn$, and $\min F_k\to\infty$. Then
\[\Big\|\sum_{i=1}^n\lambda_ie_i'\Big\|=\Big\|\sum_{i=1}^n\lambda_ie_i\Big\|=\lim_{k\to\infty}\Big\|\sum_{i=1}^n\lambda_ix_{F_k(i)}\Big\|=
\lim_{k\to\infty}\Big\|\sum_{i=1}^n\lambda_ix'_{F_k(i)}\Big\|=\Big\|\sum_{i=1}^n\lambda_i\widetilde{e}_i\Big\|\]
  \end{proof}
\begin{prop}
  Let $X$ be a Banach space and $(x_n)_{n\in\nn}$ be a
  sequence in $X$ which generates a singular spreading model
  $(e_n)_{n\in\nn}$. Let $e_n=e'_n+e$ be the natural decomposition of
  $(e_n)_{n\in\nn}$.

   Then there exists $x\in X\setminus \{0\}$
  such that $x_n\stackrel{w}{\to}x$, $\|x\|=\|e\|$ and setting
  $x'_n=x_n-x$, $(e'_n)_{n\in\nn}$ is the unique spreading model of $(x'_n)_{n\in\nn}$.
\end{prop}
\begin{proof} The proof splits into two claims as follows.
\bigskip

\textbf{Claim 1:}
  For every $M\in[\nn]^\infty$ there exist $L_M\in[M]^\infty$ and $x_M\in X\setminus \{0\}$
  such that $(x_n)_{n\in L_M}\stackrel{w}{\to}x_M$, $\|x_M\|=\|e\|$,  and setting $x'_n=x_n-x_M$, $(x'_n)_{n\in L_M}$
  generates $(e'_n)_{n\in\nn}$ as a spreading model.
\begin{proof}[Proof of Claim 1]
Notice that $(x_n)_{n\in\nn}$ does not contain any
  Schauder basic subsequence, since otherwise $(e_n)_{n\in\nn}$
  should be Schauder basic.
Let $M\in[\nn]^\infty$. Following similar arguments as in the
proof of Proposition \ref{decomp of singular spr mod} we have that
  there exist $L'_M\in[M]^\infty$ and $x_M\in
  X\setminus\{0\}$ such that $(x_n)_{n\in L_M}\stackrel{w}{\to}x_M$.
We pass to an $L_M\in[L'_M]^\infty$ such that $(x_n-x_M)_{n\in
L_M}$ generates spreading model.
   The remaining  assertions of
  the claim follow by Lemma \ref{lemmaforsingular}.
\end{proof}

\textbf{Claim 2:} There exists $x\in X$ such that  for every
$M\in[\nn]^\infty$, $x_M=x$.
\begin{proof}[Proof of Claim 2]
Let $M_1,M_2\in[\nn]^\infty$. We will show that $x_{M_1}=x_{M_2}$.
Let $\ee>0$. Since $(e'_n)_{n\in\nn}$ is spreading and Ces\'aro
summable to zero,  there exists $n_0\in\nn$ such that
\[\Big\|\frac{1}{n_0}\sum_{j=1}^{n_0}e'_j\Big\|<\frac{\ee}{4}\text{ and }
\Big\|\frac{1}{n_0}\sum_{j=1}^{n_0}e'_{n_0+j}\Big\|<\frac{\ee}{4}\]
Let $(F_k)_{k\in\nn}$ (resp. $(G_k)_{k\in\nn}$) be a sequence of
finite subsets of $L_{M_1}$ (resp. $L_{M_2}$) such that
$|F_k|=n_0$ (resp. $|G_k|=n_0$) for all $k\in\nn$,  $\min
F_k\to\infty$ (resp. $\min G_k\to\infty$) and $\max F_k<\min G_k$,
for all $k\in\nn$. Then
\[\lim_{k\to\infty}\Big\|\frac{1}{n_0}\sum_{j=1}^{n_0}x_{F_k(j)}-x_{M_1}\Big\|=\Big\|\frac{1}{n_0}\sum_{j=1}^{n_0}e'_j\Big\|<\frac{\ee}{4}\]
and
\[\lim_{k\to\infty}\Big\|\frac{1}{n_0}\sum_{j=1}^{n_0}x_{G_k(j)}-x_{M_2}\Big\|=\Big\|\frac{1}{n_0}\sum_{j=1}^{n_0}e'_j\Big\|<\frac{\ee}{4}\]
Moreover
\[\begin{split}\lim_{k\to\infty}\Big\|\frac{1}{n_0}\sum_{j=1}^{n_0}x_{F_k(j)}-&\frac{1}{n_0}\sum_{j=1}^{n_0}x_{G_k(j)}\Big\|
=\Big\|\frac{1}{n_0}\sum_{j=1}^{n_0}e_j-\frac{1}{n_0}\sum_{j=1}^{n_0}e_{n_0+j}\Big\|\\
&=\Big\|\frac{1}{n_0}\sum_{j=1}^{n_0}e'_j-\frac{1}{n_0}\sum_{j=1}^{n_0}e'_{n_0+j}\Big\|<\frac{\ee}{2}\end{split}\]
Therefore $\|x_{M_1}-x_{M_2}\|<\ee$. Since this holds for all
$\ee>0$ the proof of Claim 2 is complete.
\end{proof}
By Claim 2 we have that every subsequence of $(x_n)_{n\in\nn}$
contains a further subsequence weakly convergent to $x$. This
yields that $(x_n)_{n\in\nn}$ weakly converges to $x$ and by Lemma
\ref{lemmaforsingular} the result follows.
\end{proof}
\subsection{Singular spreading models of higher order}
The next is the main result of this subsection.
\begin{thm}\label{singhighord}
  Let $\ff$ be a regular thin family, $M\in[\nn]^\infty$ and
  $(x_s)_{s\in\ff}$ be a $\ff$-sequence in an Banach space $X$
  such that $(x_s)_{s\in\ff\upharpoonright M}$ generates a
  singular $\ff$-spreading model $(e_n)_{n\in\nn}$. Let $e_n=e'_n+e$ be the natural decomposition
  of $(e_n)_{n\in\nn}$. Then there
  exist $x\in X$ and $L\in[M]^\infty$ such that $\|x\|=\|e\|$ and setting
  $x'_s=x_s-x$, for all $s\in\ff\upharpoonright L$, we have that
  $(x'_s)_{s\in\ff\upharpoonright L}$ admits $(e'_n)_{n\in\nn}$ as
  a unique $\ff$-spreading model.
\end{thm}
For the proof of the above theorem we need some preliminary work.
We start with the following.
\begin{lem}\label{firstlem} Let $\ff$ be a regular thin family, $M\in[\nn]^\infty$ and
  $(x_s)_{s\in\ff}$ be a $\ff$-sequence in an Banach space $X$
  such that $(x_s)_{s\in\ff\upharpoonright M}$ generates a
  singular $\ff$-spreading model $(e_n)_{n\in\nn}$. Then
 for every $\ee>0$ there exists $n_\ee\in\nn$ such
  that for every $n,m\geq n_\ee$ and every
  $(s_j)_{j=1}^{n+m}\in\text{\emph{Plm}}(\ff\upharpoonright M)$
  with $s_1(1)\geq M(n+m)$, we have that
  \[\Big\|\frac{1}{n}\sum_{j=1}^nx_{s_j}-\frac{1}{m}\sum_{j=1}^mx_{s_{n+j}}\Big\|<\ee\]
\end{lem}
\begin{proof} Let $e_n=e'_n+e$ be the natural decomposition of the
singular sequence $(e_n)_{n\in\nn}$. By Proposition \ref{decomp of
singular spr mod} we have that the sequence $(e'_n)_{n\in\nn}$ is
and Ces\'aro summable to zero. Let  $\ee>0$ and choose $n_0\in\nn$
such that for all $n\geq n_0$,
\[\Big\|\frac{1}{n}\sum_{i=1}^n e'_i\Big\|<\ee/4\]
Then for all $n,m\geq n_0$,
\[\Big\|\frac{1}{n}\sum_{i=1}^n e_i-\frac{1}{m}\sum_{i=1}^m e_{n+i}\Big\|=
\Big\|\frac{1}{n}\sum_{i=1}^n e'_i-\frac{1}{m}\sum_{i=1}^m
e'_{n+i}\Big\|<\ee/2\] This yields that for some sufficiently
large $n_\ee\geq n_0$,
\[\Big\|\frac{1}{n}\sum_{j=1}^nx_{s_j}-\frac{1}{m}\sum_{j=1}^mx_{s_{n+j}}\Big\|<\ee\]
 for every $n,m\geq n_\ee$ and every
  $(s_j)_{j=1}^{n+m}\in\text{\emph{Plm}}(\ff\upharpoonright M)$
  with $s_1(1)\geq M(n+m)$.
\end{proof}

\begin{prop}\label{Schrplces}
  Let $\ff$ be a regular thin family, $M\in[\nn]^\infty$ and
  $(x_s)_{s\in\ff}$ an $\ff$-sequence in a Banach space $X$.
  Assume that for every $\ee>0$ there exists $n_\ee\in\nn$ such
  that for every $n,m\geq n_\ee$ and every
  $(s_j)_{j=1}^{n+m}\in\text{\emph{Plm}}(\ff\upharpoonright M)$
  with $s_1(1)\geq M(n+m)$ we have that
  \[\Big\|\frac{1}{n}\sum_{j=1}^nx_{s_j}-\frac{1}{m}\sum_{j=1}^mx_{s_{n+j}}\Big\|<\ee\]
  Then there exists $L\in[M]^\infty$ and  $x\in X$ such that the
  following holds.

   For every $\ee>0$
  there exists $n_0\in\nn$ such that for every $n\geq n_0$ and
  every
  $(s_j)_{j=1}^n\in\text{\emph{Plm}}(\ff\upharpoonright
  L)$ with $s_1(1)\geq L(n)$  we have that
  \[\Big\|\frac{1}{n}\sum_{j=1}^nx_{s_j}-x\Big\|<\ee\]
\end{prop}
\begin{proof} We will show the Cauchy version of the above proposition, i.e. that  there exists
$L\in[\nn]^\infty$ such that  for every $\ee>0$
  there exists $n_0\in\nn$ such that for every $n,m\geq n_0$ and
  every
  $(s_j)_{j=1}^n,(t_j)_{j=1}^m\in\text{\emph{Plm}}(\ff\upharpoonright
  L)$ with $s_1(1)\geq L(n)$ and $t_1(1)\geq L(m)$ we have that
 \begin{equation}
 \label{eqstar}\Big\|\frac{1}{n}\sum_{j=1}^nx_{s_j}-\frac{1}{m}\sum_{j=1}^mx_{t_j}\Big\|<\ee\end{equation}
 The above yields the existence of the desired $x\in X$.  Indeed, for each $k\in\nn$, let
\[A_k=\Big\{\frac{1}{n}\sum_{i=1}^n x_{s_i}: \; n\geq k,\;(s_i)_{i=1}^n\in\text{\emph{Plm}}(\ff\upharpoonright
L)\;\text{and}\;s_1(1)\geq n \Big\}\] Clearly the sequence
$(A_k)_{k\in\nn}$ is decreasing and  $\text{diam}(A_k)\to 0$.
Therefore there exists $x\in X$ such that $\cap_{k=1}^\infty
\overline{A_k}=\{x\}$. It is easy to see that $x$ satisfies the
conclusion.

 By  passing to an
  infinite subset of $M$ if necessary, we may assume that $\ff$ is very large in
  $M$.
  Let  $(\delta_k)_{k\in\nn}$ be a decreasing null sequence of positive
  reals such that $\sum_{k=1}^\infty\delta_k<\infty$. By our
  assumption there exists a sequence $(n_k)_{k\in\nn}$ of natural
  numbers such that for every $k\in\nn$, for every $n,m\geq n_k$ and every
  $(s_j)_{j=1}^{n+m}\in\text{\emph{Plm}}(\ff\upharpoonright M)$
  with $s_1(1)\geq M(n+m)$ we have that
  \[\Big\|\frac{1}{n}\sum_{j=1}^nx_{s_j}-\frac{1}{m}\sum_{j=1}^mx_{s_{n+j}}\Big\|<\delta_k\]
  We may also suppose that $(n_k)_{k\in\nn}$ is strictly
  increasing.

  Let $(F_k)_{k\in\nn}$ be a  sequence of finite
  subsets of $\nn$ such that for every $k\in\nn$ the following are
  satisfied:
  \begin{enumerate}
    \item[(1)] $\max F_k<\min F_{k+1}$,
    \item[(2)] $|F_k|=n_k$ and
    \item[(3)] $\min F_k\geq n_k+n_{k+1}$.
  \end{enumerate}
  We set \[N=\{M(\max F_k):k\in\nn\}\;\;\;\text{ and}\;\;\;
  L=\{N(2p+1):p\in\nn\}\] We claim  that   $L$ satisfies (\ref{eqstar}).
 We split the proof into two steps.

 \bigskip

  \textbf{Step 1.} For every $n\in N$, let $k_n$ be the unique positive integer such that
  $n=M(\max F_{k_n})$. Also for every $s\in\ff\upharpoonright N$, we set  $k_s=k_{s(1)}=k_{\min s}$.
  Finally for each $1\leq j\leq n_{k_s}$, let $v^s_j\in\ff\upharpoonright
  M$ be the unique element of $\ff\upharpoonright M$ such that
  \[v^s_j\sqsubseteq \big\{M(F_{k_{s(p)}}(n_{k_{s(p)}}-n_{k_s}+j)):p=1,\ldots,|s|\big\}\]
  (Using that $\ff$ is regular and very large in $M$ one can easily verify the existence of $v^s_j$, for all $1\leq j\leq n_{k_s}$).

Then the following hold.
  \begin{enumerate}
    \item[(a)] For every $s\in\ff\upharpoonright N$ we have that
    $v^s_{n_{k_s}}=s$ and $v^s_1(1)=M(F_{k_s}(1))$.
    \item[(b)] For every $m\in\nn$ and
    $(s_j)_{j=1}^m\in\text{\emph{Plm}}(\ff\upharpoonright N)$ we
    have that the $\sum_{j=1}^m n_{k_{s_j}}$-tuple
    $(v^{s_1}_1,\ldots,v^{s_1}_{n_{k_{s_1}}},v^{s_2}_1,\ldots,v^{s_2}_{n_{k_{s_2}}},\ldots,v^{s_m}_1,\ldots,v^{s_m}_{n_{k_{s_m}}})$
    is a plegma family in $\ff\upharpoonright M$.
  \end{enumerate}
 The proof of the above assertions is straightforward.

\bigskip

\textbf{Step 2.} Fix $\ee>0$. Let $k_0\in\nn$ be  such that
$\delta_{k_0}+\sum_{k=k_0}^\infty\delta_k<\frac{\ee}{2}$ and
$s_0\in\ff$ with $s_0\sqsubseteq
\{N(2p):p\in\nn\;\text{and}\;2p\geq k_0\}$. We set
\[n_0=\max(n_{k_0}, k_{\max s_0}+3)\]
Then for every $m\geq n_0$ and $(t_j)_{j=1}^m\in\ff\upharpoonright
L$ with $t_1(1)\geq L(m)$ we have that
\[\Big\|\frac{1}{n_{k_{s_0}}}\sum_{j=1}^{n_{k_{s_0}}}x_{v^{s_0}_j}-\frac{1}{m}\sum_{j=1}^mx_{t_j}\Big\|<\frac{\ee}{2}\]
\begin{proof}[Proof of Step 2]
  Let $m\geq n_0$ and
$(t_j)_{j=1}^m\in\ff\upharpoonright L$ with $t_1(1)\geq L(m)$. Let
$\widetilde{t}_1\in\ff$ be  such that
\[\widetilde{t}_1\sqsubseteq\{N(d_p-1):p=1,\ldots,|t_1|\}\] where
$d_p$ is defined to be the unique natural number such that
$t_1(p)=N(d_p)$, for all $1\leq p\leq |t_1|$. Since
$t_1\in\ff\upharpoonright L$ and $L=N(2\nn+1)$, we get that
$\widetilde{t}_1\in\ff\upharpoonright N(2\nn)$. Moreover, by the
definition of $n_0$, we  also have that $\max
s<\min\widetilde{t}_1$. Therefore by Proposition \ref{accessing
everything with plegma path of length |s_0|} we have that there
exists a plegma path $(s_l)_{l=0}^{|s|}$ from $s_0$ to
$\widetilde{t}_1$ of length $|s_0|$.

Let  $1\leq l\leq |s_0|-1$.  Then the pair $(s_l,s_{l+1})$ is a
plegma pair in $\ff\upharpoonright N$ and so, by property (b) of
Step 1, we have that
$(v^{s_l}_1,\ldots,v^{s_l}_{n_{k_{s_j}}},v^{s_{l+1}}_1,\ldots,v^{s_{l+1}}_{n_{k_{s_{l+1}}}})$
and thus
$(v^{s_j}_1,\ldots,v^{s_j}_{n_{k_{s_0}+l}},v^{s_{l+1}}_1,\ldots,v^{s_{l+1}}_{n_{k_{s_0}+l+1}})$
is a plegma family in $\ff\upharpoonright M$. Moreover by the
property (a) of Step 1 and the fact that   $\min F_k\geq
n_k+n_{k+1}$, we notice that
\[v^{s_l}_1(1)=M(F_{k_{s_l}}(1))\geq M(F_{k_{s_0}+l}(1))\geq M(n_{k_{s_0}+l}+n_{k_{s_0}+l+1})\]
Hence
\[\Big\|\frac{1}{n_{k_{s_0}+l}}\sum_{j=1}^{n_{k_{s_0}+l}}x_{v^{s_l}_j}-\frac{1}{n_{k_{s_0}+l+1}}
\sum_{j=1}^{n_{k_{s_0}+l+1}}x_{v^{s_{l+1}}_j}\Big\|<\delta_{k_0+l}\]
 Thus
\[\Big\|\frac{1}{n_{k_{s_0}}}\sum_{j=1}^{n_{k_{s_0}}}x_{v^{s_0}_j}-\frac{1}{n_{k_{s_0}+|s_0|}}
\sum_{j=1}^{n_{k_{s_0}+|s_0|}}x_{v^{\widetilde{t}_1}_j}\Big\|<\sum_{j=0}^{|s_0|-1}\delta_{k_0+j}<\sum_{k=k_0}^\infty\delta_k\]
Therefore it  suffices to show that
\[\Big\|\frac{1}{n_{k_{s_0}+|s_0|}}\sum_{j=1}^{n_{k_{s_0}+|s_0|}}x_{v^{\widetilde{t}_1}_j}-
\frac{1}{m}\sum_{j=1}^mx_{t_j}\Big\|<\delta_{k_0}\] To see this,
we first notice that $(\widetilde{t}_1, t_1,\ldots,t_m)$ is a
plegma family in $\ff\upharpoonright N$ which implies that
$(v^{\widetilde{t}_1}_1,\ldots,v^{\widetilde{t}_1}_{n_{k_{s_0}+|s_0|}},
t_1,\ldots,t_m)$ is also a plegma family in $\ff\upharpoonright
M$. Therefore, by our initial assumption  for $M$,  it would be
enough to verify that \[v^{\widetilde{t}_1}_1(1)\geq
M(n_{k_{s_0}+|s_0|}+m)\] To this end we observe that
\[v^{\widetilde{t}_1}_1(1)=M(F_{n_{k_{\widetilde{t}_1}}}(1))\geq M(n_{k_{\widetilde{t}_1}}+n_{k_{\widetilde{t}_1}+1})=
M(n_{k_{\widetilde{t}_1}}+n_{k_{t_1}})\] and so it suffices to
show that
\[n_{k_{\widetilde{t}_1}}+n_{k_{t_1}}>n_{k_{s_0}+|s_0|}+m\]
Indeed, since $\widetilde{t}_1=s_{|s_0|}$, we have that
$n_{k_{\widetilde{t}_1}}\geq n_{k_{s_0}+|s_0|}$. Also  since
\[M(\max F_{k_{t_1}})=t_1(1)\geq L(m)=N(2m+1)>N(m)=M(\max F_m)\]
we get that $k_{t_1}>m$ which implies that  $n_{k_{t_1}}\geq m$.
Therefore
$n_{k_{\widetilde{t}_1}}+n_{k_{t_1}}>n_{k_{s_0}+|s_0|}+m$ and the
proof of Step 2 is complete.
\end{proof}

By Step 2 we have that for every $\ee>0$ there exists $n_0$ in
$\nn$ such that for every $n,m\geq n_0$ and
$(s_j)_{j=1}^n,(t_j)_{j=1}^m\in\ff\upharpoonright L$ with
$s_1(1)\geq L(n)$ and $t_1(1)\geq L(m)$ we have that
\[\Big\|\frac{1}{n_{k_{s_0}}}\sum_{j=1}^{n_{k_{s_0}}}x_{v^{s_0}_j}-\frac{1}{n}\sum_{j=1}^mx_{s_j}\Big\|<\frac{\ee}{2}\]
and
\[\Big\|\frac{1}{n_{k_{s_0}}}\sum_{j=1}^{n_{k_{s_0}}}x_{v^{s_0}_j}-\frac{1}{m}\sum_{j=1}^mx_{t_j}\Big\|<\frac{\ee}{2}\]
Therefore
\[\Big\|\frac{1}{n}\sum_{j=1}^mx_{s_j}-\frac{1}{m}\sum_{j=1}^mx_{t_j}\Big\|<\ee\]
and the proof is complete.
\end{proof}
We are now ready for the proof of the main result.
\begin{proof}[Proof of Theorem  \ref{singhighord}]
Let $\ff$ be a regular thin family, $M\in [\nn]^\infty$ and
  $(x_s)_{s\in\ff}$ be a $\ff$-sequence in an Banach space $X$
  such that $(x_s)_{s\in\ff\upharpoonright M}$ generates a
  singular $\ff$-spreading model $(e_n)_{n\in\nn}$. By Lemma \ref{firstlem} and Proposition
  \ref{Schrplces} there exists  $L\in[M]^\infty$ and $x\in X$ such
  that for every $\ee>0$
  there exists $n_0\in\nn$ such that
  \begin{equation}\label{eq2star}\Big\|\frac{1}{n}\sum_{j=1}^nx_{s_j}-x\Big\|<\ee\end{equation}
for all $n\geq n_0$ and
$(s_j)_{j=1}^n\in\text{\emph{Plm}}(\ff\upharpoonright
  L)$ with $s_1(1)\geq L(n)$. For every $s\in\ff\upharpoonright L$ we set
\[x'_s=x_s-x\]
Let $e_n=e'_n+e$ be the natural decomposition
  of $(e_n)_{n\in\nn}$. We will first show that $\|e\|=\|x\|$. Indeed, since
$(e'_n)_{n\in\nn}$ is Ces\`aro summable to zero, we have that
\[\lim_n\Big\|\frac{\sum_{i=1}^n e_i}{n}-e\Big\|=0\] For
each $n\in\nn$ let
$(s_i^n)_{i=1}^n\in\text{\emph{Plm}}(\ff\upharpoonright L)$, with
$s_i^n(1)\geq L(n)$. Then by (\ref{eq2star}), we have that
\[\lim_n\Big\|\frac{1}{n}\sum_{i=1}^n x_{s_i^n}-x\Big\|=0\]
Moreover since $(x_s)_{s\in\ff\upharpoonright L}$ also generates
$(e_n)_{n\in\nn}$ as an $\ff$-spreading model, we obtain that
\[\lim_n\Bigg|\Big\|\frac{1}{n}\sum_{i=1}^n x_{s_i^n}\Big\|-\Big\|\frac{1}{n}\sum_{i=1}^n
e_i\Big\|\Bigg|=0\] Therefore we conclude that
\[\|e\|=\lim_n\Big\|\frac{1}{n}\sum_{i=1}^n e_i\Big\|=\lim_n\Big\|\frac{1}{n}\sum_{i=1}^n x_{s_i^n}\Big\|=\|x\|\]

 We claim that
$(e'_n)_{n\in\nn}$ is the unique $\ff$-spreading model of
$(x'_s)_{s\in\ff\upharpoonright L}$. Indeed let $L'\in [L]^\infty$
be such that $(x'_s)_{s\in\ff\upharpoonright L'}$ generates an
$\ff$-spreading model $(\widetilde{e}_n)_{n\in\nn}$. We have to
show that $(\widetilde{e}_n)_{n\in\nn}$ is isometric to
$(e'_n)_{n\in\nn}$. Notice that for all $n\in\nn$,
 \[\frac{1}{n}\sum_{j=1}^nx'_{s_j}=\frac{1}{n}\sum_{j=1}^nx_{s_j}\]
 Hence by
(\ref{eq2star}) we have that  that $(\widetilde{e}_n)_{n\in\nn}$
is Ces\'aro summable to zero. Therefore by Lemma \ref{isometric
1-subsym sequences} it suffices to show that for every $n\in\nn$
and $\lambda_1,\ldots\lambda_n\in\rr$ with
$\sum_{i=1}^n\lambda_1=0$ we have that
\[\Big\|\sum_{i=1}^n\lambda_ie_i'\Big\|=\Big\|\sum_{i=1}^n\lambda_i\widetilde{e}_i\Big\|\]
Indeed, let $n\in\nn$ and $\lambda_1,\ldots\lambda_n\in\rr$ with
$\sum_{i=1}^n\lambda_1=0$. Let $((s_j^k)_{j=1}^n)_{k\in\nn}$ be a
sequence in $\text{\emph{Plm}}_n(\ff\upharpoonright L)$ such that
$\min s_1^k\to \infty$. Then
\[\Big\|\sum_{i=1}^n\lambda_ie_i'\Big\|=\Big\|\sum_{i=1}^n\lambda_ie_i\Big\|=\lim_{k\to\infty}\Big\|\sum_{i=1}^n\lambda_ix_{s^k_i}\Big\|=
\lim_{k\to\infty}\Big\|\sum_{i=1}^n\lambda_ix'_{s^k_i}\Big\|=\Big\|\sum_{i=1}^n\lambda_i\widetilde{e}_i\Big\|\]
\end{proof}
 We close this subsection by stating  a corresponding result for spaces with separable dual.
 This result will be used later. We will need the following lemmas.
 \begin{lem}\label{convex means to zero yields weak null}
  Let $\ff$ be a regular thin family, $L\in[\nn]^\infty$ and
  $(x'_s)_{s\in\ff}$ be an $\ff$-sequence in a Banach space $X$.
  Suppose that for every $N\in[L]^\infty$ and $\ee>0$ there exist
  $k\in\nn$, $\lambda_1,\ldots,\lambda_k>0$ and
  $(s_j)_{j=1}^k\in\text{\emph{Plm}}(\ff\upharpoonright N)$ such
  that
  \begin{enumerate}
    \item[(i)] $\sum_{j=1}^k\lambda_j=1$ and
    \item[(ii)] $\|\sum_{j=1}^k\lambda_jx'_{s_j}\|<\ee$.
  \end{enumerate}
  Then for every $x^*\in X^*$, $\ee>0$ and $N\in[L]^\infty$ there exists $N'\in[N]^\infty$ such that
  for every $s\in\ff\upharpoonright N'$, $|x^*(x'_s)|<\ee$.
\end{lem}
\begin{proof}
  Let $x^*\in X^*$, $\ee>0$ and $N\in[L]^\infty$. Then by
  Theorem \ref{th2} there exists $N'\in[N]^\infty$ such that one
  of the following holds:
  \begin{enumerate}
    \item[(a)] For every $s\in\ff\upharpoonright N'$,
    $|x^*(x'_s)|<\ee$.
    \item[(b)] For every $s\in\ff\upharpoonright N'$,
    $x^*(x'_s)\geq\ee$.
    \item[(c)] For every $s\in\ff\upharpoonright N'$,
    $x^*(x'_s)\leq-\ee$.
  \end{enumerate}
We will show that the case (b) is excluded.  Indeed, suppose that
case (b) holds. Then by our assumption there exist  $k\in\nn$,
$\lambda_1,\ldots,\lambda_k>0$ and
  $(s_j)_{j=1}^k\in\text{\emph{Plm}}(\ff\upharpoonright N')$ such
  that
  \begin{enumerate}
    \item[(i)] $\sum_{j=1}^k\lambda_j=1$ and
    \item[(ii)] $\|\sum_{j=1}^k\lambda_jx'_{s_j}\|<\ee$.
  \end{enumerate}
  Then \[\Big\|\sum_{j=1}^k\lambda_jx'_{s_j}\Big\|\geq x^*\Big(\sum_{j=1}^k\lambda_jx'_{s_j}\Big)=
  \sum_{j=1}^k\lambda_jx^*(x'_{s_j})\geq \ee\] which contradicts
  (ii). In  a similar way case
(c) is also excluded and therefore (a) holds.
\end{proof}

By a standard diagonalization argument one may prove the
following.
\begin{lem}\label{convex means to zero yields weak null2}
  Let $\ff$ be a regular thin family, $L\in[\nn]^\infty$ and
  $(x'_s)_{s\in\ff}$ be an $\ff$-sequence in a Banach space $X$ with separable dual.
  Suppose that for every $N\in[L]^\infty$ and $\ee>0$ there exists
  $k\in\nn$, $\lambda_1,\ldots,\lambda_k>0$ and
  $(s_j)_{j=1}^k\in\text{\emph{Plm}}(\ff\upharpoonright N)$ such
  that
  \begin{enumerate}
    \item[(i)] $\sum_{j=1}^k\lambda_j=1$ and
    \item[(ii)] $\|\sum_{j=1}^k\lambda_jx'_{s_j}\|<\ee$.
  \end{enumerate}
  Then there exists $N\in[L]^\infty$ such that the $\ff$-subsequence $(x'_s)_{s\in\ff\upharpoonright N}$ is weakly null.
\end{lem}

  \begin{cor}\label{singular in space with separable dual}
  Let $\ff$ be a regular thin family, $M\in[\nn]^\infty$ and
  $(x_s)_{s\in\ff}$ be a $\ff$-sequence in an Banach space $X$ with separable dual
  such that $(x_s)_{s\in\ff\upharpoonright M}$ generates a
  singular $\ff$-spreading model $(e_n)_{n\in\nn}$. Let $e_n=e'_n+e$ be the natural decomposition
  of  $(e_n)_{n\in\nn}$. Then there
  exist $x\in X$ and $N\in[M]^\infty$ such that $(x_s)_{s\in\ff\upharpoonright N}$ weakly converges to $x$, $\|x\|=\|e\|$ and setting
  $x'_s=x_s-x$, for all $s\in\ff\upharpoonright N$, we have that
  $(x'_s)_{s\in\ff\upharpoonright N}$ admits $(e'_n)_{n\in\nn}$ as
  a unique $\ff$-spreading model.
\end{cor}
\begin{proof}
By Theorem \ref{singhighord}, there   exist $x\in X$ and
$L\in[M]^\infty$ such that  $\|x\|=\|e\|$ and setting
  $x'_s=x_s-x$, for all $s\in\ff\upharpoonright L$, we have that
  $(x'_s)_{s\in\ff\upharpoonright L}$ admits $(e'_n)_{n\in\nn}$ as
  a unique $\ff$-spreading model. By Lemma \ref{convex means to zero yields weak null2},  there exists $N\in[L]^\infty$ such that
  $(x'_s)_{s\in\ff\upharpoonright N}$ is weakly null or equivalently $(x_s)_{s\in\ff\upharpoonright N}$
weakly converges to $x$ and the proof is complete.
\end{proof}

\section{Schauder basic spreading models}
In this section we provide sufficient conditions for an
$\ff$-sequence to generate a Schauder basic spreading model.
  \begin{defn}\label{Def of SSD}
    Let $A$ be a countable seminormalized subset of a Banach space $X$.
    We say that $A$ admits a \textit{Skipped Schauder Decomposition} (SSD) if there exist $K>0$ and
    a sequence $(F_n)_{n\in\nn}$ of finite subsets of $A$ such that
    \begin{enumerate}
      \item[(i)] $\cup_{n\in\nn}F_n=A$
      \item[(ii)] For every $L\in[\nn]^\infty$ not containing two successive integers and
      every sequence $(x_l)_{l\in L}$ with $x_l\in F_l$ for all $l\in L$,  $(x_l)_{l\in L}$ is
     a Schauder basic sequence of constant $K$.
    \end{enumerate}
  \end{defn}
  The following proposition is known but for the sake of
  completeness we outline its proof.
  \begin{prop}
    Let $(x_n)_{n\in\nn}$ be a seminormalized weakly null sequence
    in a Banach space $X$. Then for every $\varepsilon>0$ the
    sequence $(x_n)_{n\in\nn}$ admits a SSD with constant
    $1+\varepsilon$.
  \end{prop}
  \begin{proof}
    We assume that $X$ has a Schauder basis $(e_n)_{n\in\nn}$ with
    basis constant 1 (for example we may assume that $X=C[0,1]$).
    By induction and using the sliding hump argument, we define a partition $(F_n)_{n\in\nn}$ of $\nn$
    into finite pairwise disjoint sets and and a sequence
    $(y_n)_{n\in\nn}$ of finite supported vectors in $X$ such that
    the following are
    fulfilled:
    \begin{enumerate}
      \item[(i)] If $n\in F_k$,
      $\|x_n-y_n\|<\frac{\varepsilon}{2^n}$.
      \item[(ii)] If $k+1<l$, $n\in F_k$ and $m\in F_l$ then
      $\max\text{supp}(y_n)<\min\text{supp}(y_m)$.
    \end{enumerate}
    It is
    easy to check that $(F_k)_{k\in\nn}$ is the desired SSD.
  \end{proof}
  \begin{lem}\label{lem48}
    Let $A$ be a subset of a Banach space $X$ admitting a SSD with constant $K$.
    Let $M\in[\nn]^\infty$, $\ff$ be a regular thin family and $(x_s)_{s\in\ff}$ an $\ff$-sequence in $A$.
     Suppose that $(x_s)_{s\in\ff\upharpoonright M}$ is hereditarily nonconstant
    and $(x_s)_{s\in\ff\upharpoonright M}$ generates $(e_n)_{n\in\nn}$ as an $\ff$-spreading model.
    Then $(e_n)_{n\in\nn}$ is Schauder basic with constant $K$.
  \end{lem}
  \begin{proof} Let $(F_k)_{k\in\nn}$ be the
    partition of $A$ witnessing its SSD property and
    $\varphi:\ff\upharpoonright M\to \nn$,
    defined  by $\varphi(s)=k$ if $x_s\in F_k$.
    Observe that $\varphi$ is hereditarily nonconstant in $M$. Indeed, assume on the contrary.
    Then there exists $N_1\in[M]^\infty$ and  $k_0\in\nn$ such that for every  $s\in\ff\upharpoonright N_1$,
     $x_s\in F_{k_0}$ for all $s\in\ff\upharpoonright N_1$.
    Since $F_{k_0}$ is finite, by Theorem \ref{th2} we get $N_2\in[N_1]^\infty$ such that
    $(x_s)_{s\in\ff\upharpoonright N_2}$ is constant. Therefore
    by Theorem \ref{Theorem equivalent forms for having norm on the spreading model},
     we get that $(e_n)_{n\in\nn}$ is trivial which is a contradiction.

    Since $\varphi$ is hereditarily nonconstant in $M$, by Corollary \ref{Proposition combinatorial for SSD}
    there exists $N\in[M]^\infty$ such that $\varphi(s_2)-\varphi(s_1)>1$ for every plegma pair
    $(s_1,s_2)$ in $\ff\upharpoonright N$. Hence, by the SSD property of $A$ we have that for every $l\in\nn$ and for every
    plegma $l$-tuple $(s_j)_{j=1}^l$ in $\ff\upharpoonright N$ the sequence $(x_{s_j})_{j=1}^l$ is Schauder basic with
    constant $K$. This easily yields that $(e_n)_{n\in\nn}$ is Schauder basic with constant $K$.
  \end{proof}
  \begin{thm}\label{ssd giving Schauder basic}
    Let $A$ be a subset of a Banach space $X$. If $A$ admits a SSD with constant
    $K$, then every non trivial spreading model of any order in $A$
    is Schauder basic with constant $K$.
  \end{thm}
   \begin{proof}
    Let $\ff$ be a regular thin family and $(x_s)_{s\in\ff}$ an $\ff$-sequence in $A$. Let $M\in[\nn]^\infty$ such that
     $(x_s)_{s\in\ff\upharpoonright M}$ generates a non trivial $\ff$-spreading model $(e_n)_{n\in\nn}$.
    Then  $(x_s)_{s\in\ff\upharpoonright M}$ is hereditarily nonconstant. Indeed,  otherwise
     there exists $L\in[M]^\infty$ such that $(x_s)_{s\in\ff\upharpoonright L}$ is constant and
     therefore by Theorem \ref{Theorem equivalent forms for having norm on the spreading model}
     we get that $(e_n)_{n\in\nn}$ is trivial.  Hence the assumptions of  Lemma \ref{lem48} hold and the result follows.
  \end{proof}

\chapter{Weakly relatively compact $\ff$-sequences and canonical tree decompositions}
An $\ff$-sequence $(x_s)_{s\in\ff}$ in a Banach space $X$ will be
called \emph{weakly relatively compact} if the weak-closure of its
range $\overline{\{x_s: s\in\ff\}}^{w}$ is a weakly compact subset
of $X$. In this chapter we start with the study of the spreading
models generated by weakly relatively compact $\ff$-sequences.
Next we focus on weakly relatively compact $\ff$-sequences in
Banach spaces with a Schauder basis. In this case it is shown that
they are approximated by sequences of the same class which in
addition share a canonical structure. These results indicate that
weakly relatively compact $\ff$-sequences consist the higher
dimensional analogue of the classical weakly convergent sequences.
\section{Spreading models generated by weakly relatively compact
$\ff-$sequences} In this section we will classify the spreading
models generated by weakly relatively compact $\ff$-sequences. We
start with the following proposition.
\begin{prop}\label{cor for subordinating}
  Let $X$ be a Banach space, $\ff$ a regular thin family and
  $(x_s)_{s\in\ff}$ be a weakly relatively compact $\ff$-sequence in
  $X$. Then for every $M\in[\nn]^\infty$ there exists
  $L\in[M]^\infty$ such that the $\ff$-subsequence $(x_s)_{s\in\ff\upharpoonright L}$ is
  subordinated with respect to the weak topology.
 \end{prop}
 \begin{proof}
 Let $M\in[\nn]^\infty$.   Since the
weak topology on every separable weakly  compact subset of a
Banach space is metrizable, we have that $\overline{\{x_s:
s\in\ff\}}^{w}$  is compact metrizable. By Proposition \ref{Create
subordinated} the  result follows.
 \end{proof}
 We give  the next definition regarding the spreading models which are generated by weakly relatively compact $\ff$-sequences.
\begin{defn}
  Let $X$ be a Banach space and $\xi<\omega_1$. By
  $\mathcal{SM}_\xi^{wrc}(X)$ we will denote
    the set of all spreading sequences
    $(e_n)_{n\in\nn}$ such that there exists
     a weakly relatively compact subset $A$ of
    $X$ such that $A$ admits $(e_n)_{n\in\nn}$ as a $\xi$-spreading model. We  also set\[\mathcal{SM}^{wrc}(X)=\bigcup_{\xi<\omega_1}\mathcal{SM}_\xi^{wrc}(X)\]
\end{defn}
\begin{prop}\label{wrc}
 Let $X$ be a Banach space, $\xi<\omega_1$ and $(e_n)_{n\in\nn}\in
\mathcal{SM}_\xi^{wrc}(X)$. Then for every regular thin family
$\g$ with $o(\g)\geq \xi$ there exist
  a weakly relatively compact $\g$-sequence $(w_t)_{t\in\g}$ in $X$ and   $L\in[\nn]^\infty$
such that the following hold. The subsequence
$(w_t)_{t\in\g\upharpoonright L}$  is
  subordinated with respect to the weak topology and generates
  $(e_n)_{n\in\nn}$ as a $\g$-spreading model.
\end{prop}
\begin{proof}
   Since $(e_n)_{n\in\nn}\in \mathcal{SM}_\xi^{wrc}(X)$ there
   exists a weakly relatively compact subset $A$ of
    $X$ such that $A$ admits $(e_n)_{n\in\nn}$ as a $\xi$-spreading model. Hence there exists a regular thin family
    $\ff$ of order $\xi$, an $\ff$-sequence $(x_s)_{s\in\ff}$ in $A$ and $M\in[\nn]^\infty$
   such that $(x_s)_{s\in\ff\upharpoonright M}$ generates
  $(e_n)_{n\in\nn}$ as an $\ff$-spreading model. By  Proposition
  \ref{propxi} there exist  a $\g$-sequence
  $(w_t)_{t\in\g}$ and $N\in[\nn]^\infty$  such that $(w_t)_{t\in\g\upharpoonright N}$
  generates $(e_n)_{n\in\nn}$ as a $\g$-spreading model and moreover
  $\{w_t:t\in\g\}\subseteq\{x_s:s\in\ff\}\subseteq A$.
  Hence  $(w_t)_{t\in\g}$
  is a weakly relatively compact $\g$-sequence. By Proposition \ref{cor for subordinating}
  there exists $L\in[N]^\infty$ such that  $(w_t)_{t\in\g\upharpoonright L}$  is
  subordinated with respect to the weak topology. Clearly   $(w_t)_{t\in\g\upharpoonright L}$ also generates
  $(e_n)_{n\in\nn}$ as a $\g$-spreading model and the proof is
  complete.
\end{proof}
By the above proposition we obtain the following.
\begin{cor} Let $X$ be a Banach space and $1\leq \zeta<\xi$ be
countable ordinals. Then $\mathcal{SM}_\zeta^{wrc}(X)\subseteq
\mathcal{SM}_\xi^{wrc}(X)$.
\end{cor}
Proposition \ref{wrc} states  that every sequence
$(e_n)_{n\in\nn}$ in $ \mathcal{SM}^{wrc}(X)$ is generated by a
subordinated $\ff$-subsequence. In the next two lemmas we present
some useful consequences of this fact.
\begin{lem}\label{lemma either l1 or not Schauder basic}
 Let $X$ be a Banach space, $\ff$ be a regular thin family, $M\in[\nn]^\infty$  and $(x_s)_{s\in\ff}$ be
 an $\ff$-sequence in $X$ such that the following are satisfied.
\begin{enumerate}
  \item[(i)] The $\ff$-subsequence $(x_s)_{s\in\ff\upharpoonright
  M}$ generates a nontrivial  $\ff$-spreading model
  $(e_n)_{n\in\nn}$.
  \item[(ii)] The $\ff$-subsequence $(x_s)_{s\in\ff\upharpoonright
  M}$ is subordinated with respect to the weak topology of $X$
   and let $\widehat{\varphi}:\widehat{\ff}\upharpoonright M\to
   (X,w)$ be the continuous function witnessing this.
\end{enumerate}
    If $\widehat{\varphi}(\emptyset)\neq 0$ then either
  $(e_n)_{n\in\nn}$ is   equivalent to the usual basis of
$\ell^1$ or $(e_n)_{n\in\nn}$ is singular.
   \end{lem}
\begin{proof} Let $x_s'=x_s-\widehat{\varphi}(\emptyset)$, for
all $s\in\ff\upharpoonright M$. Let $(e_n')_{n\in\nn}$ be an
$\ff$-spreading model generated by a subsequence of
$(x'_s)_{s\in\ff\upharpoonright M}$. Notice that
$\widehat{\varphi'}(\emptyset)=0$ and therefore by Theorem
\ref{unconditional spreading model} we have that either
$(e'_n)_{n\in\nn}$ is trivial (with $\|\widetilde{e}_n\|_*=0$ for
all $n\in\nn$), or $(e'_n)_{n\in\nn}$ is an unconditional sequence.

In the first case, by Corollary \ref{trivial ultrafilter weak
property} we have that  $(e_n)_{n\in\nn}$ is trivial which is a
contradiction. In the second case, by Proposition \ref{l1 vs Cezaro
summability}
  we have that either $(e_n')_{n\in\nn}$ is equivalent to the usual basis of $\ell^1$ or $(e_n')_{n\in\nn}$ is Ces\`aro summable to zero.
  If $(e_n')_{n\in\nn}$ is equivalent to the basis of $\ell^1$, then by Corollary \ref{ultrafilter property for ell^1 spr mod simplified}
   the same holds for the sequence $(e_n)_{n\in\nn}$. Also since $(e_n')_{n\in\nn}$ is
  unconditional,
again by Corollary \ref{trivial ultrafilter weak property} we have
that $(e_n)_{n\in\nn}$ is nontrivial. Hence it  remains to show that
  if $(e_n')_{n\in\nn}$ is Ces\`aro summable to zero then  $(e_n)_{n\in\nn}$ is not Schauder basic.

 Let $\varepsilon>0$ and fix $n_0\in\nn$ such that
  $\|\frac{1}{n_0}\sum_{j=1}^{n_0}e_j'\|<\varepsilon$.
Notice that for every $n\in\nn$ and $a_1,\ldots,a_n\in\rr$,
\[\Big|\sum_{j=1}^na_j\Big|\cdot \|\widehat{\varphi}(\emptyset)\|-
  \Big\|\sum_{j=1}^na_je_j'\Big\|\leq\Big\|\sum_{j=1}^na_je_j\Big\|\leq \Big|\sum_{j=1}^na_j\Big|\cdot \|\widehat{\varphi}(\emptyset)\|
  +\Big\|\sum_{j=1}^na_je_j'\Big\|\]
 Hence
\[\|\widehat{\varphi}(\emptyset)\|-\varepsilon\leq\Big\|\frac{1}{n_0}\sum_{j=1}^{n_0}e_j\Big\|\]
  and
  \[\Big\|\frac{1}{n_0}\sum_{j=1}^{n_0}e_j-\frac{1}{n_0}\sum_{j=n_0+1}^{2n_0}e_j\Big\| \leq2\varepsilon\]
  Suppose that $(e_n)_{n\in\nn}$ is Schauder basic. Then there
  exists a constant $C>0$ such that
  \[\Big\|\frac{1}{n_0}\sum_{j=1}^{n_0}e_j\Big\|\leq C
  \Big\|\frac{1}{n_0}\sum_{j=1}^{n_0}e_j-\frac{1}{n_0}\sum_{j=n_0+1}^{2n_0}e_j\Big\|\]
  and therefore by the above, we would have
  \[\|\widehat{\varphi}(\emptyset)\|-\varepsilon\leq 2C \ee\]
  for all $\ee>0$. Hence  $\|\widehat{\varphi}(\emptyset)\|=0$ which is a contradiction.
\end{proof}

\begin{lem}\label{2lemwrc} Let $X$ be a Banach space, $\ff$ be a regular thin family, $M\in[\nn]^\infty$  and $(x_s)_{s\in\ff}$ be
 an $\ff$-sequence in $X$ such that the following are satisfied.
\begin{enumerate}
  \item[(i)] The $\ff$-subsequence $(x_s)_{s\in\ff\upharpoonright
  M}$ generates a singular  $\ff$-spreading model
  $(e_n)_{n\in\nn}$ and let   $e_n=e'_n+e$ be the natural
  decomposition of $(e_n)_{n\in\nn}$.
  \item[(ii)] The $\ff$-subsequence $(x_s)_{s\in\ff\upharpoonright
  M}$ is subordinated with respect to the weak topology of $X$
  and let $\widehat{\varphi}:\widehat{\ff}\upharpoonright M\to
  (X,w)$ be the continuous map witnessing this.
\end{enumerate}

Then $\|\widehat{\varphi}(\emptyset)\|=\|e\|$ and setting
  $x'_s=x_s-\widehat{\varphi}(\emptyset)$, for all $s\in\ff\upharpoonright M$, we have that
  $(x'_s)_{s\in\ff\upharpoonright M}$ admits $(e'_n)_{n\in\nn}$ as
  a unique $\ff$-spreading model.
\end{lem}
\begin{proof} As in the
proof of Lemma \ref{lemma either l1 or not Schauder basic} the map
$\widehat{\varphi}':\widehat{\ff}\upharpoonright M\to
  (X,w)$, with
  $\widehat{\varphi}'(t)=\widehat{\varphi}(t)-\widehat{\varphi}(\emptyset)$ for
  all $t\in\widehat{\ff}\upharpoonright M$,  witnesses that $(x'_s)_{s\in\ff\upharpoonright M}$ is
  subordinated with $\widehat{\varphi}'(\emptyset)=0$.
  Let $L\in[M]^\infty$ such that the $\ff$-subsequence $(x'_s)_{s\in\ff\upharpoonright
  L}$ generates an $\ff$-spreading model $(\widetilde{e}_n)_{n\in\nn}$. By Corollary \ref{trivial ultrafilter weak property} we have that $(\widetilde{e}_n)_{n\in\nn}$ is nontrivial. Therefore, since $\widehat{\varphi}'(\emptyset)=0$, by Theorem \ref{unconditional spreading model} we have that
  $(\widetilde{e}_n)_{n\in\nn}$ is unconditional. Since $(e_n)_{n\in\nn}$ is
  singular, by Corollary \ref{ultrafilter property for ell^1 spr mod simplified}, we have that $(\widetilde{e}_n)_{n\in\nn}$ is not equivalent to the
  usual basis of $\ell^1$. By Proposition \ref{l1 vs Cezaro
  summability}, we get that $(\widetilde{e}_n)_{n\in\nn}$ is Ces\'aro summable
  to zero. We will first show that $\|\widehat{\varphi}(\emptyset)\|=\|e\|$. Let
  $((s^n_j)_{j=1}^n)_{n\in\nn}$ be a sequence in
  $\text{\emph{Plm}}(\ff\upharpoonright L)$ such that $\min
  s_1^n\geq L(n)$. Using that $(\widetilde{e}_n)_{n\in\nn}$ is Ces\'aro summable
  to zero, it is see that \[\lim_{n\to\infty}\Big\|\frac{1}{n}\sum_{j=1}^nx'_{s^n_j}\Big\|=0\]
  Moreover since $(e'_n)_{n\in\nn}$ is also Ces\'aro summable to
  zero, we have that
  \[\begin{split}\|e\|&=\lim_{n\to\infty}\Big\|e+\frac{1}{n}\sum_{j=1}^ne'_j\Big\|=
  \lim_{n\to\infty}\Big\|\frac{1}{n}\sum_{j=1}^ne_j\Big\|=
  \lim_{n\to\infty}\Big\|\frac{1}{n}\sum_{j=1}^nx_{s^n_j}\Big\|\\&=
  \lim_{n\to\infty}\Big\|\widehat{\varphi}(\emptyset)+\frac{1}{n}\sum_{j=1}^nx'_{s^n_j}\Big\|=\|\widehat{\varphi}(\emptyset)\|\end{split}\]
  It remains  to show that $(e'_n)_{n\in\nn}$ and $(\widetilde{e}_n)_{n\in\nn}$ are isometric. By Lemma \ref{isometric 1-subsym
sequences} it suffices to show that for every $n\in\nn$ and
$\lambda_1,\ldots\lambda_n\in\rr$ with $\sum_{i=1}^n\lambda_1=0$ we
have that
\[\Big\|\sum_{i=1}^n\lambda_ie_i'\Big\|=\Big\|\sum_{i=1}^n\lambda_i\widetilde{e}_i\Big\|\]
Indeed, let $n\in\nn$ and $\lambda_1,\ldots\lambda_n\in\rr$ with
$\sum_{i=1}^n\lambda_i=0$. Let $((s_j^k)_{j=1}^n)_{k\in\nn}$ be a
sequence in $\text{\emph{Plm}}_n(\ff\upharpoonright L)$ such that
$\min s_1^k\to \infty$. Then
\[\Big\|\sum_{i=1}^n\lambda_ie_i'\Big\|=\Big\|\sum_{i=1}^n\lambda_ie_i\Big\|=\lim_{k\to\infty}\Big\|\sum_{i=1}^n\lambda_ix_{s^k_i}\Big\|=
\lim_{k\to\infty}\Big\|\sum_{i=1}^n\lambda_ix'_{s^k_i}\Big\|=\Big\|\sum_{i=1}^n\lambda_i\widetilde{e}_i\Big\|\]
\end{proof}
We are now ready to state the main result of this section.
\begin{thm}\label{relatively weakly compact sets have unconditional spreading models}
  Let $X$ be a Banach space, $\ff$ be a regular thin family and $(x_s)_{s\in\ff}$ be a weakly
  relatively compact $\ff$-sequence which admits a spreading sequence $(e_n)_{n\in\nn}$ as an $\ff$-spreading model. Then exactly
  one of the following holds:
  \begin{enumerate}
    \item[(i)] The sequence $(e_n)_{n\in\nn}$ is trivial.
    \item[(ii)] The sequence $(e_n)_{n\in\nn}$ is singular. In this case there exist
     $L\in[\nn]^\infty$ and $x_0\in X$ such that if $e_n=e'_n+e$
     is the natural decomposition of $(e_n)_{n\in\nn}$ then the
    $\ff$-subsequence $(x'_s)_{s\in \ff\upharpoonright L}$, defined by
    $x'_s=x_s-x_0$ for all $s\in\ff\upharpoonright L$, generates
    the sequence $(e'_n)_{n\in\nn}$ as an $\ff$-spreading model
    and $\|x_0\|=\|e\|$.
    \item[(iii)] The sequence $(e_n)_{n\in\nn}$ is Schauder basic. In this case $(e_n)_{n\in\nn}$ is unconditional.
  \end{enumerate}
\end{thm}
\begin{proof}
Let us assume that $(e_n)_{n\in\nn}$ is
nontrivial, i.e. assertion (i) does not hold.
Let $M\in[\nn]^\infty$  such that $(x_s)_{s\in\ff \upharpoonright M}$
generates $(e_n)_{n\in\nn}$. By Proposition \ref{cor for
subordinating} there exists $L\in[M]^\infty$ such that
  $(x_s)_{s\in\ff\upharpoonright L}$ is subordinated and let $\widehat{\varphi}:\widehat{\ff}\upharpoonright L\to (X,w)$
  be the continuous map witnessing this.
  We distinguish two cases. Either $\widehat{\varphi}(\emptyset)=0$, or $\widehat{\varphi}(\emptyset)\neq 0$.
  In the first case, by Theorem \ref{unconditional spreading model}, we have that $(e_n)_{n\in\nn}$
  is $1$-unconditional and in the second case, by Lemma \ref{lemma either l1 or not Schauder basic} we have that $(e_n)_{n\in\nn}$ is
  equivalent to the usual basis of $\ell^1$ or singular. Hence in both cases we conclude that if $(e_n)_{n\in\nn}$ is Schauder basic
  then it is  unconditional which yields assertion (iii).
Finally if the sequence $(e_n)_{n\in\nn}$ is singular then Lemma
\ref{2lemwrc} shows that  (ii) holds (with
$x_0=\widehat{\varphi}(\emptyset)$) and the proof is complete.
\end{proof}

\section{Tree decompositions of weakly relatively compact $\ff$-sequences}
It is well known that a spreading model generated by a weakly null
sequence $(x_n)_{n\in\nn}$ in a Banach space $X$ with a Schauder
basis is also  generated by a block sequence
$(\widetilde{x}_n)_{n\in\nn}$ which sufficiently approximates
$(x_n)_{n\in\nn}$. In this section we will show that analogues of
this fact appear in  spreading models
 which are generated by weakly relatively compact $\ff$-sequences
in $X$. To state our results we will need the concepts of the
\emph{block tree decomposition} and \emph{canonical tree
decomposition} of an $\ff$-sequence.
\subsection{Block tree decompositions of $\ff$-sequences}
\begin{defn} \label{defn block tree decomposition}
Let $X$ be a Banach space with a Schauder basis. Let $\ff$ be a regular thin family and
  $(\widetilde{x}_s)_{s\in\ff}$ be an $\ff$-sequence in $X$. Let $M\in[\nn]^\infty$
   and $(\widetilde{y}_t)_{t\in\widehat{\ff} \upharpoonright M}$ be
  a family of vectors in $X$. We will say that  $(\widetilde{y}_t)_{t\in\widehat{\ff} \upharpoonright M}$ is a
  block
  tree decomposition of $(\widetilde{x}_s)_{s\in\ff\upharpoonright M}$ (or $(\widetilde{x}_s)_{s\in\ff\upharpoonright M}$
   admits $(\widetilde{y}_t)_{t\in\widehat{\ff} \upharpoonright M}$ as a
  block tree decomposition) if the following are satisfied.
  \begin{enumerate}
    \item[(i)] For every $s\in\ff\upharpoonright M$, $\displaystyle\widetilde{x}_s=\sum_{k=0}^{|s|}\widetilde{y}_{s|k}=
    \widetilde{y}_\emptyset+\sum_{k=1}^{|s|}\widetilde{y}_{s|k}$.
    \item[(ii)] For every
    $t\in\widehat{\ff}\setminus\{\emptyset\}$,
$\text{supp}(\widetilde{y}_{t})$ is finite.
 \item[(iii)]
For every $t\in\widehat{\ff}\upharpoonright M\setminus
\ff\upharpoonright M$ and for every $l_1, l_2$ in M with $
t<l_1<l_2$ and $t\cup\{l_1\}, t\cup\{l_2\}$ in
$\widehat{\ff}\upharpoonright M$,
\[\text{supp}(\widetilde{y}_{t\cup\{l_1\}})<\text{supp}(\widetilde{y}_{t\cup\{l_2\}})\]
\end{enumerate}
\end{defn}
\begin{rem}\label{rembtd} (i) Let's point out that in the above definition we allow   some
of the vectors $\widetilde{y}_{s|k}$ to  be equal to zero (in
which case $\text{supp}(\widetilde{y}_{s|k})=\emptyset$).

(ii) It is easy to see that if
$(\widetilde{y}_t)_{t\in\widehat{\ff} \upharpoonright M}$ is a
block tree decomposition of
$(\widetilde{x}_s)_{s\in\ff\upharpoonright M}$ then for every
$L\in[M]^\infty$,  $(\widetilde{y}_t)_{t\in\widehat{\ff}
\upharpoonright L}$ is a
  block
  tree decomposition of $(\widetilde{x}_s)_{s\in\ff\upharpoonright
  L}$.
\end{rem}
\begin{prop}\label{uniqueness of block tree decomp}
   Let $X$ be a Banach space with a Schauder basis. Let $\ff$ be a regular thin family, $M\in[\nn]^\infty$ and
  $(\widetilde{x}_s)_{s\in\ff}$ be an $\ff$-sequence in $X$ such
  that $(\widetilde{x}_s)_{s\in\ff\upharpoonright M}$
  admits a block tree decomposition $(\widetilde{y}_t)_{t\in\widehat{\ff} \upharpoonright
  M}$. Then $(\widetilde{y}_t)_{t\in\widehat{\ff} \upharpoonright
  M}$ is the unique block tree decomposition of
  $(\widetilde{x}_s)_{s\in\ff\upharpoonright M}$.
\end{prop}
\begin{proof}
  We proceed by induction on the order of $\ff$. Let $o(\ff)=1$. Then without loss of generality we may assume that $\ff=[\nn]^1$.
   Let $(\widetilde{y}_t)_{t\in[M]^{\leq1}}$ and $(\widetilde{y}'_t)_{t\in[M]^{\leq1}}$ be two block tree decompositions
   of a sequence $(x_s)_{s\in[M]^1}$. Suppose that $(\widetilde{y}_t)_{t\in[M]^{\leq1}}\neq(\widetilde{y}'_t)_{t\in[M]^{\leq1}}$.
   This easily yields that $\widetilde{y}_\emptyset\neq\widetilde{y}'_\emptyset$. Let $z=\widetilde{y}_\emptyset-\widetilde{y}'_\emptyset\neq0$.
    Then for every $n\in M$, we have that \[\widetilde{x}_{\{n\}}=\widetilde{y}_{\{n\}}+\widetilde{y}_{\emptyset}=
    \widetilde{y}'_{\{n\}}+\widetilde{y}'_{\emptyset}\] Hence for every $n\in M$, we have that
    \[\widetilde{y}'_{\{n\}}-\widetilde{y}_{\{n\}}=\widetilde{y}_{\emptyset}-\widetilde{y}'_{\emptyset}=z\]
  But this contradicts the third condition  of  Definition \ref{defn block tree decomposition}.

  Suppose that for some $\xi<\omega_1$ and  for every regular thin family of order strictly less than $\xi$
  the proposition holds. Let $\ff$ be a regular thin family of order $\xi$, $M\in[\nn]^\infty$
   and $(\widetilde{x}_s)_{s\in\ff\upharpoonright M}$ an $\ff$-subsequence in $X$ which admits  two block tree
    decompositions $(\widetilde{y}_t)_{t\in\widehat{\ff}\upharpoonright M}$ and $(\widetilde{y}'_t)_{t\in\widehat{\ff}\upharpoonright M}$.

  \textbf{Claim:} For every $m\in M$, we have that $\widetilde{y}_{\{m\}\cup t}=\widetilde{y}'_{\{m\}\cup t}$,
  for all $t\in\widehat{\ff_{(m)}}\upharpoonright M$.
  \begin{proof}[Proof of claim]
  Let $m\in M$ and $\g=\ff_{(m)}$. Then $\g$ is regular thin with $o(\g)<o(\ff)=\xi$. Let $z_s=x_{\{m\}\cup s}$, for all $s\in\g$.
  Let also $w_t=\widetilde{y}_{\{m\}\cup t}$ and $w'_t=\widetilde{y}'_{\{m\}\cup t}$, for all $t\in\widehat{\g}$.
  It is easy to check that $(w_t)_{t\in \widehat{\g}\upharpoonright M}$ and $(w'_t)_{t\in \widehat{\g}\upharpoonright M}$ are both block tree
   decompositions of $(z_s)_{s\in\g\upharpoonright M}$. Hence by the  inductive hypothesis we have
    that  $w_t=w'_t$ that is $\widetilde{y}_{\{m\}\cup t}=\widetilde{y}'_{\{m\}\cup
    t}$, for all $t\in\widehat{\ff_{(m)}}\upharpoonright M$. \end{proof}
 By the claim we have that  $\widetilde{y}_{s|k}=\widetilde{y}'_{s|k}$,
  for all $s\in\ff \upharpoonright M$ and $1\leq k\leq |s|$.
 Since $\widetilde{x}_s=
    \widetilde{y}_\emptyset+\sum_{k=1}^{|s|}\widetilde{y}_{s|k}=\widetilde{y}'_\emptyset+\sum_{k=1}^{|s|}\widetilde{y}'_{s|k}$
 we also conclude that $\widetilde{y}_{\emptyset}=\widetilde{y}'_{\emptyset}$.
\end{proof}
\begin{prop}\label{block tree decomp and subordinated are the same}
  Let $X$ be a Banach space with a Schauder basis $(w_n)_{n\in\nn}$. Let $\ff$ be a regular thin family, $M\in[\nn]^\infty$ and
  $(\widetilde{x}_s)_{s\in\ff}$ be an $\ff$-sequence in $X$ such
  that $(\widetilde{x}_s)_{s\in\ff\upharpoonright M}$
  admits a block tree decomposition $(\widetilde{y}_t)_{t\in\widehat{\ff} \upharpoonright
  M}$. Also assume  that $(\widetilde{x}_s)_{s\in\ff\upharpoonright
  M}$ is subordinated with respect the weak topology and let
  $\widehat{\varphi}:\widehat{\ff}\upharpoonright M\to X$ be the
  continuous map witnessing that. Then we have that
  \[\widehat{\varphi}(t)=\sum_{t'\sqsubseteq
  t}\widetilde{y}_{t'}\] for all $t\in\widehat{\ff}\upharpoonright
  M$. In particular
  $\widehat{\varphi}(\emptyset)=\widetilde{y}_{\emptyset}$.
\end{prop}
\begin{proof}
  Suppose that
  $o(\ff)=1$. Without loss of generality, we may assume that
  $\ff=\big\{\{n\}:n\in\nn\big\}$. Then
  \[\widetilde{x}_{\{n\}}=\widehat{\varphi}(\{n\})=\widetilde{y}_{\{n\}}+\widetilde{x}_{\emptyset}\]
  for all $n\in\nn$. Let $k\in\nn$. Since
  $(\widetilde{y}_{\{n\}})_{n\in\nn}$ forms a block sequence in
  $X$, for every $k\in\nn$ we have that
  $w^*_k(\widetilde{y}_{\{n\}})\to0$. Therefore, for every
  $k\in\nn$, we have that $w^*_k(\widetilde{y}_\emptyset+\widetilde{y}_{\{n\}})\to
  w^*_k(\widetilde{y}_\emptyset)$. Since $\widehat{\varphi}$ is
  continuous with respect to the weak topology, we have that
  $w^*_k(\widehat{\varphi}(\{n\}))\to
  w^*_k(\widehat{\varphi}(\emptyset))$, for all $k\in\nn$. Hence
  $w^*_k(\widetilde{y}_\emptyset)=w^*_k(\widehat{\varphi}(\emptyset))$,
  for all $k\in\nn$, that is
  $\widetilde{y}_\emptyset=\widehat{\varphi}(\emptyset)$.
  The proof proceeds by induction on the order of $\ff$ and
  by arguing as in the proof of Lemma \ref{uniqueness of block tree
  decomp}.
\end{proof}
\subsection{Canonical tree decompositions of $\ff$-sequences}
\begin{defn}\label{Def of plegma supported}
  Let $X$ a Banach space with a Schauder basis. Let  $\ff$ be a regular thin family and
  $(\widetilde{x}_s)_{s\in\ff}$ be  an $\ff$-sequence in $X$. Let $L\in[\nn]^\infty$
   and $(\widetilde{y}_t)_{t\in\widehat{\ff} \upharpoonright L}$ be
  a family of vectors in $X$. We will say that  $(\widetilde{y}_t)_{t\in\widehat{\ff} \upharpoonright L}$ is a
  canonical tree
  decomposition of $(\widetilde{x}_s)_{s\in\ff\upharpoonright L}$ (or $(\widetilde{x}_s)_{s\in\ff\upharpoonright L}$
   admits $(\widetilde{y}_t)_{t\in\widehat{\ff} \upharpoonright L}$ as a
  canonical  tree decomposition) if the following are satisfied.
  \begin{enumerate}
    \item[(i)]  The family $(\widetilde{y}_t)_{t\in\widehat{\ff} \upharpoonright L}$ is a
  block tree
  decomposition of $(\widetilde{x}_s)_{s\in\ff\upharpoonright L}$.
   \item[(ii)] For every $s\in\ff\upharpoonright L$ and $1\leq k_1<k_2\leq|s|$, \[\text{supp}(\widetilde{y}_{s|k_1})
     <\text{supp}(\widetilde{y}_{s|k_2})\]
    \item[(iii)] For every plegma pair $(s_1,s_2)$ in $\ff\upharpoonright L$
    the following are satisfied.
        \begin{enumerate}
          \item[(a)] For every $1\leq k_1\leq|s_1|$ and $1\leq k_2\leq|s_2|$ with $k_1\leq k_2$,
           \[\text{supp}(\widetilde{y}_{s_1|k_1})
           <\text{supp}(\widetilde{y}_{s_2|k_2})\]
\item[(b)] For every $1\leq k_1<k_2\leq |s_1|$, \[
\text{supp}(\widetilde{y}_{s_2|k_1})
<\text{supp}(\widetilde{y}_{s_1|k_2})\]
        \end{enumerate}
  \end{enumerate}

\end{defn}
\begin{rem}\label{rem on canonical tree decomp}
   Let $X$ be a Banach space with a Schauder basis. Let $\ff$ be a regular thin family, $L\in[\nn]^\infty$ and
  $(\widetilde{x}_s)_{s\in\ff}$ be an $\ff$-sequence in $X$ such
  that $(\widetilde{x}_s)_{s\in\ff\upharpoonright L}$
  admits a canonical tree decomposition $(\widetilde{y}_t)_{t\in\widehat{\ff} \upharpoonright
  L}$. We have the following:
  \begin{enumerate}
    \item[(i)] For every $N\in[L]^\infty$  the $\ff$-subsequence $(\widetilde{x}_s)_{s\in\ff\upharpoonright
  N}$ admits $(\widetilde{y}_t)_{t\in\widehat{\ff} \upharpoonright
  N}$ as a canonical tree decomposition.
    \item[(ii)] The canonical
  decomposition of $(\widetilde{x}_s)_{s\in\ff\upharpoonright
  L}$ is unique. This follows by  Proposition
  \ref{uniqueness of block tree decomp} and the fact that every canonical
  tree decomposition is also a block tree decomposition.
    Moreover, if in addition we assume that $(\widetilde{x}_s)_{s\in\ff\upharpoonright
  M}$ is subordinated with $\widehat{\varphi}:\widehat{\ff}\upharpoonright L\to
  X$ being the continuous map witnessing that, then by Proposition \ref{block tree decomp and subordinated are the same}
  we have that \[\widehat{\varphi}(t)=\sum_{t'\sqsubseteq
  t}\widetilde{y}_{t'}\] for all $t\in\widehat{\ff}\upharpoonright
  L$.
  \end{enumerate}
\end{rem}

\begin{lem}\label{blcan}
   Let $X$ be a Banach space with a Schauder basis. Let $\ff$ be a regular thin family, $M\in[\nn]^\infty$ and
  $(\widetilde{x}_s)_{s\in\ff}$ be an $\ff$-sequence in $X$ such
  that $(\widetilde{x}_s)_{s\in\ff\upharpoonright M}$
  admits a block tree decomposition $(\widetilde{y}_t)_{t\in\widehat{\ff} \upharpoonright M}$. Then there exists
  $L\in[M]^\infty$ such that  $(\widetilde{x}_s)_{s\in\ff\upharpoonright
  L}$ admits  $(\widetilde{y}_t)_{t\in\widehat{\ff} \upharpoonright L}$  as a canonical tree
  decomposition.
\end{lem}
\begin{proof} Let $M_1\in[M]^\infty$ such that  $\ff$ is very large in $M_1$.
   Let
  \[\mathcal{F}_1=\big{\{}s\in\ff\upharpoonright M_1:\;\text{supp}(\widetilde{y}_{s|k_1}) <\text{supp}
  (\widetilde{y}_{s|k_2}),\;\text{for all}\;1\leq k_1<k_2\leq
  |s|\big{\}}\]
  Using   (iii) of Definition \ref{defn block tree decomposition}
   it is easy to see that $\mathcal{F}_1$ is large in $M_1$.
  Hence by Theorem \ref{th2} there exists $M_2\in[M_1]^\infty$ such that $\ff\upharpoonright M_2\subseteq \mathcal{F}_1$, that is
  condition (ii) of the Definition \ref{Def of plegma supported} is satisfied for all $s\in\ff\upharpoonright M_2$.

  Let $A$ be the set all plegma pairs in $\ff\upharpoonright M_2$ satisfying condition (iii) of
  Definition \ref{Def of plegma supported}.
  Again we see that $A$ is large in $M_2$ and therefore by Corollary \ref{corollary ramseyforplegma}
    there exists $L\in[M_2]^\infty$ such that  every plegma pair in $\ff\upharpoonright L$
  belongs to $A$. Finally, by (ii) of Remark \ref{rembtd}, we have that $(\widetilde{y}_t)_{t\in\ff\upharpoonright L}$ is also a
  block tree decomposition of $(\widetilde{x}_s)_{s\in\ff\upharpoonright L}$. Therefore  the family
    $(\widetilde{y}_t)_{t\in\widehat{\ff}\upharpoonright L}$ is a canonical tree decomposition of $(\widetilde{x}_s)_{s\in\ff\upharpoonright L}$
    and the proof is complete.
\end{proof}
\subsection{Subordinated $\ff$-sequences and canonical tree
decompositions}
\begin{thm}\label{canonical tree}
  Let $X$ be a Banach space with a Schauder basis. Let $\ff$ be a regular thin family, $M\in[\nn]^\infty$ and
  $(x_s)_{s\in\ff}$ be an $\ff$-sequence in $X$ such
  that $(x_s)_{s\in\ff\upharpoonright M}$ is
  subordinated with respect to the weak topology of $X$ and let $\widehat{\varphi}:\widehat{\ff}\upharpoonright M\to (X,w)$
   be the continuous map witnessing this. Then for every  sequence
  $(\ee_n)_{n\in\nn}$ of positive reals, there exist $L\in[M]^\infty$ and an
  $\ff$-subsequence $(\widetilde{x}_s)_{s\in\ff\upharpoonright L}$
  in $X$ satisfying the following.
  \begin{enumerate}
    \item[(i)] For every $s\in\ff\upharpoonright L$, $\|x_s-\widetilde{x}_s\|<\ee_n$, where $\min
  s=L(n)$.
\item[(ii)] The $\ff$-subsequence
$(\widetilde{x}_s)_{s\in\ff\upharpoonright
    L}$ is subordinated with respect to the weak topology of $X$. Moreover, if
    $\widetilde{\varphi}:\widehat{\ff}\upharpoonright L\to (X,w)$ is the continuous map witnessing
    this, then
    $\widetilde{\varphi}(\emptyset)=\widehat{\varphi}(\emptyset)$.
    \item[(iii)] The $\ff$-subsequence $(\widetilde{x}_s)_{s\in\ff\upharpoonright
    L}$ admits a canonical tree decomposition.
\end{enumerate}
\end{thm}
\begin{proof}
  Passing to an infinite subset of $M$ if necessary, we may assume
  that $\ff$ is very large in $M$.
  Let $(\ee_n)_{n\in\nn}$ be a sequence of positive reals. We may
  suppose that $(\ee_n)_{n\in\nn}$ is decreasing.
 For every
 $s\in \ff\upharpoonright M$  and $1\leq k\leq |s|$,  let \[y_{s|k}=\widehat{\varphi}(s|k)-\widehat{\varphi}(s|k-1)\]
Clearly, for every $t\in\widehat{\ff}\upharpoonright M\setminus
\ff\upharpoonright M$, the sequence $(y_{t\cup\{m\}})_{m\in M}$ is
weakly null.

\textbf{Claim 1.} There exist $M_1\in[M]^\infty$ and a family
$(\widetilde{y}_t)_{t\in\widehat{\ff} \upharpoonright M_1}$ such
that the following are satisfied:
\begin{enumerate}
\item[(i)] $\widetilde{y}_\emptyset=y_\emptyset$ and for every
$t\in(\widehat{\ff}\upharpoonright M_1)\setminus\{\emptyset\}$,
$\|y_t-\widetilde{y}_t\|<\frac{1}{2^{|t|}}\ee_{n_t}$, where $\max
t=M_1(n_t)$.

\item[(ii)] For every $t\in(\widehat{\ff}\upharpoonright
M_1)\setminus\{\emptyset\}$, the support of $\widetilde{y}_t$ is
finite.

\item[(iii)] For every $t\in(\widehat{\ff}\upharpoonright
M_1)\setminus\widehat{\ff}$ and $m_1, m_2$ in $M_1$ with $\max
t<m_1<m_2$, we have that
$\text{supp}(\widetilde{y}_{t\cup\{m_1\}})<\text{supp}(\widetilde{y}_{t\cup\{m_2\}})$.
\end{enumerate}
\begin{proof}[proof of Claim 1]
  Let $L_0=M$ and $\widetilde{y}_\emptyset=y_\emptyset$. We
  inductively construct a decreasing sequence $(L_n)_{n\in\nn}$ in
  $[M]^\infty$, a strictly increasing sequence $(l_n)_{n\in\nn}$
  in $M$
  and $(\widetilde{y}_{t\cup\{l\}})_{l\in L_n}$, for all $n\in\nn$ and $t\subseteq\{l_i:1\leq
  i<n\}$ with $t\in\widehat{\ff}\upharpoonright L\setminus
\ff\upharpoonright L$ and $\max t=l_{n-1}$, if $n>1$, such that
for every $n\in\nn$ the following
  are satisfied.
  \begin{enumerate}
    \item[(a)] $l_n=\min L_n$ and $l_n<\min L_{n+1}$.
    \item[(b)] For every $t\subseteq\{l_i:1\leq
  i<n\}$ with $t\in\widehat{\ff}\setminus
\ff$ and $\max t=l_{n-1}$, if $n>1$, we have that
\begin{enumerate}
  \item[(i)] For every $k\in\nn$, the vector
  $\widetilde{y}_{t\cup\{L_n(k)\}}$ is of finite support.
  \item[(ii)] For every $k_1<k_2$ in $\nn$, we have that
  \[\text{supp}(\widetilde{y}_{t\cup\{L_n(k_2)\}})<\text{supp}(\widetilde{y}_{t\cup\{L_n(k_2)\}})\]
  \item[(iii)] For every $k\in\nn$ we that
  \[\|y_{t\cup\{L_n(k)\}}-\widetilde{y}_{t\cup\{L_n(k)\}}\|<\frac{1}{2^{|t|+1}}\ee_{n+k-1}\]
\end{enumerate}
  \end{enumerate}
  We set $M_1=\{l_n:n\in\nn\}$. Since $(\ee_n)_{n\in\nn}$ is
  decreasing, it is easy to see that $M_1$ and $(\widetilde{y}_t)_{t\in\widehat{\ff} \upharpoonright
  M_1}$ are the desirable ones.
\end{proof}
For all $t\in\widehat{\ff}\upharpoonright M_1$, we set
$\displaystyle\widetilde{x}_t=\sum_{t'\sqsubseteq
t}\widetilde{y}_{t'}$.

\textbf{Claim 2.} The following hold.
\begin{enumerate}
\item[(i)] The family
$(\widetilde{x}_t)_{t\in\widehat{\ff}\upharpoonright
  M_1}$ is a weak-convergent $\widehat{\ff}$-tree, i.e. the sequence $(\widetilde{x}_{t\cup\{m\}})_{m\in M_1}$ is
  weakly convergent to $\widetilde{x}_t$, for all $t\in\widehat{\ff}\upharpoonright
  M_1\setminus\ff\upharpoonright M_1$.
  \item[(ii)] The $\ff$-subsequence $(\widetilde{x}_s)_{s\in\ff\upharpoonright
  M_1}$ admits $(\widetilde{y}_t)_{t\in\widehat{\ff} \upharpoonright
  M_1}$ as a block tree decomposition.
  \item[(iii)] For every $s\in \ff\upharpoonright M_1\setminus \{\emptyset\}$,
  $\|x_s-\widetilde{x}_s\|<\ee_{k_s}$, where $\min(s)=M_1(k_s)$
  and $\widetilde{x}_\emptyset=x_\emptyset$.
\end{enumerate}
\begin{proof}[proof of Claim 2]
(i)  Recall that for every $t\in\widehat{\ff}\upharpoonright
M\setminus \ff\upharpoonright M_1$, the sequence
$(y_{t\cup\{m\}})_{m\in M_1}$ is weakly null and by (i) of Claim 1
we have that
\[\lim_{m}\|y_{t\cup\{m\}}-\widetilde{y}_{t\cup\{m\}}\|=0\] Hence
the sequence $(\widehat{y}_{t\cup\{m\}})_{m\in M_1}$ is also
weakly null and since
$\widetilde{x}_{t\cup\{m\}}=\widetilde{x}_{t}+\widetilde{y}_{t\cup\{m\}}$
the conclusion follows.

(ii) It is straightforward by (i) and (ii) of Claim 1.

(iii) It is clear that $\widetilde{x}_\emptyset=x_\emptyset$. Also
using (i) of Claim 1, we get that
\[\Big\|x_s-\widetilde{x}_s\Big\|\leq\sum_{\emptyset\neq t\sqsubseteq
s}\Big\|y_t-\widetilde{y}_{t}\Big\|\leq \ee_{k_s}\]
 for every $s\in
\ff\upharpoonright M_1\setminus \{\emptyset\}$.
\end{proof}
By (ii) of Claim 2 and  Lemma \ref{blcan}  there exists
$M_2\in[M_1]^\infty$ such that the $\ff$-subsequence
$(\widetilde{x}_s)_{s\in\ff\upharpoonright
  M_2}$ admits $(\widetilde{y}_t)_{t\in\widehat{\ff} \upharpoonright
  M_2}$ as a canonical  tree decomposition. Moreover, by Proposition
\ref{prpconv}  there exists $M_3\in[M_2]^\infty$ such that the
$\ff$-subsequence $(\widetilde{x}_s)_{s\in\ff\upharpoonright
  M_3}$ is subordinated with respect to the weak topology.

We set $L=M_3$ and it is readily checked that  the
$\ff$-subsequence $(\widetilde{x}_s)_{s\in\ff\upharpoonright L}$
 is as desired.
\end{proof}
\begin{cor}\label{cor canonical tree with spr mod}
  Let $X$ be a Banach space with a Schauder basis. Let $\ff$ be a regular thin family, $M\in[\nn]^\infty$ and
  $(x_s)_{s\in\ff}$ be an $\ff$-sequence in $X$ such
  that $(x_s)_{s\in\ff\upharpoonright M}$ is
  subordinated with respect to the weak topology of $X$ and let $\widehat{\varphi}:\widehat{\ff}\upharpoonright M\to (X,w)$
   be the continuous map witnessing this. Assume also that $(x_s)_{s\in\ff\upharpoonright
   M}$ generates an $\ff$-spreading model $(e_n)_{n\in\nn}$. Then there exist $L\in[M]^\infty$ and an
  $\ff$-subsequence $(\widetilde{x}_s)_{s\in\ff\upharpoonright L}$
  in $X$ satisfying the following.
  \begin{enumerate}
    \item[(i)] The $\ff$-subsequence
$(\widetilde{x}_s)_{s\in\ff\upharpoonright
    L}$ generates $(e_n)_{n\in\nn}$ as an $\ff$-spreading model.
\item[(ii)] The $\ff$-subsequence
$(\widetilde{x}_s)_{s\in\ff\upharpoonright
    L}$ is subordinated with respect to the weak topology of $X$. Moreover, if
    $\widetilde{\varphi}:\widehat{\ff}\upharpoonright L\to (X,w)$ is the continuous maps witnessing
    this, then
    $\widetilde{\varphi}(\emptyset)=\widehat{\varphi}(\emptyset)$.
    \item[(iii)] The $\ff$-subsequence $(\widetilde{x}_s)_{s\in\ff\upharpoonright
    L}$ admits a canonical tree decomposition.
  \end{enumerate}
\end{cor}
\begin{cor}\label{assuming canonical decomposition}
  Let $X$ be a Banach space with a Schauder basis. Let
  $ \xi<\omega_1$, $(e_n)_{n\in\nn}\in\mathcal{SM}_\xi^{wrc}(X)$ and $\ff$ be a regular thin family of order
  $\xi$.
 Then there
  exist $L\in[\nn]^\infty$ and an $\ff$-subsequence $(\widetilde{x}_s)_{s\in\ff\upharpoonright
  L}$ in $X$ having the following properties.
  \begin{enumerate}\item[(i)] It  generates
  $(e_n)_{n\in\nn}$ as an $\ff$-spreading model.\item[(ii)] It is
   subordinated with respect to the weak topology. \item[(iii)] It  admits a canonical tree decomposition. \end{enumerate}
\end{cor}
\begin{proof}
By Proposition \ref{wrc} there exists a weakly relatively compact
$\ff$-sequence $(x_s)_{s\in\ff}$ and $M\in[\nn]^\infty$ such that
$(x_s)_{s\in\ff\upharpoonright
  M}$ generates $(e_n)_{n\in\nn}$ as an $\ff$-spreading model and moreover $(x_s)_{s\in\ff\upharpoonright
  M}$ is
   subordinated with respect to the weak topology. Corollary
   \ref{cor canonical tree with spr mod} completes the proof.
\end{proof}

We close this section by stating some corresponding to the above
results concerning $\mathcal{SM}(X^*)$, where $X^*$ is the dual of
a Banach space $X$ with a shrinking Schauder basis. First one
needs to state an analogue to Theorem \ref{canonical tree}. Namely
we have the following.
\begin{thm}\label{canonical tree decomp on duals}
  Let $X$ be a Banach space with a shrinking Schauder basis $(w_n)_{n\in\nn}$. Let $\ff$ be a regular thin family, $M\in[\nn]^\infty$ and
  $(x^*_s)_{s\in\ff}$ be an $\ff$-sequence in $X^*$ such
  that $(x^*_s)_{s\in\ff\upharpoonright M}$ is
  subordinated with respect to the weak* topology of $X$ and let $\widehat{\varphi}:\widehat{\ff}\upharpoonright M\to (X,w^*)$
   be the continuous map witnessing this. Then for every  sequence
  $(\ee_n)_{n\in\nn}$ of positive reals, there exist $L\in[M]^\infty$ and an
  $\ff$-subsequence $(\widetilde{x}^*_s)_{s\in\ff\upharpoonright L}$
  in $X^*$ satisfying the following.
  \begin{enumerate}
    \item[(i)] For every $s\in\ff\upharpoonright L$, $\|x_s^*-\widetilde{x}^*_s\|<\ee_n$, where $\min
  s=L(n)$.
\item[(ii)] The $\ff$-subsequence
$(\widetilde{x}^*_s)_{s\in\ff\upharpoonright
    L}$ is subordinated with respect to the weak* topology of $X$. Moreover, if
    $\widetilde{\varphi}:\widehat{\ff}\upharpoonright L\to (X,w^*)$ is the continuous map witnessing
    this, then
    $\widetilde{\varphi}(\emptyset)=\widehat{\varphi}(\emptyset)$.
    \item[(iii)] The $\ff$-subsequence $(\widetilde{x}^*_s)_{s\in\ff\upharpoonright
    L}$ admits a canonical tree decomposition with respect to $(w^*_n)_{n\in\nn}$.
\end{enumerate}
\end{thm}
We omit the proof of the above theorem since it is almost
identical to the one of Theorem \ref{canonical tree}. Finally
since every bounded subset of the dual of a separable Banach space
endowed with the weak* topology is compact metrizable, by
Proposition \ref{Create subordinated} and the above theorem we
obtain the following analogue of Corollary \ref{assuming canonical
decomposition}.
\begin{cor}\label{rem on the w* generated spreading models}
  Let $X$ be a Banach space with a shrinking Schauder basis $(w_n)_{n\in\nn}$. Let
  $\xi<\omega_1$, $(e_n)_{n\in\nn}\in\mathcal{SM}_\xi(X^*)$ and $\ff$ be a regular thin family of order
  $\xi$.
 Then there
  exist $L\in[\nn]^\infty$ and an $\ff$-subsequence $(\widetilde{x}^*_s)_{s\in\ff\upharpoonright
  L}$ in $X^*$ having the following properties.
  \begin{enumerate}\item[(i)] It generates
  $(e_n)_{n\in\nn}$ as an $\ff$-spreading model.\item[(ii)] It is
   subordinated with respect to the weak* topology of $X^*$. \item[(iii)] It  admits a canonical tree decomposition with respect to $(w_n^*)_{n\in\nn}$. \end{enumerate}
\end{cor}

\subsection{Disjoint canonical tree decompositions of $\ff$-sequences}
The following definition extends the classical notion of the
disjointly supported sequences for $\ff$-sequences in a Banach
space with a Schauder basis.
\begin{defn}
  Let  $X$ be a Banach space with a Schauder basis. Let $\ff$ be a regular thin family, $M\in[\nn]^\infty$ and
  $(x_s)_{s\in\ff}$ an $\ff$-sequence in $X$ of finitely supported
  vectors. We will say that the $\ff$-subsequence $(x_s)_{s\in\ff\upharpoonright
  M}$ is plegma disjointly supported  if for
  every plegma pair $(s_1,s_2)$ in $\ff\upharpoonright M$ we have
  that $\text{supp}(x_{s_1})\cap\text{supp}(x_{s_2})=\emptyset$.
\end{defn}

In connection with the canonical tree decompositions we observe
the following. Let $(\widetilde{x}_s)_{s\in\ff\upharpoonright L}$
be an $\ff$-subsequence in a Banach space with a Schauder basis,
which admits a canonical tree decomposition
$(\widetilde{y}_t)_{t\in\widehat{\ff} \upharpoonright L}$. It is
easy to see that $(\widetilde{x}_s)_{s\in\ff\upharpoonright L}$ is
plegma disjointly supported if and only if
$\widetilde{y}_\emptyset=0$. In this case we will say that
$(\widetilde{x}_s)_{s\in\ff\upharpoonright L}$ admits
$(\widetilde{y}_t)_{t\in\widehat{\ff} \upharpoonright L}$ as a
\emph{disjoint canonical tree
  decomposition}.
\begin{rem}\label{rem cutting the root}
   Let $X$ be a Banach space with a Schauder basis. Let $\ff$ be a regular thin family, $L\in[\nn]^\infty$ and
  $(\widetilde{x}_s)_{s\in\ff}$ be an $\ff$-sequence in $X$ such
  that $(\widetilde{x}_s)_{s\in\ff\upharpoonright M}$
  admits a canonical tree decomposition $(\widetilde{y}_t)_{t\in\widehat{\ff} \upharpoonright
  L}$.  Let $x'_s=\widetilde{x}_s-\widetilde{y}_\emptyset$ for all $s\in\ff\upharpoonright
  L$. Then notice that the $\ff$-subsequence $(x'_s)_{s\in\ff\upharpoonright
  L}$ admits $(y'_t)_{t\in\widehat{\ff}\upharpoonright
  L}$ as a disjoint canonical tree decomposition, where  $y'_t=\widetilde{y}_t$, for all
  $t\in\widehat{\ff}\upharpoonright
  L\setminus\{\emptyset\}$. Hence $(x'_s)_{s\in\ff\upharpoonright
  L}$ is a plegma disjointly supported $\ff$-subsequence.
\end{rem}
In the sequel an $\ff$-spreading model will be called \emph{plegma
disjointly generated} if it is generated by a plegma disjointly
supported $\ff$-sequence.

In Corollary \ref{assuming canonical decomposition} we have shown
that every $(e_n)_{n\in\nn}\in\mathcal{SM}^{wrc}(X)$, where $X$ is
a Banach space with a Schauder basis, is generated by an
$\ff$-subsequence which is subordinated with respect to the weak
topology and admits a canonical tree decomposition. The next
corollary asserts that if $(e_n)_{n\in\nn}$ is in addition a
Schauder basic sequence not equivalent to the usual basis of
$\ell^1$, then it is plegma disjointly generated. In the case of
$(e_n)_{n\in\nn}$ being equivalent to the usual basis of $\ell^1$,
this cannot in general be achieved because of Corollary
\ref{ultrafilter property for ell^1 spr mod simplified}. However
Corollary \ref{ultrafilter property for ell^1 spr mod simplified}
implies an isomorphic version of the aforementioned fact.

\begin{cor}\label{thm disjointly generic decomposition for wekly relatively compact ff-sequences}
  Let $X$ be a Banach space with a Schauder basis. Let
  $\xi<\omega_1$, $(e_n)_{n\in\nn}$ be a Schauder basic sequence in $\mathcal{SM}^{wrc}_\xi(X)$
  and $\ff$ be a regular thin family of order $\xi$. Then there
  exist exist $L\in[\nn]^\infty$ and an $\ff$-subsequence $(\widetilde{x}_s)_{s\in\ff\upharpoonright
  L}$ in $X$ having the following properties.
 \begin{enumerate}
    \item[(i)] If $(e_n)_{n\in\nn}$ is not equivalent to the usual basis of $\ell^1$, then $(\widetilde{x}_s)_{s\in\ff\upharpoonright L}$
               generates $(e_n)_{n\in\nn}$ as an $\ff$-spreading model.
    \item[(ii)] If $(e_n)_{n\in\nn}$ is equivalent to the usual basis of $\ell^1$, then $(\widetilde{x}_s)_{s\in\ff\upharpoonright L}$
               generates an $\ff$-spreading model which is also equivalent to the usual basis of $\ell^1$.
    \item[(iii)] The $\ff$-subsequence $(\widetilde{x}_s)_{s\in\ff\upharpoonright L}$ is
   subordinated with respect to the weak topology.
    \item[(iv)] The $\ff$-subsequence $(\widetilde{x}_s)_{s\in\ff\upharpoonright L}$   admits a disjoint canonical tree decomposition.
  \end{enumerate}
\end{cor}
\begin{proof}
By
  Corollary \ref{assuming canonical decomposition}, we have that there
  exist $L\in[\nn]^\infty$ and a subordinated $\ff$-sequence $(x_s)_{s\in\ff\upharpoonright
  L}$ which admits a canonical tree decomposition $(y_t)_{t\in\widehat{\ff}\upharpoonright
  L}$ and generates
  $(e_n)_{n\in\nn}$ as an $\ff$-spreading model. By Remark \ref{rem on canonical tree
  decomp} we have that $\widehat{\varphi}(\emptyset)=y_\emptyset$.

  If $\widehat{\varphi}(\emptyset)=0$, the proof is completed.
  Suppose that $\widehat{\varphi}(\emptyset)\neq0$. Since $(e_n)_{n\in\nn}$ is Schauder basic,
   by Lemma \ref{lemma either l1 or not Schauder basic} we
  have that $(e_n)_{n\in\nn}$ is equivalent to the usual basis of $\ell^1$. For every $s\in \ff\upharpoonright L$ we set
  $x'_s=x_s-\widehat{\varphi}(\emptyset)$. By Remark \ref{rem cutting the
  root}, we have that $(x'_s)_{s\in\ff\upharpoonright M}$ is
  subordinated and admits a disjoint canonical tree decomposition.
  We pass to an $L'\in[L]^\infty$ such that $(x'_s)_{s\in\ff\upharpoonright
  M}$ generates an $\ff$-spreading model, which, by Corollary
  \ref{ultrafilter property for ell^1 spr mod simplified}, is
  equivalent to the usual basis of $\ell^1$.
\end{proof}

\chapter{The spreading models of $\ell^p$, $1\leq
p<\infty$, and $c_0$}\label{chapter 7}  In this chapter we present
the spreading models of  some classical spaces.  In particular it
is shown  that every nontrivial spreading model of $\ell^p$ with
$1<p<\infty$ generates an isometric copy of $\ell^p$ and every
nontrivial spreading model of $\ell^1$ is equivalent to the usual
basis of $\ell^1$. Moreover the class of  spreading models of
order two of  $c_0$ contains up to equivalence all the Schauder
basic spreading sequences.
\section{The spreading models  of $\ell^p$, $1\leq
p<\infty$}
\begin{lem}\label{lemma for the spreading models of l^p}
  Let $1\leq p<\infty$, $\ff$ be a regular thin family,
  $M\in[\nn]^\infty$ and $(x_s)_{s\in\ff\upharpoonright M}$ be a
  seminormalized plegma disjointly supported $\ff$-subsequence in
  $\ell^p$. Then every $\ff$-spreading model $(e_n)_{n\in\nn}$ of $(x_s)_{s\in\ff\upharpoonright
  M}$ is equivalent to the usual basis of $\ell^p$. In particular for every $n\in\nn$ and
  $a_1,\ldots,a_n\in\rr$ we have that \[\Big\|\sum_{j=1}^na_je_j\Big\|=\|e_1\|\Big(\sum_{j=1}^n|a_j|^p\Big)^\frac{1}{p}\]
\end{lem} The proof of the above lemma is omitted, since it is
similar to the proof of the fact that every normalized disjointly
supported sequence in $\ell^p$ is isometric to the usual basis of
$\ell^p$.
\subsection{The spreading models  of $\ell^p$, $1<
p<\infty$}
\begin{lem}\label{lp spr mod phi keno}
  Let $1< p<\infty$, $\ff$ be a regular thin family,
  $M\in[\nn]^\infty$ and $(x_s)_{s\in\ff}$ be an
   $\ff$-sequence in
  $\ell^p$ which  generates a nontrivial $\ff$-spreading
model $(e_n)_{n\in\nn}$. Also suppose that
$(x_s)_{s\in\ff\upharpoonright M}$ is subordinated and let
$\widehat{\varphi}:\widehat{\ff}\upharpoonright M\to(\ell^p,w)$ be
the continuous map witnessing that.
\begin{enumerate}
\item[(i)] If  $\widehat{\varphi}(\emptyset)=0$ then
$(e_n)_{n\in\nn}$ is equivalent to the usual basis of $\ell^p$. In
particular for every $n\in\nn$ and
  $a_1,\ldots,a_n\in\rr$ we have that \[\Big\|\sum_{j=1}^na_je_j\Big\|=\|e_1\|\Big(\sum_{j=1}^n|a_j|^p\Big)^\frac{1}{p}\]
\item[(ii)] If $\widehat{\varphi}(\emptyset)\neq 0$ then
$(e_n)_{n\in\nn}$ is singular.
  In particular, for every $k\in\nn$ and
  $a_1,\ldots,a_k\in\rr$, we have that
  \begin{equation}\label{022}\Big\|\sum_{j=1}^na_je_j\Big\|=\Bigg(\|e\|^p\Big|\sum_{j=1}^na_j\Big|^p+
  \|e'_1\|^p\sum_{j=1}^n|a_j|^p\Bigg)^\frac{1}{p}\end{equation}
  where $e_n=e'_n+e$ is the natural decomposition of
  $(e_n)_{n\in\nn}$.\end{enumerate}
\end{lem}
\begin{proof} By Corollary \ref{cor canonical tree with spr mod} we may assume
  that $(x_s)_{s\in\ff\upharpoonright M}$ admits a canonical tree
  decomposition.

 (i) Since $\widehat{\varphi}(\emptyset)=0$ we have that
$(x_s)_{s\in\ff\upharpoonright L}$ is plegma disjointly supported.
Lemma \ref{lemma for the spreading models of l^p} completes the
proof.

 (ii) Let $x'_s=x_s-\widehat{\varphi}(\emptyset)$, for all
  $s\in\ff\upharpoonright M$. Let $L\in[M]^\infty$ such that $(x'_s)_{s\in\ff\upharpoonright
  L}$ generates an $\ff$-spreading model $(\widetilde{e}_n)_{n\in\nn}$.
  By Lemma \ref{lp spr mod phi keno} we have that
  $(\widetilde{e}_n)_{n\in\nn}$ is not equivalent to the usual
  basis of $\ell^1$. Corollary \ref{ultrafilter property for ell^1 spr mod
  simplified} yields that $(e_n)_{n\in\nn}$ is not equivalent to the usual
  basis of $\ell^1$. By Lemma \ref{lemma either l1 or not Schauder
  basic} we have that $(e_n)_{n\in\nn}$ is singular. By Lemma
  \ref{2lemwrc} we have that $(x'_s)_{s\in\ff\upharpoonright L}$ generates
  $(e'_n)_{n\in\nn}$ as an $\ff$-spreading model and
  $\|\widehat{\varphi}(\emptyset)\|=\|e\|$. Notice that for every sequence $(s_n)_{n\in\nn}$
   in $\ff\upharpoonright L$ with $\min s_n\to\infty$, we have
   that $\|x'_{s_n}\|\to \|e'_1\|$ and
   $\min\text{supp}(x'_{s_n})\to\infty$.
    These observations easily yield that equation (\ref{022}) holds for every choice of $k\in\nn$ and $a_1,\ldots,a_k\in\rr$.
\end{proof}
\begin{thm}\label{The spr mods of l^p}
  Let $1<p<\infty$ and $(e_n)_{n\in\nn}$ a spreading model of $\ell^p$ of order $\xi$, for some $\xi<\omega_1$. Then one of the following holds.
  \begin{enumerate}
    \item[(i)] The sequence $(e_n)_{n\in\nn}$ is trivial.
    \item[(ii)] The sequence $(e_n)_{n\in\nn}$ is singular. In this case if $e_n=e_n'+e$ is
     the natural decomposition of $(e_n)_{n\in\nn}$ then for every $k\in\nn$ and $a_1,\ldots,a_k\in\rr$,
     we have that \[\Big\|\sum_{j=1}^na_je_j\Big\|=\Bigg(\|e\|^p\Big(\Big|\sum_{j=1}^na_j\Big|\Big)^p+
     \sum_{j=1}^n\|e_1'\|| a_j|^p\Bigg)^\frac{1}{p}\]
    \item[(iii)] The sequence $(e_n)_{n\in\nn}$ is Schauder basic. In this case it is equivalent to the usual basis of $\ell^p$.
     In particular for every $n\in\nn$ and $a_1,\ldots,a_n\in\rr$ we have that
  \[\Big\|\sum_{j=1}^na_je_j\Big\|=\|e_1\|\Big(\sum_{j=1}^n|a_j|^p\Big)^\frac{1}{p}\]
  \end{enumerate}
\end{thm}
\begin{proof}
  Let $\ff$ be a regular thin family of order $\xi$.
  Since $\ell^p$ is reflexive, we have that $(e_n)_{n\in\nn}$ is generated by a weakly relatively
compact $\ff$-sequence. By Proposition \ref{wrc} we have that
$(e_n)_{n\in\nn}$ is generated by a subordinated $\ff$-sequence.
Lemma \ref{lp spr mod phi keno} completes the proof.
\end{proof}
\begin{cor} Let $1<p<\infty$ and $(e_n)_{n\in\nn}$ be a nontrivial  spreading model of $\ell^p$ of order $\xi$, for some
$\xi<\omega_1$. Then the space $E$ generated by $(e_n)_{n\in\nn}$
is isometric to $\ell^p$.
\end{cor}
\begin{proof}
  By Theorem \ref{The spr mods of l^p} it suffices to treat the
  case of $(e_n)_{n\in\nn}$ being singular.
  Let $(u_n)_{n\in\nn}$ be the usual basis of $\ell^p$. Let
  $T:E\to\ell^p$ be the linear operator with
  $T(e_n)=\|e\|u_1+\|e_1'\|u_{n+1}$. It is easy to see that $T$ is
  an isometry. Since $T(e_n)\stackrel{w}{\to}\|e\|u_1$, we have
  that $u_1\in T[E]$ and therefore $T$ is onto.
\end{proof}
\begin{rem}
  By Theorem \ref{The spr mods of l^p} it is immediate that every
  spreading model of $\ell^p$ with $1<p<\infty$ of any order is
  isometric to a sequence in $\ell^p$. Thus,
  $\mathcal{SM}_1(\ell^p)=\mathcal{SM}(\ell^p)$, for all
  $1<p<\infty$.
\end{rem}
\subsection{The spreading models of $\ell^1$}

 Concerning $\ell^1$ we have the following.
\begin{thm}\label{the spreading models of l^1}
  Every nontrivial spreading model of any order of
  $\ell^1$ is equivalent to the usual basis of $\ell^1$.
\end{thm}
\begin{proof}
Let $(e_n)_{n\in\nn}$ be a nontrivial spreading model of order
$\xi$ of $\ell^1$, for some
  $\xi<\omega_1$. Then there exists an $\ff$-sequence $(x_s)_{s\in\ff}$ in $\ell^1$
  and $M\in[\nn]^\infty$ such that the $\ff$-subsequence
  $(x_s)_{s\in\ff\upharpoonright M}$ generates $(e_n)_{n\in\nn}$
  as an $\ff$-spreading model.
  By Corollary \ref{rem on the w* generated spreading models} we may assume that $(x_s)_{s\in\ff\upharpoonright M}$ admits a
  canonical tree decomposition $(y_t)_{t\in\widehat{\ff}\upharpoonright
  M}$. We set $x'_s=x_s-y_\emptyset$, for all $s\in\ff\upharpoonright M$, and we pass to an
  $L\in[M]^\infty$ such that $(x'_s)_{s\in\ff\upharpoonright L}$
  generates an $\ff$-spreading model $(e'_n)_{n\in\nn}$. By
  Corollary \ref{trivial ultrafilter weak property} we have that
  $(e'_n)_{n\in\nn}$ is nontrivial and therefore, passing to a
  further infinite subset of $L$ if necessary, the $\ff$-subsequence
  $(x'_s)_{s\in\ff\upharpoonright L}$ is
  seminormalized. Since  $(x'_s)_{s\in\ff\upharpoonright L}$ is plegma disjointly
  supported (see  Remark  \ref{rem cutting the root}), by  Lemma \ref{lemma for the spreading models of l^p}
  we get  that $(e'_n)_{n\in\nn}$ is equivalent to the usual basis
  of $\ell^1$. By Corollary \ref{ultrafilter property for ell^1 spr mod
  simplified} we have that $(e_n)_{n\in\nn}$ is equivalent to the usual basis
  of $\ell^1$.
\end{proof}
\begin{rem}
  Using similar arguments as in the case of $\ell^p$, for
  $1<p<\infty$, it can be shown that for every
  $(e_n)_{n\in\nn}\in\mathcal{SM}(\ell^1)$ there exist $c_1,c_2\geq0$ such that \[\Big\|\sum_{j=1}^la_je_j\Big\|=c_1\Big|\sum_{j=1}^la_j\Big|+c_2\sum_{j=1}^l|a_j|\]
  for all $l\in\nn$ and $a_1,\ldots,a_l\in\rr$. Therefore, every spreading model of any order of $\ell^1$ is
  isometric to a sequence in $\ell^1$. Thus
  $\mathcal{SM}_1(\ell^1)=\mathcal{SM}(\ell^1)$.
\end{rem}
\section{The spreading models of $c_0$}
\subsection{Weakly relatively compact generated spreading models of $c_0$}
Following the same procedure as in the case of $\ell^p$, with
$1<p<\infty$, one can obtain the following analogue of the Theorem
\ref{The spr mods of l^p}.
\begin{thm}\label{The spr mods of c_0}
  Let $(e_n)_{n\in\nn}\in\mathcal{SM}^{wrc}(c_0)$. Then one of the following holds.
  \begin{enumerate}
    \item[(i)] The sequence $(e_n)_{n\in\nn}$ is trivial.
    \item[(ii)] The sequence $(e_n)_{n\in\nn}$ is singular. In this case if $e_n=e_n'+e$ is
     the natural decomposition of $(e_n)_{n\in\nn}$ then for every $k\in\nn$ and $a_1,\ldots,a_k\in\rr$,
     we have that \[\Big\|\sum_{j=1}^na_je_j\Big\|=\max\Big\{\|e\|\cdot\Big|\sum_{j=1}^na_j\Big|,\|e_1'\|\cdot\max_{1\leq j\leq n}|a_j|\Big\}\]
    \item[(iii)] The sequence $(e_n)_{n\in\nn}$ is Schauder basic. In this case it is equivalent to the usual basis of $c_0$.
    In particular for every $n\in\nn$ and $a_1,\ldots,a_n\in\rr$ we have that
  \[\Big\|\sum_{j=1}^na_je_j\Big\|=\|e_1\|\cdot\max_{1\leq j\leq n}|a_j|\]
  \end{enumerate}
\end{thm}
\begin{cor}
Let $(e_n)_{n\in\nn}$ be a nontrivial spreading model in
$\mathcal{SM}^{wrc}(c_0)$. Then the space $E$ generated by
$(e_n)_{n\in\nn}$  is isometric to $c_0$.\end{cor}
\subsection{The  classical spreading models of $c_0$}
In this subsection we review some rather well known facts
concerning the classical spreading models of $c_0$. We will make
use of some results related to nontrivial weak-Cauchy sequences in
$c_0$. We start with the following lemma.
\begin{lem}\label{upper c_0 gives upper summing}
  Let $(x_n)_{n\in\nn}$ be a sequence in a Banach space.
  Let $v_1=x_1$ and $v_n=x_n-x_{n-1}$, for all $n>1$. Suppose that $(v_n)_{n\in\nn}$
   admits an upper $c_0$ estimate, i.e. there exists $c>0$ such that for every $k\in\nn$ and $b_1,\ldots,b_k\in\rr$
  \[\Big\|\sum_{j=1}^kb_jv_j\Big\|\leq c\cdot\max_{1\leq j\leq k}|b_j|\]
  Then $(x_n)_{n\in\nn}$ is dominated by the summing basis of
  $c_0$ i.e. there exists $c>0$ such that for every $k\in\nn$ and $a_1,\ldots,a_k\in\rr$
  we have that \[\Big\|\sum_{j=1}^ka_jx_{j}\Big\|\leq c\cdot\max_{l\leq k}\Big|\sum_{j=1}^la_j\Big|\]

\end{lem}
\begin{proof}
  First notice that for every $j\in\nn$, $x_j=\sum_{i=1}^jv_i$. Let $k\in\nn$ and $a_1,\ldots,a_k\in\rr$. We set \[b_j=\sum_{i=j}^k a_i\] for all $1\leq j\leq k$. Then we have that
  \[\sum_{j=1}^ka_jx_j=\sum_{j=1}^ka_j\sum_{i=1}^jv_i=\sum_{j=1}^k\Big(\sum_{i=j}^ka_i\Big)v_j=\sum_{j=1}^kb_jv_j \]
  Therefore
  \[\Big\|\sum_{j=1}^ka_jx_j\Big\| =\Big\|\sum_{j=1}^kb_jv_j\Big\| \leq c\cdot\max_{1\leq j\leq k} |b_j|= c\cdot\max_{1\leq j\leq k}\Big|\sum_{i=j}^ka_i\Big|\leq 2c\cdot\max_{1\leq j\leq k}\Big|\sum_{i=1}^ja_i\Big|\]
\end{proof}
\begin{lem}\label{r3}
  Let $(x_n)_{n\in\nn}$ be a nontrivial weak-Cauchy sequence in $c_0$. Then the sequence
  $(x_n)_{n\in\nn}$ contains a subsequence $(x_{k_n})_{n\in\nn}$ such that $(x_{k_n})_{n\in\nn}$ is dominated
  by the summing basis of $c_0$.\end{lem}
\begin{proof}
  For every $x\in\ell^\infty$ and $F\subseteq\nn$ we denote as $x|_F$ the sequence in $\ell^\infty$ such that $x|_F(n)=x(n)$, for all $n\in F$, and $x|_F(n)=0$, for all $n\not\in F$.

  Since $y\in\ell^\infty\setminus c_0$, there exist $L\in[\nn]^\infty$ and $\theta>0$ such that $|y(l)|>\theta$, for all $l\in L$. Let $(\ee_n)_{n\in\nn}$ be a decreasing null sequence of positive reals, such that $\sum_{n=1}^\infty 4\ee_n<\frac{\theta}{2}$. We inductively choose a strictly increasing sequence $(l_n)_{n\in\nn}$ in $L$ and a subsequence $(x_{k_n})_{n\in\nn}$ of $(x_n)_{n\in\nn}$ satisfying for every $n\in\nn$ the following.
  \begin{enumerate}
    \item[(i)] $\displaystyle \big\|x_{k_n}|_{\{1,\ldots,l_n\}}-y|_{\{1,\ldots,l_n\}}\big\|<\ee_{n+1}$ and
    \item[(ii)] $\displaystyle \big\|x_{k_n}|_{\{l_{n+1},l_{n+1}+1,\ldots\}}\big\|<\ee_{n+1}$.
  \end{enumerate}
  Let $l_1\in L$. Pick $k_1\in\nn$ such that \[\big\|x_{k_1}|_{\{1,\ldots,l_1\}}-y|_{\{1,\ldots,l_1\}}\big\|<\ee_2\]
  Suppose that for some $n\in\nn$, $l_1<\ldots<l_n$ and $x_{k_1},\ldots,x_{k_n}$ have been chosen satisfying (i) and (ii). Then let $l_{n+1}\in L$ such that $l_{n+1}>l_n$ and \[\big\|x_{k_n}|_{\{l_{n+1},\ldots\}}\big\|<\ee_{n+1}\]
  Finally, pick $k_{n+1}>k_n$ such that \[\big\|x_{k_n+1}|_{\{1,\ldots,l_{n+1}\}}-y|_{\{1,\ldots,l_{n+1}\}}\big\|<\ee_{n+2}\]

  Let $F_1=\{1,\ldots,l_2-1\}$ and $F_n=\{l_{n-1}+1,\ldots,l_{n+1}-1\}$, for all $n>1$. Let also $u_1=x_{k_1}|_{F_1}$ and $u_n=(x_{k_n}-x_{k_{n-1}})|_{F_n}$. Then we have that
  \[\|u_1-x_{k_1}\|=\big\|x_{k_1}|_{\{l_2,l_2+1,\ldots\}}\big\|<\ee_2\leq\ee_1<4\ee_1\]
  and for every $n>1$, we have that
  \[\begin{split}
    \|u_n-(x_{k_n}&-x_{k_{n-1}})\| =\big\|(x_{k_n}-x_{k_{n-1}})_{F_n^c}\big\|\\ &\leq\big\|(x_{k_n}-x_{k_{n-1}})|_{\{1,\ldots,l_{n-1}\}}\big\|+\big\|x_{k_n}|_{\{l_{n+1},\ldots\}}\big\|+\big\|x_{k_{n-1}}|_{\{l_{n+1},\ldots\}}\big\|\\ &\leq\big\|(x_{k_n}-y)|_{\{1,\ldots,l_{n-1}\}}\big\|+\big\|(x_{k_{n-1}}-y)|_{\{1,\ldots,l_{n-1}\}}\big\|+\ee_{n+1}+\ee_n\\
    &\leq\ee_{n+1}+\ee_n+\ee_{n+1}+\ee_n\leq4\ee_n
  \end{split}\]
Also we have that
\[\|u_1\|\geq|u_1(l_1)|=|x_{k_1}(l_1)|>|y(l_1)|-\ee_2>\frac{\theta}{2}\]
and for every $n>1$, we have that
\[\begin{split}
  \|u_n\|&\geq|u_n(l_n)|\geq \|x_{k_n}(l_n)|-|x_{k_{n-1}}(l_n)|\geq |y(l_n)|-\ee_{n+1}-\ee_n>\frac{\theta}{2}
\end{split}\]
 Therefore $(u_n)_{n\in\nn}$ is seminormalized. The boundness of $(u_n)_{n\in\nn}$ follows by the boundness of $(x_n)_{n\in\nn}$, which is a consequence of the fact that $(x_n)_{n\in\nn}$ $w^*$ convergences to $y$. By the definition of the sequence $(F_n)_{n\in\nn}$, it is immediate that the sequences $(u_{2n-1})_{n\in\nn}$ and $(u_{2n})_{n\in\nn}$ are block sequence in $c_0$. Therefore, both $(u_{2n-1})_{n\in\nn}$ and $(u_{2n})_{n\in\nn}$ are equivalent to the usual basis of $c_0$ with basis constant 1. Since $\|u_n\|\geq\frac{\theta}{2}$ and $\sum_{n=1}^\infty 4\ee_n<\frac{\theta}{2}$, we have that both $(x_{k_{2n-1}}-x_{k_{2n-2}})_{n\in\nn}$ and $(x_{k_{2n}}-x_{k_{2n-1}})_{n\in\nn}$ are equivalent to the usual basis of $c_0$, where we have set $x_{k_0}=0$. This easily yields, using the triangle inequality of the norm, that $(x_n-x_{n-1})_{n\in\nn}$ admits upper $c_0$ estimate. Lemma \ref{upper c_0 gives upper summing} completes the proof.
\end{proof}
Finally we state the next well known result (see Proposition 2.2
in \cite{Ro2}).
\begin{prop} \label{r2} Every nontrivial weak-Cauchy sequence in a Banach
space contains a subsequence which dominates the summing basis.
\end{prop}
By the above we obtain the following.
\begin{prop}\label{sum}
  Every nontrivial weak-Cauchy sequence in $c_0$ contains  a subsequence
   equivalent to the summing basis.
\end{prop}

\begin{prop}
  Let $(e_n)_{n\in\nn}$ be a spreading model of order one of $c_0$ generated by a sequence $(x_n)_{n\in\nn}$.
  Then one of the following holds.
  \begin{enumerate}
    \item[(a)] The sequence  $(x_n)_{n\in\nn}$ contains a  weakly convergent subsequence. In this case
    $(e_n)_{n\in\nn}$ satisfies exactly one of the alternatives of Theorem
    \ref{The spr mods of c_0}.
    \item[(b)] The sequence  $(x_n)_{n\in\nn}$ contains a  nontrivial weak-Cauchy subsequence. In this case
$(e_n)_{n\in\nn}$ is equivalent to the summing basis of $c_0$.
  \end{enumerate}
\end{prop}
\begin{proof}
  It is clear that $(x_n)_{n\in\nn}$ contains a subsequence
  $(x_{k_n})_{n\in\nn}$ which is  weak-Cauchy.
 Then $(x_{k_n})_{n\in\nn}$ is either
  weak convergent, or it is nontrivial weak-Cauchy.
   If $(x_{k_n})_{n\in\nn}$ is weakly convergent then by Theorem \ref{The spr mods of c_0},
    one of the alternatives (i)-(iii) will occur. In the second case by Proposition \ref{sum}  $(x_{k_n})_{n\in\nn}$
   contains a further subsequence equivalent to the summing basis of
   $c_0$. Since this subsequence also  generates $(e_n)_{n\in\nn}$
   as a spreading model we have that $(e_n)_{n\in\nn}$ is also
   equivalent to the summing basis.
\end{proof}
\begin{rem}
  By the results of Chapter \ref{Chapter 5} we may conclude that
  in  case (a) of the above proposition  the whole
  sequence $(x_n)_{n\in\nn}$ is weakly convergent.
\end{rem}

\subsection{The higher order spreading models of $c_0$}
In this subsection we deal with the higher order spreading models
of $c_0$. As we have already mentioned the class of  spreading
models of order two of $c_0$ essentially contains all Schauder
basic spreading sequences. Moreover it is shown that a singular
spreading model of (any order) of  $c_0$ is actually a second
order one.
\begin{lem}\label{on spr
mod of c_0 first lem}
  Let $(x_n)_{n\in\nn}$ be a sequence in $\ell^\infty$. For every
  $n<k$ in $\nn$ let \[x_{\{n,k\}}=(x_n(1),\ldots
  x_n(k),0,\ldots)\] Then for every $l\in\nn$,
  $n_1,\ldots,n_l\in\nn$, $k_1<\ldots<k_l$ in $\nn$ and $a_1,\ldots,a_l\in\rr$ we have
  that \[\Big\|\sum_{i=1}^la_ix_{\{n_i,k_i\}}\Big\|_\infty\leq\max_{1\leq j\leq l}\Big\|\sum_{i=j}^l a_i x_{n_i}\Big\|_\infty\]
\end{lem}
\begin{proof}
  Let $l\in\nn$,
  $n_1,\ldots,n_l\in\nn$, $k_1<\ldots<k_l$ in $\nn$ and
  $a_1,\ldots,a_l\in\rr$. We set
  $x=\sum_{i=1}^la_ix_{\{n_i,k_i\}}$.
  There exists $m_0\leq k_l$ such that $|x(m_0)|=\|x\|_\infty$. Let $k_0=0$
  and let $1\leq j_0\leq l$ such that $k_{j_0-1}<
  m_0\leq k_{j_0}$. Then we have that
  \[\begin{split}\|x\|_\infty&=|x(m_0)|=\Big|\sum_{i=j_0}^la_ix_{n_i,k_i}(m_0)\Big|=\Big|\sum_{i=j_0}^la_ix_{n_i}(m_0)\Big|
  \\
  &\leq\Big\|\sum_{i=j_0}^la_ix_{n_i}\Big\|_\infty \leq \max_{1\leq j\leq l}\Big\|\sum_{i=j}^l a_i x_{n_i}\Big\|_\infty
  \end{split}\]
\end{proof}
\begin{lem}\label{on spr mod of c_0 second lem}
  Let $(e_n)_{n\in\nn}$ be a spreading sequence in $\ell^\infty$. For every
  $n<k$ in $\nn$ let \[x_{\{n,k\}}=(e_n(1),\ldots
  e_n(k),0,\ldots)\] Then for every spreading
  model $(e'_n)_{n\in\nn}$ of $(x_s)_{s\in[\nn]^2}$, we have that
  \[\Big\|\sum_{i=1}^la_ie_i\Big\|_\infty\leq\Big\|\sum_{i=1}^la_ie'_i\Big\|\leq\max_{1\leq j\leq l}\Big\|\sum_{i=j}^l a_i
  e_i\Big\|_\infty\]
  for all $l\in\nn$ and $a_1,\ldots,a_l\in\rr$.
\end{lem}
\begin{proof}
  Let $M\in[\nn]^\infty$ and $(e'_n)_{n\in\nn}$ such that
  $(x_s)_{s\in[M]^2}$ generates $(e'_n)_{n\in\nn}$ as a spreading
  model. Let $l\in\nn$ and $a_1,\ldots, a_l\in\rr$. We will first
  show that \[\Big\|\sum_{i=1}^la_ie'_i\Big\|\leq\max_{1\leq j\leq l}\Big\|\sum_{i=j}^l a_i
  e_i\Big\|_\infty\]
  Let $((s_j^n)_{j=1}^l)_{n\in\nn}$ be a sequence in
  $\text{\emph{Plm}}_l([M]^2)$ such that $\min s_1^n\to\infty$. Then \[\Big\|\sum_{j=1}^la_jx_{s^n_j}\Big\|\to\Big\|\sum_{j=1}^la_je'_j\Big\|\]
  By Lemma \ref{on spr mod of c_0 first lem} and as $(e_n)_{n\in\nn}$ is spreading we have
  that  for all $n\in\nn$,
  \[\Big\|\sum_{i=1}^la_ix_{s^n_j}\Big\|\leq\max_{1\leq j\leq l}\Big\|\sum_{i=j}^l a_i
  e_{s^n_j(1)}\Big\|_\infty
  =\max_{1\leq j\leq l}\Big\|\sum_{i=j}^l a_i
  e_i\Big\|_\infty\]
  Therefore
  \[\Big\|\sum_{i=1}^la_ie'_i\Big\|\leq\max_{1\leq j\leq l}\Big\|\sum_{i=j}^l a_i
  e_i\Big\|_\infty\]
  It remains to show that  \[\Big\|\sum_{i=1}^la_ie_i\Big\|_\infty\leq\Big\|\sum_{i=1}^la_ie'_i\Big\|\] To this end let $\ee>0$ and
  choose  $m_0\in\nn$ such that
  \[\Big\|\sum_{i=1}^la_ie_i\Big\|_\infty-\ee\leq \Big|\sum_{i=1}^la_ie_i(m_0)\Big|\]
  Also let   $k_0\in\nn$ such that  for every
  $(s_j)_{j=1}^l\in\text{\emph{Plm}}([M]^2)$ with $s_1(1)\geq
  M(k_0)$, we have that \[\Bigg|\Big\|\sum_{j=1}^la_je_j'\Big\|-\Big\|\sum_{j=1}^la_jx_{s_j}\Big\|_\infty\Bigg|<\ee\]
  Let  $(s_j)_{j=1}^l\in\text{\emph{Plm}}([M]^2)$ with $s_1(1)\geq
  M(k_0)$ and $s_1(2)\geq m_0$. Then by the above we get that
  \[\begin{split}\Big\|\sum_{i=1}^la_ie_i\Big\|_\infty&\leq \Big|\sum_{i=1}^la_ie_i(m_0)\Big|+\ee=\Big|\sum_{i=1}^la_ix_{s_i}(m_0)\Big|+\ee\\
  &\leq\Big\|\sum_{i=1}^la_ix_{s_i}\Big\|+\ee\leq \Big\|\sum_{j=1}^la_je_j'\Big\|+2\ee\end{split}\]
  Since this holds for all $\ee>0$, we obtain that
  \[\Big\|\sum_{i=1}^la_ie_i\Big\|_\infty\leq\Big\|\sum_{i=1}^la_ie'_i\Big\|\]
  and the proof is complete.
\end{proof}
\begin{prop}\label{c_0 has every bimonotone as order 2}
  For every Schauder basic spreading sequence $(e_n)_{n\in\nn}$ there exists
  $(e_n')_{n\in\nn}\in\mathcal{SM}_2(c_0)$ equivalent to
  $(e_n)_{n\in\nn}$. In particular, if $(e_n)_{n\in\nn}$ is
  bimonotone, then $(e_n)_{n\in\nn}\in\mathcal{SM}_2(c_0)$.
\end{prop}
\begin{proof}
  We may assume that $(e_n)_{n\in\nn}$ is a sequence in
  $\ell^\infty$. Let $C$ be the basis constant of $(e_n)_{n\in\nn}$.
  By Lemma \ref{on spr mod of c_0 second lem} there exists
  $(e_n')_{n\in\nn}\in\mathcal{SM}_2(c_0)$ such that
  \[\Big\|\sum_{i=1}^la_ie_i\Big\|_\infty\leq\Big\|\sum_{i=1}^la_ie'_i\Big\|\leq\max_{1\leq j\leq l}\Big\|\sum_{i=j}^l a_i
  e_i\Big\|_\infty\leq (1+C)\Big\|\sum_{i=1}^la_ie_i\Big\|_\infty\]
  for all $l\in\nn$ and $a_1,\ldots,a_l\in\rr$. Therefore
  $(e_n)_{n\in\nn}$ and $(e_n')_{n\in\nn}$ are equivalent. Moreover, if in
  addition $(e_n)_{n\in\nn}$ is bimonotone then
  \[\max_{1\leq j\leq l}\Big\|\sum_{i=j}^l a_i
  e_i\Big\|_\infty\leq \Big\|\sum_{i=1}^l a_i
  e_i\Big\|_\infty\] for all  $l\in\nn$ and
  $a_1,\ldots,a_l\in\rr$. Hence $(e_n')_{n\in\nn}$ is isometric to
  $(e_n)_{n\in\nn}$.
\end{proof}
\begin{cor}\label{max of bimonotone and trivial in c_0}
  Let $(e'_n)_{n\in\nn}$ be a bimonotone Schauder basic spreading sequence
  and $c>0$. We define the norm $\||\cdot\||$ on $c_{00}(\nn)$ as
  follows. We set
  \[\Big\|\Big|\sum_{j=1}^na_je_j\Big\|\Big|=\max\Big(c\cdot\Big|\sum_{j=1}^na_j\Big|,\|\sum_{j=1}^na_je'_j\Big\|\Big)\]
  for all $n\in\nn$ and $a_1,\ldots,a_n\in\rr$, where
  $(e_n)_{n\in\nn}$ is the natural Hamel basis of $c_{00}(\nn)$.
  Then $\||\cdot\||$ is well defined, $(e_n)_{n\in\nn}$ is a
  nontrivial
  spreading sequence and $(e_n)_{n\in\nn}\in\mathcal{SM}_2(c_0)$.
\end{cor}
\begin{proof}
  It is immediate that $\||\cdot\||$ is well defined and $(e_n)_{n\in\nn}$
  is a nontrivial spreading sequence. Since $(e'_n)_{n\in\nn}$ is a
  bimonotone Schauder basic sequence, by Proposition
  \ref{c_0 has every bimonotone as order 2}, there exists an
  $[\nn]^2$-sequence $(x_s)_{s\in[\nn]^2}$ in $c_0$ which
  generates $(e'_n)_{n\in\nn}$ as a $2$-spreading model.
  Clearly we may assume that for every $s\in[\nn]^2$, $x_s(1)=0$.
  Let, for every $s\in[\nn]^2$, $y_s$ to be the sequence in $c_0$
  such that $y_s(1)=c$ and $y_s(n)=x_s(n)$, for all $n>1$. It is
  easy to see that $(y_s)_{s\in[\nn]^2}$ generates
  $(e_n)_{n\in\nn}$ as a $2$-spreading model. Hence
  $(e_n)_{n\in\nn}\in\mathcal{SM}_2(c_0)$.
\end{proof}
By Proposition \ref{singular splitted} and the above corollary we
have the following.
\begin{cor}\label{c_0 universcal for singular}
  For every singular spreading sequence $(e_n)_{n\in\nn}$, we have
  that there exists
  $(\widetilde{e}_n)_{n\in\nn}\in\mathcal{SM}_2(c_0)$ such that
  the sequences $(e_n)_{n\in\nn}$ and
  $(\widetilde{e}_n)_{n\in\nn}$ are equivalent.
\end{cor}
\begin{prop}
  For every singular $(e_n)_{n\in\nn}\in\mathcal{SM}(c_0)$, we have that $(e_n)_{n\in\nn}$ belongs to $\mathcal{SM}_2(c_0)$. Moreover, if $e_n=e'_n+e$ is the natural decomposition of $(e_n)_{n\in\nn}$, then we have that \[\Big\|\sum_{j=1}^ka_je_j\Big\|=\max\Big\{\|e_1\|\cdot\Big|\sum_{j=1}^ka_j\Big|,\Big\|\sum_{j=1}^ka_je'_j\Big\|\Big\}\] for all $k\in\nn$ and $a_1,\ldots,a_k$.
\end{prop}
\begin{proof}
  Let $(e_n)_{n\in\nn}\in\mathcal{SM}(c_0)$ be  singular and let $e_n=e'_n+e$ be the natural decomposition of  $(e_n)_{n\in\nn}$.
  Let $\ff$ be a regular thin family, $M\in[\nn]^\infty$ and $(x_s)_{s\in\ff}$ an
  $\ff$-sequence in $c_0$ such that $(x_s)_{s\in\ff\upharpoonright M}$ generates
  $(e_n)_{n\in\nn}$ as an $\ff$-spreading model. By Corollary \ref{singular in space with separable dual}, we have
that  there exist $x\in c_0$ and $N\in[M]^\infty$ such that
$(x_s)_{s\in\ff\upharpoonright N}$ weakly converges to $x$,
$\|x\|=\|e\|$ and setting $x'_s=x_s-x$, for all
$s\in\ff\upharpoonright N$, we have that
$(x'_s)_{s\in\ff\upharpoonright N}$ admits $(e'_n)_{n\in\nn}$ as a
unique $\ff$-spreading model. We pass to an $L\in[N]^\infty$ such
that $(x'_s)_{s\in\ff\upharpoonright L}$ generates
$(e'_n)_{n\in\nn}$ as an $\ff$-spreading model. Let $k\in\nn$ and
$a_1,\ldots,a_k$. Let $((s_j^n)_{j=1}^k)_{n\in\nn}$ be a sequence
in $\text{\emph{Plm}}(\ff\upharpoonright L)$, such that $\min
s_1^n\to\infty$. Notice that  $(x'_s)_{s\in\ff\upharpoonright L}$
is weakly null. Therefore we have that
  \[\begin{split}
  \Big\|\sum_{j=1}^ka_je_j\Big\|& =\lim_{n\to\infty}\Big\|\sum_{j=1}^ka_jx_{s_j}\Big\|=\lim_{n\to\infty}\Big\|\sum_{j=1}^ka_jx+\sum_{j=1}^ka_jx'_{s_j}\Big\|\\ &=\max\Big\{\Big\|\sum_{j=1}^ka_jx\Big\|,\lim_{n\to\infty}\Big\|\sum_{j=1}^ka_jx'_{s_j}\Big\|\Big\}\\ &=\max\Big\{\|e_1\|\cdot\Big|\sum_{j=1}^ka_j\Big|,\Big\|\sum_{j=1}^ka_je'_j\Big\|\Big\}
  \end{split}\]
Since $(e'_n)_{n\in\nn}$ is 1-unconditional, we have that
$(e'_n)_{n\in\nn}$ is also bimonotone Schauder basic. The above
equation and Corollary \ref{max of bimonotone and trivial in c_0}
yield that $(e_n)_{n\in\nn}\in\mathcal{SM}_2(c_0)$.
\end{proof}
By Proposition \ref{c_0 has every bimonotone as order 2} and
Corollary \ref{c_0 universcal for singular} we have the following.
\begin{cor}
  The set $\mathcal{SM}_2(c_0)$ is isomorphically universal for all
  spreading sequences.
\end{cor}
\begin{rem}
  In Subsection \ref{subsection Tsirelson} of Chapter \ref{chapter 9} we describe the spreading models of the
  generalized Tsirelson's space $T_\alpha$, $\alpha<\omega_1$.
\end{rem}
\chapter{Composition of the spreading models} In this chapter we
present some composition properties for spreading models. Among
others we  show that the class of spreading models, which we call
block $k$-iterated spreading models, introduced in \cite{O-S} are
also spreading models of the same order in our context. Moreover,
we present several related results for $\ell^p$ and $c_0$
spreading models.
\section{The composition property}
We well need the notion of the plegma block $\ff$-sequence.
\begin{defn}
  Let  $X$ be a Banach space with a Schauder basis. Let $\ff$ be a regular thin family, $M\in[\nn]^\infty$ and
  $(x_s)_{s\in\ff}$ be an $\ff$-sequence of finitely supported
  vectors in $X$. We will say that the $\ff$-subsequence $(x_s)_{s\in\ff\upharpoonright
  M}$ is  plegma block if for
  every plegma pair $(s_1,s_2)$ in $\ff\upharpoonright M$ we have
  that
  $\text{supp}(x_{s_1})<\text{supp}(x_{s_2})$.
\end{defn}
In the sequel an $\ff$-spreading model will be called \emph{plegma
block generated} if it is generated by a plegma block
$\ff$-subsequence.
\begin{thm}\label{composition thm}
  Let  $X$ be a Banach space and  $(e_n)_{n\in\nn}$ be a Schauder basic sequence in $\mathcal{SM}_\xi(X)$, for some  $\xi<\omega_1$. Let
  $E$ be the Banach space generated by $(e_n)_{n\in\nn}$ and let $(\overline{e}_n)_{n\in\nn}$ be a plegma block generated $k$-spreading model of $E$, for some  $k\in\nn$.
  Then $$(\overline{e}_n)_{n\in\nn}\in\mathcal{SM}_{\xi+k}(X)$$
\end{thm}
For the proof of the above theorem we need the following
definition.
\begin{defn}
  Let $\ff,\g$ be families of finite subsets of $\nn$. We define the ordered direct sum of $\g$ and $\ff$ to be the family
  \[\g\oplus\ff=\Big{\{}s\cup t:s\in\g,t\in\ff\;\;\text{and}\;\;s<t\Big{\}}\]
\end{defn}
\begin{rem}
  It is easy to see that if $\ff$ and $\g$ are regular thin families, then $\g\oplus\ff$ is also a regular
  thin family and $o(\g\oplus\ff)=o(\ff)+o(\g)$. In particular $o([\nn]^{k}\oplus\ff)=o(\ff)+k$, for every $k\in\nn$.
\end{rem}
\begin{proof}[Proof of Theorem \ref{composition thm}]
  Let $\ff$ be a regular thin family with $o(\ff)=\xi$,
   $(x_s)_{s\in\ff}$ be an $\ff$-sequence in $X$ and $M\in[\nn]^\infty$ such that $(x_s)_{s\in\ff\upharpoonright M}$ generates $(e_n)_{n\in\nn}$
   as an $\ff$-spreading model.
  We may also suppose that $\ff$ is very large in $M$. Using Remark \ref{remark for k spr mod} we may
  assume that there exists a plegma block $[\nn]^k$-sequence
  $(y_t)_{t\in[\nn]^k}$ which generates $(\overline{e}_n)_{n\in\nn}$
  as an $[\nn]^k$-spreading model.

  For every $t\in[\nn]^k$ we set $F_t=\text{supp}(y_t)$ with respect to $(e_n)_{n\in\nn}$ and $l_t=|F_t|$.
  Then for every $t\in[\nn]^k$
  the vector $y_t$ is of the form \[y_t=\sum_{j=1}^{l_t}a^t_{F_t(j)}e_{F_t(j)}^{\;}\]
  We set $\g=[\nn]^{k}\oplus\ff$ and for every $v\in\g$ we set $t_v$ and $s_v$ the unique elements
  in $[\nn]^k$ and $\ff$ respectively such that $v=t_v\cup s_v$ and $t_v<s_v$. We split the proof into three steps.\\
  \textbf{Step 1:} We set \[\g^*=\Big{\{} v\in\g:\min s_v\geq M(l_{t_v}) \Big{\}}\]
  It is easy to verify that $\g^*$ is large in $M$. Hence by Theorem \ref{th2} there exists $L_0\in[M]^\infty$ such that
  $\g\upharpoonright L_0\subseteq \g^*$. For every $v\in\g\upharpoonright L_0$ we define a finite sequence $(s^v_1,\ldots,s^v_{l_{t_v}})$
  as follows. Let $s_v=\{M(q^v_1),\ldots,M(q^v_{|s_v|})\}$.
  Then for every $j=1,\ldots,l_{t_v}$ we set $s^v_j$ to be the unique element of $\ff\upharpoonright M$ such that
  $$s^v_j\sqsubseteq \Big{\{}M(q_i^v-l_{t_v}+j):i=1,\ldots,|s_v|\Big{\}}$$
  We define the family $(z_v)_{v\in\g\upharpoonright L_0}$ by setting
  $$z_v=\sum_{j=1}^{l_{t_v}}a^{t_v}_{F_{t_v}(j)}x_{s^v_j}$$
  for every $v\in\g\upharpoonright L_0$.\\
  \textbf{Step 2:} For every $n\in\nn$ we set
  \[\begin{split}\mathcal{A}_n=\Big{\{} (v_p)_{p=1}^n\in &\text{\emph{Plm}}_n(\g\upharpoonright L_0):
  \; s_1^{v_1}(1)\geq M\Big(\sum_{p=1}^{n}l_{t_{v_p}}\Big) \\
  &\text{and}\;\;
  (s_j^{v_1})_{j=1}^{l_{t_{v_1}}}\;^\frown \ldots \;^\frown (s_j^{v_n})_{j=1}^{l_{t_{v_n}}}\;\text{is plegma}
   \Big\}\end{split}\]
  It is easily verified that $\mathcal{A}_n$ is large in $L_0$. Inductively using Theorem \ref{ramseyforplegma} we construct a decreasing sequence $(L_n)_{n\in\nn}$
  in $[L_0]^\infty$ such that $\text{\emph{Plm}}_n(\g\upharpoonright L_n)\subseteq\mathcal{A}_n$ for all $n\in\nn$.
  Let $L$ be a diagonalization of $(L_n)_{n\in\nn}$,
  that is $L(n)\in L_n$ for each $n$.\\
  \textbf{Step 3:} In this step we will show that $(z_s)_{s\in\g\upharpoonright L}$ admits $(\overline{e}_n)_{n\in\nn}$ as a
  $(\xi+k)$- spreading model.
  Indeed, first
  notice that the family $\g$ is of order $\xi+k$. Let $(\delta^1_n)_{n\in\nn}$ be the null sequence of positive numbers such that
  $(x_s)_{s\in\ff\upharpoonright M}$ generates $(e_n)_{n\in\nn}$ with respect
  to $(\delta^1_n)_{n\in\nn}$. Let also $(\delta^2_n)_{n\in\nn}$ be the null sequence of positive numbers such that
  $(y_t)_{t\in[L]^k}$ generates $(\overline{e_n})_{n\in\nn}$ with respect
  to $(\delta^2_n)_{n\in\nn}$. Let $C$ be  the basis constant of $(e_n)_{n\in\nn}$, $K=\sup\{\|y_t\|:t\in[L]^k\}$
  and set $\delta_n=2CK\delta^1_n+\delta^2_n$, $n\in\nn$.

  We will show that $(z_v)_{v\in\g\upharpoonright L}$ satisfies the conditions of Remark \ref{old defn yields new}
  concerning the sequence $(\overline{e_n})_{n\in\nn}$
  with respect to $(\delta_n)_{n\in\nn}$.
Let $l\in\nn$, $(v_i)_{i=1}^l$ a plegma $l$-tuple in
$\g\upharpoonright L$
  with $v_1(1)\geq L(l)$ and $b_1,\ldots,b_l\in[-1,1]$. By Step 2, $(v_i)_{i=1}^l$ belongs to $\mathcal{A}_l$. Notice that
  \begin{equation}\begin{split}\label{eq10}
   \Bigg| \Big\|\sum_{i=1}^l b_i z_{v_i} \Big\|-\Big\|\sum_{i=1}^l b_i \overline{e}_i \Big\| \Bigg|
    \leq& \Bigg|\Big\|\sum_{i=1}^l b_i z_{v_i}   \Big\|-\Big\|\sum_{i=1}^l b_i y_{t_{v_i}}  \Big\| \Bigg|\\
    &+ \Bigg| \Big\|\sum_{i=1}^l b_i y_{t_{v_i}}  \Big\|- \Big\|\sum_{i=1}^l b_i \overline{e}_i \Big\| \Bigg|
  \end{split}\end{equation}
  It is straightforward that
  \begin{equation}\label{eq11}
    \Bigg| \Big\|\sum_{i=1}^l b_i y_{t_{v_i}}  \Big\|-\Big\| \sum_{i=1}^l b_i \overline{e}_i  \Big\|\Bigg|<\delta^2_l
  \end{equation}
  Since $(e_n)_{n\in\nn}$ is Schauder basic with basis constant $C$ and for every $t\in[L]^k$, $\|y_t\|\leq K$,
  we have that for every $t\in[L]^k$ and $1\leq j\leq l_t$, $|a^t_{F_t(j)}|\leq 2CK$. Hence for every $1\leq i\leq l$, $t\in[L]^k$ and $1\leq j\leq l_t$,
  we have that $b_ia^t_{F_t(j)}\in [-2CK,2CK]$. Since $(v_i)_{i=1}^l$ belongs to $\mathcal{A}_l$ we have that $v_1(1)\geq M (\sum_{i=1}^l l_{t_{v_i}})
  \geq L(\sum_{i=1}^l l_{t_{v_i}})$. So
  \begin{equation}\begin{split}\label{eq12}
   \Bigg|\Big\|\sum_{i=1}^l b_i z_{v_i}   \Big\|-\Big\|\sum_{i=1}^l b_i y_{t_{v_i}}  \Big\| \Bigg|=&
   \Bigg|\Big\|\sum_{i=1}^l\sum_{j=1}^{l_{t_{v_i}}} b_ia_{F_{t_{v_i}}(j)}^{t_{v_i}} x^{s_{v_i}}_j  \Big\|\\
   &-\Big\|\sum_{i=1}^l\sum_{j=1}^{l_{t_{v_i}}} b_ia_{F_{t_{v_i}}(j)}^{t_{v_i}} e_{F_{t_{v_i}}(j)}^{\;}  \Big\|  \Bigg|<2CK\delta^1_l
  \end{split}\end{equation}
  The inequalities (\ref{eq10}), (\ref{eq11}) and (\ref{eq12}) yield that
  \[\Bigg| \Big\|\sum_{i=1}^l b_i z_{v_i} \Big\|-\Big\|\sum_{i=1}^l b_i \overline{e}_i \Big\| \Bigg|<\delta^2_l+2CK\delta^1_l=\delta_l\]
  Hence by Remark \ref{old defn yields new}, we get that for some $L'\in[L]^\infty$,
   $(z_v)_{v\in\g\upharpoonright L'}$ generates $(\overline{e_n})_{n\in\nn}$ as a $\g$-spreading model.
\end{proof}
\begin{rem}\label{composition rem}
  If we additionally
  assume that $X$ has a Schauder basis then the above proof actually provides more information concerning the structure of the sequence  $(z_v)_{v\in\g\upharpoonright L}$.
  First notice that by
  Step 2, we have the following.
  \begin{enumerate}
    \item[(i)] If $(e_n)_{n\in\nn}$ is plegma disjointly (resp. block) generated
    by $(x_s)_{s\in\ff\upharpoonright M}$ then $(\overline{e}_n)_{n\in\nn}$
    is plegma disjointly (resp. block) generated by $(z_v)_{v\in\g\upharpoonright L'}$.
    \item[(ii)] If
   $(x_s)_{s\in\ff\upharpoonright M}$ admits a disjoint canonical tree decomposition then $(z_v)_{v\in\g\upharpoonright L}$
   also does.
  \end{enumerate}
   Also notice that for every $v\in\g\upharpoonright L$ we have that
   that  $|s_j^v|<|v|$, for all $1\leq j\leq l_{t_v}$.
\end{rem}

  Let us point out that Theorem \ref{composition thm} does not seem extendable for arbitrary thin family $\g$ in place of $[\nn]^k$.
  The main difficulty for this is that the elements of $\g$ when $o(\g)\geq \omega$ are not of equal length. Thus in the new thin
  family $\g\oplus\ff$
  the plegma pairs are not decomposed into two plegma pairs from
  the families
  $\ff$ and $\g$.
  However a general composition result (i.e. for arbitrary thin family $\g$)
  seems provable after a modification
  of the notion of plegma
   on $\g\oplus\ff$, but this is beyond the purposes of the present paper.

  \section{The $k$-iterated spreading models}
  A different notion of $k$-spreading model is provided in \cite{O-S}.
  In this section we will discuss its relation with the present context.
  According to \cite{O-S} we have the following terminology.

   1) Let $E_0,E$ be Banach spaces. We write $E_0\to E$ if $E$ has a Schauder basis
    which is a spreading model of some seminormalized basic sequence
    in $E_0$ and $E_0\stackrel{k}{\to}E$ if $E_0\to E_1\to\ldots\to E_{k-1}\to E$
    for some sequence of Banach spaces $E_1,\ldots, E_{k-1}$. Note that for every
     $k\in\nn$ if $E_0\stackrel{k}{\to}E$ then $E$ has a spreading Schauder basis.

  2)  Let $E_1,E_2$ be Banach spaces with Schauder bases
    $(e_n^1)_{n\in\nn}$ and $(e_n^2)_{n\in\nn}$ respectively. We
    will write $E_1\substack{\; \\\longrightarrow\\\text{bl}} E_2$ if
    $(e_n^2)_{n\in\nn}$ is a spreading model of some seminormalized
    block subsequence of $(e_n^1)_{n\in\nn}$. Let $E_0$ be a Banach space and $E$ be a Banach space with a
    Schauder basis. Similarly for some
    $k>1$, we say that $E_0\substack{k \\\longrightarrow\\\text{bl}}
    E$, if there exists a sequence $E_1,\ldots,E_{k-1}$ of Banach
    spaces with Schauder bases
    $(e_n^1)_{n\in\nn},\ldots,(e_n^{k-1})_{n\in\nn}$ respectively such that
    $E_0\substack{\;
    \\\longrightarrow\\\text{ }} E_1 \substack{\;
    \\\longrightarrow\\\text{bl}} E_2 \substack{\;
    \\\longrightarrow\\\text{bl}}\ldots\substack{\;
    \\\longrightarrow\\\text{bl}} E_{k-1} \substack{\;
    \\\longrightarrow\\\text{bl}} E$.
\begin{defn}
    Let $E_0,E$ be Banach spaces and $k\in\nn$. We say that $E$ is a
    $k$-\textit{iterated} spreading model of $E_0$ if $E_0\stackrel{k}{\to}E$.
    Additionally, $E$ is a \textit{block} $k$-\textit{iterated} spreading model of $E_0$ if
    $E_0\substack{k\\\longrightarrow\\\text{bl}} E$.
  \end{defn}
Under the above definition we have the following which is a direct consequence of Theorem \ref{composition thm}.
  \begin{cor}\label{Proposition correlating the two kinds of bloch spreading model notions}
    Let $X$ be a Banach space and $E$ be a Banach space with Schauder basis $(e_n)_{n\in\nn}$ such that
      $X\substack{k \\\longrightarrow\\\text{bl}} E$, for some  $k>1$. Then $(e_n)_{n\in\nn}$  is a
   $k$-spreading model of $X$.
  \end{cor}
  \begin{rem} A space which does not contain any $\ell^p$, for $1\leq p\leq\infty$,
    or $c_0$ spreading model  is constructed in \cite{O-S}.
    In the same paper it is asked if there exists a space which does not contain any $\ell^p$, for $1\leq p\leq\infty$,
    or $c_0$ $k$-iterated spreading model of any $k\in\nn$. In
    Chapter \ref{space Odel_Schlum} we answer affirmatively this
    problem.
  \end{rem}
  \begin{rem}
    Let us also mention that generally the class of block $k$-iterated spreading models is strictly smaller than the
    corresponding one of the plegma block generated $k$-spreading
    models. In Chapter \ref{space separeting strong k-order from
    k-order} we construct a space having this property.
  \end{rem}
  \section{Applications to $\ell^p$ and $c_0$ spreading models}
  \subsection{$\ell^p$ spreading models}
  \begin{defn}
    Let $X$ be a Banach space, $\ff$ be a regular thin family,
and $(x_s)_{s\in\ff}$ be an $\ff$-sequence in $X$. Also let
$M\in[\nn]^\infty$ and  $p\in[1,\infty)$. We say that the
    $\ff$-subsequence $(x_s)_{s\in\ff\upharpoonright M}$ generates
    $\ell^p$ (resp. $c_0$) as an $\ff$-spreading model if $(x_s)_{s\in\ff\upharpoonright M}$
    generates as an $\ff$-spreading model a sequence equivalent
    to the usual basis of $\ell^p$ (resp. $c_0$).
  \end{defn}
  \begin{rem}
    It is easy to see that an $\ff$-subsequence
    $(x_s)_{s\in\ff\upharpoonright M}$ generates $\ell^p$
    as an $\ff$-spreading model iff $(x_s)_{s\in\ff\upharpoonright M}$
    generates an $\ff$-spreading model and there exist $C,c>0$
    such that
    \[c\Big(\sum_{j=1}^n|a_j|^p\Big)^\frac{1}{p}\leq\Big\|\sum_{j=1}^na_jx_{s_j}\Big\|\leq C\Big(\sum_{j=1}^n|a_j|^p\Big)^\frac{1}{p}\]
    for all $n\in\nn$, $a_1,\ldots,a_n\in\rr$ and $(s_j)_{j=1}^n$
    plegma $n$-tuple in $\ff\upharpoonright M$ with $s_1(1)\geq
    M(n)$. A  similar remark also holds for $\ff$-subsequences
    which generate $c_0$ as a spreading model.
  \end{rem}

  \begin{prop}\label{getting l^p spr mod by containing}
    Let $X$ be a Banach space, $\xi<\omega_1$ and $(e_n)_{n\in\nn}$
    be a nontrivial sequence in $\mathcal{SM}_\xi(X)$. Assume  that
    the space $E$ generated by $(e_n)_{n\in\nn}$ contains an isomorphic copy of
    $\ell^p$ (resp. $c_0$), for some $p\in[1,\infty)$. Then $X$
    admits an $\ell^p$ (resp. $c_0$) spreading model of order
    $\xi+1$.
  \end{prop}
  \begin{proof}
    Since $(e_n)_{n\in\nn}$ is nontrivial,
    we have that $(e_n)_{n\in\nn}$ is either a Schauder basic
    or a singular spreading sequence. Suppose that $(e_n)_{n\in\nn}$
    is Schauder basic. Since $E$ contains an isomorphic copy of
    $\ell^p$  for some $p\in[1,\infty)$ (resp. $c_0$), by standard
    arguments, there exists a block subsequence $(z_n)_{n\in\nn}$ of
    $(e_n)_{n\in\nn}$ equivalent to the usual basis of $\ell^p$
    (resp. $c_0$). Notice that every spreading model of
    $(z_n)_{n\in\nn}$ is also equivalent to the usual basis of $\ell^p$
    (resp. $c_0$) and therefore the result  follows by Theorem
    \ref{composition thm} for $k=1$.

  Suppose that $(e_n)_{n\in\nn}$ is a singular spreading sequence.
  Let $e_n=e'_n+e$ be the natural decomposition of
$(e_n)_{n\in\nn}$ and $E'$ be the Banach generated by
$(e'_n)_{n\in\nn}$.
  Let $(z_n)_{n\in\nn}$ be a sequence in $E$ equivalent to the usual
basis of $\ell^p$ (resp. $c_0$). By Corollary \ref{isomorphy_between_E_and_E'} we have that $E$ and $E'$ are isomorphic and therefore there exists a sequence $(z'_n)_{n\in\nn}$ in $E'$ isomorphic to the usual basis of $\ell^p$. Recall that $(e'_n)_{n\in\nn}$ is an unconditional basis for $E'$ and therefore the proof is complete by following similar arguments as in the first case above of $(e_n)_{n\in\nn}$ being Schauder basic.
  \end{proof}
\subsection{Isometric $\ell^1$ and $c_0$ spreading models} The
following is one of the well known results due to R. James (c.f.
\cite{J}).
  \begin{prop}
    Let $(x_n)_{n\in\nn}$ be a normalized sequence in a Banach space $X$ equivalent to the usual basis of $\ell^1$ (resp. $c_0$).
     Then for every $\ee>0$ there exists a block normalized subsequence $(y_n)_{n\in\nn}$ of $(x_n)_{n\in\nn}$ which admits
     lower $\ell^1$ constant $1-\ee$ (resp. lower $c_0$ constant $1-\ee$ and upper $c_0$ constant $1+\ee$).
  \end{prop}
  Using the standard diagonolization argument the above
  proposition yields the following.
  \begin{prop}
    Let $X$ be a Banach space with a Schauder basis. If $X$ contains an isomorphic copy of $\ell^1$ (resp. $c_0$) then $X$
    admits the usual basis of $\ell^1$ (resp. $c_0$) as a block generated spreading model of order one.
  \end{prop}
  The above proposition, Theorem \ref{composition thm} and Remark
  \ref{composition rem} readily yield the following.
  \begin{cor}\label{one order more to get the isometric}
    Let $X$ be a Banach space  and $(e_n)_{n\in\nn}$ be a Schauder basic sequence in $\mathcal{SM}_\xi(X)$,
    for some  $\xi<\omega_1$. Suppose that  the space generated by $(e_n)_{n\in\nn}$ contains an isomorphic copy of $\ell^1$ (resp. $c_0$).
    Then $X$ admits the usual basis of $\ell^1$ (resp. $c_0$) as a $(\xi+1)$-spreading model.

  If in addition $X$ has a Schauder basis then for every  regular thin family $\ff$ of order $\xi$,
   $M\in[\nn]^\infty$ and $\ff$-sequence $(x_s)_{s\in\ff}$
    in $X$ such that the $\ff$-subsequence $(x_s)_{s\in\ff\upharpoonright M}$ generates $(e_n)_{n\in\nn}$ as an $\ff$-spreading model,
    there exist $L\in [M]^\infty$ and a $\g$-sequence $(z_v)_{v\in\g}$, where $\g=[\nn]^1\oplus\ff$, which satisfy the following:
    \begin{enumerate}
      \item[(i)] The $\g$-subsequence $(z_v)_{v\in\g\upharpoonright L}$ generates the usual basis of $\ell^1$ (resp. $c_0$)
       as a $\g$-spreading model. Moreover, if the $\ff$-subsequence $(x_s)_{s\in\ff\upharpoonright M}$ is plegma block, then
     the $\g$-subsequence $(z_v)_{v\in\g\upharpoonright L}$ is also plegma block.
      \item[(ii)] For every $v\in\g\upharpoonright L$ there exist $m\in\nn$ and $s_1,\ldots,s_m\in\ff$ such that
      $z_v\in<x_{s_1},\ldots,x_{s_m}>$ and $|s_j|<|v|$, for all $1\leq j\leq m$.
    \end{enumerate}
  \end{cor}
  \begin{cor}\label{reflexive spaces have c_0 or l^1 or rflxv spr mod}
     Let $X$ be a reflexive space. Then one of the following
     holds.
     \begin{enumerate}
       \item[(i)] The space $X$ admits the usual basis of $\ell^1$
       as a $\xi$-spreading model, for some
       $\xi<\omega_1$.
       \item[(ii)] The space $X$ admits  the usual basis of $c_0$
       as a $\xi$-spreading model, for some
       $\xi<\omega_1$.
       \item[(iii)] Every nontrivial spreading model of any
       order of $X$, generates a reflexive space.
\end{enumerate}
       Moreover  every Schauder basic $(e_n)_{n\in\nn}\in\mathcal{SM}(X)$  is
    unconditional.
  \end{cor}
  \begin{proof}
    Suppose that neither (i) nor (ii) holds. Let $1\leq \xi<\omega_1$ and a nontrivial
    $(e_n)_{n\in\nn}\in\mathcal{SM}_\xi(X)$. We will show that the space
      $E$ generated by $(e_n)_{n\in\nn}$ is reflexive and hence
      (iii) is true. First notice that the space $E$ cannot
contain an isomorphic copy of $c_0$ or $\ell^1$. Indeed otherwise
by Corollary \ref{one order more to get the isometric}, the space
$X$ would admit the usual basis of $\ell^1$ or $c_0$ as a
$(\xi+1)$-spreading model. We distinguish two cases. Either
$(e_n)_{n\in\nn}$
      is Schauder basic or it is singular.

      If $(e_n)_{n\in\nn}$ is Schauder basic then by
      Theorem \ref{relatively weakly compact sets have unconditional spreading models}
we have that it is unconditional. Hence by James' theorem  we have
that $E$ is reflexive.

Assume that  $(e_n)_{n\in\nn}$ is singular and let $e_n=e'_n+e$ be
its natural decomposition.   Let $E'$ be the space generated by
$(e'_n)_{n\in\nn}$. Then we have that $(e'_n)_{n\in\nn}$ is
unconditional and $E=E'\oplus <e>$. Since $E$ does not contain any
isomorphic copy of $c_0$ or $\ell^1$, $E'$ also does. Hence by
James' theorem  we have that $E'$ is reflexive. Since $E=E'\oplus
<e>$, we have that $E$ is also reflexive.
  \end{proof}

\chapter{$\ell^1$ spreading models}\label{chapter 9}
In this chapter we study $\ell^1$ spreading models generated by
$\ff$-sequences. In the first section we establish the non
distortion of $\ell^1$ spreading models and thus we extend the
corresponding known result which holds for the classical spreading
models. In the second section we present a technique of splitting
a canonical tree decomposition of an $\ff$-sequence. This
technique will be used to produce plegma block generated $\ell^1$
spreading models. Finally in the last section we extend the well
known dichotomy of H.P. Rosenthal concerning Ces\`aro summability
and $\ell^1$ spreading models (see \cite{AT}, \cite{M}).
\section{Almost isometric $\ell^1$ spreading models}

  Let $M\in[\nn]^\infty$, $C\geq c>0$, $\ff$ be a regular thin family and $(x_s)_{s\in\ff}$ be an $\ff$-sequence in a
  Banach space $X$. We will say that the $\ff$-subsequence $(x_s)_{s\in\ff\upharpoonright M}$ \textit{generates} $\ell^1$
  \textit{as an $\ff$-spreading model
  with constants} $c,C$ if $(x_s)_{s\in\ff\upharpoonright M}$ generates an $\ff$-spreading model $(e_n)_n$
  which is $c, C$ isomorphic to the standard basis of $\ell^1$,
  that is
  \[c\sum_{j=1}^n|a_j|\leq\Big\|\sum_{j=1}^n a_j e_j\Big\|\leq C \sum_{j=1}^n|a_j|\]
   for all $n\in\nn$ and  $a_1,\ldots,a_n\in\rr$.
    In particular  if $\|x_s\|\leq 1$, for all $s\in\ff\upharpoonright M$,
  then we may assume that $C=1$ and in this case we will say that
 $(x_s)_{s\in\ff\upharpoonright M}$ \textit{generates} $\ell^1$ \textit{as an $\ff$-spreading model
  with constant} $c$.

\begin{prop}\label{Prop on almost isometric l^1 spr mod}
  Let $M\in[\nn]^\infty$, $1> c>0$, $\ff$ be a regular thin family and $(x_s)_{s\in\ff}$
  be  an $\ff$-sequence in the
  unit ball $B_X$ of a Banach space $X$. Suppose that $(x_s)_{s\in\ff\upharpoonright M}$ generates $\ell^1$ as an
   $\ff$-spreading model with constant $c$.
  Then for every $\varepsilon>0$ there exist $L\in[M]^\infty$ and an $\ff$-subsequence $(y_s)_{s\in\ff\upharpoonright L}$ in $B_X$
  such that the $\ff$-subsequence $(y_s)_{s\in\ff\upharpoonright L}$ generates $\ell^1$ as an $\ff$-spreading model with constant $1-\varepsilon$.
\end{prop}
\begin{proof}Let $(e_n)_{n\in\nn}$ be the $\ff$-spreading model
generated by $(x_s)_{s\in\ff\upharpoonright M}$. Then  for every
$n\in\nn$ and
  $a_1,\ldots,a_n\in\rr$, we have that
  \[c\sum_{j=1}^n|a_j|\leq\Big\|\sum_{j=1}^na_je_j\Big\|\leq\sum_{j=1}^n|a_j|\]
Clearly we  may assume that
\begin{equation}\label{eqell1} c=\inf\Bigg\{\Big\|\sum_{j=1}^na_je_j\Big\|:\;n\in\nn\;\;\text{and}\;\;a_1,\ldots,a_n\in\rr\;\;\text{such
that}\;\;\sum_{j=1}^n|a_j|=1\Bigg\}\end{equation} Let
$\varepsilon>0$ and choose $0<\varepsilon'<c$   such that
\[\frac{c-\varepsilon'}{c+2\varepsilon'}>1-\varepsilon\]
By passing to an infinite subset of $M$, if it is necessary, we
may assume the following:
\begin{enumerate}
\item[(a)] The family $\ff$ is very large in $M$. \item[(b)] For
every $n\in\nn$, $a_1,\ldots,a_n\in[-1,1]$ and  every plegma
$n$-tuple $(s_j)_{j=1}^n$
  in $\ff\upharpoonright M$ with $\min s_1\geq M(n)$,
  \[\Bigg|\Big\|\sum_{j=1}^na_j x_{s_j}\Big\|-\Big\|\sum_{j=1}^na_j e_j\Big\|\Bigg|<\varepsilon'\]
\end{enumerate}
  By  (\ref{eqell1}), there exist $k\in\nn$ and $b_1,\ldots, b_k\in[-1,1]$ with $\sum_{i=1}^k|b_i|=1$  such that
  \[\Big\|\sum_{i=1}^k b_i e_i\Big\|<c+\varepsilon'\]
  Hence by (b), for every plegma $k$-tuple $(s_i)_{i=1}^k$
  in $\ff\upharpoonright M$ with $\min s_1\geq M(k)$ we have that
  \[\Big\|\sum_{i=1}^k b_i x_{s_i}\Big\|<c+2\varepsilon'\]
  For each $n\in\nn$, set $I_n=\{M(n\cdot k+1),..., M((n+1)\cdot k)\}$. Then
  obviously,
  $\max(I_n)<\min(I_{n+1})$, $|I_n|=k$ and $\min(I_n)>
    M(n\cdot k)$.
We set \[L=\{\max I_n:\;n\in\nn\}=\{M((n+1)\cdot k):\;n\in\nn\}\]
Since $\widehat{\ff}$ is regular and $\ff$ is very large in $M$,
it is easy to see that
  for every $1\leq i\leq k$ there exists a
  unique $t_i^s\sqsubseteq \{I_{n_j}(i): 1\leq j\leq |s|\}$ with
  $t_i^s\in\ff$. We set
\[y_s=\frac{\sum_{j=1}^kb_j x_{t_j^s}}{c+2\ee'},\]
for all $s\in \ff\upharpoonright L$.

  We claim  that the $\ff$-subsequence $(y_s)_{s\in\ff\upharpoonright L}$ generates $\ell^1$ as an
  $\ff$-spreading model with constant $1-\varepsilon$. Indeed, let $n\in\nn$, $a_1,\ldots,a_n\in [-1,1]$ and $(s_j)_{j=1}^n$
  plegma $n$-tuple in $\ff\upharpoonright L$ with $s_1(1)\geq L(n)$. First notice that the $n\cdot k$-tuple
  $(t_1^{s_1},\ldots,t_k^{s_1},\ldots,t_1^{s_n},\ldots,t_k^{s_n})$ is plegma and $t_1^{s_1}(1)\geq\min
  I_n>
    M(n\cdot k)$. Hence
\[\begin{split}
    \Big\|\sum_{j=1}^n a_j y_{s_j}\Big\|&=
    \Big\|\sum_{j=1}^na_j \cdot
    \sum_{i=1}^k  \frac{b_i x_{t_i^{s_j}}}{c+2\ee'}\Big\|=
    \Big\|\sum_{j=1}^n\sum_{i=1}^k\frac{a_j}{c+2\ee'}b_i
    x_{t_i^{s_j}}\Big\|\\ &\geq
    (c-\varepsilon')\sum_{j=1}^n\sum_{i=1}^k\frac{|a_j|\cdot|b_i|}{c+2\ee'}
    =\frac{c-\varepsilon'}{c+2\varepsilon'}\sum_{j=1}^n|a_j|\sum_{i=1}^k|b_i|=
    \frac{c-\varepsilon'}{c+2\varepsilon'}\sum_{j=1}^n|a_j|
  \end{split}\]
\end{proof}
\begin{rem}\label{Rem on almost isometric l^1 spr mod}
 If we additionally
  assume that $X$ has a Schauder basis then by the above proof we
  have the following.
  \begin{enumerate}
    \item[(i)] If $(x_s)_{s\in\ff\upharpoonright M}$ is plegma
    block (resp. plegma disjointly supported) then $(y_s)_{s\in\ff\upharpoonright
    L}$ is plegma
    block (resp. plegma disjointly supported).
    \item[(ii)] If $(x_s)_{s\in\ff\upharpoonright M}$ admits a canonical (resp. disjoint canonical) tree decomposition then the same holds for $(y_s)_{s\in\ff\upharpoonright
    L}$.
  \end{enumerate}
\end{rem}
\section{Plegma block generated $\ell^1$ spreading models}
\subsection{Splittings of canonical tree decompositions}
\begin{defn}\label{defn g splitting}
  Let $X$ be a Banach space with Schauder basis, $\ff$ be a regular thin family,
  $M\in[\nn]^\infty$ and $(x_s)_{s\in\ff}$ be an $\ff$-sequence in
  $X$ such that the $\ff$-subsequence
  $(x_s)_{s\in\ff\upharpoonright M}$ admits a canonical tree
  decomposition $(y_t)_{t\in\widehat{\ff}\upharpoonright M}$. Let
  $\g$ be a thin family of finite subsets of $\nn$ such that
  $\g\sqsubset\ff\upharpoonright M$. We define the $\g$-splitting
  of $(x_s)_{s\in\ff\upharpoonright M}$ to be the
  $\ff$-subsequences $\big(x^{(1,\g)}_s\big)_{s\in\ff\upharpoonright M}$ and
  $(x^{(2,\g)}_s)_{s\in\ff\upharpoonright M}$, which are defined
  as follows. For every $s\in\ff\upharpoonright M$ we set
  \[x_s^{(1,\g)}=\sum_{t\sqsubseteq s_\g}y_t\;\;\text{and}\;\;
  x_s^{(2,\g)}=\sum_{s_\g\sqsubset t\sqsubseteq s}y_t\]
  where $s_\g$ is the unique element of $\g$ satisfying
  $s_\g\sqsubset s$.
\end{defn}
\begin{rem}
  Under the assumptions of Definition \ref{defn g splitting}, we
  easily observe the following.
  \begin{enumerate}
    \item[(i)] For each $s\in\ff\upharpoonright M$,
    $x_s=x_s^{(1,\g)}+x_s^{(2,\g)}$.
    \item[(ii)] The $\ff$-subsequences $\big(x^{(1,\g)}_s\big)_{s\in\ff\upharpoonright M}$ and
  $\big(x^{(2,\g)}_s\big)_{s\in\ff\upharpoonright M}$ admit canonical tree
  decompositions $\big(y^{(1,\g)}_t\big)_{t\in\widehat{\ff}\upharpoonright M}$ and
  $\big(y^{(2,\g)}_t\big)_{t\in\widehat{\ff}\upharpoonright M}$
  respectively which are of the following form.  For every $t\in\widehat{\g}$,
  \[y^{(1,\g)}_t=y_t\;\;\text{and}\;\;y^{(2,\g)}_t=
  0\]
    while for every $t\in (\widehat{\ff}\upharpoonright M)\setminus\widehat{\g}$,
  \[y^{(1,\g)}_t=0\;\;\text{and}\;\;y^{(2,\g)}_t=
  y_t\]
  It is obvious that $\widetilde{y}_t=\widetilde{y}_t^{(1,\g)}+\widetilde{y}_t^{(2,\g)}$.
  In the sequel we will refer to $\big(y^{(1,\g)}_t\big)_{t\in\widehat{\ff}\upharpoonright M}$ and
  $\big(y^{(2,\g)}_t\big)_{t\in\widehat{\ff}\upharpoonright M}$ as
  the $\g$-\textit{splitting of} $(y_t)_{t\in\widehat{\ff}\upharpoonright
  M}$.
  \item[(iii)] Assume, in addition, that the $\ff$-subsequence
  $(x_s)_{s\in\ff\upharpoonright M}$ is
  subordinated with respect to the weak topology and let
  $\widehat{\varphi}:\widehat{\ff}\upharpoonright M\to (X,w)$ be
  the continuous map witnessing this. Then the $\ff$-subsequences
  $\big(x^{(1,\g)}_s\big)_{s\in\ff\upharpoonright M}$ and
  $\big(x^{(2,\g)}_s\big)_{s\in\ff\upharpoonright M}$ are also
  subordinated with respect to the weak topology. More precisely,
  the continuous maps $\widehat{\varphi}^{(1,\g)},\widehat{\varphi}^{(2,\g)}:\widehat{\ff}\upharpoonright M\to
  (X,w)$ witnessing this fact
  are defined as follows. For every $t\in\widehat{\ff}\upharpoonright
  M$,
   \[\widehat{\varphi}^{(1,\g)}(t)=\sum_{t'\sqsubseteq t}y_{t'}^{(1,\g)}\;\;\text{and}\;\;\widehat{\varphi}^{(2,\g)}(t)=\sum_{t'\sqsubseteq t}y_{t'}^{(2,\g)}\]
   It is immediate that
   $\widehat{\varphi}(t)=\widehat{\varphi}^{(1,\g)}(t)+\widehat{\varphi}^{(2,\g)}(t)$.
    Moreover
  notice that
  \[\widehat{\varphi}^{(1,\g)}(\emptyset)=\widehat{\varphi}(\emptyset)\;\;\text{and}\;\;\widehat{\varphi}^{(2,\g)}(\emptyset)=0\]
   In the sequel we will refer to $\widehat{\varphi}^{(1,\g)}$ and
   $\widehat{\varphi}^{(2,\g)}$ as the $\g$-\textit{splitting of}
   $\widehat{\varphi}$.
  \end{enumerate}
\end{rem}

\begin{lem}\label{up g splitting spr mod}
  Let $X$ be a Banach space with a Schauder basis. Let $\ff$ be a regular thin family, $M\in[\nn]^\infty$ and $\g\subseteq[M]^{<\infty}$
  such that  $\g\sqsubset\ff\upharpoonright M$. Suppose that the family $\g(M^{-1})=\{s\in[\nn]^{<\infty}:\;M(s)\in\g\}$ is regular thin.
  Let $(x_s)_{s\in\ff}$ be an $\ff$-sequence in $X$
   such that $(x_s)_{s\in\ff\upharpoonright M}$ admits a canonical tree decomposition.
    Then every spreading
  model of $(x_s^{(1,\g)})_{s\in\ff\upharpoonright
  M}$ belongs to $\mathcal{SM}_{o(\g)}(X)$.
\end{lem}
\begin{proof}
  Let
  $L\in[M]^\infty$ such that $(x_s^{(1,\g)})_{s\in\ff\upharpoonright
  L}$ generates an $\ff$-spreading model $(e_n)_{n\in\nn}$. We set
  $\mathcal{H}=\g(L^{-1})=\{s\in[\nn]^{<\infty}:\;L(s)\in\g\}$. By the remarks after Definition \ref{defn backwards shifting} we have that
  $\mathcal{H}$ is a regular thin family with $o(\mathcal{H})=o(\g\upharpoonright
  L)=o(\g)$. For every $t\in\mathcal{H}$ we set
  $w_t=x_s^{(1,\g)}$, where $s\in\ff\upharpoonright M$ and $L(t)\sqsubset
  s$. By the definition of $(x_s^{(1,\g)})_{s\in\ff\upharpoonright
  M}$ we have that $w_t$ is well defined. It is easy to see that
  $(w_t)_{t\in\mathcal{H}}$ generates $(e_n)_{n\in\nn}$ as an
  $\mathcal{H}$-spreading model.
\end{proof}
By the above and Corollary \ref{ultrafilter property for ell^1
spreading models} we have the following.
\begin{cor}\label{splitting to continue in example}
  Let $X$ be a Banach space with a Schauder basis. Suppose that
  $\mathcal{SM}^{wrc}(X)$ contains a sequence equivalent to the
  usual basis of $\ell^1$ and let $\xi_0$ be the minimum of all countable
  ordinals $\xi$ such that $\mathcal{SM}^{wrc}_\xi(X)$ contains such a sequence.

   Let $\ff$ be regular thin of order $\xi_0$, $M\in[\nn]^\infty$ and
  $(x_s)_{s\in\ff}$ a weakly relatively
 compact $\ff$-sequence in $X$ such that the $\ff$-subsequence
  $(x_s)_{s\in\ff\upharpoonright M}$ generates  $\ell^1$ as an $\ff$-spreading model with lower constant $c$
   and admits a canonical tree decomposition.

   Let $\g\subseteq[M]^{<\infty}$
  such that  $\g\sqsubset\ff\upharpoonright M$
  and  $\g(M^{-1})=\{s\in[\nn]^{<\infty}:\;M(s)\in\g\}$ is regular
  thin with  $o(\g)<\xi_0$.

   Then for every $L\in[M]^\infty$ there exists
$N\in[L]^\infty$ such that $(x_s^{(2,\g)})_{s\in
\ff\upharpoonright N}$ generates
  $\ell^1$ as an $\ff$-spreading model with the same lower constant $c$.
\end{cor}
\subsection{Plegma cuts of regular thin families}
Let $\ff$ be a regular thin family. For every plegma pair
$(s_1,s_2)$ in $\ff$ the \emph{plegma cut of $s_2$ with respect
to} $s_1$, denoted by $ s_2/_{s_1}$, is defined to be the  set
\[s_2/_{s_1}=s_2\cap\{1,\ldots,\max
s_1\}=\{s_2(1),\ldots,s_2(|s_1|-1)\}\] It is clear that
$s_2/_{s_1}$ is an initial segment of $s_2$ with length depending
on the choice of $s_1$. Moreover, plegma cuts are defined for all
$s\in\ff$ with gaps between its elements and in general they are
not uniquely determined. In the sequel, we will mainly use plegma
cuts with minimum or maximal length.

\begin{defn}\label{notation hunging ff by min L}
  Let $\ff$ be regular thin and $L=\{l_n:n\in\nn\}\in[\nn]^\infty$ such that $\ff$ is very large in $L$. We define
  \[\begin{split}\ff/_L=\Big\{s/_{s'}:&s\in\ff\upharpoonright L(2\nn)\;\text{and if}\;s=\{L(2k_1),\ldots,L(2k_m)\}\;\text{then}\;s'\in\ff\\
  &\text{and}\;s'\sqsubset\{L(1),L(2k_1+1),\ldots,L(2k_m+1)\} \Big\}
  \end{split}\]
\end{defn}
 Notice that $\ff/_L$ is well defined and
 consists of all plegma cuts of minimal length and of the form $s/_{s'}$, with
 $s\in\ff\upharpoonright
 L(2\nn)$ and $s'\in\ff\upharpoonright L$. Moreover, it easy to
 see that
 \[\begin{split}\ff/_L=\Big\{\{L(k_1),&\ldots,L(k_m)\}: m\in\nn\;\;\text{and}\\
 & \{L(k_1+1),\ldots,L(k_m+1)\}\in\ff_{(L(1))}\upharpoonright L(2\nn-1)  \Big\}\end{split}\]
   This representation of $\ff/_L$ yields the
  following.
\begin{prop}\label{rem f/L}
  Let $\ff$ be regular thin and $L\in[\nn]^\infty$ such that $\ff$ is very large in $L$. Then the following
  hold.
  \begin{enumerate}
    \item[(i)] The family $\ff/_L$ is thin and $\ff/_L\sqsubset \ff\upharpoonright L(2\nn)$.
    \item[(ii)] The family  $\ff/_L(L(2\nn)^{-1})=\{s\in[\nn]^{<\infty}:L(2s)\in\ff/_L\}$ is regular thin.
    \item[(iii)] $o(\ff/_L)=o(\ff_{(l_1)})<o(\ff)$, where $l_1=\min L$.
  \end{enumerate}
\end{prop}
The next lemma uses plegma cuts of maximal length to produce
plegma block $\ff$-subsequences from $\ff$-subsequences admitting
a canonical tree decomposition.
  \begin{lem}\label{block plegma}
    Let $X$ be a Banach space with a Schauder basis, $\ff$ a regular thin family
    and $M\in[\nn]^\infty$ such that $\ff$ is very large in
    $M$. Let $(x_s)_{s\in\ff}$
     an $\ff$-sequence in $X$, such that
     $(x_t)_{t\in\widehat{\ff} \upharpoonright M}$
     admits a canonical tree decomposition
     $(y_t)_{t\in\widehat{\ff} \upharpoonright M}$. For every
     $L\in[M]^\infty$ we define the $\ff$-subsequence $(z^L_s)_{s\in\ff\upharpoonright
     L(2\nn)}$ as follows.
     For every
  $s\in\ff\upharpoonright
  L(2\nn)$, with $s=\{L(2\rho_1)<\ldots<L(2\rho_k)\}$, we set $s'$ to be the unique initial segment of
  $\{L(2\rho_1-1)<\ldots<L(2\rho_k-1)\}$ in
  $\ff$. Let $s^*=s/_{s'}$ and for every $s\in\ff\upharpoonright
  L(2\nn)$, let \[z^L_s=x_s-\sum_{t\sqsubseteq
  s^*}y_t=\sum_{s^*\sqsubset t\sqsubseteq s}y_t\]
  Then the following are satisfied.
\begin{enumerate}
  \item[(a)] For every $L\in[M]^\infty$ the $\ff$-subsequence $(z^L_s)_{s\in\ff\upharpoonright
     L(2\nn)}$ is plegma block.
  \item[(b)] For every $\delta>0$ there exists $L\in[M]^\infty$
  such that either
  \begin{enumerate}
    \item[(i)] for every $s\in\ff\upharpoonright L(2\nn)$,
    $\|x_s^{(1,\ff/_L)}\|\geq\delta$, or
    \item[(ii)] for every $s\in\ff\upharpoonright L(2\nn)$,
    $\|x_s-z^L_s\|<\delta$.
  \end{enumerate}
\end{enumerate}
  \end{lem}
  \begin{proof}
    (a) Let $L\in[M]^\infty$ and $(s_1,s_2)$ be a plegma pair in $\ff\upharpoonright
    L(2\nn)$. Then it is easy to see that
    $(s_1',s_1,s_2',s_2)$ is a plegma 4-tuple and therefore
    $|s_1|\leq|s_2'|$. Moreover notice that
    \[\max(s_1\setminus
    s_1/_{s'_1})=\max(s_1)=s_1(|s_1|)\;\text{and}\;\min(s_2\setminus
    s_2/_{s'_2})=s_2(|s'_2|)\]
    By
    the definition of the canonical tree decomposition we have that
    $z_{s_1}<z_{s_2}$.

  (b) Let $\delta>0$.
  By Theorem \ref{ramseyforplegma} there exists
  $L\in[M]^\infty$ such that either
  \begin{enumerate}
\item[(1)] for every plegma pair $(s_1,s_2)$ in
    $\ff\upharpoonright L$, \[\Big\|\sum_{t\sqsubseteq
    s_2/_{s_1}}y_t\Big\|\geq\delta\] or
    \item[(2)] for every plegma pair $(s_1,s_2)$ in
    $\ff\upharpoonright L$, \[\Big\|\sum_{t\sqsubseteq s_2/_{s_1}}y_t\Big\|<\delta\]
  \end{enumerate}
  Assume that (1) holds. Since
  $\ff/_L\subseteq\{s/_{s'}:(s',s)\in\text{\emph{Plm}}(\ff\upharpoonright
  L)\}$, by the definition of
  $\big(x_s^{(1,\ff/_L)}\big)_{s\in\ff\upharpoonright L(2\nn)}$, we
  get that assertion (i) of the lemma holds. Similarly it is shown that (2) implies assertion (ii) of the
  lemma.
  \end{proof}
  \subsection{Banach spaces admitting $\ell^1$ as a unique spreading
  model}
  \begin{thm}\label{unique l^1 spr mod}
    Let $X$ be a Banach space with an unconditional basis. Assume that every  plegma block
    generated spreading
    model of any order of $X$  is equivalent to the usual basis of
    $\ell^1$. Then every nontrivial spreading model in
    $\mathcal{SM}^{wrc}(X)$ is equivalent to the usual basis of
    $\ell^1$.
  \end{thm}
\begin{proof}
  We proceed by induction on $1\leq\xi<\omega_1$. For $\xi=1$
  we have the following. Let
  $(e_n)_{n\in\nn}\in\mathcal{SM}_1^{wrc}(X)$ be nontrivial. By standard arguments or using Corollary \ref{assuming canonical
  decomposition} there exist a block sequence $(y_n)_{\in\nn}$ and
  $y_\emptyset\in X$ such that the sequence $(x_n)_{n\in\nn}$,
  defined by $x_n=y_n+y_\emptyset$ for all $n\in\nn$, generates
  $(e_n)_{n\in\nn}$ as a spreading model. Let $L\in[\nn]^\infty$
  such that the subsequence $(y_n)_{n\in L}$ generates a spreading
  model $(e'_n)_{n\in\nn}$. By Corollary \ref{trivial ultrafilter weak
  property} we get that $(e'_n)_{n\in\nn}$ is nontrivial. By our
  assumptions we have that $(e'_n)_{n\in\nn}$ is equivalent to the
  usual basis of $\ell^1$. By Corollary \ref{ultrafilter property for ell^1 spr mod
  simplified} we have that $(e_n)_{n\in\nn}$ is also equivalent to
  the usual basis of $\ell^1$.

    Suppose that for some $\xi<\omega_1$ we have that for every
  $\zeta<\xi$, every nontrivial spreading model in $\mathcal{SM}^{wrc}_\zeta(X)$
  is equivalent to the usual basis of $\ell^1$. Let
  $(e_n)_{n\in\nn}$ be a nontrivial spreading model in $\mathcal{SM}^{wrc}_\xi(X)$.
  Let also $\ff$ be a regular thin family of order $\xi$. By
  Corollary \ref{assuming canonical decomposition} there
  exist $M\in[\nn]^\infty$ and an $\ff$-subsequence $(x_s)_{s\in\ff\upharpoonright
  M}$ in $X$ having the following properties.
  \begin{enumerate}
    \item[(i)] It  generates
  $(e_n)_{n\in\nn}$ as an $\ff$-spreading model.
  \item[(ii)] It is subordinated and let $\widehat{\varphi}:\widehat{\ff}\upharpoonright M\to$
  be the continuous map witnessing this.
     \item[(iii)] It  admits a canonical tree decomposition $(y_t)_{t\in\widehat{\ff}\upharpoonright M}$.
  \end{enumerate}
  By Corollaries \ref{ultrafilter property for ell^1 spr mod
  simplified} and \ref{trivial ultrafilter weak
  property} we may assume that $y_\emptyset=0$.
  Let $c>0$ such that for every $s\in\ff\upharpoonright L$,
  $\|x_s\|>c$.
  By Lemma \ref{block plegma} there exists
  $L\in[M]^\infty$ such that either
\begin{enumerate}
    \item[(i)] for every $s\in\ff\upharpoonright L(2\nn)$,
    $\|x_s^{(1,\ff/_L)}\|\geq\frac{c}{2}$, or
    \item[(ii)] for every $s\in\ff\upharpoonright L(2\nn)$,
    $\|x_s-z^L_s\|<\frac{c}{2}$.
  \end{enumerate}

Suppose that the first alternative holds. Then
$(x_s^{(1,\ff/_L)})_{s\in\ff\upharpoonright L(2\nn)}$ is
seminormalized. Moreover
$(x_s^{(1,\ff/_L)})_{s\in\ff\upharpoonright L(2\nn)}$ is weakly
null, since it is weakly convergent to
$\widehat{\varphi}^{(1,\g)}(\emptyset)=\widehat{\varphi}(\emptyset)=y_\emptyset=0$.
Therefore $(x_s^{(1,\ff/_L)})_{s\in\ff\upharpoonright L(2\nn)}$
does not contain any norm convergent $\ff$-subsequence. By Theorem
\ref{Theorem equivalent forms for having norm on the spreading
model} we have that $(x_s^{(1,\ff/_L)})_{s\in\ff\upharpoonright
L(2\nn)}$ admits an nontrivial spreading model $(e'_n)_{n\in\nn}$.
Lemma \ref{up g splitting spr mod} and our inductive hypothesis
yields that $(e'_n)_{n\in\nn}$ is equivalent to the usual basis of
$\ell^1$. The unconditionality of the basis of $X$ yields that
$(e_n)_{n\in\nn}$ is equivalent to the usual basis of $\ell^1$.

Suppose that the second alternative holds. Then
$(z^L_s)_{s\in\ff\upharpoonright L(2\nn)}$ is a seminormalized
plegma block $\ff$-subsequence. By our assumption we have that
every $\ff$-spreading model of $(z^L_s)_{s\in\ff\upharpoonright
L(2\nn)}$ is equivalent to the usual basis of $\ell^1$. Again the
unconditionality of the basis of $X$ yields that $(e_n)_{n\in\nn}$
is equivalent to the usual basis of $\ell^1$.
\end{proof}

\subsection{Banach spaces admitting $\ell^1$ as a plegma block generated spreading model}
In this subsection we provide sufficient conditions for a Banach
space $X$ with a Schauder basis which ensures that the $\ell^1$
spreading models of $X$ also appear as plegma block generated.
\begin{defn}\label{notation of P property}
  Let $X$ be a Banach space with a Schauder basis. We say
  that $X$ satisfies the property $\mathcal{P}$, if for every
  $\delta>0$ there exists $k\in\nn$ such that for every finite block sequence $(x_j)_{j=1}^k$, with
  $\|x_j\|\geq\delta$ for all $j=1,\ldots,k$, we have that $\|\sum_{j=1}^k x_j\|>1$.
\end{defn}
\begin{rem}\label{remark p property for block generating}
  $\;$
  \begin{enumerate}
    \item[(i)] It is easy to see that the property $\mathcal{P}$
    is preserved under equivalent renormings.
    \item[(ii)] If $X$ is a Banach space with an unconditional basis
    $(e_n)_{n\in\nn}$ then $X$ satisfies the property
    $\mathcal{P}$ iff $c_0$ is not finitely block representable in $X$.
  \end{enumerate}
\end{rem}
\begin{thm}\label{getting block generated ell^1 spreading model}
  Let $X$ be a Banach space with a Schauder basis  satisfying property $\mathcal{P}$.
If for some $\xi<\omega_1$, $\ell^1$ is a $\xi$-spreading model of
a weakly relatively compact
  subset of $X$, then $\ell^1$ is plegma block generated as a $\xi$-spreading model.
\end{thm}
For the proof of the above theorem we will need the following
lemma.
\begin{lem}\label{lemma getting block generated ell^1 spreading model finishing}
  Let $X$ be a Banach space with a Schauder
  basis. Let $\ff$ be regular thin, $M\in[\nn]^\infty$ and $0<c\leq 1$.
  Let $(x_s)_{s\in\ff}$ be an $\ff$-sequence in $X$, such that $(x_s)_{s\in\ff\upharpoonright M}$ generates $\ell^1$ as an $\ff$-spreading model with constant
  $c$ and also admits
  a canonical tree decomposition.
  If $X$ does not admit $\ell^1$ as a plegma block
  generated spreading model then for every $0<\delta<c$ there
  exists $L\in[M]^\infty$ such that $\|x_s^{(1,\ff/_L)}\|\geq\delta$,
  for all $s\in\ff\upharpoonright L(2\nn)$.
\end{lem}
\begin{proof}
  We may suppose that $\ff$ is very large in $M$. Let
  $0<\delta<c$. By Lemma \ref{block plegma} there exists $L\in[M]^\infty$
  such that either
  \begin{enumerate}
    \item[(i)] for every $s\in\ff\upharpoonright L(2\nn)$,
    $\|x_s^{(1,\ff/_L)}\|\geq\delta$, or
    \item[(ii)] for every $s\in\ff\upharpoonright L(2\nn)$,
    $\|x_s-z^L_s\|<\delta$.
  \end{enumerate}
  We have to show that alternative (ii) is impossible. Indeed, otherwise we have that every
  $\ff$-spreading model of $(z^L_s)_{s\in\ff\upharpoonright
  L(2\nn)}$ is equivalent to the usual basis of $\ell_1$ with lower constant $c-\delta$.
 This yields that $\ell_1$ is plegma block generated by  the
$\ff$-subsequence  $(z^L_s)_{s\in\ff\upharpoonright
  L(2\nn)}$ as an
  $\ff$-spreading model, which is a contradiction.
\end{proof}
\begin{proof}[Proof of Theorem \ref{getting block generated ell^1 spreading model}]
  Assume that $X$ does not admits $\ell^1$ as a plegma block generated spreading model. By Remark \ref{remark on the definition of spreading model} and
  Remark \ref{remark p property for block generating} we may suppose
  that the basis of $X$ is bimonotone. Let $\xi_0<\omega_1$ be the minimum countable ordinal $\xi$ such
  that there exists a weakly relatively compact subset $A$ of $X$
  which admits $\ell^1$ as a $\xi$-spreading model.

 Let $\ff$ be regular thin of order $\xi_0$. By Corollaries \ref{ultrafilter property for ell^1 spr mod simplified} and \ref{assuming canonical
  decomposition} there
  exist $M_1\in[\nn]^\infty$ and an $\ff$-subsequence $(x_s)_{s\in\ff\upharpoonright
  M_1}$ in $X$ having the following properties.
  \begin{enumerate}\item[(i)] It  generates
  an $\ell^1$ $\ff$-spreading model with lower constant $c$.
  \item[(ii)] It is
   subordinated with respect to the weak topology and if
   $\widehat{\varphi}:\widehat{\ff}\upharpoonright M_1\to X$ is
   the continuous map witnessing this then
   $\widehat{\varphi}(\emptyset)=0$.
   \item[(iii)] It  admits a canonical tree decomposition.
   \end{enumerate}

  Let $0<\delta<c$ and let $k\in\nn$ such that for every finite block sequence $(x_j)_{j=1}^k$, with $\|x_j\|\geq\delta$ for all $j=1,\ldots,k$,
  $\|\sum_{j=1}^k x_j\|>1$.

  By Lemma \ref{lemma getting block generated ell^1 spreading model
  finishing} there
  exists $L_1\in[M_1]^\infty$ such that $\|x_s^{(1,\ff/_{L_1})}\|\geq\delta$,
  for all $s\in\ff\upharpoonright L_1(2\nn)$.
  Since $o(\ff/_{L_1})<\xi_0$,
  by Corollary \ref{splitting to continue in example} there exists $M_2\in[L_1(2\nn)]^\infty$ such that
  $\big(x_s^{(2,\ff/_{L_1})}\big)_{s\in \ff\upharpoonright M_2}$ generates $\ell^1$ as an $\ff$-spreading model with the same lower
   constant $c$.

  We set $x^2_s=x_s^{(2,\ff/_{L_1})}$, for all
   $s\in\ff\upharpoonright M_2$. Following the same arguments as
   above there exist $L_2\in[M_2]^\infty$ and
   $M_3\in[L_2(2\nn)]^\infty$ such that $\|x_s^{2(1,\ff/_{L_2})}\|\geq\delta$,
  for all $s\in\ff\upharpoonright L_2(2\nn)$ and $\big(x_s^{(2,\ff/_{L_2})}\big)_{s\in \ff\upharpoonright M_3}$ generates $\ell^1$ as an $\ff$-spreading model with the same lower
   constant $c$.

  Notice that for every $s\in \ff\upharpoonright M_3$, we have
  that
  \[x_s^{(1,\g_1)}<x_s^{2\;(1,\g_2)},\;
  \|x_s^{(1,\g_1)}\|\geq\delta\;\text{and}\;
  \|x_s^{2\;(1,\g_2)}\|\geq\delta\]
  Moreover since for every $s\in \ff\upharpoonright M_3$
  $\|x_s\|\leq1$, $\widehat{\varphi}(\emptyset)=0$ and the basis
  is bimonotone, we have that
  $\|\widetilde{x}_s^{(1,\g_1)}+\widetilde{x}_s^{2\;(1,\g_2)}\|\leq1$,
  for all $s\in \ff\upharpoonright M_3$.
  Continuing in the same way it is clear that after $k$ steps the property $\mathcal{P}$
  will lead us to a contradiction.
\end{proof}
\begin{cor}\label{getting block generated ell^1 spreading model corollary for reflexive}
  Let $X$ be a reflexive Banach space with a Schauder basis  satisfying property $\mathcal{P}$.
If for some $\xi<\omega_1$, $\ell^1$ is a $\xi$-spreading model of
$X$ then $\ell^1$ is a plegma block generated $\xi$-spreading
model of $X$.
\end{cor}
\begin{rem} As we show in Chapter \ref{space without block}, the assumption that the
space $X$ satisfies property $\mathcal{P}$ is necessary in Theorem
\ref{getting block generated ell^1 spreading model}.
\end{rem}
\subsection{The spreading models of the Tsirelson spaces}\label{subsection Tsirelson} The definition of the
Tsirelson space $T_\alpha$, $1\leq\alpha<\omega_1$ makes use of
the Schreier families $(\mathcal{S}_\alpha)_{\alpha<\omega_1}$
(see Definition \ref{Schreier families defn}). For the sake of
completeness we will prove the following property concerning
$(\mathcal{S}_\alpha)_{\alpha<\omega_1}$, which will be used in
the sequel.
\begin{lem}\label{S_1 is the smaller}
  For every $1\leq\alpha\leq\omega_1$, we have that $\mathcal{S}_1\subseteq\mathcal{S}_\alpha$.
\end{lem}
\begin{proof}
  We proceed by induction on $\alpha$. For $\alpha=1$ the result trivially holds. Suppose that for some $\alpha<\omega_1$, $\mathcal{S}_1\subseteq\mathcal{S}_{\alpha'}$, for all $1\leq\alpha'<\alpha$. If $\alpha$ is a successor, $\alpha=\beta+1$, then by our inductive hypothesis and the definition of $\mathcal{S}_\alpha$, we have that $\mathcal{S}_1\subseteq\mathcal{S}_\beta\subseteq\mathcal{S}_\alpha$. In the case of a limit countable ordinal $\alpha$, let $(\alpha_n)_{n\in\nn}$ be the increasing to $\alpha$ sequence of countable ordinals that takes part in the definition of $\mathcal{S}_\alpha$. Notice that $\mathcal{S}_{\alpha_1}\subseteq\mathcal{S}_\alpha$. Therefore by our inductive hypothesis we have that $\mathcal{S}_1\subseteq\mathcal{S}_{\alpha_1}\subseteq\mathcal{S}_\alpha$
\end{proof}

We ready to state the definition of the Tsirelson spaces
$T_\alpha$. Let $\alpha\geq1$ be a countable ordinal. Let
$W_\alpha$ be the minimal subset of the algebraic dual
$c_{00}(\nn)^\#$ of $c_{00}(\nn)$ satisfying the following
properties:
\begin{enumerate}
  \item[(i)] $\pm e_n^*\in W_\alpha$, for all $n\in\nn$.
  \item[(ii)] For every $n\in\nn$ and $f_1,\ldots,f_n\in W_\alpha$ such that $f_1<\ldots<f_n$ and $\{\min\text{supp}(f_i):1\leq i\leq n\}\in\mathcal{S}_\alpha$, we have that \[\frac{1}{2}\sum_{i=1}^nf_i\in W_\alpha\]
\end{enumerate}
It is easy to check that every element of $W_a$ is of finite
support. Moreover for every $f=\sum_{i=1}^na_ie_i^*\in W_\alpha$
and $F\subseteq\{1,\ldots,n\}$ we have that $\sum_{i\in F}
a_ie_i^*\in W_\alpha$. Also $W_\alpha$ is symmetric, i.e. for
every $f\in W_\alpha$ we have that $-f\in W_\alpha$.

Let $\|\cdot\|_{T_\alpha}:c_{00}(\nn)\to\rr$ be the norm defined
as follows:
\[\|x\|_{T_\alpha}=\sup\{f(x):f\in W_\alpha\}\]
The space $T_\alpha$ is defined to be the completion of
$c_{00}(\nn)$ under the norm $\|\cdot\|_{T_\alpha}$.

It is easy to see that the usual Hamel basis $(e_n)_{n\in\nn}$ of
$c_{00}(\nn)$ forms an unconditional basis for the space
$T_\alpha$. Moreover it is well known that the spaces $T_\alpha$,
for all $1\leq\alpha<\omega_1$, are reflexive.  Concerning the
spreading models of these spaces we have the following result.
\begin{thm}\label{Tsirelson's spr mods thm}
  Let $1\leq\alpha<\omega_1$. Then every nontrivial spreading model of any order of the space $T_\alpha$ is equivalent to the usual basis of $\ell^1$.
\end{thm}
By the reflexivity of the space $T_\alpha$ and Theorem \ref{unique
l^1 spr mod} it suffices to show the following proposition.
\begin{prop}\label{Tsirelson plegma block spr mod}
  Let $1\leq\alpha<\omega_1$, $\ff$ be a regular thin family, $M\in[\nn]^\infty$ and $(x_s)_{s\in\ff\upharpoonright M}$
   a seminormalized plegma block $\ff$-subsequence in $T_\alpha$.
  Then every $\ff$-spreading model of $(x_s)_{s\in\ff\upharpoonright M}$ is equivalent to the usual basis of $\ell^1$.
\end{prop}
For the proof of Proposition \ref{Tsirelson plegma block spr mod}
we will need the following lemma.
\begin{lem}\label{Tsirleson pushing up the supports}
  Let $X$ be a Banach space with a Schauder basis. Let $\ff$ be a regular thin family, $M\in[\nn]^\infty$ and
   $(x_s)_{s\in\ff\upharpoonright M}$ a plegma block $\ff$-subsequence in $X$
   of nonzero elements. Then there exists $L\in[M]^\infty$ such that for
   every $s\in \ff\upharpoonright L$, if $\min s=L(n)$ then $\min \text{supp}(x_s)\geq n$.
\end{lem}
\begin{proof}
  Let $M'\in[M]^\infty$ such that $\ff$ is very large in $M'$ and $L=M'(2\nn)$.
  Then $L$ is the desired set. Indeed, let $s\in\ff\upharpoonright L$.
   Let $n_1,\ldots,n_{|s|}\in\nn$ such that $s(j)=L(n_j)$, for all $1\leq j\leq |s|$.
    Then $s(j)=M'(2n_j)$, for all $1\leq j\leq |s|$. Let, for every $1\leq p\leq 2n_1$, $s_p\in\ff$ such that
  \[s_p\sqsubseteq \big\{M'(2n_j+p-2n_1):1\leq j\leq|s|\big\}\]
  Notice that $(s_p)_{p=1}^{2n_1}$ is a plegma path and $s=s_{2n_1}$.
   Therefore \[1\leq \min\text{supp}(x_{s_1})<\ldots<\min\text{supp}(x_{s_{2n_1}})=\min\text{supp}(x_s)\]
  This easily yields that $\min\text{supp}(x_s)\geq 2n_1\geq n_1$.
\end{proof}
\begin{proof}[Proof of Proposition \ref{Tsirelson plegma block spr mod}]
  Let $c>0$ be such that $\|x_s\|_{T_\alpha}>c$, for all $s\in\ff\upharpoonright M$.
   We pass to a $M'\in[M]^\infty$ such that $(x_s)_{s\in\ff\upharpoonright M'}$
    generates an $\ff$-spreading model $(e'_n)_{n\in\nn}$.
    By Lemma \ref{Tsirleson pushing up the supports} there exists $L\in[M']^\infty$ such that
     for every $s\in \ff\upharpoonright L$, if $\min s=L(n)$ then $\min \text{supp}(x_s)\geq n$.
      To complete the proof, it suffices to show that for every $n\in\nn$,
       $a_1,\ldots,a_n\in\rr$ and $(s_j)_{j=1}^n\in\text{\emph{Plm}}(\ff\upharpoonright L)$ with $s_1(1)\geq L(n)$,
        we have that \[\Big\|\sum_{j=1}^na_jx_{s_j}\Big\|_{T_\alpha}\geq \frac{c}{2}\sum_{j=1}^n|a_j|\]
  Indeed, let $n\in\nn$, $a_1,\ldots,a_n\in\rr$ and $(s_j)_{j=1}^n\in\text{\emph{Plm}}(\ff\upharpoonright L)$ with
   $s_1(1)\geq L(n)$. Since $\|x_{s_i}\|_{T_\alpha}>c$, for all $1\leq i\leq n$, there exist $f_1,\ldots,f_n\in W_\alpha$
    such that $f_i(x_{s_i})>c$. By the properties of $W_\alpha$,
     we may suppose that $\text{supp}(f_i)\subseteq\text{supp}(x_{s_i})$, for all $1\leq i\leq n$.
     Therefore $\min\text{supp}(f_i)\geq\min\text{supp}(x_{s_i})$, for all $1\leq i\leq n$.
  By the choice of $(x_s)_{s\in\ff \upharpoonright L}$ and Lemma \ref{S_1 is the smaller},
  it is easy to check that $\{\min\text{supp}(x_{s_j}):1\leq j\leq n)\}\in\mathcal{S}_1\subseteq\mathcal{S}_\alpha$.
   Therefore by the spreading property of $\mathcal{S}_\alpha$, we have that $\{\min\text{supp}(f_j):1\leq j\leq n)\}\in\mathcal{S}_\alpha$.
    Hence, if $\ee_i=\text{sign}(a_i)$, for all $1\leq i\leq n$, we have that $\varphi=\frac{1}{2}\sum_{i=1}^n \ee_if_i\in W_\alpha$. Thus
  \[\Big\|\sum_{j=1}^na_jx_{s_j}\Big\|_{T_\alpha}\geq \varphi\Big(\sum_{j=1}^na_jx_{s_j}\Big)\geq\frac{c}{2}\sum_{j=1}^n|a_j|\]
\end{proof}

\chapter{$c_0$ spreading models}
In this chapter we study $c_0$ spreading models generated by
$\ff$-sequences. In the first section we present a combinatorial
result concerning partial unconditionality in infinitely branching
trees in Banach spaces. Based on this result  and using the
splitting technique of canonical tree decompositions, presented in
the previous chapter, we establish a corresponding to $\ell^1$
result for plegma block generated $c_0$ spreading models. Finally
in the last section we deal with the duality between  $c_0$ and
$\ell^1$ spreading models.
\section{On partial unconditionality of trees in Banach spaces}
In this section we present a Ramsey result concerning partial
unconditionality in trees in Banach spaces.  Our approach is
related to the corresponding one stated for sequences instead of
trees and which is followed by  several papers (see \cite{AG},
\cite{AGR}, \cite{E}, \cite{Od}).

 We start with some definitions.
\begin{defn}\label{Delta partition}
  Let $k\in\nn$ and $\Delta=(N^{(1)},\ldots,N^{(k)})$ be a
  partition of $\nn$ into $k$ infinite disjoint sets, i.e.
  $N^{(i)}\in[\nn]^\infty$ for all $1\leq i\leq k$, $N^{(i)}\cap
  N^{(j)}=\emptyset$ for all $i\neq j$ in $\{1,\ldots,k\}$ and
  $\nn=\cup_{i=1}^k N^{(i)}$. For every $L\in[\nn]^\infty$ we
  define the modulo $\Delta$ partition of $L$ as the $k$-tuple
  $(L_{(\Delta,1)},\ldots,L_{(\Delta,k)})$, where
  $L_{(\Delta,i)}=L(N^{(i)})$ for all $1\leq i\leq k$. Moreover for
  every nonempty $F\in[\nn]^{<\infty}$ we define the modulo $\Delta$ partition of $F$ as the $k$-tuple
  $(F_{(\Delta,1)},\ldots,F_{(\Delta,k)})$, where
  $F_{(\Delta,i)}=F(\{1,\ldots,|F|\}\cap N^{(i)})$ for all $1\leq i\leq
  k$, and for $F=\emptyset$, $\emptyset_{\Delta,i}=\emptyset$, for all
  $1\leq i\leq k$.

  Finally we define the map $i_\Delta:\nn\to\{1,\ldots,k\}$ such
  that $i_\Delta(n)=i$ if $n\in N^{(i)}$, for all $n\in\nn$.
\end{defn}
\begin{rem}\label{null trees rem with properties}
  It is immediate that the following are satisfied:
  \begin{enumerate}
    \item[($\Delta1$)] For every $L\in[\nn]^\infty$ we have that $L_{(\Delta,i)}=\{L(n):i_\Delta(n)=i\}$, for all $1\leq i\leq k$.
     Moreover $\cup_{i=1}^k L_{(\Delta,i)}=L$ and $L_{(\Delta,i)}\cap
        L_{(\Delta,j)}=\emptyset$, for all $i\neq j$.
    \item[($\Delta2$)] For every $F\in[\nn]^{<\infty}$ and
        every $L\in[\nn]^\infty$ with $F\sqsubseteq L$ we have that
        \[F_{(\Delta,i)}=L_{(\Delta,i)}\cap F\sqsubseteq L_{(\Delta,i)}\] for all $1\leq i\leq
        k$. Therefore for every $F\in[\nn]^{<\infty}$ and
        every $L,L'\in[\nn]^\infty$ with $F\sqsubseteq L$ and
        $F\sqsubseteq L'$, we have that \[L_{(\Delta,i)}\cap F=L'_{(\Delta,i)}\cap
        F=F_{(\Delta,i)}\] for all $1\leq i\leq k$.
    \item[($\Delta3$)] Let $L,L'\in[\nn]^\infty$ and $n_0\in\nn$ such
    that for every $n\neq n_0$ in $\nn$, $L(n)=L'(n)$. Then
    \begin{enumerate}
      \item[(a)] For every $i\neq i_\Delta(n_0)$ in
      $\{1,\ldots,k\}$, $L_{(\Delta,i)}=L'_{(\Delta,i)}$,
      \item[(b)] $L(n_0)\in L_{(\Delta,i_\Delta(n_0))}$ and $L'(n_0)\in
      L'_{(\Delta,i_\Delta(n_0))}$,
      \item[(c)]
      $L_{(\Delta,i_\Delta(n_0))}\setminus\{L(n_0)\}=L'_{(\Delta,i_\Delta(n_0))}\setminus\{L'(n_0)\}$.
    \end{enumerate}
  \end{enumerate}
\end{rem}
\begin{defn}
  Let $(y_t)_{t\in[\nn]^{<\infty}}$ be a family of vectors in a
  Banach space $X$. We will say that $(y_t)_{t\in[\nn]^{<\infty}}$ is a weakly null
  tree for every $t\in[\nn]^{<\infty}$, setting
  $N_t=\{n\in\nn:n>t\}$, we have that $(y_{t\cup\{n\}})_{n\in
  N_t}\stackrel{w}{\longrightarrow}0$. A weakly null tree $(y_t)_{t\in[\nn]^{<\infty}}$
  will be called bounded if the family
  $(x_s)_{s\in[\nn]^{<\infty}}$ is bounded, where $x_s=\sum_{t\sqsubseteq
  s}y_t$ for all $s\in[\nn]^{<\infty}$.
\end{defn}
\begin{notation}
  Let $N\in[\nn]^\infty$ and $\g$ a thin family very large in $N$.
  For every $L\in[
  N]^\infty$ we set $I_\g(L)$ to be the unique initial segment
  of $L$ in $\g$.
\end{notation}
\begin{defn}\label{null trees definition}
  Let $k\in\nn$, $\Delta=(N^{(1)},\ldots,N^{(k)})$ be a
  partition of $\nn$ into $k$ infinite disjoint sets,
  $N\in[\nn]^\infty$, $\mathfrak{G}=(\g_1,\ldots,\g_k)$ be a $k$-tuple of thin families which are very
  large in $N$ and $(y_t)_{t\in[\nn]^{<\infty}}$ be a bounded weakly null
  tree in a Banach space $X$. Also let $\ee>0$ and
  $(\delta_n)_{n\in\nn}$ be a sequence of positive reals.

  We will say that an infinite subset $L$ of $N$ is
  $(\ee,(\delta_n)_{n\in\nn})$-good (with respect to
  $k,\Delta,N,\mathfrak{G}$ and $(y_t)_{t\in[\nn]^{<\infty}}$) if
  for every $x^*\in B_{X^*}$ there exists $y^*\in
  B_{X^*}$ such that, setting for  $1\leq i\leq k$, $v_i=I_{\g_i}(L_{(\Delta,i)})$,
 the following are satisfied.
  \begin{enumerate}
  \item [(a)]   $\Big|(x^*-y^*)\Big(\sum_{t\sqsubseteq v_i}y_t\Big)\Big|\leq
  \ee$.
  \item[(b)] For all $t\in[L]^{<\infty}$ with $v_i\sqsubset t\sqsubseteq
  L_{(\Delta,i)}$, for some
  $i\in\{1,\ldots,k\}$, we have that
  $\big|y^*(y_t)\big|\leq\delta_{|t|}$.
  \end{enumerate}
\end{defn}
The main goal of this section is to prove the next result.
\begin{thm}\label{null trees Theorem}
  Let $k,\Delta,N,\mathfrak{G}$ and $(y_t)_{t\in[\nn]^{<\infty}}$ be
  as in Definition \ref{null trees definition}. Also let $\ee>0$
  and $(\delta_n)_{n\in\nn}$ be a sequence of positive reals.
  Then there exists $M\in[N]^\infty$ such that every
  $L\in[M]^\infty$ is $(\ee,(\delta_n)_{n\in\nn})$-good.
\end{thm}
We fix for the sequel  $k,\Delta,N,\mathfrak{G}$ and
$(y_t)_{t\in[\nn]^{<\infty}}$
  as in Definition \ref{null trees definition}, $\ee>0$
  and $(\delta_n)_{n\in\nn}$  a sequence of positive reals.
  Let \[\mathbb{G}=\big\{
  L\in[N]^\infty:\;L\;\;\text{is}\;\;(\ee,(\delta_n)_{n\in\nn})\text{-good}\big\}\]
  The proof of the following lemma is easy.
\begin{lem}\label{null trees Lemma 1} The set $\mathbb{G}$
  is closed in $[\nn]^\infty$.
\end{lem}
\begin{lem}\label{null trees Lemma 2}
  The set $\mathbb{G}$ is dense in $[N]^\infty$, i.e. for each
  $M\in[N]^\infty$ there exists $L\in\mathbb{G}$ with
  $L\in[M]^\infty$.
\end{lem}
 Lemma \ref{null trees Lemma 1}, Lemma \ref{null trees Lemma 2}
and Galvin-Prikry's theorem yield   Theorem \ref{null trees
Theorem}. Therefore it remains to show Lemma \ref{null trees Lemma
2}.

To this end we will need the next definition.
\begin{defn}
  An $F\in[N]^{<\infty}$ will be called $\mathfrak{G}$-free
  if for $i=i_{\Delta}(|F|+1)$ we have that $F_{(\Delta,i)}\not\in
  \widehat{\g}_{i}\setminus\g_{i}$.
\end{defn}
\begin{rem}
 Notice that by Def. \ref{Delta partition}, we have that for every
 $F\in [N]^{<\infty}$ and every $n\in N$ with $F<n$, setting
 $F'=F\cup \{n\}$,
  \[F'_{(\Delta,i_{\Delta}(|F'|))}=F_{(\Delta,i_{\Delta}(|F|+1))}\cup\{n\}\]
Hence if $F\in[N]^{<\infty}$ is $\mathfrak{G}$-free then
$F'_{(\Delta,i_{\Delta}(|F'|))}\not\in\widehat{\g}_{i_{\Delta}(|F'|)}$.
   \end{rem}
Let for every $s\in[\nn]^{<\infty}$, $x_s=\sum_{t\sqsubseteq
s}y_t$ and $K=\sup\{\|x_s\|:s\in[\nn]^{<\infty}\}$. Let also
$\delta_0>0$.
   \begin{sublem}\label{null trees sublemma}
     Let $F\in[\nn]^{<\infty}$ be $\mathfrak{G}$-free and
     $M\in[N]^\infty$ with $F\sqsubseteq M$. Also let  $((a_t)_{t\sqsubseteq
     F_{(\Delta,i)}})_{i=1}^k$ and $(b_i)_{i=1,\;i\neq
     i_0}^k$ in
     $[-2K,2K]$, where $i_0=i_{\Delta}(|F|+1)$.
For every $L'\in [M]^\infty$ we set $B^*(L')$ to be the set of all
$x^*\in B_{X^*}$ satisfying the following.
     \begin{enumerate}
       \item[(i)] For all
       $1\leq i \leq k$ and $t\sqsubseteq F_{(\Delta,i)}$, $|x^*(y_t)-a_t|\leq\frac{\delta_{|t|}}{2^{|F|+3}}$.
       \item[(ii)]  For all $1\leq i\leq k$ with $i\neq
       i_0$,
       $|x^*(\sum_{F_{(\Delta,i)}\sqsubset t \sqsubseteq v'_i}y_t)-b_i|\leq\frac{\ee}{2^{|F|+3}}$,
     where $v'_i=I_{\g_i}(L'_{(\Delta, i)})$.
     \end{enumerate}
Then  there exists
     $L\in[M]^\infty$ with $F\sqsubseteq L$ such that  for every
     $L'\in[L]^\infty$ with $F\sqsubseteq L'$ and every
     $x^*\in B^*(L')$
     there exists $y^*\in B^*(L')$ such that
      \[|y^*(y_{F_{(\Delta,i_0)}\cup\{
     L'(|F|+1)\}})|\leq \frac{\delta_{|F_{(\Delta,i_0)}|+1}}{2^{|F|+1}}\]
\end{sublem}
   \begin{proof}
     We set $\mathcal{A}$ to be the set of all $L'$ in $[M]^\infty$
     with $F\sqsubset L'$
     such that for every
     $x^*\in B^*(L')$
     there exists $y^*\in B^*(L')$ such that
      \[|y^*(y_{F_{(\Delta,i_0)}\cup\{
     L'(|F|+1)\}})|\leq \frac{\delta_{|F_{(\Delta,i_0)}|+1}}{2^{|F|+1}}\]
     By property $(\Delta2)$ it is easy to see that the set $\mathcal{A}$ is open (actually it is clopen).
     Hence
     by Galvin-Prikry's theorem there exists $L$ in $[M]^\infty$
     with $F\sqsubset L$ such that either $L'\in\mathcal{A}$ for all $L'\in[L]^\infty$ with
     $F\sqsubset L'$, or $L'\not\in\mathcal{A}$ for all $L'\in[L]^\infty$ with
     $F\sqsubset L'$. We will show that the second alternative is
     impossible.

     Indeed suppose that $L'\not\in\mathcal{A}$ for all $L'\in[L]^\infty$ with
     $F\sqsubset L'$. Then for each $L'\in[L]^\infty$ with
     $F\sqsubset L'$ we have that $B^*(L')\neq\emptyset$
     and for every $y^*\in B^*(L')$
     \[|y^*(y_{F_{(\Delta,i_0)}\cup\{
     L'(|F|+1)\}})|>\frac{\delta_{|F_{(\Delta,i_0)}|+1}}{2^{|F|+1}}\]
     Let $P=L\setminus
     F=\{p_1<p_2<\ldots\}$. For every $n\in\nn$ and $1\leq j\leq
     n$ we set \[L'_{n,j}=F\cup \{p_j\}\cup\{p_l:l>n\}\]
     By property $(\Delta3)$ we have that for every $n\in\nn$ and $1\leq j_1,j_2\leq
     n$, $B^*(L'_{n,j_1})=B^*(L'_{n,j_2})$. Hence
     for every $n\in\nn$ there exists $y^*_n\in B_{X^*}$ such that
     \[|y_n^*(y_{F_{(\Delta,i_0)}\cup\{p_j\}})|=|y_n^*(y_{F_{(\Delta,i_0)}\cup\{
     L'_{n,j}(|F|+1)\}})|> \frac{\delta_{|F_{(\Delta,i_0)}|+1}}{2^{|F|+1}}\]
     for all $1\leq j\leq n$. Let $y^*$ be a w$^*$-limit of
     $(y_n^*)_{n\in\nn}$. Then
     \[|y^*(y_{F_{(\Delta,i_0)}\cup\{p_j\}})|\geq \frac{\delta_{|F_{(\Delta,i_0)}|+1}}{2^{|F|+1}}\]
     which is a contradiction since
     $(y_{F_{(\Delta,i_0)}\cup\{n\}})_{n>F_{(\Delta,i_0)}}$ is weakly null.
   \end{proof}
   \begin{cor}\label{null tree corollary}
     Let $F\in[\nn]^{<\infty}$ be $\mathfrak{G}$-free and
     $M\in[N]^\infty$ with $F\sqsubseteq M$. Then there exists
     $L\in[M]^\infty$ with $F\sqsubset L$ such that for every
     $L'\in[M]^\infty$ with $F\sqsubset L'$ and $x^*\in B_{X^*}$
     there exists $y^*\in B_{X^*}$ satisfying the following.
     \begin{enumerate}
       \item[(i)] For all
       $1\leq i \leq k$ and $t\sqsubseteq F_{(\Delta,i)}$, $|x^*(y_t)-y^*(y_t)|\leq\frac{\delta_{|t|}}{2^{|F|+2}}$.
       \item[(ii)] For all $1\leq i\leq k$ with $i\neq
       i_0$,
       \[\Big|(x^*-
       y^*)\Big(\sum_{F_{(\Delta,i)}\sqsubset t \sqsubseteq
       v'_i}y_t\Big)\Big|\leq\frac{\ee}{2^{|F|+2}}\]
     where $v'_i=I_{\g_i}(L'_{(\Delta, i)})$.
     \item[(iii)]$|y^*(y_{F_{(\Delta,i_0)}\cup\{
     L'(|F|+1)\}})|\leq
     \frac{\delta_{|F_{(\Delta,i_0)}|+1}}{2^{|F|+1}}$.
     \end{enumerate}
   \end{cor}
   \begin{proof}
     For every $1\leq i \leq k$ and $t\sqsubseteq F_{(\Delta,i)}$
     let $A_t$ be a $\frac{\delta_{|t|}}{2^{|F|+3}}$-net of $[-2K,2K]$. For every $1\leq i\leq k$ with $i\neq
       i_0$ let $B_i$ be a $\frac{\ee}{2^{|F|+3}}$-net of $[-2K,2K]$. Let
       \[(((a_t^q)_{t\sqsubseteq
     F_{(\Delta,i)}})_{i=1}^k,(b_i^q)_{i=1,\;i\neq
     i_0}^k)_{q=1}^m\] be an enumeration of the set $(\prod_{i=1}^k(\prod_{t\sqsubseteq
     F_{(\Delta,i)}}A_t))\times(\prod_{i=1,\;i\neq
     i_0}^k B_i)$. We set $L_0=M$ and inductively for $q=1,\ldots,m$, using Sublemma \ref{null trees sublemma} for
     $((a_t^q)_{t\sqsubseteq
     F_{(\Delta,i)}})_{i=1}^k$ and $(b_i^q)_{i=1,\;i\neq
     i_0}^k$, we construct a
     decreasing sequence of infinite subsets $(L_q)_{q=1}^m$ of $L_0$ with
     $F\sqsubset
     L_q$, for all $1\leq q\leq m$. It is easy to check that $L_m$
     is the desired set.
   \end{proof}
   \begin{proof}
     [Proof of Lemma \ref{null trees Lemma 2}] We may suppose that
     $(\delta_n)_{n=1}^\infty$ is a decreasing sequence and
     $\sum_{n=1}^\infty\delta_n<\ee$. Let also $\delta_1<\delta_0<\ee/2$ and $M=M_1\in[N]^\infty$.
      Let $F_1$ be the $\sqsubseteq$-minimal
     initial segment of $M_1$ which is $\mathfrak{G}$-free.
     Applying Corollary \ref{null tree corollary} we obtain
     $M_2\in[M_1]^\infty$ with $F_1\sqsubset M_2$ satisfying the
     conclusion of Corollary \ref{null tree corollary} for
     $F=F_1$ and $L=M_2$. In the same way  we construct by induction a sequence
     $(F_n)_n\in\nn$ in $[M_1]^{<\infty}$
     and a sequence $(M_n)_{n\in\nn}$ in $[M_1]^\infty$ satisfying
     the following for every $n\in\nn$.
     \begin{enumerate}
       \item[(i)] For every $n\in\nn$, $F_n\sqsubset M_n$.
       \item[(ii)] For every $n\in\nn$, $F_n\sqsubset F_{n+1}$ and
       $M_{n+1}\in[M_n]^\infty$.
       \item[(iii)] For every $n\in\nn$, the conclusion of
       Corollary \ref{null tree corollary} is satisfied for
       $F=F_n$ and $L=M_{n+1}$.
     \end{enumerate}

     We set $L=\cup_{n=1}^\infty F_n$. We claim that  $L$ is
     $(\ee,(\delta_n)_{n\in\nn})$-good. Indeed, let $x^*\in
     B_{X^*}$ and set $y_0^*=x^*$.
     For every $1\leq i\leq k$ let $v_i$ be the unique element of
     $\g_i$ such that $v_i\sqsubset L_{(\Delta,i)}$. For every
     $n\in\nn$ we set
     \[t_i^n=v_i\cap F_n=v_i\cap (F_n)_{(\Delta,i)}\]
     and for every $1\leq i\leq k$ and
     $v_i\sqsubset t\sqsubset L_{(\Delta,i)}$ we set
     \[n_t=\min\big\{n\in\nn:t\subseteq L|(|F_n|+1)\big\}\]
     Notice that \[t=(F_{n_t})_{(\Delta,i)}\cup\{L(|F_{n_t}|+1)\}\]
     for all $1\leq i\leq k$ and
     $v_i\sqsubset t\sqsubset L_{(\Delta,i)}$. Finally let, for every $n\in\nn$,
     $i_n=i_{\Delta}(|F_n|+1)$.

     By the construction of $(F_n)_{n\in\nn}$ we may inductively
     choose a sequence $(y^*_n)_{n\in\nn}$ in $B_{X^*}$ such that
     for every $n\in\nn$ the following are satisfied:
     \begin{enumerate}
       \item[(i)] For all
       $1\leq i \leq k$ and $t\sqsubseteq (F_n)_{(\Delta,i)}$, $|y_{n-1}^*(y_t)-y_n^*(y_t)|\leq\frac{\delta_{|t|}}{2^{|F_n|+2}}$.
       \item[(ii)] For all $1\leq i\leq k$ with $i\neq
       i_n$,
       \[\Big|(y_{n-1}^*-
       y_n^*)(x_{v_i}-x_{t_i^n})\Big|=
       \Big|(y_{n-1}^*-
       y_n^*)\Big(\sum_{(F_n)_{(\Delta,i)}\sqsubset t \sqsubseteq
       v^n_i}y_t\Big)\Big|\leq\frac{\ee}{2^{|F_n|+2}}\]
     \item[(iii)]
     $|y_n^*(y_{(F_n)_{(\Delta,i_n)}\cup\{
     L(|F_n|+1)\}})|\leq
     \frac{\delta_{|(F_n)_{(\Delta,i_n)}|+1}}{2^{|F_n|+1}}$.
     \end{enumerate}
     By (i) and (ii) we have that for every $n\in\nn$ and $1\leq i\leq k$,
     \[\begin{split}
     |(y_{n-1}^*-y_n^*)(x_{v_i})|
     &\leq |(y_{n-1}^*-y_n^*)(x_{v_i}-x_{t^n_i})|+\sum_{t\sqsubseteq
     t_i^n}|(y_{n-1}^*-y_n^*)(y_t)|\\
     &\leq \frac{\ee}{2^{|F_n|+2}}+\sum_{t\sqsubseteq t_i^n}\frac{\delta_{|t|}}{2^{|F_n|+2}}
     <\frac{\ee}{2^{|F_n|+1}}\end{split}\]
     This yields that $|(x^*-y_n^*)(x_{v_i})|<\ee$, for all
     $n\in\nn$. By (iii) and the definition of the natural number $n_t$,
     we have that for every $1\leq i\leq k$
     and every $v_i\sqsubset t\sqsubset L_{(\Delta,i)}$,
     \[|y_{n_t}^*(y_t)|\leq
     \frac{\delta_{|(F_{n_t})_{(\Delta,i_0)}|+1}}{2^{|F_{n_t}|+1}}
     =\frac{\delta_{|t|}}{2^{|F_{n_t}|+1}}\]
     Hence by (i) and (iii) for every $1\leq i\leq k$,
     every $v_i\sqsubset t\sqsubset L_{(\Delta,i)}$ and $n>
     n_t$, we have that $t\sqsubseteq (F_m)_{(\Delta,i)}$, for all
     $n_t<m\leq n$, and
     \[\begin{split}|y^*_n(y_t)|&\leq
     |y^*_{n_t}(y_t)|+\sum_{m=n_t+1}^n|(y^*_{m-1}-y^*_m)(y_t)|\\
     &\leq\frac{\delta_{|t|}}{2^{|F_{n_t}|+1}}+\frac{\delta_{|t|}}{2^{|F_m|+2}}\leq\delta_{|t|}
     \end{split}\]
     Let $y^*$ be a $w^*$-limit of $(y_n)_{n\in\nn}$. Then $y^*\in
     B_{X^*}$ and for each $1\leq i\leq k$,
  \begin{enumerate}
  \item [(a)]   $\Big|(x^*-y^*)(x_{v_i})\Big|\leq
  \ee$.
  \item[(b)] For all $t\in[L]^{<\infty}$ with $v_i\sqsubset t\sqsubset
  L_{(\Delta,i)}$, for some
  $i\in\{1,\ldots,k\}$, we have that
  $\big|y^*(y_t)\big|\leq\delta_{|t|}$.
  \end{enumerate}
  Since for every $x^*\in B_{X^*}$ there exists $y^*\in B_{X^*}$
  satisfying (a) and (b), we have that $L$ is
     $(\ee,(\delta_n)_{n\in\nn})$-good.
   \end{proof}
\section{Dominated spreading models}
\begin{thm}\label{c_0 block 1 domination prop}
  Let $\ff,\g$ be regular thin families and $N\in[\nn]^\infty$ such
  that $\g\upharpoonright N\sqsubset \ff\upharpoonright N$. Let
  $(x_s)_{s\in\ff}$ be a bounded $\ff$-sequence in a Banach space
  $X$ such that the $\ff$-subsequence
  $(x_s)_{s\in\ff\upharpoonright N}$ is subordinated and let
  $\widehat{\varphi}:\widehat{F}\upharpoonright N\to (X,w)$ be the
  continuous map witnessing this. For every $v\in
  \g\upharpoonright N$, let $z_v=\widehat{\varphi}(v)$. Suppose that
  $(x_s)_{s\in\ff\upharpoonright N}$ and $(z_v)_{v\in\g\upharpoonright
  N}$ generate $(e^1_n)_{n\in\nn}$ and
  $(e^2_n)_{n\in\nn}$ as spreading models respectively. Then
  for every $k\in\nn$ and $a_1,\ldots,a_k\in\rr$ we have that
  \[\Big\|\sum_{j=1}^ka_je_j^2\Big\|\leq\Big\|\sum_{j=1}^ka_je_j^1\Big\|\]
\end{thm}
The proof of the above theorem relies on a series of lemmas which
are presented below.
\begin{lem}\label{c_0 block lemma from thm of weakly null trees}
  Let $k\in\nn$ and $\ee>0$. Then
  there exists $M\in[N]^\infty$ such that for every plegma
  $k$-tuple $(s_j)_{j=1}^k$ in $\ff\upharpoonright M$ and
  $a_1,\ldots,a_k\in\rr$, we have
  that
  \[\Big\|\sum_{j=1}^ka_jx_{s_j}\Big\|\geq\Big\|\sum_{j=1}^ka_jz_{v_j}\Big\|-\ee\sum_{j=1}^k|a_j|\]
  where for each $1\leq j\leq k$, $v_j$ is
  the unique element in $\g$ such that $v_j\sqsubset s_j$.
\end{lem}
\begin{proof}
  We define a bounded weakly null tree $(y_t)_{t\in[\nn]^{<\infty}}$ as follows. For
  each $t\not\in\widehat{\ff}\upharpoonright N$ we set $y_t=0$. We
  also
  set $y_\emptyset=\widehat{\varphi}(\emptyset)$ and for each nonempty
  $t\in\widehat{\ff}\upharpoonright N$ we set
  $y_t=\widehat{\varphi}(t)-\widehat{\varphi}(t^*)$, where
  $t^*=t\setminus\{\max t\}$. It is immediate that
  $(y_t)_{t\in[\nn]^{<\infty}}$ is a bounded weakly null tree. Let
  $\ee>0$ and
  $(\delta_{n})_{n\in\nn}$ be a decreasing null sequence of
  positive reals such that
  $\sum_{n=1}^\infty\delta_n<\frac{\ee}{2}$. Let $k\in\nn$ and
  $\Delta^k=(N_k^{(1)},\ldots,N_k^{(k)})$ be the partition of $\nn$ into
  disjoint infinite sets such that
  $N_k^{(j)}=\{n\in\nn:n=j(\text{mod}k)\}$, for all $1\leq j<k$, and
  $N_k^{(k)}=\{n\in\nn:n=0(\text{mod}k)\}$. Applying Theorem \ref{null trees
  Theorem} for
  $k,\Delta^k,N,(\g,\ldots,\g),(y_t)_{t\in[\nn]^{<\infty}},\frac{\ee}{2}$
  and $(\delta_{n})_{n\in\nn}$, we obtain $M'\in[N]^\infty$
  such that every $L\in[M']^\infty$ is
  $(\frac{\ee}{2},(\delta_{n})_{n\in\nn})$-good.

  We set $M=\{M'(k\cdot n):n\in\nn\}$. It is easy to check that
  for every plegma $k$-tuple $(s_j)_{j=1}^k$ in $\ff\upharpoonright
  M$ there exists $L\in[M']^\infty$ such that $s_j\sqsubset
  L_{(\Delta^k,j)}$, for all $1\leq j\leq k$. The set $M$ is the
  desired one.

  Indeed, let $a_1,\ldots,a_k\in\rr$ and $(s_j)_{j=1}^k$ be a plegma
  $k$-tuple in $\ff\upharpoonright M$. Let $L\in[M']^\infty$ such
  that $s_j\sqsubset L_{(\Delta^k,j)}$, for all $1\leq j\leq k$.
  Let also for each $1\leq j\leq k$, $v_j$
  the unique element in $\g$ such that $v_j\sqsubset s_j$. Then
  for every $x^*\in B_{X^*}$ there exists $y^*\in B_{X^*}$ such
  that for every $1\leq j\leq k$
  \begin{enumerate}
    \item[(i)] $|x^*(z_{v_j})-y^*(z_{v_j})|\leq\frac{\ee}{2}$ and
    \item[(ii)] $|y^*(y_t)|\leq\delta_{|t|}$, for all $t\sqsupset
    v_j$.
  \end{enumerate}
  By (ii) and the choice of the family
  $(y_t)_{t\in[\nn]^{<\infty}}$ it is easy to see that
  \[|y^*(x_{s_j}-z_{v_j})|\leq\frac{\ee}{2}\] for all $1\leq j\leq
  k$. Therefore we have that
  \begin{equation}
    \label{eq15}
    \begin{split}
      \Big|y^*\Big(\sum_{j=1}^ka_jx_{s_j}\Big)\Big|&
      \geq\Big|y^*\Big(\sum_{j=1}^ka_jz_{v_j}\Big)\Big|-\frac{\ee}{2}\sum_{j=1}^k|a_j|\\
      &\geq\Big|x^*\Big(\sum_{j=1}^ka_jz_{v_j}\Big)\Big|-\ee\sum_{j=1}^k|a_j|
    \end{split}
  \end{equation}
  Since for every $x^*\in B_{X^*}$ there exists $y^*\in B_{X^*}$
  which satisfies (\ref{eq15}), we have that
  \[\Big\|\sum_{j=1}^ka_jx_{s_j}\Big\|\geq\Big\|\sum_{j=1}^ka_jz_{v_j}\Big\|-\ee\sum_{j=1}^k|a_j|\]
\end{proof}
\begin{lem}
  For every $(\ee_n)_{n\in\nn}$ decreasing null
  sequence of positive reals, there exists  $M\in[N]^\infty$ such
  that for every $k\leq l$ in $\nn$, every plegma
  $k$-tuple $(s_j)_{j=1}^k$ in $\ff\upharpoonright M$ with $s_1(1)= M(l)$ and
  $a_1,\ldots,a_k\in\rr$ we have
  that
  \[\Big\|\sum_{j=1}^ka_jx_{s_j}\Big\|\geq\Big\|\sum_{j=1}^ka_jz_{v_j}\Big\|-\ee_l\sum_{j=1}^k|a_j|\]
  where for each $1\leq j\leq k$,
  $v_j$ is
  the unique element in $\g$ such that $v_j\sqsubset s_j$.
\end{lem}
\begin{proof}
  Using Lemma \ref{c_0 block lemma from thm of weakly null trees}
  we inductively construct a decreasing sequence $(M_l)_{l\in\nn}$
  in $[N]^\infty$ such that for every $l\in\nn$ we have that for every $1\leq k\leq l$, every plegma
  $k$-tuple $(s_j)_{j=1}^k$ in $\ff\upharpoonright M_l$ and
  $a_1,\ldots,a_k\in\rr$
  \[\Big\|\sum_{j=1}^ka_jx_{s_j}\Big\|\geq\Big\|\sum_{j=1}^ka_jz_{v_j}\Big\|-\ee_l\sum_{j=1}^k|a_j|\]
  where for each $1\leq j\leq k$,
  $v_j$ is
  the unique element in $\g$ such that $v_j\sqsubset s_j$. Let
  $M\in[N]^\infty$ such that $M(l)\in M_l$, for all $l\in\nn$. It
  is easy to check that $M$ is the desired set.
\end{proof} By the above lemma the proof of Theorem \ref{c_0 block 1
domination prop} is immediate.

\section{Plegma block generated $c_0$ spreading models}
In this section we show that if $X$ is a Banach space with a
Schauder basis and
   $\mathcal{SM}^{wrc}(X)$
  contains a sequence equivalent to the usual basis of $c_0$, then
  $X$ admits $c_0$ as a plegma block generated
  spreading model. We start with the following lemma.
\begin{lem}\label{c_0 block thm}
  Let $X$ be a Banach space with a Schauder basis such that
  $\mathcal{SM}^{wrc}(X)$ contains a sequence equivalent to the
  usual basis of $c_0$ and let $\xi_0$ be the minimum countable
  ordinal $\xi$ such that $\mathcal{SM}^{wrc}_\xi(X)$ contains such a sequence.

  Let $\ff$ be a regular thin family of order $\xi_0$, $M\in[\nn]^\infty$
  and $(x_s)_{s\in\ff}$ be a weakly relatively compact $\ff$-sequence
  in $X$
  such that $(x_s)_{s\in\ff\upharpoonright M}$ generates $c_0$ as
  an $\ff$-spreading model.
  Then for every sequence $(\delta_n)_{n\in\nn}$ of positive real
  numbers there exist $N\in[M]^\infty$ and an $\ff$-subsequence $(z_s)_{s\in\ff\upharpoonright
  N}$ such  that the following are satisfied:
  \begin{enumerate}
    \item[(i)]  For every $s\in\ff\upharpoonright N$ if $\min
  s=N(n)$, then $\|z_s-x_s\|<\delta_n$.
    \item[(ii)] For every plegma pair $(s_1,s_2)$ in
    $\ff\upharpoonright N$ we have that $z_{s_1}<z_{s_2}$, that is
    $(z_s)_{s\in\ff\upharpoonright N}$ is a plegma block
    subsequence.
  \end{enumerate}
\end{lem}
\begin{proof}
  We may suppose that $\ff$ is very large
  in $M$. By Proposition \ref{Create subordinated} we may suppose that
  $(x_s)_{s\in\ff\upharpoonright M}$ is subordinated and let $\widehat{\varphi}:\widehat{\ff}\upharpoonright L\to (X,w)$ be the continuous map witnessing
    this.
    Since $(x_s)_{s\in\ff\upharpoonright L}$ generates an
  $\ff$-spreading model which is Schauder basic and not equivalent
  to the usual basis of $\ell^1$, by Lemma \ref{lemma either l1 or not Schauder
  basic}, we have that
  $\widehat{\varphi}(\emptyset)=0$.
    By Theorem
  \ref{canonical tree}
  there exist $L\in[M]^\infty$ and an
  $\ff$-subsequence $(\widetilde{x}_s)_{s\in\ff\upharpoonright L}$
  in $X$ satisfying the following.
  \begin{enumerate}
    \item[(i)] For every $s\in\ff\upharpoonright L$, $\|x_s-\widetilde{x}_s\|<\delta_n/2$, where $\min
  s=L(n)$.
\item[(ii)] The $\ff$-subsequence
$(\widetilde{x}_s)_{s\in\ff\upharpoonright
    L}$ is subordinated with respect to the weak topology of $X$. Moreover, if
    $\widetilde{\varphi}:\widehat{\ff}\upharpoonright L\to (X,w)$ is the continuous map witnessing
    this, then
    $\widetilde{\varphi}(\emptyset)=\widehat{\varphi}(\emptyset)=0$.
    \item[(iii)] The $\ff$-subsequence $(\widetilde{x}_s)_{s\in\ff\upharpoonright
    L}$ admits a canonical tree decomposition\\
    $(\widetilde{y}_t)_{t\in\ff\upharpoonright L}$.
\end{enumerate}

  \textbf{Claim:} For every $\delta>0$ and $L'\in[L]^\infty$ there
  exists $L''\in[L']^\infty$ such that for every plegma pair
  $(s_1,s_2)$ in $\ff\upharpoonright L''$ we have that
  \[\Big\|\sum_{t\sqsubseteq
  s_2/_{s_1}}\widetilde{y}_t\Big\|<\delta\]
  where $s_2/_{s_1}=s_2\cap\{1,\ldots,\max s_1\}$.
  \begin{proof}[Proof of the Claim]
    Let $\delta>0$ and $L'\in[L]^\infty$. Then by Theorem
    \ref{ramseyforplegma} there exists $L''\in[L']^\infty$
    such that either \[\Big\|\sum_{t\sqsubseteq
  s_2/_{s_1}}\widetilde{y}_t\Big\|<\delta\] for every plegma pair
  $(s_1,s_2)$ in $\ff\upharpoonright L''$, or
  \[\Big\|\sum_{t\sqsubseteq
  s_2/_{s_1}}\widetilde{y}_t\Big\|\geq\delta\] for every plegma pair
  $(s_1,s_2)$ in $\ff\upharpoonright L''$.
  We will show that the second alternative is impossible.
Otherwise let
  $\g=\ff/_{L''}$ (see Definition \ref{notation hunging ff by min
  L}) and notice that $\big(\widetilde{x}^{(1,\g)}_s\big)_{s\in\ff\upharpoonright
  L''(2\nn)}$ is seminormalized. Moreover $\big(\widetilde{x}^{(1,\g)}_s\big)_{s\in\ff\upharpoonright
  L''(2\nn)}$ is subordinated with
  $\widetilde{\varphi}^{(1,\g)}(\emptyset)=\widetilde{\varphi}(\emptyset)=0$.
  Let $(e_n)_{n\in\nn}$ be an $\ff$-spreading model of $\big(\widetilde{x}^{(1,\g)}_s\big)_{s\in\ff\upharpoonright
  L''(2\nn)}$.
  Therefore by Theorem \ref{unconditional spreading model} we have that $(e_n)_{n\in\nn}$
  is nontrivial and unconditional. The latter, by Proposition \ref{c_0 block 1 domination
  prop}, yields that $(e_n)_{n\in\nn}$ is equivalent to the usual
  basis of $c_0$. By Lemma \ref{up g splitting spr mod} we
  contradict the assumption of the minimality of
  $\xi_0$.
  \end{proof}
  Using the above claim
  we inductively construct a decreasing sequence
  $(L_n)_{n\in\nn}$ in $[L]^\infty$ such that for every
  $n\in\nn$ and every plegma pair $(s_1,s_2)$  in $\ff\upharpoonright L_n$ we have that
  \[\Big\|\sum_{t\sqsubseteq
  s_2/_{s_1}}\widetilde{y}_t\Big\|<\delta_n/2\]
  Let $L_\infty\in[L]^\infty$ such that $L_\infty(n)\in L_n$,
  for all $n\in\nn$. We set $N=L_\infty(2\nn)$. Let $(z^N_s)_{s\in\ff\upharpoonright
  N}$ defined as in Lemma \ref{block plegma} (with $N$ in place of
  $L$). Then $(z^N_s)_{s\in\ff\upharpoonright
  N}$ is plegma block. Moreover for every $s\in \ff\upharpoonright N$ with $\min s=N(n)-L_\infty(2\nn)$
  we easily conclude that
$\|z_s^N-\widetilde{x}_s\|<\frac{\delta_n}{2}$ and therefore
$\|z_s^N-x_s\|<\delta_n$.
\end{proof}
\begin{rem}
  It is easy to see that $(z_s)_{s\in\ff\upharpoonright N}$
  generates the same $\ff$-spreading model as the
  $\ff$-subsequence $(x_s)_{s\in\ff\upharpoonright M}$.
\end{rem}
\begin{thm}\label{c_0 block spreading model cor}
  Let $X$ be a Banach space with a Schauder basis. For every
  $\xi<\omega_1$ we have that if $\mathcal{SM}^{wrc}_\xi(X)$
  contains a sequence equivalent to the usual basis of $c_0$, then
  $X$ admits $c_0$ as a plegma block generated $\xi$-spreading model.
\end{thm}

\section{Duality of $c_0$ and $\ell^1$ spreading models}
The main aim of this section Theorem \ref{duality c_0 l^1 thm}
which generalize the classical fact that if a Banach space admits
$c_0$ as a spreading model then its dual admits $\ell^1$ as a
spreading model.
\begin{defn}
  Let $X$ be a Banach space, $c>0$, $\ff$ a regular thin family, $M\in[\nn]^\infty$, $(x_s)_{s\in\ff}$ a normalized $\ff$-sequence in $X$ and $(x_s^*)_{s\in\ff}$ a bounded $\ff$-sequence in $X^*$. We will say that $(x_s)_{s\in\ff\upharpoonright M}$ and $(x_s^*)_{s\in\ff\upharpoonright M}$ are $\ell^1$-associated over $c$ if for every $k\in\nn$, $a_1,\ldots,a_k\in\rr$ and every plegma $k$-tuple $(s_j)_{j=1}^k$ in $\ff\upharpoonright M$ with $s_1(1)\geq M(k)$ we have that \[\sum_{j=1}^ka_jx_{s_j}^*\Big(\sum_{j=1}^k\text{sign}(a_j)x_{s_j}\Big)\geq c\sum_{j=1}^k|a_j|\]
  We will say that $(x_s)_{s\in\ff\upharpoonright M}$ and $(x_s^*)_{s\in\ff\upharpoonright M}$ are $\ell^1$-associated if there exists $c>0$ such that $(x_s)_{s\in\ff\upharpoonright M}$ and $(x_s^*)_{s\in\ff\upharpoonright M}$ are $\ell^1$-associated over $c$.
\end{defn}
\begin{rem}\label{duality c_0 l^1 rem associated l^1 giving l^1 spr mod}
  Suppose that $(x_s)_{s\in\ff\upharpoonright M}$ and $(x_s^*)_{s\in\ff\upharpoonright M}$ are $\ell^1$-associated. Then if $(x_s)_{s\in\ff\upharpoonright M}$ generates $c_0$ as an $\ff$-spreading model, then it is easy to see that every $\ff$-spreading model admitted by $(x_s^*)_{s\in\ff\upharpoonright M}$ is $\ell^1$.
\end{rem}
\begin{defn}
  Let $X$ be a Banach space, $\ff$ a regular thin family, $M\in[\nn]^\infty$, $(x_s)_{s\in\ff}$ a normalized $\ff$-sequence in $X$ and $(x_s^*)_{s\in\ff}$ a bounded $\ff$-sequence in $X^*$. The $\ff$-subsequence $(x_s^*)_{s\in\ff\upharpoonright M}$ is called biorthogonal to $(x_s)_{s\in\ff\upharpoonright M}$ if for every $k\in \nn$ and every plegma $k$-tuple $(s_j)_{j=1}^k$ in $\ff\upharpoonright M$ with $s_1(1)\geq M(k)$ we have that
  \[x_{s_j}^*(x_{s_i})=\delta_{ij},\;\;\text{for all}\;\;i,j\in\{1,\ldots,k\}\]
\end{defn}
\begin{rem}\label{duality c_0 l^1 rem orthogonals are associated l^1}
  It is immediate that if $(x_s^*)_{s\in\ff\upharpoonright M}$ is biorthogonal to $(x_s)_{s\in\ff\upharpoonright M}$ then $(x_s)_{s\in\ff\upharpoonright M}$ and $(x_s^*)_{s\in\ff\upharpoonright M}$ are $\ell^1$-associated over 1.
\end{rem}
\begin{lem}\label{duality c_0 l^1 lem get biorthogonals}
  Let $X$ be a Banach space with a Schauder basis $(e_n)_{n\in\nn}$, $\ff$ a regular thin family, $M\in[\nn]^\infty$ and $(x_s)_{s\in\ff}$
  a normalized plegma block $\ff$-sequence  in $X$. Then there exists an $\ff$-subsequence
  $(x_s^*)_{s\in\ff\upharpoonright M}$ biorthogonal to $(x_s)_{s\in\ff\upharpoonright M}$ with
  respect to $(e_n^*)_{n\in\nn}$.
\end{lem}
\begin{proof}
  By Hahn-Banach theorem, for every $s\in\ff\upharpoonright M$ there exists $\widetilde{x}_s^*\in B_{X^*}$ such that $\widetilde{x}_s^*(x_s)=\|x_s\|=1$. For every $s\in\ff\upharpoonright M$ we set
  \[x_s^*=\sum_{n\in\text{range}(x_s)}\widetilde{x}_s^*(e_n)e_n^*\] where $\text{range}(x_s)=\{\min(\text{supp}(x_s)),\ldots,\max(\text{supp}(x_s))\}$. It is easy to check that $(x_s^*)_{s\in\ff\upharpoonright M}$ satisfies the conclusion of the lemma.
\end{proof}
\begin{prop}\label{plegma block l^1 in dual}
  Let $X$ be a Banach space with a Schauder basis $(e_n)_{n\in\nn}$. If for some
  $\xi<\omega_1$ the set $\mathcal{SM}^{wrc}_\xi(X)$ contains a
  sequence equivalent to the usual basis of $c_0$, then $\overline{<(e^*_n)_{n\in\nn}>}$ admits $l^1$ as a plegma block generated $\xi$-spreading model.
\end{prop}
\begin{proof}
  By Theorem \ref{c_0 block spreading model
  cor} there exist a regular thin family $\ff$ of order $\xi$, $M\in[\nn]^\infty$ and
  $(x_s)_{s\in\ff}$ an $\ff$-sequence in
  $X$ such that $(x_s)_{s\in\ff\upharpoonright M}$ plegma block generates $c_0$
  as an $\ff$-spreading model. We may also assume that $(x_s)_{s\in\ff\upharpoonright M}$ is normalized.
  By Lemma \ref{duality c_0 l^1 lem get biorthogonals} there exists a bounded $\ff$-subsequence
  $(x_s^*)_{s\in\ff\upharpoonright M}$ biorthogonal to $(x_s)_{s\in\ff\upharpoonright M}$ which is plegma block with respect to $(e_n^*)_{n\in\nn}$.
  By Remarks  \ref{duality c_0 l^1 rem orthogonals are associated l^1}
  and \ref{duality c_0 l^1 rem associated l^1 giving l^1 spr mod}
  we have that $(x_s^*)_{s\in\ff\upharpoonright M}$ admits $\ell^1$ as an $\ff$-spreading model.
\end{proof}
\begin{lem}\label{duality c_0 l^1 lem approximation for l^1 associating}
  Let $X$ be a Banach space, $\ff$ a regular thin family,
  $M\in[\nn]^\infty$, $(x_s)_{s\in\ff}$ a
  normalized $\ff$-sequence in $X$ and $(x_s^*)_{s\in\ff}$ a bounded
  $\ff$-sequence in $X^*$ such that $(x_s)_{s\in\ff\upharpoonright M}$ and $(x_s^*)_{s\in\ff\upharpoonright M}$
  are $\ell^1$-associated over $c$, for some $c>0$. Let $(\delta_n)_{n\in\nn}$ a
  sequence of positive reals and
  $(z_s)_{s\in\ff}$ a normalized $\ff$-sequence in $X$ such
  that $\|x_s-z_s\|\leq\delta_n$, for all $n\in\nn$ and $s\in\ff\upharpoonright M$ with $\min s=M(n)$.
 If $\sum_{n=1}^\infty\delta_n<\frac{c}{2K}$, where
 $K=\sup\{\|x_s^*\|:s\in\ff\upharpoonright M\}$
 then
  $(z_s)_{s\in\ff\upharpoonright M}$ and $(x_s^*)_{s\in\ff\upharpoonright M}$ are
  $\ell^1$-associated over $c/2$.
\end{lem}
\begin{proof}
  Indeed, for every $k\in\nn$, $a_1,\ldots,a_k\in\rr$ and every plegma $k$-tuple
  $(s_j)_{j=1}^k$ in $\ff\upharpoonright M$ with $s_1(1)\geq M(k)$ we have that
  \[\begin{split}\sum_{j=1}^ka_jx_{s_j}^*\Big(&\sum_{j=1}^k\text{sign}(a_j)z_{s_j}\Big)\\
  &\geq\sum_{j=1}^ka_jx_{s_j}^*\Big(\sum_{j=1}^k\text{sign}(a_j)x_{s_j}\Big)-
  \sum_{j=1}^k|a_j|\cdot\|x_{s_j}^*\|\Big\|\sum_{i=1}^k\|z_{s_i}-x_{s_i}\Big\|\\
   &\geq\frac{c}{2}\sum_{j=1}^k|a_j|
   \end{split}\]
\end{proof}
\begin{thm}\label{duality c_0 l^1 thm}
Let $X$ be a Banach space. If for some
  $\xi<\omega_1$ the set $\mathcal{SM}^{wrc}_\xi(X)$ contains a
  sequence equivalent to the usual basis of $c_0$, then $X^*$ admits $l^1$ as a $\xi$-spreading model.
\end{thm}
\begin{proof}
  Let $\xi_0$ be the minimum countable ordinal $\xi$ such that
  $\mathcal{SM}^{wrc}_\xi(X)$ contains a sequence equivalent to
  the usual basis of $c_0$. Notice that it suffices to prove the theorem for $\xi=\xi_0$.
  Let $\ff$ be a regular thin family of order $\xi_0$,
  $M\in[\nn]^\infty$ and $(x_s)_{s\in\ff}$ a weakly relatively compact $\ff$-sequence in
  $X$ such that $(x_s)_{s\in\ff\upharpoonright M}$ generates $c_0$ as an $\ff$-spreading model.
  Observe that we may also assume that $(x_s)_{s\in\ff}$ is normalized. Let $Y$ be a separable (closed) subspace
  of $X$ such that $\{x_s:s\in\ff\}\subseteq Y$ and $T:Y\to C[0,1]$ an isometric (linear)
  embedding. Let $(e_n)_{n\in\nn}$  be as Schauder basis of $C[0,1]$. Then $(T(x_s))_{s\in\ff}$
  is a normalized weakly relatively compact $\ff$-sequence in $C[0,1]$. Let $(\delta_n)_{n\in\nn}$
  be a sequence of positive reals such that $\sum_{n=1}^\infty\delta_n<\infty$. By
  Lemma
  \ref{c_0 block thm} there exist $L_1\in[M]^\infty$ and $(\widetilde{z}_s)_{s\in\ff\upharpoonright L_1}$
  such that the following are satisfied:
  \begin{enumerate}
    \item[(i)] For every $s\in\ff\upharpoonright L_1$, if $\min s=L_1(l)$, then $\|\widetilde{z}_s-T(x_s)\|<\frac{\delta_l}{2}$.
    \item[(ii)] For every plegma pair $(s_1,s_2)$ in $\ff\upharpoonright L_1$ we have that $\widetilde{z}_{s_1}<\widetilde{z}_{s_2}$.
  \end{enumerate}
  For every $s\in\ff\upharpoonright L_1$ we set $z_s=\frac{\widetilde{z}_s}{\|\widetilde{z}_s\|}$.
  Then $(z_s)_{s\in\ff\upharpoonright L_1}$ is a normalized $\ff$-subsequence such that
  \begin{enumerate}
    \item[(a)] For every $s\in\ff\upharpoonright L_1$, if $\min s=L_1(l)$, then $\|z_s-T(x_s)\|<\delta_l$.
    \item[(b)] For every plegma pair $(s_1,s_2)$ in $\ff\upharpoonright L_1$ we have that $z_{s_1}<z_{s_2}$.
  \end{enumerate}
  By Lemma \ref{duality c_0 l^1 lem get biorthogonals} there exists an $\ff$-subsequence
  $(z_s^*)_{s\in\ff\upharpoonright L_1}$ in $(C[0,1])^*$ biorthogonal to
  $(z_s)_{s\in\ff\upharpoonright L_1}$. By Remark
  \ref{duality c_0 l^1 rem orthogonals are associated l^1} we have that
  $(z_s)_{s\in\ff\upharpoonright L_1}$ and $(z_s^*)_{s\in\ff\upharpoonright L_1}$ are $\ell^1$-associated over 1.
  Let $K=\sup\{\|z_s^*\|:s\in\ff\upharpoonright L_1\}$, $n_0\in\nn$ such that $\sum_{n=n_0}^\infty\delta_n<\frac{1}{2K}$ and
  $L=\{L_1(n):n\geq n_0\}$. By Lemma \ref{duality c_0 l^1 lem approximation for l^1 associating} $(T(x_s))_{s\in\ff\upharpoonright L}$
  and $(z_s^*)_{s\in\ff\upharpoonright L}$ are $\ell^1$-associated over $\frac{1}{2}$. This easily yields
  that $(x_s)_{s\in\ff\upharpoonright L}$ and $(T^*(z_s^*))_{s\in\ff\upharpoonright L}$ are $\ell^1$-associated
  over $\frac{1}{2}$. Notice that $(T^*(z_s^*))_{s\in\ff\upharpoonright L}$ is a bounded $\ff$-subsequence in $Y^*$. For every
  $s\in\ff\upharpoonright L$, by Hahn-Banach theorem there exists $x_s^*\in X^*$ such that $\|x_s^*\|=\|T^*(z_s^*)\|$
  and $x_s^*|_Y=T^*(z_s^*)|_Y$. Therefore $(x_s^*)_{s\in\ff\upharpoonright L}$ is a bounded $\ff$-subsequence in
  $X^*$ and $(x_s)_{s\in\ff\upharpoonright L}$, $(x_s^*)_{s\in\ff\upharpoonright L}$ are $\ell^1$-associated (over $\frac{1}{2}$).
  Hence by Remark \ref{duality c_0 l^1 rem associated l^1 giving l^1 spr mod} we have that $(x_s^*)_{s\in\ff\upharpoonright L}$
   admits $\ell^1$ as an $\ff$-spreading model.
\end{proof}
\chapter{Establishing the hierarchy of  spreading models} \label{separating k l^1 spr mod space}
In this chapter we deal with the problem of the existence of
spaces $X$ admitting $\ell^1$ as $\xi$-spreading model but not
less. We present two examples. The first one  answers the problem
for $\xi<\omega$ and the second concerns transfinite countable
ordinals. (Anafora gia to mikro gia ta mikris taksis)

\section{Spaces admitting $\ell^1$ as $\xi$-spreading
model but not less}\label{l^1 on ksi} In this section we show that
for every countable ordinal $\xi$ there exists a reflexive space
$\mathfrak{X}_\xi$ with an unconditional basis satisfying the
following properties:
\begin{enumerate}
  \item[(i)] The space $\mathfrak{X}_\xi$ admits $\ell^1$ as a $\xi$-spreading model.
  \item[(ii)] For every ordinal $\zeta$ such that $\zeta+2<\xi$,
  the space $\mathfrak{X}_\xi$ does not admit $\ell^1$ as a $\zeta$-spreading model.
\end{enumerate}
Therefore, if $\xi$ is a limit countable ordinal, then
$\mathfrak{X}_\xi$ is the minimum countable ordinal $\zeta$ such
that $\mathcal{SM}_\zeta(\mathfrak{X}_\xi)$ contains a sequence
equivalent to the usual basis of $\ell^1$.

 \subsection{The definition of the space $\mathfrak{X}_\xi$}

  Let $\xi<\omega_1$ and $\ff$ be a regular
thin family of order $\xi$. We define the norm
$\|\cdot\|:c_{00}(\ff)\to\rr$ by setting
\[\begin{split}
  \|x\|=\sup\Bigg\{\Big(\sum_{i=1}^d\Big(\sum_{j=1}^{l_i}|x(t_j^i)|\Big)^2\Big)^\frac{1}{2}:\;&
  d\in\nn,(t_j^1)_{j=1}^{l_1},\ldots,(t_j^d)_{j=1}^{l_d}\in\text{\emph{Plm}}(\ff)\;\;\text{and}\\
  &\text{for every}\;\;1\leq i_1<i_2\leq d,\\
   &\{t_j^{i_1}:\;1\leq j\leq l_{i_1}\}\cap\{t_j^{i_2}:\;1\leq j\leq l_{i_2}\}=\emptyset
  \Bigg\}
\end{split}\] for all $x\in c_{00}(\ff)$.
We define $X_\xi=\overline{(c_{00}(\ff),\|\cdot\|)}$. It is easy
to see that a norming set for this space is the smallest
$W\subseteq c_{00}(\ff)^\#$ such that the following are satisfied:
\begin{enumerate}
  \item[(i)] For every $d\in\nn$, every plegma $d$-tuple $(t_j)_{j=1}^d$ in $\ff$ and $\ee_1,\ldots,\ee_d\in\{0,1\}$,
   the functional $f=\sum_{j=1}^d(-1)^{\ee_j}e_{t_j}^*$ belongs to $W$ and is called of type I.
  \item[(ii)] For every $d\in\nn$, $f_1,\ldots,f_d\in W$ of type I with disjoint supports and $a_1,\ldots,a_d\in\rr$
  with $\sum_{j=1}^d a_j^2\leq1$, the functional $\varphi=\sum_{j=1}^da_jf_j$ belongs to $W$ and is called of type II.
\end{enumerate}
It is immediate by the definition of the space $X_\xi$ that its
natural basis $(e_s)_{s\in\ff}$ is unconditional. Our first aim is
to prove that the space $X_\xi$ does not contain any isomorphic
copy of $\ell^1$. To this end we need the following notation and
lemmas.
\begin{notation}
  Let $(x_n)_{n\in\nn}$ be a sequence of finite supported vectors in $X_\xi$. We say that $(x_n)_{n\in\nn}$ is an $\ff$-block sequence if for
  every $n\in\nn$, \[\max\{\max t:t\in\text{supp}(x_n)\}<\min\{\min t:t\in\text{supp}(x_{n+1})\}\]
\end{notation}
The following lemma is immediate by the definition of the norm
$\|\cdot\|$ and the observation that if $(x_n)_{n\in\nn}$ is an
$\ff$-block sequence in $X_\xi$, then for every $n,m\in\nn$, with
$n\neq m$, there does not exist any plegma pair $(s_1, s_2)$, such
that $s_1\in\text{supp}(x_n)$ and $s_2\in\text{supp}(x_m)$.
\begin{lem}\label{X_xi block are l^2}
  Every seminormalized $\ff$-block sequence in $X_\xi$ is equivalent to the usual basis of $\ell^2$.
\end{lem}
For every $l\in\nn$ we recall that $\ff_{[l]}=\{s\in\ff:\min
s=l\}$. We define
$X_l=\overline{<(e_s)_{s\in\ff_{[l]}}>}^{\|\cdot\|}$ and
$P_l:X_\xi\to X_l$ such that for every $x\in c_{00}(\ff)$,
$P_l(x)=\sum_{s\in\ff_{[l]}}(x)$. Since for every $l\in\nn$ there
does not exist any plegma pair in $\ff_{[l]}$, the space $X_l$ is
isometric to $\ell^2$. Hence for every $l\in\nn$ the space $X_l$
is reflexive.
\begin{prop}\label{X_xi either l^2 or something in X_l}
  Every subspace $Z$ of $X_\xi$ contains a
  further subspace $W$ such that
  either there exists $l_0\in\nn$ such that $P_{l_0}|_W$ is an
  isomorphic embedding, or $W$ is an isomorphic copy of $\ell^2$.
\end{prop}
\begin{proof}
  Either there exists an $l_0\in\nn$ such that the operator
  $P_{l_0}|_Z:Z\to X_{l_0}$ is not strictly singular or for every
  $l\in\nn$ the operator
  $P_{l_0}|_Z:Z\to X_{l_0}$ is strictly singular. In the first
  case it is immediate that there exists $W$ infinite dimensional
  subspace of $Z$ such that the operator $P_{l_0}|_W$ is an
  isomorphic embedding.

  Suppose that the second case occurs. Let $(\ee_n)_{n\in\nn}$ be
  a sequence of positive reals such that $\sum_{n=1}^\infty
  \ee_n<\frac{1}{3}$. By induction we construct an $\ff$-block
  sequence $(\widetilde{w}_n)_{n\in\nn}$ and a sequence
  $(w_n)_{n\in\nn}$ in $S_Z$ such that
  $\|w_n-\widetilde{w}_n\|<\ee_n$, for all $n\in\nn$.
  Let $w_1\in S_Z$ and $\widetilde{w}_1\in c_{00}(\ff)$ such
  that $\|w_1-\widetilde{w}_1\|<\ee_1$. Suppose that
  $(w_i)_{i=1}^n$ and $(\widetilde{w}_i)_{i=1}^n$ have been
  chosen. Let $l_0=\max\{\max s:s\in
  \text{supp}(\widetilde{w}_n)\}$. Since, for every $1\leq l\leq
  l_0$, the operator $P_l|_Z$ is strictly singular,
  there exists a subspace $W$ of $Z$ such that
  $\|P_l|_W\|<\frac{\ee_{n+1}}{2l_0}$, for all $1\leq l\leq l_0$.
  Let $w_{n+1}\in S_W$ and
  $w_{n+1}'=w_{n+1}-\sum_{l=1}^{l_0}P_l(w_{n+1})$. Pick also
  $\widetilde{w}_{n+1}\in c_{00}([\nn]^k)$ such that
  $\text{supp}(\widetilde{w}_{n+1})\subseteq
  \text{supp}(w_{n+1}')$ and
  $\|w_{n+1}'-\widetilde{w}_{n+1}\|<\frac{\ee_{n+1}}{2}$. One can
  easily check that $\|w_{n+1}-\widetilde{w}_{n+1}\|<\ee_{n+1}$
  and $\max\{\max s:s\in \text{supp}(\widetilde{w}_n)\}<\min\{\min
  s:s\in\text{supp}(\widetilde{w}_{n+1})\}$.

  It is immediate that the sequence $(\widetilde{w}_n)_{n\in\nn}$
  is 1-unconditional and seminormalized. By the choice of the
  sequence $(\ee_n)_{n\in\nn}$ we have that the sequences
  $(\widetilde{w}_n)_{n\in\nn}$ and $(w_n)_{n\in\nn}$ are
  equivalent. By Lemma \ref{X_xi block are l^2} the sequence
  $(\widetilde{w}_n)_{n\in\nn}$ is equivalent to the usual basis
  of $\ell^2$. Hence the subspace $\overline{<(w_n)_{n\in\nn}>}$
  consists an isomorphic copy of $\ell^2$.
\end{proof}
\begin{cor}
  The space $X_\xi$ is $\ell^2$ saturated.
\end{cor}
Since the natural basis $(e_s)_{s\in\ff}$ of the space $X_\xi$ is
unconditional, using James' theorem (\cite{J}) and the above
corollary we get the following.
\begin{cor}
  The space $X_\xi$ is reflexive.
\end{cor}
We close this section with the following proposition which will be
used in the next subsection.
\begin{prop}\label{X_xi getting l^2}
  Let $\g$ be a regular thin family, $M\in[\nn]^\infty$ and $(x_s)_{s\in\g}$ a seminormalized $\g$-sequence in $X_\xi$ satisfying the following:
  \begin{enumerate}
    \item[(i)] The $\g$-subsequence is plegma disjointly
    supported.
    \item[(ii)] There exists $K\in\nn$ such that $\max\{t(1):t\in\text{supp}(x_s)\}\leq K$ for all $s\in\g\upharpoonright M$.
  \end{enumerate}
  Then every $\g$-spreading model admitted by $(x_s)_{s\in\g\upharpoonright M}$ is equivalent to the usual basis of $\ell^2$.
\end{prop}
\begin{proof}
  Let $c,C>0$ such that $c\leq\|x_s\|\leq C$. Let also $L\in[M]^\infty$ such that the $\g$-subsequence $(x_s)_{s\in\g\upharpoonright L}$
   generates a $\g$-spreading model. Let $n\in\nn$, $(s_j)_{j=1}^n$ be a plegma $n$-tuple in $\g\upharpoonright L$ with $s_1(1)\geq L(n)$
    and $a_1,\ldots,a_n\in\rr$.
  Since $(a_j\cdot x_{s_j})_{j=1}^n$ are disjoint supported, we have that
  \[\Big\|\sum_{j=1}^na_jx_{s_j}\Big\|\geq\Big(\sum_{j=1}^n\|a_j x_{s_j}\|^2\Big)^\frac{1}{2}\geq c\Big(\sum_{j=1}^n|a_j |^2\Big)^\frac{1}{2}\]
   We will complete the proof by showing that \[\Big\|\sum_{j=1}^na_jx_{s_j}\Big\|\leq CK^\frac{1}{2}\Big(\sum_{j=1}^n|a_j |^2\Big)^\frac{1}{2}\]
  Indeed, let $\varphi\in W$. Then there exist $d\in\nn$, $f_1,\ldots,f_d\in W$ of type I and $b_1,\ldots,b_d\in\rr$,
   with $\sum_{q=1}^db_q^2\leq1$, such that $\varphi=\sum_{q=1}^db_qf_q$. For every $1\leq q\leq d$ we set
   \[E_q=\big\{j\in\{1,\ldots,n\}:\;\text{supp}(f_q)\cap\text{supp}(x_{s_j})\neq\emptyset\big\}\]
    It is easy to check that for every $1\leq q\leq d$ we have that $|E_q|\leq K$.
  \[\begin{split}
    \Big|\varphi\Big(\sum_{j=1}^na_jx_{s_j}\Big)\Big|&\leq\sum_{q=1}^d\sum_{j=1}^n|b_qa_jf_q(x_{s_j})|=
    \sum_{q=1}^d\sum_{j\in E_q}|b_qa_jf_q(x_{s_j})|\\
    &\leq \Big(\sum_{q=1}^d\sum_{j\in E_q}|b_q|^2 \Big)^\frac{1}{2}\cdot \Big(\sum_{q=1}^d\sum_{j\in E_q}|a_jf_q(x_{s_j})|^2 \Big)^\frac{1}{2}\\
    &\leq K^\frac{1}{2}\cdot \Big(\sum_{j=1}^n|a_j|^2\sum_{q=1}^d|f_q(x_{s_j})|^2 \Big)^\frac{1}{2}
    \leq K^\frac{1}{2}\cdot C \Big(\sum_{j=1}^n|a_j|^2 \Big)^\frac{1}{2}
  \end{split}\]
\end{proof}
\subsection{$\ell^1$ spreading models of $\mathfrak{X}_\xi$}

In this subsection we study the $\ell^1$ spreading models of the
space $\mathfrak{X}_\xi$. It is a direct consequence of the
definition of the norm of the space that the natural basis
$(e_s)_{s\in\ff}$ generates the usual basis of $\ell^1$ as an
$\ff$-spreading model. The main aim of this subsection is to show
that $\mathfrak{X}_\xi$ does not admit $\ell^1$ as
$\zeta$-spreading model for any $\zeta<\omega_1$ such that
$\zeta+2<\xi$. The proof of this result goes as follows. Let $\g$
be  a regular thin family  with $o(\g)<\xi$ and $(x_s)_{s\in\g}$ a
bounded $\g$-sequence  of finitely supported elements of
$\mathfrak{X}_\xi$. We consider the next two cases. In the first
one we assume that for every $s\in\g$ and $t\in\text{supp}(x_s)$,
$|s|<|t|$.  Then under the additional assumption that
$(x_s)_{s\in\g}$ is plegma disjointly supported, we show that
$(x_s)_{s\in\g}$ does not admit $\ell^1$ as a $\g$-spreading
model. The second case is the complemented one. Namely we assume
that for every $s\in\g$ and $t\in\text{supp}(x_s)$, $|t|\leq|s|$
and again we show that $(x_s)_{s\in\g}$ does not admit $\ell^1$ as
a $\g$-spreading model. The final result follows from the above
two cases.

Let us point out that a similar  method was used in the proof of
Theorem \ref{non plegma preserving maps}. To some extent Theorem
\ref{non plegma preserving maps} can be viewed as the set
theoretical analogue of the present result.
\subsubsection{Case I}
The following lemma is similar to Lemma \ref{Lemma for 2 vectos of
norm almost 2}. For the sequel let $h:D\to\rr$ be the function
defined in Definition \ref{defoffunction}.
\begin{lem}\label{X xi going forward}
  Let $(\ee,\delta)\in D$, with $\ee>0$. Let also $x_1,x_2\in B_{X_\xi}$ with disjoint finite supports satisfying the  following:
  \begin{enumerate}
    \item[(a)] $\|x_1+x_2\|>2-2\ee$ and
    \item[(b)] For every $t_1\in\text{supp}(x_1)$ and
    $t_2\in\text{supp}(x_2)$ the pair $(t_2,t_1)$ is not
    plegma.
  \end{enumerate} Then for every $G_1\subseteq\text{supp}(x_1)$ such that $\|x_1|_{G_1^c}\|\leq\delta$
  there exists $G_2\subseteq\text{supp}(x_2)$ such that
  \begin{enumerate}
    \item[(i)] $\|x_2|_{G_2^c}\|\leq h(\ee,\delta)$.
    \item[(ii)] For every $t_2\in G_2$ there exists $t_1\in G_1$ such that the pair $(t_1,t_2)$ is plegma.
  \end{enumerate}
\end{lem}
\begin{proof}
  Let $\varphi \in W$ such that $\varphi(x_1+x_2)>2-2\ee$. Then there exist $d\in\nn$, $f_1,\ldots,f_d\in W$ of type
   I with disjoint supports and $a_1,\ldots,a_d\in\rr$ with $\sum_{q=1}^da_q^2\leq1$ such that $\varphi=\sum_{q=1}^da_qf_q$.
    It is immediate that $\varphi(x_1),\varphi(x_2)>1-2\ee$.
  Let $A_1=\{q\in\{1,\ldots,d\}:\;\text{supp}(f_q)\cap G_1\neq\emptyset\}$ and $A_2=\{1,\ldots,d\}\setminus A_1$.
  We define $\varphi_1=\sum_{q\in A_1}a_q f_q$ and $\varphi_2=\varphi-\varphi_1=\sum_{q\in A_2}a_qf_q$.
   Let $G_2=\text{supp}(x_2)\cap\text{supp}(\varphi_1)$.
  It is easy to see that (ii) is satisfied and $\varphi(x_2|_{G_2})=\varphi_1(x_2|_{G_2})=\varphi_1(x_2)$. Notice that
  \[\begin{split}1-2\ee-\delta&<\varphi(x_1|_{G_1})=\varphi_1(x_1|_{G_1})=\sum_{q\in A_1}a_qf_q(x_1|_{G_1})\\
  &\leq \Big(\sum_{q\in A_1}a_q^2\Big)^\frac{1}{2}\Big(\sum_{q\in A_1}f_q(x_1|_{G_1})^2\Big)^\frac{1}{2}\leq
   \Big(\sum_{q\in A_1}a_q^2\Big)^\frac{1}{2}\end{split}\]
  Since $\sum_{q=1}^da_q^2\leq1$, we get that
  \[\sum_{q\in A_2}a_q^2<1-(1-2\ee-\delta)^2\]
  Hence \[|\varphi_2(x_2)|\leq \sum_{q\in A_2}|a_qf_q(x_2)|\leq\Big(\sum_{q\in A_2}|a_q|^2\Big)^\frac{1}{2}
   \Big(\sum_{q\in A_2}|f_q(x_2)|^2\Big)^\frac{1}{2}<(1-(1-2\ee-\delta)^2)^\frac{1}{2}\]
  Thus \[\|x_2|_{G_2}\|\geq \varphi(x_2|_{G_2})=\varphi_1(x_2)=\varphi(x_2)-\varphi_2(x_2)>1-2\ee-(1-(1-2\ee-\delta)^2)^\frac{1}{2}\]
  By the definition of the space $X_\xi$ we have that
  $\|x_2\|^2\geq\|x_2|_{G_2}\|^2+\|x_2|_{G_2^c}\|^2$.
  Hence $\|x_2|_{G_2^c}\|\leq(1-(1-2\ee-(1-(1-2\ee-\delta)^2)^\frac{1}{2})^2)^\frac{1}{2}=h(\ee,\delta)$.
\end{proof}
The proof of the following lemma is similar to the one of the
previous lemma.
\begin{lem}\label{X xi going back}
  Let $\ee,\delta,x_1,x_2$ be as in Lemma \ref{X xi going
  forward}. Then for every $G_2\subseteq\text{supp}(x_2)$ such that $\|x_2|_{G_2^c}\|\leq\delta$ there exists
   $G_1\subseteq\text{supp}(x_1)$ such that
  \begin{enumerate}
    \item[(i)] $\|x_1|_{G_1^c}\|\leq h(\ee,\delta)$.
    \item[(ii)] For every $t_1\in G_1$ there exists $t_2\in G_2$ such that the pair $(t_1,t_2)$ is plegma.
  \end{enumerate}
\end{lem}
Under the above lemmas we have the following.
\begin{prop}\label{X_xi non l^1 isometric spr mod with long lengths}
  Let $\g$ regular thin family, $M\in[\nn]^\infty$ and $(x_s)_{s\in\g}$ a $\g$-sequence in $B_{X_\xi}$ satisfying the following
  \begin{enumerate}
    \item[(i)] The $\g$-subsequence $(x_s)_{s\in\g\upharpoonright
    M}$ is plegma disjointly supported.
    \item[(ii)] For every plegma pair $(s_1,s_2)$ in $\g\upharpoonright M$, every $t_1\in\text{supp}(x_{s_1})$ and
    $t_2\in\text{supp}(x_{s_2})$ the pair $(t_2,t_1)$ is not
    plegma.
    \item[(iii)]For every $s\in\g\upharpoonright M$ and every $t\in\text{supp}(x_s)$ we have that $|t|> |s|$.
  \end{enumerate}
  Then the $\g$-subsequence $(x_s)_{s\in\g\upharpoonright M}$ does not admit the usual basis of $\ell^1$ as a $\g$-spreading model.
\end{prop}
\begin{proof}
  Assume on the contrary that the $\g$-subsequence $(x_s)_{s\in\g\upharpoonright M}$ admits the
  usual basis of $\ell^1$ as $\g$-spreading model. We inductively choose sequences
  $(\delta_n)_{n=0}^\infty$ and $(\ee_n)_{n\in\nn}$ as follows.
  Let $\delta_0=0$ and pick $0<\ee_1<\frac12-\frac14\sqrt{2}$.
  Then $(\ee_1,\delta_0)\in D\setminus J_3$ and therefore $0<h(\ee_1,\delta_0)<1$.
  We set $\delta_1=h(\ee_1,\delta_0)$. Suppose that $\ee_1,\ldots,\ee_n$ and $\delta_0,\ldots,\delta_n$ have been
   chosen such that for every $1\leq k\leq n$
  \[1>\delta_k=h(\ee_k,\delta_{k-1})>0\]
  Then pick sufficiently small $\ee_{n+1}>0$ such that $(\ee_{n+1},\delta_n)\in D\setminus J_3$. Then we get $1>h(\ee_{n+1},\delta_n)>0$ and
  we set
  \[\delta_{n+1}=h(\ee_{n+1},\delta_n)\]
  It is clear that for every $n\in\nn$ we have that $0<\delta_n<1$.

  Let $L\in[M]^\infty$ such that the $\g$-subsequence $(x_s)_{s\in\g\upharpoonright L}$ generates
  the usual basis of $\ell^1$ as $\g$-spreading model with respect to $(\ee_n)_{n\in\nn}$.
  We assume passing to an infinite subset of $L$, that $\g$ is very large in $L$. Let $s_0\in\g$ such that $s_0\sqsubseteq L(2\nn)$.
  We set
  \[K=\max\big\{\max t:t\in\text{supp}(x_{s_0})\big\}\]

  \textbf{Claim:} Let $L_1=\{m\in L(2\nn):m>\max s_0\}$. For every
  $s\in\g\upharpoonright L_1$ there exists $G_s\subseteq\text{supp}(x_s)$ such that
  $\|x_s|_{G_s^c}\|\leq\delta_{|s_0|}$ and for all $t\in G_s$, $t(1)<K$ .
\begin{proof}
  [Proof of Claim]  Let $s\in\g\upharpoonright L_1$. Then $s\in\g\upharpoonright L(2\nn)$
  satisfying $\max s_0<\min s$ and therefore by Proposition \ref{accessing everything with plegma path of length |s_0|}
  there exists plegma path $(s_j)_{j=0}^{|s_0|}$ in $\g\upharpoonright\upharpoonright L$ from $s_0$ to $s$. Since
  the $\g$-subsequence $(x_s)_{s\in\g\upharpoonright L}$ generates the usual basis of $\ell^1$ as $\g$-spreading model with respect to
  $(\ee_n)_{n\in\nn}$,
  we have that for every $1\leq j\leq |s_0|$,
  \[\|x_{s_j}+x_{s_{j-1}}\|>2-2\ee_j\] We set $G_0=\text{supp}(x_{s_0})$. Using
  Lemma \ref{X xi going forward} inductively for $j=1,\ldots,|s_0|$,
  we obtain $G_1,\ldots,G_{|s_0|}$ satisfying the following:
  \begin{enumerate}
    \item[(i)] $G_j\subseteq\text{supp}(x_{s_j})$
    \item[(ii)] $\|x_{s_j}|_{G_j^c}\|\leq \delta_j$ and
    \item[(iii)] for every $t\in G_j$ there exists $t'\in G_{j-1}$ such that the pair $(t',t)$ is plegma.
  \end{enumerate}
  Hence for every $t\in G_{|s_0|}$ there exists plegma path $(t_j)_{j=0}^{|s_0|}$ of length $|s_0|$ such that
  $t_{|s_0|}=t$ and $t_j\in G_j$ for all $0\leq j\leq |s_0|$. Since $\ff$ is regular thin
  we have that $|t_j|\geq|t_0|$, for all $1\leq j\leq |s_0|$.
  Hence by assumption (iii), we have that $|t_j|>|s_0|$ for all $0\leq j\leq |s_0|$.
  Therefore
  \[t(1)=t_{|s_0|}(1)<t_{|s_0|-1}(2)<\ldots<t_{|s_0|-j}(j+1)<\ldots<t_0(|s_0|+1)\leq
  K\] and the proof of the claim is complete.
\end{proof}
  For every $s\in\g\upharpoonright L_1$ we set $x_s^1=x_s|_{G_s}$ and $x_s^2=x_s-x_s^1$. We choose $L_2\in[L_1]^\infty$ such that
  $(x_s^1)_{s\in\g\upharpoonright L_2}$ and $(x_s^2)_{s\in\g\upharpoonright L_2}$ generate $(e^1_n)_{n\in\nn}$ and $(e^2_n)_{n\in\nn}$
  respectively as $\g$-spreading models. Then $(e^1_n)_{n\in\nn}$
  is either trivial or by Lemma \ref{X_xi getting l^2}
  and the above claim is equivalent to the usual basis of
  $\ell^2$.
  Hence by Corollary \ref{ultrafilter property for ell^1 spreading
  models} $(e^2_n)_{n\in\nn}$ is
  the usual basis of $\ell^1$ and thus $\|e_1^2\|=1$. The latter consists a contradiction since
   $\|x_s^2\|=\|x_s|_{G_s^c}\|\leq\delta_{|s_0|}<1$, for all $s\in\g\upharpoonright L_2$.
\end{proof}
\subsubsection{Case II}
\begin{lem}\label{X_xi not existing map Phi}
  Let $\mathcal{H}$ regular thin family with $o(\mathcal{H})<o(\ff)$ and $M\in[\nn]^\infty$. Then there is no map
  $\Phi:\mathcal{H}\upharpoonright M\to\mathcal{P}(\ff)$ satisfying for every $v\in\mathcal{H}\upharpoonright M$ the following:
  \begin{enumerate}
    \item[(i)] $\Phi(v)\neq\emptyset$.
    \item[(ii)] $|t|\leq|v|$, for all $t\in\Phi(v)$.
    \item[(iii)] $v(i)\leq t(i)$, for all $t\in\Phi(v)$ and $1\leq i\leq |t|$.
  \end{enumerate}
\end{lem}
\begin{proof}
  Suppose on the contrary that there exists a map $\Phi:\mathcal{H}\upharpoonright M\to\mathcal{P}(\ff)$ satisfying (i)-(iii).
  By Proposition \ref{corollary by Gasparis}
  there exists $L\in[M]^\infty$ such that $\mathcal{H}\upharpoonright L\sqsubset\ff\upharpoonright L$.
   Let $v\in\mathcal{H}\upharpoonright L$ and
  $u\in \ff\upharpoonright L$ such that $v\sqsubset u$. Let
  $t\in\Phi(v)$.
  Then $|t|<|u|$ and $u(i)=v(i)\leq t(i)$, for all $1\leq i\leq |t|$. The latter contradicts that $u\in\ff$.
\end{proof}
\begin{prop}\label{X_xi non spr mod with big min and small lengths}
  Let $\g$ be a regular thin family with $o(\g)<\xi$, $M\in[\nn]^\infty$ and
  $(x_s)_{s\in\g}$ a $\g$-sequence in $B_{X_\xi}$ satisfying the following:
  \begin{enumerate}
    \item[(i)] For every $s\in\g\upharpoonright M$ the vector $x_s$ is of finite support.
    \item[(ii)] For every plegma pair $(s_1,s_2)$ and every $t_1\in\text{supp}(x_{s_1})$,
    $t_2\in\text{supp}(x_{s_2})$ the pair
    $(t_2,t_1)$ is not plegma.
    \item[(iii)] For every $s\in\g\upharpoonright M$ and $t\in\text{supp}(x_s)$, we have that $|t|\leq|s|$.
    \item[(iv)] For every $s\in\g\upharpoonright M$ and $t\in\text{supp}(x_s)$, if $\min s=M(k)$ then $t(1)\geq k$.
  \end{enumerate}
  Then the $\g$-subsequence $(x_s)_{s\in\g\upharpoonright M}$ does not admit the usual basis of $\ell^1$ as a $\g$-spreading model.
\end{prop}
\begin{proof}
  Suppose on the contrary that the $\g$-subsequence $(x_s)_{s\in\g\upharpoonright M}$
  admits the usual basis of $\ell^1$ as a $\g$-spreading model.
  Let $\delta_0=0,01$. Since
  the function $h$ is continuous at $(0,0)$  and $h(0,0)=0$ we can
  inductively construct strictly decreasing null sequences of reals
  $(\ee_n)_{n\in\nn}$ and $(\delta_n)_{n\in\nn}$ such that
  $\delta_1<\delta_0$ and $h(\ee_n,\delta_n)<\delta_{n-1}$, for
  all $n\in\nn$.

   We pass to some $L\in[M]^\infty$ such that $\g$ is very large in $L$ and  the $\g$-subsequence $(x_s)_{s\in\g\upharpoonright L}$
   generates the usual basis of $\ell^1$ as a $\g$-spreading model with respect to $(\ee_n)_{n\in\nn}$.
   By property (iv) we have that for every $s\in\g\upharpoonright L$ and $t\in\text{supp}(x_s)$, if $\min s=L(k)$ then $t(1)\geq k$. We set
   \[\mathcal{H}=\{v\in[\nn]^{<\infty}:L(v)\in\g\}\] and for every $v\in\mathcal{H}$ we set $z_v=x_{L(v)}$.
  It is immediate that the $\mathcal{H}$-sequence $(z_v)_{v\in\mathcal{H}}$ generates the usual basis of $\ell^1$ as
  an $\mathcal{H}$-spreading model with respect to
  $(\ee_n)_{n\in\nn}$ and for every $v\in\mathcal{H}$ we have that
  $v(1)\leq t(1)$, for all $t\in\text{supp}(z_v)$. Notice also
  that $o(\mathcal{H})=o(\g)<\xi=o(\ff)$.

  For every $v\in\mathcal{H}\upharpoonright 2\nn$, we select a plegma path $(v_j)_{j=1}^{|v|}$ in $\mathcal{H}$ such that $v_1=v$ and
  \[\big\{n-1:n\in v_{j-1}\setminus\{\min v_{j-1}\}\big\}\sqsubseteq v_j\] for all $1<j\leq|v|$.
  Hence $v_{j}(1)=v(j)-(j-1)$, for all $1\leq j\leq |v|$. Notice also that for every $1<j\leq|v|$ we have that $\|x_{v_j}+x_{v_{j-1}}\|>2-2\ee_j$.
  We set $G_{|v|}=\text{supp}(z_{v_{|v|}})$ and using for $j=|v|,\ldots,2$, Lemma \ref{X xi going back} we may get
  $G_{|v|-1},\ldots,G_1$ satisfying the following:
  \begin{enumerate}
    \item[(a)] $G_j\subseteq\text{supp}(x_{v_j})$, for all $1\leq j\leq |v|$.
    \item[(b)] $\|x_{v_j}|_{G_j^c}\|<\delta_{j-1}$, for all $1\leq j\leq |v|$.
    \item[(c)] For every $t\in G_j$ there exists $t'\in G_{j+1}$ such that the pair $(t,t')$ is plegma.
  \end{enumerate}
  Hence for every $t_1\in G_1$ there exists a plegma path $(t_j)_{j=1}^{|v|}$ such that $t_j\in G_j$,
  for all $1\leq j\leq |v|$. Thus for every $t_1\in G_1$ we have that for every $1\leq j\leq |t_1|\leq|v|$,
  $t_1(j)\geq t_j(1)+(j-1)\geq v_j(1)+(j-1)=v_1(j)=v(j)$. We set $G_v=G_1$, which by (b) is
  nonempty. Hence there exists a map
  $\Phi:\mathcal{H}\upharpoonright 2\nn\to\mathcal{P}(\ff)$ satisfying for every $v\in\mathcal{H}\upharpoonright M$ the following:
  \begin{enumerate}
    \item[(i)] $\Phi(v)\neq\emptyset$.
    \item[(ii)] $|t|\leq|v|$, for all $t\in\Phi(v)$.
    \item[(iii)] $v(i)\leq t(i)$, for all $t\in\Phi(v)$ and $1\leq i\leq |t|$.
  \end{enumerate}
  This contradicts Lemma \ref{X_xi not existing map Phi}.
\end{proof}
\subsubsection{The main result}
\begin{thm}
  For every $\zeta<\omega_1$, with $\zeta+2<\xi$, the space $X_\xi$ does not admit $\ell^1$ as a $\zeta$-spreading model.
\end{thm}
\begin{proof}
  Assume on the contrary that for some $\zeta<\omega_1$ and $\zeta+2<\xi$, the space $X_\xi$ admits $\ell^1$ as a $\zeta$-spreading model.
  We choose $\phi:\ff\to\nn$ to be an onto and 1-1 map such that for every $s_1,s_2\in\ff$, if $\max s_1<\max s_2$, then $\phi(s_1)<\phi(s_2)$.
  For every $n\in\nn$ we set $e_n=e_{\phi^{-1}(n)}$. It is easy to check that the space $X_\xi$ satisfies the property $\mathcal{P}$ given in
  Definition \ref{notation of P property}.  Since $X_\xi$ is reflexive, by Corollary
  \ref{getting block generated ell^1 spreading model corollary for reflexive} the space $X_\xi$ admits $\ell^1$ as a plegma block generated
  spreading model of order $\zeta$. By Corollary \ref{one order more to get the isometric} there exist a regular thin family $\g_1$ of order
  $\zeta+1$, $M_1\in[M]^\infty$ and a $\g_1$-sequence $(x_s)_{s\in\g_1}$ in $X_\xi$ such that
  $(x_s)_{s\in\g_1\upharpoonright M_1}$ plegma block generates the usual basis of $\ell^1$ as a $\g_1$-spreading
  model. We may also assume that $(x_s)_{s\in\g_1\upharpoonright
  M_1}$ is normalized.
  For every $s\in \g_1\upharpoonright M_1$ we define $G_s^1=\{t\in\text{supp}(x_s):|t|\leq|s|\}$,
  $G_s^2=\{t\in\text{supp}(x_s):|t|>|s|\}$, $x_s^1=x_s|_{G_s^1}$ and $x_s^2=x_s|_{G_s^2}$.

  First we will show that $(x_s^1)_{s\in\g_1\upharpoonright M_1}$ does not admit $\ell^1$ as a $\g_1$-spreading model.
  Indeed, assume
  on the contrary. Then there is $M_2\in[M_1]^\infty$ such that $(x_s^1)_{s\in\g_1\upharpoonright M_2}$
   plegma block generates $\ell^1$ as a $\g_1$-spreading model.
  By Corollary \ref{one order more to get the isometric}
  there exist $M_3\in [M_2]^\infty$ and a $\g_2$-sequence $(z_v)_{v\in\g_2}$, where $\g_2=[\nn]^1\oplus\g_1$, which satisfy the following:
    \begin{enumerate}
      \item[(i)] The $\g_2$-subsequence $(z_v)_{v\in\g_2\upharpoonright M_3}$ plegma block generates the
       usual basis of $\ell^1$ as a $\g$-spreading model.
      \item[(ii)] For every $v\in\g_2\upharpoonright M_3$ there exist $m\in\nn$ and $s_1,\ldots,s_m\in\g_1$ satisfying the following:
          \begin{enumerate}
         \item[(a)] $z_v\in<x_{s_1}^1,\ldots,x_{s_m}^1>$
         \item[(b)] $|s_j|<|v|$, for all $1\leq j\leq m$.
       \end{enumerate}
    \end{enumerate}
    We may also assume that the $\g_2$-subsequence $(z_v)_{v\in\g_2\upharpoonright M_3}$ is normalized.
    For every $k\in\nn$ and $v\in\g_2\upharpoonright M_3$, we define $F_v^k=\{t\in\text{supp}(z_v):\min t<k\}$,
    $z_v^{1,k}=z_v|_{F_v^k}$ and $z_v^{2,k}=z_v-z_v^{1,k}$. By Proposition \ref{X_xi getting l^2}, for every $k\in\nn$ the
    $\g_2$-subsequence $(z_v^{1,k})_{v\in\g_2\upharpoonright M_3}$ does not admit $\ell^1$ as a $\g_2$-spreading model.
    Using Corollary \ref{ultrafilter property for ell^1 spreading models} we inductively construct a decreasing sequence $(M'_k)_{k\in\nn}$
    of infinite subsets of $M_3$ such that $(z_v^{2,k})_{v\in\g_2\upharpoonright M'_k}$ plegma block generates the usual basis of $\ell^1$
    and $\|z_v^{1,k}\|<\frac{1}{k}$, for all $v\in\g_2\upharpoonright M'_k$. We pick $M_4\in[M_3]^\infty$ such that
    $M_4(k)\in M'_k$, for all $k\in\nn$.
    For every $v\in\g_2\upharpoonright M_4$, we define $w_v=z_v^{2,k_v}$, where $k_v\in\nn$ satisfying $\min v=M_4(k_v)$.
    Since $\|w_v-z_v\|<\frac{1}{k_v}$, for all $v\in\g_2\upharpoonright M_4$, it is easy to check that
    the $\g_2$-subsequence $(w_v)_{v\in\g_2\upharpoonright M_4}$ generates the usual basis of $\ell^1$ as a
    $\g_2$-spreading model, which contradicts Proposition \ref{X_xi non spr mod with big min and small lengths}.

    Since the $\g_1$-subsequence $(x_s^1)_{s\in\g_1\upharpoonright M_1}$ does not admit $\ell^1$ as a
    $\g_1$-spreading model, by Corollary \ref{ultrafilter property for ell^1 spreading models} we get
    that the $\g_1$-subsequence $(x_s^2)_{s\in\g_1\upharpoonright M_1}$ admits the usual basis of
     $\ell^1$ as a $\g_1$-spreading model, which contradicts Proposition \ref{X_xi non l^1 isometric spr mod with long lengths}.
\end{proof}

\chapter{$k$-spreading models are not $k$-iterated spreading models}
\label{space separeting strong k-order from k-order} In this
chapter we construct two spaces, a non reflexive and a reflexive
one, which show that for $k>1$, the $k$-iterated spreading models
are a distinct subclass of the ones of order $k$. We first present
and study a general class of norms. The desired examples are
special cases of that class.
\section{The general construction}\label{general constraxion for separete spr mod form strong}
Let $k\in\nn$, with $k>1$. For every $l\in\nn$ let
$C_l=\{s\in[\nn]^k:\min s=l\}$ and $P_l:c_{00}([\nn]^k)\to
c_{00}(C_l)$ defined by $P_l(x)=\sum_{s\in C_l}x(s)e_s$, for all
$x\in c_{00}([\nn]^k)$. Let $(\|\cdot\|_l)_{l\in\nn}$ be a
sequence of norms defined on $c_{00}(\nn)$ such that for every
$l\in\nn$ the following are satisfied:
\begin{enumerate}
  \item[(i)] The basis $(e_n)_{n\in\nn}$ is 1-unconditional under the norm $\|\cdot\|_l$.
  \item[(ii)] For every $n\in\nn$, $\|e_n\|_l=1$.
\end{enumerate}
By the unconditionality property, for every $l\in\nn$, we consider
the norm $\|\cdot\|_l$ be also defined on $c_{00}(C_l)$. Fix
$1<q<p<\infty$. We define the norm
$\|\cdot\|_{q,p}:c_{00}([\nn]^k)\to\rr$ such that
\[\begin{split}
  \|x\|_{q,p}=\sup\Big\{\Big(\sum_{i=1}^d&\Big(\sum_{j=1}^{m_i}
  |x(s_j^i)|^q\Big)^\frac{p}{q}\Big)^\frac{1}{p}:\;
  d\in\nn, m_i\in\nn, (s_j^i)_{j=1}^{m_i}\;\;\text{is plegma}\\
   &\text{in}\;\;[\nn]^k\;\;\text{for all}\;\;1\leq i\leq d\;\;\text{and}\;\;\Big\{s_j^{i_1}(1)\Big\}_{j=1}^{m_{i_1}}\cap \Big\{s_j^{i_2}(1)\Big\}_{j=1}^{m_{i_2}}=\emptyset,\\
   &\text{for all}\;\;1\leq i_1<i_2\leq d\Big\}
\end{split}\]
Let $\|\cdot\|_{(1)}:c_{00}([\nn]^k)\to\rr$ be the norm defined by
\[\|x\|_{(1)}=\Big(\sum_{l=1}^\infty\|P_l(x)\|_l^p\Big)^\frac{1}{p}\]
for all $x\in c_{00}([\nn]^k)$. We also set
$\|\cdot\|_{(2)}=\|\cdot\|_{q,p}$. Finally we define the norm
$\|\cdot\|:c_{00}([\nn]^k)\to\rr$ by setting
\[\|x\|=\max\{\|x\|_{(1)},\|x\|_{(2)}\}\]
for all $x\in c_{00}([\nn]^k)$ and we define
$X=\overline{(c_{00}([\nn]^k),\|\cdot\|)}$. It is immediate that
the sequence $(e_s)_{s\in[\nn]^k}$ forms an 1-unconditional basis
for the space $X$. Notice also that for every $l\in\nn$ and $x\in
c_{00}(C_l)$ we have that $\|x\|=\|x\|_l$. Hence the subspace
$c_{00}(C_l)$ of $X$ is isometric to $(c_{00}(\nn),\|\cdot\|_l)$,
for all $l\in\nn$. For every $l\in\nn$ we define
$X_l=\overline{(c_{00}(C_l),\|\cdot\|)}$.
\begin{notation}
  Let $(x_n)_{n\in\nn}$ be a sequence in $c_{00}([\nn]^k)$. We
  will say that $(x_n)_{n\in\nn}$ is $[\nn]^k$-block if for every
  $n_1<n_2$ in $\nn$ and for every $s_1\in\text{supp}(x_1)$ and
  $s_2\in\text{supp}(x_2)$ we have that $\max s_2<\min s_2$.
\end{notation}
\begin{rem}\label{the i norm are l^p on block}
  It is easy to see that if $(x_n)_{n\in\nn}$ is a $[\nn]^k$-block
  sequence in $X$, then
  \[\Big\|\sum_{l=1}^r a_lx_l\Big\|_{(i)}=\Big(\sum_{l=1}^r|a_l|^p\cdot\|x_l\|_{(i)}^p\Big)^\frac{1}{p}\]
  for all $r\in\nn,\; a_1,\ldots,a_r\in\rr$ and $i\in\{1,2\}$.
\end{rem}
\begin{lem} \label{N^k block are l^p}
  Every seminormalized $[\nn]^k$-block
  sequence in $X$ is equivalent to the usual
  basis of $\ell^p$.
\end{lem}
\begin{proof}
  Let $(x_n)_{n\in\nn}$ be a seminormalized $[\nn]^k$-block
  sequence in $X$ and $c,C>0$ such that $c\leq\|x_n\|\leq C$, for
  all $n\in\nn$.
  Let $r\in\rr$ and $a_1,\ldots,a_r\in\rr$. Then by Remark
  \ref{the i norm are l^p on block} for every
  $i\in\{1,2\}$ we have that
  \[\Big\|\sum_{l=1}^r a_lx_l\Big\|_{(i)}=\Big(\sum_{l=1}^r|a_l|^p\cdot\|x_l\|_{(i)}^p\Big)^\frac{1}{p}
  \leq \Big(\sum_{l=1}^r|a_l|^p\cdot\|x_l\|^p\Big)^\frac{1}{p}\leq C \Big(\sum_{l=1}^r|a_l|^p\Big)^\frac{1}{p}\]
  Hence \[\Big\|\sum_{l=1}^r a_lx_l\Big\|\leq C \Big(\sum_{l=1}^r|a_l|^p\Big)^\frac{1}{p}\]
  Let $E_1=\{l\in\{1,\ldots,r\}:\|x_l\|_{(1)}\geq\|x_l\|_{(2)}\}$
  and $E_2=\{1,\ldots,r\}\setminus E_1$. Let also $i_0\in\{1,2\}$
  such that
  \[\Big(\sum_{l\in E_{i_0}}|a_l|^p\Big)^\frac{1}{p}\geq\frac{1}{2^\frac{1}{p}}\Big(\sum_{l=1}^r|a_l|^p\Big)^\frac{1}{p}\]
  Then by Remark \ref{the i norm are l^p on block}, we have that
  \[\begin{split}
    \Big\|\sum_{l=1}^r a_lx_l\Big\|&\geq \Big\|\sum_{l=1}^r
    a_lx_l\Big\|_{(i_0)}=
    \Big(\sum_{l=1}^r|a_l|^p\cdot\|x_l\|_{(i_0)}^p\Big)^\frac{1}{p}\\
    &\geq \Big(\sum_{l\in
    E_{i_0}}|a_l|^p\cdot\|x_l\|_{(i_0)}^p\Big)^\frac{1}{p}=
    \Big(\sum_{l\in
    E_{i_0}}|a_l|^p\cdot\|x_l\|^p\Big)^\frac{1}{p}\\
    &\geq c\Big(\sum_{l\in E_{i_0}}|a_l|^p\Big)^\frac{1}{p}\geq\frac{c}{2^\frac{1}{p}}\Big(\sum_{l=1}^r|a_l|^p\Big)^\frac{1}{p}
  \end{split}\]
\end{proof}
The following corollary follows by a sliding hump argument and
Lemma \ref{N^k block are l^p}.
\begin{cor}\label{seperating example limits to zero}
  Let $(x_n)_{n\in\nn}$ be a seminormalized
  sequence in $X$ such that $P_l(x_n)\to 0$, for all $l\in\nn$. Then $(x_n)_{n\in\nn}$
  contains a subsequence equivalent to the usual basis of
  $\ell^p$.
\end{cor}
\begin{prop}\label{either l^p or something in X_l}
  Every subspace $Z$ of $X$ contains a
  further subspace $W$ such that
  either there exists $l_0\in\nn$ such that $P_{l_0}|_W$ is an
  isomorphic embedding, or $W$ is an isomorphic copy of $\ell^p$.
\end{prop}
\begin{proof}
  Either there exists an $l_0\in\nn$ such that the operator
  $P_{l_0}|_Z:Z\to X_{l_0}$ is not strictly singular or for every
  $l\in\nn$ the operator
  $P_{l_0}|_Z:Z\to X_{l_0}$ is strictly singular. In the first
  case it is immediate that there exists $W$ infinite dimensional
  subspace of $Z$ such that the operator $P_{l_0}|_W$ is an
  isomorphic embedding.

  Suppose that the second case occurs. Let $(\ee_n)_{n\in\nn}$ be
  a sequence of positive reals such that $\sum_{n=1}^\infty
  \ee_n<\frac{1}{3}$. By induction we construct an $[\nn]^k$-block
  sequence $(\widetilde{w}_n)_{n\in\nn}$ and a sequence
  $(w_n)_{n\in\nn}$ in $S_Z$ such that
  $\|w_n-\widetilde{w}_n\|<\ee_n$, for all $n\in\nn$.
  Let $w_1\in S_Z$ and $\widetilde{w}_1\in c_{00}([\nn]^k)$ such
  that $\|w_1-\widetilde{w}_1\|<\ee_1$. Suppose that
  $(w_i)_{i=1}^n$ and $(\widetilde{w}_i)_{i=1}^n$ have been
  chosen. Let $l_0=\max\{\max s:s\in
  \text{supp}(\widetilde{w}_n)\}$. Since, for every $1\leq l\leq
  l_0$, the operator $P_l|_Z$ is strictly singular,
  there exists a subspace $W$ of $Z$ such that
  $\|P_l|_W\|<\frac{\ee_{n+1}}{2l_0}$, for all $1\leq l\leq l_0$.
  Let $w_{n+1}\in S_W$ and
  $w_{n+1}'=w_{n+1}-\sum_{l=1}^{l_0}P_l(w_{n+1})$. Pick also
  $\widetilde{w}_{n+1}\in c_{00}([\nn]^k)$ such that
  $\text{supp}(\widetilde{w}_{n+1})\subseteq
  \text{supp}(w_{n+1}')$ and
  $\|w_{n+1}'-\widetilde{w}_{n+1}\|<\frac{\ee_{n+1}}{2}$. One can
  easily check that $\|w_{n+1}-\widetilde{w}_{n+1}\|<\ee_{n+1}$
  and $\max\{\max s:s\in \text{supp}(\widetilde{w}_n)\}<\min\{\min
  s:s\in\text{supp}(\widetilde{w}_{n+1})\}$.

  It is immediate that the sequence $(\widetilde{w}_n)_{n\in\nn}$
  is 1-unconditional and seminormalized. By the choice of the
  sequence $(\ee_n)_{n\in\nn}$ we have that the sequences
  $(\widetilde{w}_n)_{n\in\nn}$ and $(w_n)_{n\in\nn}$ are
  equivalent. By Lemma \ref{N^k block are l^p} the sequence
  $(\widetilde{w}_n)_{n\in\nn}$ is equivalent to the usual basis
  of $\ell^p$. Hence the subspace $\overline{<(w_n)_{n\in\nn}>}$
  consists an isomorphic copy of $\ell^p$.
\end{proof}
\begin{cor}\label{separation example condition for reflexivity}
  The following are satisfied:
  \begin{enumerate}
    \item[(i)] The space $X$ contains an isomorphic copy of $\ell^r$
    (resp. $c_0$), for every $r\neq p$, if and only if there exists $l\in\nn$ such that
    the space $X_l$ contains an isomorphic copy of $\ell^r$
    (resp. $c_0$).
    \item[(ii)] The space $X$ is reflexive if and only if, for each
    $l\in\nn$, the space $X_l$ is reflexive.
  \end{enumerate}
\end{cor}
\begin{proof}
  (i) Let
  $r\in[1,\infty)$, with $r\neq p$. Suppose that $X$ contains an
  isomorphic copy $Z$ of $\ell^r$. Then by
  Proposition \ref{either l^p or something in X_l} we conclude that there exists
  subspace $W$ of $Z$ such that either $W$ is an isomorphic copy of $\ell^r$ or for some $l\in\nn$, $P_l|_W$ is an isomorphic embedding of $W$ into
  $X_l$.
  Since $W$ is a subspace of $Z$, it contains a further subspace
  $W'$ isomorphic to $\ell^r$. Therefore the first alternative is
  impossible and we get that there exists an $l\in\nn$ such that
  $X_l$ contains an isomorphic copy of $\ell^r$. Conversely, since for every $l\in\nn$, $X_l$ is a subspace of $X$,
  if $X_l$ contains an isomorphic copy of $\ell^r$, then so does $X$. The arguments concerning
  $c_0$ are identical.

  (ii) If $X$ is reflexive then the same holds for every subspace
  of $X$. In particular $X_l$ is reflexive for all $l\in\nn$.
  Conversely suppose that $X_l$ is reflexive for every $l\in\nn$.
  By (i) we have that $X$ does not contain any isomorphic copy od
  $\ell^1$ or $c_0$. Since $X$ has an unconditional basis, by James'
  theorem (c.f. \cite{J}), we conclude that $X$ is reflexive.
\end{proof}
\begin{lem}\label{seperating example getting l^p}
  Assume that the space $X$ does not contain any isomorphic
  copy of $c_0$. Let $(x_n)_{n\in\nn}$ be a seminormalized Schauder basic
  sequence in $X$ such that for every $l\in\nn$  the sequence
  $(P_l(x_n))_{n\in\nn}$ is norm convergent. Then for every $l\in\nn$ the sequence
  $(P_l(x_n))_{n\in\nn}$ is a null sequence and $(x_n)_{n\in\nn}$
  contains a subsequence equivalent to the usual basis of
  $\ell^p$.
\end{lem}
\begin{proof}
  For every $l\in\nn$, let $y_l\in X_l$ be the norm limit of
  $(P_l(x_n))_{n\in\nn}$. By Corollary \ref{seperating example limits to
  zero} it suffices to show that $y_l=0$ for all $l\in\nn$. Since $(e_s)_{s\in[\nn]^k}$
  is unconditional and $X$ does not contain any isomorphic
  copy of $c_0$ we get that $(e_s)_{s\in[\nn]^k}$ is boundedly complete.

  For every $l_0\in\nn$, we have
  \[\sum_{l=1}^{l_0}P_l(x_n)\stackrel{n\to\infty}{\longrightarrow}\sum_{l=1}^{l_0}y_l\]
  Hence
  $(\|\sum_{l=1}^{l_0}y_l\|)_{l_0=1}^\infty$ is a  bounded sequence and therefore
    there exists $y\in X$
  such that
  $\sum_{l=1}^{l_0}y_l\stackrel{l_0\to\infty}{\longrightarrow}y$.
  Hence  for each $l\in\nn$, $P_l(y)=y_l$ and so we need   to show  that $y=0$.

  Suppose on the contrary that $y\neq 0$ and let $z_n=x_n-y$ for all
  $n\in\nn$. Since $(x_n)_{n\in\nn}$ is Schauder basic, we have
  that $(z_n)_{n\in\nn}$ is not a null sequence. Notice that for
  every $l\in\nn$, we have that
  $P_l(z_n)\stackrel{n\to\infty}{\longrightarrow}0$. By Corollary \ref{seperating example limits to
  zero} we can pass to a subsequence
  $(z_{k_n})_{n\in\nn}$  such
  that $(z_{k_n})_{n\in\nn}$ is equivalent to the usual basis of
  $\ell^p$. In particular we have that  the sequence $(z_{k_n})_{n\in\nn}$ is Ces\`aro
  summable to zero. Therefore there exists $n_0>0$ such
  that
  \[\Big\|\frac{1}{n_0}\sum_{n=1}^{n_0}z_{k_n}\Big\|<\frac{\|y\|}{3C}\;\;\text{and}\;\;\Big\|\frac{1}{n_0}\sum_{n=n_0+1}^{2n_0}z_{k_n}\Big\|<\frac{\|y\|}{3C}\]
  where $C$ is the basis constant of $(x_n)_n$.
  This yields that
  \[\frac{2\|y\|}{3}<\Big\|\frac{1}{n_0}\sum_{n=1}^{n_0}x_{k_n}\Big\|\leq C\Big\|\frac{1}{n_0}
  \sum_{n=1}^{n_0}x_{k_n}-\frac{1}{n_0}\sum_{n=n_0+1}^{2n_0}x_{k_n}\Big\|<\frac{2\|y\|}{3}\]
  which is a contradiction.
\end{proof}
\begin{prop}\label{the basis gen. lq}
  The space $X$ admits $\ell^q$ as a spreading model of order $k$.
  In particular, the natural basis $(e_s)_{s\in[\nn]^k}$ of $X$ generates
  $\ell^q$ as an $[\nn]^k$-spreading model.
\end{prop}
\begin{proof}
  Let $n\in\nn$, $a_1,\ldots,a_n\in\rr$ and $(s_j)_{j=1}^n$
  be a plegma $n-$tuple in $[\nn]^k$ with $s_1(1)\geq n$. It is
  immediate that
  \[\Big\|\sum_{j=1}^na_je_{s_j}\Big\|_{(1)}=\Big(\sum_{j=1}^n|a_j|^p\Big)^\frac{1}{p}\]
  Let $F_1,\ldots,F_d$ disjoint subsets of $\{1,\ldots,n\}$. Since
  $1<q<p$, we have that
  \[\Big(\sum_{i=1}^d\Big(\sum_{j\in F_i}|a_j|^q\Big)^\frac{p}{q}\Big)^\frac{1}{p}\leq
  \Big(\sum_{i=1}^d\Big(\sum_{j\in F_i}|a_j|^q\Big)^\frac{q}{q}\Big)^\frac{1}{q}
  =\Big(\sum_{j=1}^n|a_j|^q\Big)^\frac{1}{q}\]
  the above imply that \[\Big\|\sum_{j=1}^na_je_{s_j}\Big\|=\Big\|\sum_{j=1}^na_je_{s_j}\Big\|_{(2)}=\Big(\sum_{j=1}^n|a_j|^q\Big)^\frac{1}{q}\]
\end{proof}
\section{The nonreflexive case}
The main result of this section is the following.
\begin{thm}\label{the space X_1,q,p}
  For every $1<q<p<\infty$ and $k>1$, there exists a Banach space $X_{1,p,q}^k$
  with an unconditional basis such that every seminormalized Schauder basic
  sequence $(x_n)_{n\in\nn}$ in $X_{1,p,q}^k$ contains a subsequence which
  is  equivalent either to the usual basis of $\ell^1$ or to the usual basis of $\ell^p$.
  Moreover the space $X_{1,p,q}^k$
  admits $\ell^q$ as a spreading model of order $k$.
\end{thm}
 \begin{proof} For $1<q<p$, let
$X_{1,q,p}$ be the space resulting from the construction of the
previous section by setting for every $l\in\nn$,
$\|\cdot\|_l=\|\cdot\|_{\ell^1}$. Let $(x_n)_{n\in\nn}$ be a
seminormalized Schauder basic sequence in $X$. Then one of the
following holds:
  \begin{enumerate}
    \item[(i)] There exist $l_0\in\nn$ and $M_0\in[\nn]^\infty$ such
    that the sequence $(P_{l_0}(x_n))_{n\in M_0}$ does not contain any
    norm convergent subsequence.
    \item[(ii)] For every $l\in\nn$ and $M\in[\nn]^\infty$, the
    sequence $(P_l(x_n))_{n\in M}$ contains a norm
    convergent subsequence.
  \end{enumerate}
  Suppose that (i) holds. Then $(P_{l_0}(x_n))_{n\in M_0}$ is a
  seminormalized sequence in $\ell^1$ with no norm Cauchy
  subsequence. By  the well known Rosenthal's $\ell^1$-
theorem and the Schur property of $\ell^1$, we conclude that there
exists $L\in [M_0]^\infty$  such that $(P_{l_0}(x_n))_{n\in L}$ is
equivalent to the usual basis of $\ell^1$. Since the basis
$(e_s)_{s\in[\nn]^k}$
  is unconditional, we get that the sequence $(x_n)_{n\in L}$
  is also equivalent to the usual basis of $\ell^1$.

  Suppose that (ii) holds. Then  we may
  pass to an $M\in[\nn]^\infty$ such that the sequence
  $(P_l(x_n))_{n\in M}$ converges. By Corollary \ref{separation example condition for
  reflexivity}, we have that $X_{1,p,q}^k$ does not contain any
  isomorphic copy of $c_0$ and by Lemma \ref{seperating example getting
  l^p}, we have that  there exists a further  subsequence of $(x_n)_{n\in M}$
   equivalent to the usual basis of
  $\ell^p$.

  Finally, by Proposition \ref{the basis gen. lq} we have that the
  basis $(e_s)_{s\in[\nn]^k}$ of the space $X_{1,p,q}^k$
  generates $\ell^q$ as a $k$-spreading model.
\end{proof}
\begin{cor}\label{X_1,q,p spr mods}
  The space $X_{1,q,p}^k$ does not admit $\ell^q$ as an iterated
  spreading model of any order. Precisely,
  every iterated spreading model of
  any order admitted by $X_{1,q,p}^k$ is equivalent either to the
  usual basis of $\ell^1$
  or to the usual basis of $\ell^p$.
\end{cor}
\begin{proof}
  Since for every $r\in[1,\infty)$ every iterated spreading model (of order one) of $\ell^1$ or
  $\ell^p$ is equivalent to the usual basis of $\ell^1$ or
  $\ell^p$ respectively, it suffices to show that every iterated spreading model of order one of $X$ is either $\ell^1$ or $\ell^p$.

  Indeed, let $(x_n)_{n\in\nn}$ be a seminormalized Schauder basic sequence in
  $X$ which generates a spreading model $(e_n)_{n\in\nn}$. By
  Theorem \ref{the space X_1,q,p} we have that $(x_n)_{n\in\nn}$
  contains a subsequence $(x_{k_n})_{n\in\nn}$ equivalent either to the
  usual basis of $\ell^1$
  or to the usual basis of $\ell^p$. Since $(x_{k_n})_{n\in\nn}$
  also generates $(e_n)_{n\in\nn}$, we have that $(e_n)_{n\in\nn}$
  is equivalent either to the
  usual basis of $\ell^1$
  or to the usual basis of $\ell^p$.
\end{proof}
\section{The reflexive case}
The main result of this section is the following.
\begin{thm}\label{the refspace X_1,q,p}
  For every $1<q<p<\infty$ and $k>1$, there exists a reflexive space $X_{T,p,q}^k$
  with an unconditional basis such that every seminormalized Schauder basic
  sequence $(x_n)_{n\in\nn}$ in $X_{T,p,q}^k$ contains a subsequence which
   either is equivalent to the usual basis of $\ell^p$ or generates a spreading model (of order one) equivalent to the usual basis of $\ell^1$.
  Moreover the space $X_{T,p,q}^k$
  admits $\ell^q$ as a spreading model of order $k$.
\end{thm}
\begin{proof}
  For $1<q<p$, let
$X_{T,q,p}^k$ be the space resulting from the general construction
by setting for every $l\in\nn$, $\|\cdot\|_l=\|\cdot\|_{T}$, where
$\|\cdot\|_T$ denotes the norm defined on Tsirelson's space. Since
Tsirelson's space is reflexive, by Corollary \ref{separation
example condition for reflexivity} we get that $X_{T,q,p}^k$ is
reflexive.
  Let $(x_n)_{n\in\nn}$ be a seminormalized Schauder basic sequence in $X_{T,q,p}^k$.
  Since $X_{T,q,p}^k$ is reflexive and $(x_n)_{n\in\nn}$ is
  Schauder basic, we get that $(x_n)_{n\in\nn}$ is weakly null. By
  the standard sliding hump argument and by passing to a subsequence
  of $(x_n)_{n\in\nn}$, we may suppose that
  the sequence $(x_n)_{n\in\nn}$ is finitely disjointly supported.
  Observe that one of the following holds:
  \begin{enumerate}
    \item[(i)] For every $l\in\nn$,
    $P_l(x_n)\stackrel{n\to\infty}{\longrightarrow}0$.
    \item[(ii)] There exist $l_0\in\nn$, $\theta>0$ and $M_0\in[\nn]^\infty$ such
    that $\|P_{l_0}(x_n)\|>\theta$ for all $n\in M_0$.
  \end{enumerate}

  Suppose that (i) holds.  By Corollary \ref{separation example condition for
  reflexivity}, we have that $X_{T,p,q}^k$ does not contain any
  isomorphic copy of $c_0$ and by Lemma \ref{seperating example getting
  l^p}, we have that  there exists a further  subsequence of $(x_n)_{n\in M_0}$
   equivalent to the usual basis of
  $\ell^p$.

  Suppose that (ii) holds. Then the sequence $(P_{l_0}(x_n))_{n\in M_0}$
  actually
  forms a seminormalized block sequence in Tsirelson's space.
  Since every spreading model generated by a
seminormalized Schauder basic sequence in Tsirelson's space is
equivalent to the usual basis of $\ell^1$, there exist
$L\in[M_0]^\infty$ such that the sequence $(P_{l_0}(x_n))_{n\in
L}$ generate $\ell^1$ as spreading model. Since the basis
$(e_s)_{s\in[\nn]^k}$ is
  unconditional, we easily get that the subsequence $(x_n)_{n\in
  L}$ also generates $\ell^1$ as spreading model.

  Moreover, by Proposition \ref{the basis gen. lq} we have that the
  basis $(e_s)_{s\in[\nn]^k}$ of the space $X_{T,p,q}^k$
  generates $\ell^q$ as a $k$-spreading model.
  \end{proof}
The proof of the following corollary is similar to the one of
Corollary \ref{X_1,q,p spr mods}.
\begin{cor}
  The space $X_{T,q,p}^k$ does not admit $\ell^q$ as an iterated
  spreading model of any order. Precisely,
  every iterated spreading model of
  any order of $X_{T,q,p}^k$ is equivalent either to the
  usual basis of $\ell^1$
  or to the usual basis of $\ell^p$.
\end{cor}

\chapter{Spreading models
do not occur everywhere as  plegma block generated}\label{space
without block} In this chapter we construct a reflexive space $X$
with an unconditional basis such that $X$ admits $\ell^1$ as a
spreading model (of order $\omega$) but not as a block generated
spreading model (of any order). The space $X$ does not satisfy the
property $\mathcal{P}$ (see Definition \ref{notation of P
property}) and therefore the condition concerning this property in
Theorem \ref{getting block generated ell^1 spreading model} is
necessary.
\section{The construction and the reflexivity of the space $X$}
We recall that the Schreier family
$\mathcal{S}=\{s\in[\nn]^{<\infty}:|s|=\min s\}$ is a regular thin
family of order $\omega$. Let
\[\begin{split}
  \mathcal{H}=\big\{(t_j)_{j=1}^k:k\in\nn, (t_j)_{j=1}^k\in\text{\emph{Plm}}_k(\widehat{\mathcal{S}})\;\;\text{and}\;\;k-1\leq|t_1|=\ldots=|t_k|\big\}
\end{split}\]
Notice that for every $t\in\widehat{\mathcal{S}}$, we have that
$(t)\in\mathcal{H}$. Moreover for every
$(t_j)_{j=1}^k\in\mathcal{H}$, we have that the elements of the
set $\{t_j:j=1,\ldots,k\}$ are pairwise
$\sqsubseteq$-incomparable. Let $d\in\nn$ and for every $1\leq
q\leq d$, let $k_d\in\nn$ and $(t_j^q)_{j=1}^{k_q}\in\mathcal{H}$.
We will say that $(t_j^1)_{j=1}^{k_1},\ldots,(t_j^d)_{j=1}^{k_d}$
are \emph{incomparable} if the elements of the set
$\cup_{q=1}^d\{t_j^q:1\leq j\leq k_q\}$ are pairwise
$\sqsubseteq$-incomparable.

We define the following norm on $c_{00}(\widehat{\mathcal{S}})$:
\[\begin{split}\|x\|=\sup\Bigg\{\Big(\sum_{q=1}^d\Big(\sum_{j=1}^{k_q}|x(t_j^q)|\Big)^2\Big)^\frac{1}{2}\Bigg\}\end{split}\] where the supremum is
taken over all $d\in\nn$ and incomparable
$(t_j^1)_{j=1}^{k_1},\ldots,(t_j^d)_{j=1}^{k_d}$ in $\mathcal{H}$.
We set $X=\overline{c_{00}(\widehat{\mathcal{S}},\|\cdot\|)}$. It
is immediate that the natural basis
$(e_t)_{t\in\widehat{\mathcal{S}}}$ is 1-unconditional.

We may enumerate the basis of $X$ as $(e_n)_{n\in\nn}$ as follows.
Let $\phi:\widehat{\mathcal{S}}\to\nn$ be an 1-1 and onto map such
that $\phi(\emptyset)=1$ and for every
$t_1,t_2\in\widehat{\mathcal{S}}$ with $\max t_1<\max t_2$,
$\phi(t_1)<\phi(t_2)$. We set $e_n=e_{\phi^{-1}(n)}$, for all
$n\in\nn$. It is easy to see that the space $X$ does not have the
property $\mathcal{P}$ (see Definition \ref{notation of P
property}). Indeed for every $k\in\nn$ let $t_j=\{i:k\leq
i<k+j\}$, for all $1\leq j\leq k$. Then $\|e_{t_j}\|=1$, for all
$1\leq j\leq k$, and $\|\sum_{j=1}^k e_{t_j}\|=1$.
\begin{prop}
  The space $X$ admits the usual basis of $\ell^1$ as an
  $\omega$-spreading model.
\end{prop}
\begin{proof}
  For every $s\in\mathcal{S}$ let \[x_s=\sum_{\emptyset\sqsubset t\sqsubseteq
  s}e_t\] By the definition of the norm it is easy to see that
  $\|x_s\|=1$, for all $s\in\mathcal{S}$, and that
  $(x_s)_{s\in\mathcal{S}}$ generates the usual basis of $\ell^1$
  as an $\mathcal{S}$-spreading model.
\end{proof}
Our next aim is to show the reflexivity of the space $X$. To this
end we will first show that the space $X$ is $\ell^2$-saturated,
that is every subspace of $X$ contains an isomorphic copy of
$\ell^2$.
\begin{lem}\label{X_nbl block with no hitable and all incomparable giving l^2}
  Let $(x_n)_{n\in\nn}$ be a seminormalized sequence in $X$ such
  that for every $n\neq m$ in $\nn$ the following are satisfied:
  \begin{enumerate}
    \item[(i)] For every $t_1\in\text{supp}x_n$ and
    $t_2\in\text{supp}x_m$, $t_1,t_2$ are incomparable.
    \item[(ii)] For every $t_1\in\text{supp}x_n$ and
    $t_2\in\text{supp}x_m$, neither the pair $(t_1,t_2)$ nor
    $(t_2,t_1)$ belongs to $\mathcal{H}$.
  \end{enumerate}
  Then the sequence $(x_n)_{n\in\nn}$ is equivalent to the usual
  basis of $\ell^2$.
\end{lem}
\begin{proof}
  By the definition of the norm we have that if (i) (resp. (ii))
  holds then $(x_n)_{n\in\nn}$ admits a lower (resp. upper)
  $\ell^2$ estimate.
\end{proof}
\begin{notation}
  A sequence $(x_n)_{n\in\nn}$ in $X$ is called
  $\widehat{\mathcal{S}}$-block if for every $n<m$ in $\nn$,
  $t_1\in\text{supp}x_n$ and $t_2\in\text{supp}x_m$ we have that
  $t_1,t_2\neq\emptyset$ and
  $\max t_1<\min t_2$.
\end{notation}
The following is immediate from the previous lemma.
\begin{cor}\label{X_nbl seminorm block give l^2}
  Every seminormalized $\widehat{\mathcal{S}}$-block sequence in
  $X$ is equivalent to the usual basis of $\ell^2$.
\end{cor}
We recall that for every $t\in\widehat{\mathcal{S}}$,
\[\widehat{\mathcal{S}}_{[t]}=\{t'\in\widehat{\mathcal{S}}:t\sqsubseteq
t'\}\]
\begin{prop}
  For every $n\geq2$ the subspace
  $\overline{c_{00}(\widehat{\mathcal{S}}_{[\{n\}]})}$ is isomorphic to
  $\ell^2$. More precisely, for every $n\geq2$ the basis
  $(e_t)_{t\in\widehat{\mathcal{S}}_{[\{n\}]}}$ is equivalent to
  the usual basis of $\ell^2$.
\end{prop}
\begin{proof}
  We will prove it by induction on $n\geq2$. For $n=2$ we have the following. Let $k\in\nn$ and
  $a_0,\ldots,a_k\in\rr$. Then
  \[\Big\|a_0e_{\{2\}}+\sum_{j=1}^ka_je_{\{2,2+j\}}\Big\|=\max\Big(|a_0|,\Big(\sum_{j=1}^ka_j^2\Big)^\frac{1}{2}\Big)
  \leq\Big(\sum_{j=0}^ka_j^2\Big)^\frac{1}{2}\]
  Notice also that $\max(|a_0|^2,\sum_{j=1}^ka_j^2)\geq\frac{1}{2}\sum_{j=0}^ka_j^2$. Hence
  \[\frac{1}{\sqrt{2}}\Big(\sum_{j=0}^ka_j^2\Big)^\frac{1}{2}\leq\Big\|a_0e_{\{2\}}+\sum_{j=1}^ka_je_{\{2,2+j\}}\Big\|
  \leq\Big(\sum_{j=0}^ka_j^2\Big)^\frac{1}{2}\]
and the proof for $n=2$ is complete.
  Let $n\geq 2$ and suppose
  that for every $l\in\nn$, $a_1,\ldots,a_l\in\rr$ and
  $t_1,\ldots,t_l\in\widehat{\mathcal{S}}_{[\{n\}]}$ we have that
  \[(\sqrt{2})^{-(n-1)}\Big(\sum_{j=1}^la_j^2\Big)^\frac{1}{2}
  \leq\Big\|\sum_{j=1}^la_je_{t_j}\Big\|
  \leq\Big(\sum_{j=1}^la_j^2\Big)^\frac{1}{2}\]
  It is easy to see that for every $k\in\nn$ we have that the two
  1-unconditional sequences
  $(e_t)_{t\in\widehat{\mathcal{S}}_{[\{n\}]}}$ and
  $(e_t)_{t\in\widehat{\mathcal{S}}_{[\{n+1,n+1+k\}]}}$ are
  1-equivalent. Hence for every $k\in\nn$, $l\in\nn$, $a_1,\ldots,a_l\in\rr$ and
  $t_1,\ldots,t_l\in\widehat{\mathcal{S}}_{[\{n+1,n+1+k\}]}$ we have that
  \[(\sqrt{2})^{-(n-1)}\Big(\sum_{j=1}^la_j^2\Big)^\frac{1}{2}
  \leq\Big\|\sum_{j=1}^la_je_{t_j}\Big\|
  \leq\Big(\sum_{j=1}^la_j^2\Big)^\frac{1}{2}\]
  Let $k\in\nn$. Let $a_0\in\rr$ and for each $1\leq j\leq k$, let
  $l_j\in\nn$, $(a_q^j)_{q=1}^{l_j}$ in $\rr$ and
  $(t_q^j)_{q=1}^{l_j}$ in
  $\widehat{\mathcal{S}}_{[\{n+1,n+1+j\}]}$.
  Then we have that
  \[\Big\|a_0e_{\{n+1\}}+\sum_{j=1}^k\sum_{q=1}^{l_j}a^j_qe_{t^j_q}\Big\|
  =\max\Bigg\{|a_0|,\Big\|\sum_{j=1}^k\sum_{q=1}^{l_j}a^j_qe_{t^j_q}\Big\|\Bigg\}
  \leq \Big(a_0^2+\sum_{j=1}^k\sum_{q=1}^{l_j}(a^j_q)^2\Big)^\frac{1}{2}\]
  Notice that $\displaystyle\max\Big(a_0^2,\sum_{j=1}^k\sum_{q=1}^{l_j}(a^j_q)^2\Big)\geq\frac{1}{2}\Big(a_0^2+\sum_{j=1}^k\sum_{q=1}^{l_j}(a^j_q)^2\Big)$.
  Hence
  \[(\sqrt{2})^{-n}\Big(a_0^2+\sum_{j=1}^k\sum_{q=1}^{l_j}(a^j_q)^2\Big)^\frac{1}{2}
  \leq \Big\|a_0e_{\{n+1\}}+\sum_{j=1}^k\sum_{q=1}^{l_j}a^j_qe_{t^j_q}\Big\|
  \leq \Big(a_0^2+\sum_{j=1}^k\sum_{q=1}^{l_j}(a^j_q)^2\Big)^\frac{1}{2}\]
\end{proof}
\begin{notation}
  For every $l\in\nn$ we set
  $X_l=\overline{c_{00}(\widehat{\mathcal{S}}_{[\{l\}]})}$ and
  $P_l:X\to X_l$ such that
  $P_l(x)=\sum_{t\in\widehat{\mathcal{S}}_{[\{l\}]}}x(t)e_t$, for
  all $x\in X$. We also set $X_0=<e_\emptyset>$ and $P_0:X\to X_0$
  such that $P_0(x)=x(\emptyset)e_\emptyset$, for all $x\in X$.
  Clearly the spaces $X_0$ and $X_1$ are of dimension 1 and
  therefore the projections $P_0$ and $P_1$ are compact.
\end{notation}
\begin{prop}\label{X_nbl l^2 saturated}
  The space $X$ is $\ell^2$ saturated.
\end{prop}
\begin{proof}
  Let $Y$ be a subspace of $X$. Then either there exists $l\geq 2$
  such that $P_l|_Y$ is not strictly singular or for every
  $l\geq0$ the operator $P_l|_Y$ is strictly singular.

  In the first case the result is immediate. Suppose that the
  second case holds. Let $(\ee_n)_{n\in\nn}$ be a decreasing
  sequence of positive reals such that $\sum_{n=1}^\infty
  \ee_n<\frac{1}{3}$. Since $P_0$ is strictly singular, there exists $y_1\in S_Y$ such that $\|P_0(y_1)\|<\frac{\ee_1}{2}$. We pick $w_1\in X$ of
  finite support such that $\text{supp}w_1\subseteq
  \text{supp}(y_1-P_0(y_1))$ and $\|w_1-(y_1-P_0(y_1))\|<\frac{\ee_1}{2}$. Therefore $\|y_1-w_1\|<\ee_1$ and $\emptyset\not\in\text{supp} w_1$. Let $l_1=\max\{\max
  t:t\in\text{supp}w_1\}$. Since $P_l$ is strictly singular for
  all $0\leq l\leq l_1$, there exists a subspace $Y_1$ of $Y$ such that
  $\|P_l|_{Y_1}\|<\frac{\ee_2}{2(l_1+1)}$. Let $y_2\in S_{Y_2}$
  and $\widetilde{w}_2=y_2-\sum_{l=0}^{l_1}P_l(y_2)$. Then
  $\|\widetilde{w}_2-y_2\|<\frac{\ee_2}{2}$. Let $w_2\in X$ of
  finite support such that
  $\text{supp}w_2\subseteq\text{supp}\widetilde{w}_2$
  and $\|w_2-\widetilde{w}_2\|<\frac{\ee_2}{2}$. Hence
  $\|w_2-y_2\|<\ee_2$ and \[\max\{\max t:t\in\text{supp} w_1\}<\min\{\min t:t\in\text{supp}w_2\}\]

  Proceeding in the same way we may construct a normalized
  sequence $(y_n)_{n\in\nn}$ in $Y$ and an
  $\widehat{\mathcal{S}}$-block sequence $(w_n)_{n\in\nn}$ such
  that $\|w_n-y_n\|<\ee_n$, for all $n\in\nn$. The latter yields
  that the sequences $(y_n)_{n\in\nn}$ and $(w_n)_{n\in\nn}$ are
  equivalent. By Corollary \ref{X_nbl seminorm block give l^2}
  we have that $(w_n)_{n\in\nn}$ is equivalent to the usual basis
  of $\ell^2$. Hence $(y_n)_{n\in\nn}$ is equivalent to the usual basis
  of $\ell^2$ and the proof is complete.
\end{proof}
\begin{prop}
  The space $X$ is reflexive.
\end{prop}
\begin{proof}
  First recall that the space $X$ has an unconditional basis.
  Since, by Proposition \ref{X_nbl l^2 saturated}, $X$ is $\ell^2$
  saturated, we have that $X$ does not contain any isomorphic copy
  of $c_0$ or $\ell^1$. Hence by James' theorem (c.f. \cite{J}) we
  have that $X$ is reflexive.
\end{proof}
\section{The space $X$ does not admit any plegma block generated $\ell^1$ spreading model}
\begin{notation}
  Let $\g\subseteq [\nn]^{<\infty}$, $k\geq2$ and
  $s_0,\ldots,s_k\in\g$. We will say that $(s_j)_{j=0}^k$ is a
  3-plegma path from $s_0$ to $s_k$ in $\g$ of length $k$, if
  $(s_{j},s_{j+1},s_{j+2})$ is plegma for all $0\leq j\leq k-2$.
\end{notation}
We will need a strengthening of Proposition \ref{accessing
everything with plegma path of length |s_0|}.
\begin{cor}\label{corollary for strong plegma paths}
  Let $\g$ be a regular thin family and $L\in[\nn]^\infty$ such
  that $\g$ is very large in $L$. Then for every
  $s_0,s\in\g\upharpoonright\upharpoonright L(2\nn)$, with
  $s_0<s$, there exists a 3-plegma path from $s_0$ to $s$
  in $\g\upharpoonright\upharpoonright L$ of length $2|s_0|$.
\end{cor}
\begin{proof}
  By Proposition \ref{accessing everything with plegma path of length
  |s_0|} there exists a plegma path $(s'_j)_{j=0}^{|s_0|}$
  from $s_0$ to $s$ in $\g\upharpoonright\upharpoonright L(2\nn)$
  of length $|s_0|$. For every $1\leq j\leq |s_0|$, if
  $s_j'=\{L(n_1^j)<\ldots<L(n_{|s_j'|}^j)\}$, we set
  $\widetilde{s}_j$ the unique element of $\g$ such that
  \[\widetilde{s}_j\sqsubseteq\{L(n_1^j-1),\ldots,L(n_{|s_j'|}^j-1)\}\]
  We set $s_{2j}=s_j'$, for all $0\leq j\leq |s_0|$, and
  $s_{2j-1}=\widetilde{s}_j$, for all $1\leq j\leq |s_0|$. It is
  easy to check that $(s_j)_{j=0}^{2|s_0|}$ is a 3-plegma
  path from $s_0$ to $s$ in $\g\upharpoonright\upharpoonright L$
  of length $2|s_0|$.
\end{proof}

Let $W\subseteq c_{00}(\widehat{\mathcal{S}})$ be the minimal set
satisfying the following.
\begin{enumerate}
  \item[(i)] For every $l\in\nn$, $(t_j)_{j=1}^l\in\mathcal{H}$ and $\ee_1,\ldots,\ee_l\in\{-1,1\}$
  the functional $f=\sum_{j=1}^l\ee_je_{t_j}^*$ belongs to $W$ and will be called of type I with weight $|t_1|$.
  \item[(ii)] For every $d\in\nn$, every collection $f_1,\ldots,f_d$ of functionals
   of type I, such that the set $\cup_{q=1}^d\text{supp}(f_q)$
     consists of pairwise $\sqsubseteq$-incomparable elements,
     and $a_1,\ldots,a_d\in\rr$ with $\sum_{q=1}^da_q^2\leq1$
     the functional $\varphi=\sum_{q=1}^d a_qf_q$ belongs to $W$ and will be called of type II.
\end{enumerate}
It is easy to check that the set $W$ is a norming set for the
space $X$. Moreover, for $d\in\nn$, the collection
$f_1,\ldots,f_d$ of functionals of type I is called incomparable
if the elements of the set $\cup_{q=1}^d\text{supp}(f_q)$ are
pairwise $\sqsubseteq$-incomparable.
\begin{lem}\label{X_nbl short in supp does not admit l^1}
  Let $\g$ be a regular thin family and $(x_v)_{v\in\g}$ a $\g$-sequence in $B_X$. Suppose that there exists $k_0\in\nn$ and $L\in[\nn]^\infty$
  such that for all $v\in\g\upharpoonright L$ and $t\in\text{supp}(x_v)$, $|t|\leq k_0$ and  for every plegma pair $(v_1,v_2)$
   in $\g\upharpoonright L$, we have that $\text{supp}(x_{v_1})\cap\text{supp}(x_{v_2})=\emptyset$. Then the $\g$-subsequence
   $(x_v)_{v\in\g\upharpoonright L}$ does not admit $\ell^1$ as a $\g$-spreading model.
\end{lem}
\begin{proof}
  Let $L'\in[L]^\infty$. For every $\ee>0$ there exists $l\in\nn$ such that $\frac{\sqrt{k_0+1}}{l}<\ee$. We pick
   a plegma $l$-tuple $(v_j)_{j=1}^l$ in $\g\upharpoonright L'$ with $v_1(1)\geq L'(l)$. We will show that
   $\|\frac{1}{l}\sum_{j=1}^lx_{v_j}\|<\ee$. Indeed, for every $\varphi\in W$ there exist $d\in\nn$,
   $f_1,\ldots,f_d$ of type I incomparable and $a_1,\ldots, a_d\in\rr$ with $\sum_{q=1}^d a_q^2\leq1$,
    such that $\varphi=\sum_{q=1}^da_qf_q$. We are interested to estimate the quantity $\varphi(\frac{1}{l}\sum_{j=1}^lx_{v_j})$.
    Since for every $t\in\text{supp}(\sum_{j=1}^lx_{v_j})$, $|t|\leq k_0$, we may assume that the weight of $f_q$ is at most $k_0$,
     for all $1\leq q\leq d$. For every $1\leq q\leq d$, we set
  \[E_q=\{j\in\{1,\ldots,l\}:\text{supp}(f_q)\cap\text{supp}(x_{v_j})\neq\emptyset\}\]
  Notice that $|E_q|\leq|\text{supp}(f_q)|\leq k_0+1$, for all
  $1\leq q\leq d$. We have the following
  \[\begin{split}
    \Big|\varphi\Big(\frac{1}{l}\sum_{j=1}^lx_{v_j}\Big)\Big|& \leq\sum_{q=1}^d\Big|\frac{1}{l}a_qf_q\Big(\sum_{j=1}^lx_{v_j}\Big)\Big|
    \leq\sum_{q=1}^d\sum_{j\in E_q}\big|\frac{1}{l}a_qf_q(x_{v_j})\big|\\
    &\leq\Big(\sum_{q=1}^d\sum_{j\in E_q} a_q^2\Big)^\frac{1}{2}\Big(\sum_{q=1}^d\sum_{j\in E_q} \big(\frac{1}{l}f_q(x_{v_j})\big)^2\Big)^\frac{1}{2}\\
    &\leq\sqrt{k_0+1}\Big(\frac{1}{l^2}\sum_{j=1}^l\sum_{q=1}^d\big(f_q(x_{v_j})\big)^2\Big)^\frac{1}{2}\leq
    \frac{\sqrt{k_0+1}}{l}<\ee
  \end{split}\]
\end{proof}
\begin{rem}\label{X_nbl blck give no hitable remark}
  Let $x_1,x_2\in X$ such that $x_1<x_2$ with respect to $(e_n)_{n\in\nn}$, where $e_n=e_{\phi^{-1}(n)}$ for all $n\in\nn$. Let $F\subseteq\text{supp}(x_2)$ such that for every $t\in F$ there exists $t'\in \text{supp}(x_2)$ with $t'\sqsubset t$. Then there is no $t_1\in \text{supp}(x_1)$ and $t_2\in F$ such that $(t_1,t_2)$ or $(t_2,t_1)$ belong to $\mathcal{H}$.

  Indeed, let $t_1\in \text{supp}(x_1)$ and $t_2\in F$ with $|t_1|=|t_2|$. The pair $(t_2,t_1)\not\in\mathcal{H}$. Indeed, since otherwise we would
  have that $\max t_2<\max t_1$ and therefore $\phi(t_2)<\phi(t_1)$, which contradicts that $x_1<x_2$.

  Suppose that the pair $(t_1,t_2)\in\mathcal{H}$. Thus $(t_1,t_2)$ is plegma. Since $t_2\in F$ there exists
  $t_2'\in \text{supp}(x_2)$ such that $t_2'\sqsubset t_2$. Hence $(t_1,t_2')$ is plegma and $|t_1|>|t_2'|$.
  The latter easily yields that $\max t_1>\max t_2'$. Hence $\phi(t_1)>\phi(t_2')$, which contradicts that $x_1<x_2$.
\end{rem}
\begin{cor}\label{X_nbl x_v^3 does not admit l^1}
  Let $\g$ be a regular thin family, $l\in\nn$ and $(x_v)_{v\in \g}$ a $\g$-sequence in $B_X$ such that for
  every plegma pair $(v_1,v_2)$ in $\g\upharpoonright L$, $ x_{v_1}<x_{v_2}$. Suppose that for every
  $v\in\g\upharpoonright L$ there exist $F_v^1,F_v^2\subseteq\text{supp}(x_v)$ such that for every $t\in F_v^2$
  there exists $t'\in F_v^1$ such that $t'\sqsubset t$. Let $x_v^2=x_v|_{F_v^2}$ for every $v\in \g\upharpoonright L$. Then $(x_v^2)_{v\in\g\upharpoonright L}$ does not admit $\ell^1$ as a $\g$-spreading model.
\end{cor}
\begin{proof}
  Indeed, let $L'\in[L]^\infty$. For every $\ee>0$ we pick $l_0\in\nn$ such that $\frac{1}{l_0}<\ee$. Let $(v_j)_{j=1}^{l_0}$ be a plegma $l_0$-tuple in $\g\upharpoonright L'$. We will show that $\|\frac{1}{l_0}\sum_{j=1}^{l_0}x_{v_j}^2\|<\ee$. Indeed let $\varphi\in W$. Then there exist $d\in\nn$, $f_1,\ldots,f_d$ of type I incomparable and $a_1,\ldots,a_d\in\rr$ with $\sum_{q=1}^da_1^2\leq1$ such that $\varphi=\sum_{q=1}^d a_q f_q$. For every $1\leq q\leq d$, we set $E_q=\{j\in\{1,\ldots,l\}:\text{supp}(f_q)\cap\text{supp}(x_{v_j}^2)\neq\emptyset\}$. By Remark \ref{X_nbl blck give no hitable remark} we have that $|E_q|\leq1$ for all $1\leq q\leq d$. Hence
  \[\begin{split}
    \Big|\varphi\Big(\frac{1}{l_0}\sum_{j=1}^{l_0}x_{v_j}^2\Big)\Big|& \leq\sum_{q=1}^d\Big|\frac{1}{l_0}a_qf_q\Big(\sum_{j=1}^{l_0}x_{v_j}^2\Big)\Big|
    \leq\sum_{q=1}^d\sum_{j\in E_q}\big|\frac{1}{l_0}a_qf_q(x_{v_j}^2)\big|\\
    &\leq\Big(\sum_{q=1}^d\sum_{j\in E_q} a_q^2\Big)^\frac{1}{2}\Big(\sum_{q=1}^d\sum_{j\in E_q} \big(\frac{1}{l_0}f_q(x_{v_j}^2)\big)^2\Big)^\frac{1}{2}\\
    &\leq\Big(\frac{1}{l_0^2}\sum_{j=1}^{l_0}\sum_{q=1}^d\big(f_q(x_{v_j}^2)\big)^2\Big)^\frac{1}{2}\leq
    \frac{1}{l_0}<\ee
  \end{split}\]
\end{proof}
\begin{lem}\label{X_nbl lem contradiction for small x_v^3}
  Let $\g$ be a regular thin family, $L\in[\nn]^\infty$ and
  $(x_v)_{v\in\g}$ a $\g$-sequence in $B_X$. Let $k_0\in\nn$ and
  set for every $v\in\g\upharpoonright L$,
  \[\begin{split}&F_v^1=\{t\in\text{supp}(x_v):|t|\leq k_0\}\;\;\text{and}\\
  &F_v^3=\{t\in\text{supp}(x_v): \forall t'\in F_v^1,\;\; t,t'\;\text{are incomparable}\}\end{split}\]
  Suppose that there exists $\delta<1$ such that for every
  $v\in\g\upharpoonright L$, $\|x_v|_{F_v^3}\|\leq\delta$. Then
  $(x_v)_{v\in\g\upharpoonright L}$ does not admit the usual basis
  of $\ell^1$ as a plegma block generated $\g$-spreading model.
\end{lem}
\begin{proof}
  For every $v\in\g\upharpoonright L$ we define
  \[F_v^2=\{t\in\text{supp}(x_v)\setminus F_v^1: \exists t'\in\ F_v^1\;\;\text{such
  that}\;\;
  t'\sqsubset t\}\]
  It is immediate that for every $v\in\g\upharpoonright L$,
  $(F_v^i)_{i=1}^3$ is a partition of $\text{supp}(x_v)$ and for
  every $t\in F_v^2\cup F_v^3$ we have that  $|t|>k_0$. For every
  $v\in\g\upharpoonright L$ and $1\leq i\leq 3$ we set
  $x_v^i=x_v|_{F_v^i}$.

  Suppose on the contrary that there exists $L_1\in[L]^\infty$
  such that the $\g$-subsequence $(x_v)_{v\in\g\upharpoonright
  L_1}$ plegma block generates the usual basis
  of $\ell^1$ as a $\g$-spreading model. By Lemma \ref{X_nbl short in supp does not admit
  l^1} (resp. Corollary \ref{X_nbl x_v^3 does not admit l^1}) we
  have that $(x_v^1)_{v\in\g\upharpoonright L_1}$ (resp. $(x_v^2)_{v\in\g\upharpoonright
  L_1}$) does not admit $\ell^1$ as a $\g$-spreading model. Hence
  by Corollary \ref{ultrafilter property for ell^1 spreading
  models} $(x_v^3)_{v\in\g\upharpoonright L_1}$ admits the usual
  basis of $\ell^1$ as a $\g$-spreading model, which is impossible
  since $\|x_v^3\|\leq\delta<1$, for all $v\in \g\upharpoonright
  L_1$.
\end{proof}
\begin{notation}$\;$
 Let
  \[D=\Big\{(\ee,\delta)\in[0,\frac{1}{3}-\frac{\sqrt{3}}{6})\times[0,1):\;(\sqrt{3}-3\sqrt{3}\ee)^2+(1-3\ee-\delta)^2\geq3\Big\}\]
  and $h:D\to\rr$ be the
  function defined by
  \[h(\ee,\delta)=(1-(1-3\ee-(3-2(1-3\ee)^2-(1-3\ee-\delta)^2)^\frac{1}{2})^2)^\frac{1}{2}\]
  Let us note that the curve $\mathcal{E}=\{(\ee,\delta)\in\rr^2:\;(\sqrt{3}-3\sqrt{3}\ee)^2+(1-3\ee-\delta)^2=3\}$ is an ellipse, since its image
  through the linear transformation $T:\rr^2\to\rr^2$, defined by
\[T(\ee,\delta)= \left[ \begin{array}{cc}
3\sqrt{3} & 0  \\
3 & 1  \end{array} \right] \cdot \left[\begin{array}
{c}\ee\\\delta
\end{array}\right]\]
is a circle centered at $(\sqrt{3},1)$ and of radius $\sqrt{3}$. Moreover notice
that $(\frac{1}{3}-\frac{\sqrt{3}}{6},0)$ is the first
intersection point of the curve $\mathcal{E}$ and the $\ee$-axis.
Also the point $(0,1)$ belongs to $\mathcal{E}$ and the
$\delta$-axis is the tangent of $\mathcal{E}$ at $(0,1)$.
Therefore the set $D$ is a curved triangle with edges
$J_1,J_2,J_3$, where $J_1$ (respectively $J_2$) is the segment
with endpoints $(\frac{1}{3}-\frac{\sqrt{3}}{6},0)$ and $(0,0)$
(respectively $(0,0)$ and $(0,1)$) and $J_3$ is the arc of
$\mathcal{E}$ which joins the $(0,1)$ with
$(\frac{1}{3}-\frac{\sqrt{3}}{6},0)$.

 It is
easy to see that the function $h$ is well defined on $D$. Moreover
the function $h$ is strictly increasing on $D$ in the following
sense: for all $(\ee',\delta'),(\ee,\delta)\in D$, with either
$0\leq\ee'<\ee$ and $0\leq\delta'\leq\delta$ or $0\leq\ee'\leq\ee$
and $0\leq\delta'<\delta$, we have that
$h(\ee',\delta')<h(\ee,\delta)$. Finally $h[D]=[0,1]$,
$h^{-1}(\{0\})=\{(0,0)\}$ and $h^{-1}(\{1\})=J_3$.
\end{notation}
\begin{lem}\label{X_nbl continuing with 3}
  Let $x_1<x_2<x_3$ (with respect to $(e_n)_{n\in\nn}$, where
  $e_n=e_{\phi^{-1}(\{n\})}$ for all $n\in\nn$) in $B_X$ and
  $k_0\in\nn$. For every $j=1,2,3$ let
  \[\begin{split}
    &F^1_j=\{t\in\text{supp}(x_j):|t|\leq k_0\}\\
    &F^2_j=\{t\in\text{supp}(x_j)\setminus F_j^1:\exists t'\in F_j^1\;\;\text{such that}\;\;t'\sqsubset
    t\}\\
    &F_j^3=\{t\in\text{supp}(x_j):\forall t'\in
    F_j^1,\;\;t,t'\;\;\text{incomparable}\}\\
    &x_j^1=x_j|_{F_j^1},x_j^2=x_j|_{F_j^2}\;\;\text{and}\;\;x_j^3=x_j|_{F_j^3}
  \end{split}\]
  Let $(\ee,\delta)\in D$. Suppose that the following are
  satisfied:
  \begin{enumerate}
    \item[(i)] $\|x_1+x_2+x_3\|>3-3\ee$ and
    \item[(ii)] $\|x_2^3\|<\delta$.
  \end{enumerate}
  Then $\|x_3^3\|<h(\ee,\delta)$.
\end{lem}
\begin{proof}
  Since $\|x_1+x_2+x_3\|>3-3\ee$, there exists $\varphi\in W$ such that $\varphi(x_1+x_2+x_3)>3-3\ee$. Then there exist $d\in\nn$, $f_1,\ldots,f_d$ of type I incomparable and $a_1,\ldots,a_d\in\rr$, with $\sum_{q=1}^d a_q^2\leq1$ such that $\varphi=\sum_{q=1}^da_qf_q$. Let $I=\{1,\ldots,d\}$. For every $F\subseteq\{1,2,3\}$ nonempty we set
  \[\begin{split}I_F=\{q\in I:\;&\text{supp}(f_q)\cap\text{supp}(x_i)\neq\emptyset,\;\;\forall i\in F,\\
  &\text{and}\;\;\text{supp}(f_q)\cap\text{supp}(x_i)=\emptyset,\;\;\forall i\not\in F \}\end{split}\]
  and $\varphi_F=\sum_{q\in I_F}a_qf_q$. Moreover we set
  \[\begin{split}
    &I_{\leq k_0}=\{q\in I_{\{1,2,3\}}:w(f_q)\leq k_0\},\;\;I_{> k_0}=\{q\in I_{\{1,2,3\}}:w(f_q)> k_0\},\\
    &\varphi_{\leq k_0}=\sum_{q\in I_{\leq k_0}}a_qf_q\;\;\text{and}\;\;\varphi_{> k_0}=\sum_{q\in I_{> k_0}}a_qf_q
  \end{split}\]

  Since $\varphi(x_1+x_2+x_3)>3-3\ee$ and $\varphi(x_i)\leq1$, for all $1\leq i\leq 3$, we have that $\varphi(x_i)>1-3\ee$, for all $1\leq i\leq 3$. Hence
  \[\begin{split}
    1-3\ee&<\varphi(x_1)=\sum_{1\in F\subseteq\{1,2,3\}}\varphi_F(x_1)=\sum_{1\in F\subseteq\{1,2,3\}}\sum_{q\in I_F}a_qf_q(x_1)\\
    &\leq\Big(\sum_{1\in F\subseteq\{1,2,3\}}\sum_{q\in I_F}a_q^2\Big)^\frac{1}{2}
  \end{split}\]
  Thus \[\Big(\sum_{q\in I_{\{2\}}\cup I_{\{3\}}\cup I_{\{2,3\}}}a_q^2\Big)^\frac{1}{2}<(1-(1-3\ee)^2)^\frac{1}{2}\]
  Similarly by $1-3\ee<\varphi(x_3)$ we get that
  \[\Big(\sum_{q\in I_{\{1\}}\cup I_{\{2\}}\cup I_{\{1,2\}}}a_q^2\Big)^\frac{1}{2}<(1-(1-3\ee)^2)^\frac{1}{2}\]
  By Remark \ref{X_nbl blck give no hitable remark} we have that $\text{supp}(\varphi_{\{1,2,3\}})\cap\text{supp}(x_2^2)=\emptyset$. Hence it is easy to see that
  \[\varphi_{\{1,2,3\}}(x_2)=\varphi_{\{1,2,3\}}(x_2^1+x_2^3)=\varphi_{\leq k_0}(x_2^1)+\varphi_{> k_0}(x_2^3)\]
  Thus
  \[1-3\ee<\varphi(x_2)=\sum_{\substack{2\in F\subseteq\{1,2,3\}\\F\neq\{1,2,3\}}}\sum_{q\in I_F}a_qf_q(x_2)+\varphi_{\leq k_0}(x_2^1)+\varphi_{>k_0}(x_2^3)\]
  Since $\|x_2^3\|<\delta$, we have that
  \[1-3\ee-\delta<\Big(\sum_{\substack{2\in F\subseteq\{1,2,3\}\\F\neq\{1,2,3\}}}\sum_{q\in I_F}a_q^2+\sum_{q\in I_{\leq k_0}}a_q^2\Big)^\frac{1}{2}\]
  Hence
  \[\Big(\sum_{q\in I_{\{1\}}\cup I_{\{3\}}\cup I_{\{1,3\}}\cup I_{>k_0}}a_q^2\Big)^\frac{1}{2}<(1-(1-3\ee-\delta)^2)^\frac{1}{2}\]
  By the above we have that
  \[\Big(\sum_{\substack{F\subseteq\{1,2,3\}\\F\neq\emptyset,\{1,2,3\}}}\sum_{q\in I_F}a_q^2+\sum_{q\in I_{>k_0}}a_q^2\Big)^\frac{1}{2}<(3-2(1-3\ee)^2-(1-3\ee-\delta)^2)^\frac{1}{2}\]
  The latter yields that
  \[\sum_{\substack{F\subseteq\{1,2,3\}\\F\neq\emptyset,\{1,2,3\}}}\varphi_F(x_3)+\varphi_{>k_0}(x_3)<(3-2(1-3\ee)^2-(1-3\ee-\delta)^2)^\frac{1}{2}\]
  Hence $\varphi_{\leq k_0}(x_3)>1-3\ee-(3-2(1-3\ee)^2-(1-3\ee-\delta)^2)^\frac{1}{2}$.
  Since $\varphi_{\leq k_0}(x_3)=\varphi_{\leq k_0}(x_3^1)$ we have that
  $\|x_3^1\|>1-3\ee-(3-2(1-3\ee)^2-(1-3\ee-\delta)^2)^\frac{1}{2}$.
  Since $1\geq\|x_3\|\geq\|x_3^1+x_3^3\|\geq(\|x_3^1\|^2+\|x_3^3\|^2)^\frac{1}{2}$, we have that
  \[\|x_3^3\|<(1-(1-3\ee-(3-2(1-3\ee)^2-(1-3\ee-\delta)^2)^\frac{1}{2})^2)^\frac{1}{2}=h(\ee,\delta)\]
\end{proof}
The proof of the following lemma is similar to the above and we omit it.
\begin{lem}\label{X_nbl lem starting with 3}
  Let $x_1<x_2<x_3$ in $B_X$ and $\ee>0$ such that $(\ee,0)\in D$ and $\|x_1+x_2+x_3\|>3-3\ee$. We set
  \[\begin{split}
    &k_0=\max\{|t|:t\in\text{supp}(x_1)\},\; F_3^1=\{t\in\text{supp}(x_3):|t|\leq k_0\},\\
    &F_3^3=\{t\in\text{supp}(x_3):\;\forall t'\in F_3^1,\;\;t,t'\;\;\text{incomparable}\}\;\;\text{and}\;\; x_3^3=x_3|_{F_3^3}
  \end{split}\]
  Then $\|x_3^3\|<h(\ee,0)$.
\end{lem}
\begin{thm}\label{X_nbl the final thm}
  The space $X$ does not contain any plegma block generated $\ell^1$ spreading model.
\end{thm}
\begin{proof}
  Suppose on the contrary that $X$ admits $\ell^1$ as a plegma block generated $\xi$-spreading model, for some $\xi<\omega_1$.
   Then by Corollary \ref{one order more to get the isometric} the space $X$ admits the usual basis of $\ell^1$ as
    a plegma block generated $(\xi+1)$-spreading model. That is there exist a regular thin family $\g$ of order $\xi+1$,
    $M\in[\nn]^\infty$ and a $\g$-sequence $(x_v)_{v\in\g}$ such that the $\g$-subsequence $(x_v)_{v\in\g\upharpoonright M}$
    plegma block generates the usual basis of $\ell^1$ as a $\g$-spreading model. Clearly we may suppose that the $\g$-sequence
     $(x_v)_{v\in\g}$ is normalized.

  We inductively choose sequence $(\delta_n)_{n=0}^\infty$ and $(\ee_n)_{n\in\nn}$ as follows. We set $\delta_0=0$ and we pick
   $0<\ee_1<\frac{1}{3}-\frac{\sqrt{3}}{6}$. Then $(\ee_1,\delta_0)\in D\setminus J_3$ and therefore $0<h(\ee_1,\delta_0)<1$.
    We set $\delta_1=h(\ee_1,\delta_0)$. Suppose that $\ee_1>\ldots>\ee_n$ and $\delta_0,\ldots,\delta_{n}$
    have been chosen such that for every $1\leq k\leq n$
  \[0<\delta_{k}=h(\ee_k,\delta_{k-1})<1\]
  We pick $\ee_{n+1}<\ee_n$, such that $(\ee_{n+1},\delta_{n})\in D\setminus J_3$. Thus $0<h(\ee_{n+1},\delta_{n})<1$
   and we set \[\delta_{n+1}=h(\ee_{n+1},\delta_{n})\]
  It is clear that for every $n\in\nn$, $\ee_n>\ee_{n+1}$ and $0<\delta_n<1$.

  We pass to $M_1\in[M]^\infty$ such that $\g$ is very large in $M_1$ and the $\g$-subsequence $(x_v)_{v\in\g\upharpoonright M_1}$
   plegma block generates the usual basis of $\ell^1$  as a $\g$-spreading model with respect to $(\frac{\ee_n}{2})_{n\in\nn}$.
    Let $v_0$ be the unique element of $\g$ such that $v_0\sqsubset M_1(4\nn)$ and $k_0=\max\{|t|:t\in\text{supp}(x_{v_0})\}$.
     For every $v\in\g$ we set
  \[\begin{split}
    &F_v^1=\{t\in\text{supp}(x_v):|t|\leq k_0\},\\
    &F_v^2=\{t\in\text{supp}(x_v)\setminus F_v^1:\;\exists t'\in F_v^1\;\;\text{such that}\;\;t'\sqsubset t\},\\
    &F_v^3=\{t\in\text{supp}(x_v):\;\forall t'\in F_v^1,\;\; t',t\;\;\text{are incomparable}\},\\
    &x_v^1=x_v|_{F_v^1},\; x_v^2=x_v|_{F_v^2}\;\;\text{and}\;\;x_v^3=x_v|_{F_v^3}
  \end{split}\]
  Let also $M_2=\{n\in M_1(4\nn):n>\max v_0\}$.

  \textbf{Claim:} For every $v\in \g\upharpoonright M_2$ we have that $\|x_v^3\|<\delta_{2|v_0|}$.
  \begin{proof}
    [Proof of Claim] Let $v\in \g\upharpoonright M_2$. By Corollary \ref{corollary for strong plegma paths} there
     exists a 3-plegma path $(v_j)_{j=0}^{2|v_0|}$ from $v_0$ to $v$ in $\g\upharpoonright\upharpoonright M_1$
      of length $2|v_0|$. Since $(v_0,v_1,v_2)$ is plegma we have that $x_{v_0}<x_{v_1}<x_{v_2}$ and
      $\|x_{v_0}+x_{v_1}+x_{v_2}\|>3-3\ee_1$.  By Lemma \ref{X_nbl lem starting with 3} we have that
       $\|x_{v_1}^3\|<\delta_1$. Inductively for $j=1,\ldots,2|v_0|-1$ we have that
        $(v_{j-1},v_j,v_{j+1})$ is plegma. Consequentially we have that
         $x_{v_{j-1}}<x_{v_j}<x_{v_{j+1}}$. Notice also that $\|x_{v_{j-1}}+x_{v_j}+x_{v_{j+1}}\|>3-3\ee_{j+1}$.
         Having inductively that $\|x_{v_j}^3\|<\delta_j$ (notice that for $j=1$ it is true)
          by Lemma \ref{X_nbl continuing with 3} we get that $\|x_{v_{j+1}}^3\|<h(\ee_j,\delta_j)=\delta_{j+1}$.
           Hence $\|x_v^3\|=\|x_{v_{2|v_0|}}^3\|<\delta_{2|v_0|}$.
  \end{proof}
  The validity of the above claim and the observation that the $\g$-subsequence $(x_v)_{v\in\g\upharpoonright M_2}$ also
   plegma block generates the usual basis of $\ell^1$ as a $\g$-spreading model contradicts Lemma \ref{X_nbl lem contradiction for small x_v^3}.
\end{proof}
By the reflexivity of the space $X$, Theorem \ref{X_nbl the final
thm} and Proposition \ref{plegma block l^1 in dual} we have the
following.
\begin{cor}
  The dual space $X^*$ of $X$ does not admit $c_0$ as spreading model of any
  order.
\end{cor}
\chapter{A reflexive space not admitting $\ell^p$
 or $c_0$ as a spreading model}\label{space Odel_Schlum}
In this chapter we present an example of a reflexive space $X$
having the property that every spreading model, of any order, of
$X$ does not contain any isomorphic copy of  $c_0$ or $\ell^p$,
for every $p\in[1,\infty)$. This example answers in the
affirmative a related problem posed in \cite{O-S} and shows that
Krivine's theorem \cite{Kr} concerning $\ell^p$ or $c_0$ block
finite representability cannot be captured by the notion of
spreading models.

\section{The definition of the space $X$}
We start with the definition of the space. Its construction is
closely related to the corresponding one in \cite{O-S}.
 Let
$(n_j)_{j\in\nn}$ and $(m_j)_{j\in\nn}$ be two strictly increasing
sequences of natural numbers satisfying the following:
\begin{enumerate}
  \item[(i)] $\sum_{j=1}^\infty \frac{1}{m_j}\leq 0,1$.
  \item[(ii)] For every $a>0$, we have that
  $\frac{n_j^a}{m_j}\stackrel{j\to\infty}{\longrightarrow}\infty$.
  \item[(iii)] For every $j\in\nn$, we have that $\frac{n_j}{n_{j+1}}<\frac{1}{m_{j}}$.
\end{enumerate}
We consider the minimal subset $W\subset (c_{00}(\nn))^\#$
satisfying the following:
\begin{enumerate}
  \item[(i)] $\pm e_n^*\in W$, for all $n\in\nn$.
  \item[(ii)] Functionals of type I: For every $j\in\nn$, $d\leq
  n_j$ and $f_1<\ldots<f_d$ in $W$, the functional
  $\varphi=\frac{1}{m_j}\sum_{q=1}^d f_q$ belongs to $W$. The
  functional $\varphi$ is defined to be of type I and we associate
  to it its weight to be $w(\varphi)=m_j$.
  \item[(iii)]Functionals of type II: For every $d\in\nn$,
  $a_1,\dots,a_d\in\rr$ with $\sum_{k=1}^d a_k^2\leq 1$ and
  $f_1,\ldots,f_d$ in $W$ of type I with pairwise different
  weights, the functional $\varphi=\sum_{k=1}^d a_k f_k$ belongs
  to $W$ and is defined to be of type II.
\end{enumerate}
We define the norm $\|\cdot\|$ on $c_{00}(\nn)$, by setting for
every $x\in c_{00}(\nn)$
\[\|x\|=\sup\{\varphi(x):\varphi \in W\}\] Let $X$ be the
completion of $c_{00}(\nn)$ under the above norm.  It is easy to
see that the Hamel basis $(e_n)_{n\in\nn}$ of $c_{00}(\nn)$ is an
unconditional basis of the space  $X$.

Also for every $j\in\nn$ we define on $X$ the norm $\|\cdot\|_j$
by setting for every $x\in X$
\[\|x\|_j=\sup\{f(x):\; f\;\;\text{is of type I with}\;\;w(f)=m_j\}\]
Notice that for every $j\in\nn$ the norms $\|\cdot\|$ and
$\|\cdot\|_j$ are equivalent. Precisely it is easily shown that
for every $x\in X$, $\|x\|_j\leq \|x\|\leq m_j\|x\|_j$. It is also
easy to check that for every $x\in X$ we have that
\begin{equation}\label{eq17}\|x\|=\max \Big\{ \|x\|_\infty,\big( \sum_{j=1}^\infty \|x\|_j^2 \big)^\frac{1}{2}
\Big\}\end{equation}
where $\|\cdot\|_\infty$ denotes the supremum norm. Hence
for every $x\in X$ the sequence $w=(\|x\|_j)_{j\in\nn}$ belongs to
$\ell^2$ and  $\big( \sum_{j=1}^\infty \|x\|_j^2
\big)^\frac{1}{2}=\|w\|_{\ell^2}\leq \|x\|$.
\section{On the spreading models of the space  $X$}
In this section we will show that  every spreading model, of any
order, of $X$ does not contain any isomorphic copy of  $c_0$ or
$\ell^p$, for every $p\in[1,\infty)$.
\subsection{The space $X$ does not admit $\ell^1$ as a spreading model}
We first show that $X$ does not admit  $\ell^1$ as a  spreading
model of $X$ of any order. We start with the following lemmas.
\begin{lem}\label{small_estimation_on_means}
  Let $j<j_0$ in $\nn$ and $(x_q)_{q=1}^{n_{j_0}}$ be a block
  sequence in the unit ball $B_X$ of $X$. Then
  \[\Big\| \frac{x_1+\ldots+x_{n_{j_0}}}{n_{j_0}}
  \Big\|_j\leq\frac{2}{m_j}\]
\end{lem}
\begin{proof}
  It suffices to show that for every $f$ in $W$ of type I with $w(f)=m_j$, we have that
  \[\Big|f\Big(\frac{1}{n_{j_0}}\sum_{p=1}^{n_{j_0}}x_p\Big)\Big|\leq\frac{2}{m_j}\]
    Indeed let $f$ in $W$ of type I with $w(f)=m_j$.
  Then there exist $1\leq d\leq n_i$ and $f_1<\ldots< f_d$ in $W$, such that $f=\frac{1}{m_i}\sum_{q=1}^d f_q$. We set $I=\{1,\ldots,n_{j_0}\}$,
  \[\begin{split}
    A=\{p\in I:\;&\text{there exists at most one}\;\;q\in\{1,\ldots,n_i\}\;\;\text{such that}\\
    &\text{supp}f_q\cap\text{supp}x_p\neq\emptyset\}\;\;\text{and}
  \end{split}\]
  \[\begin{split}
    B=\{p\in I:\;&\text{there exists at least two}\;\;q\in\{1,\ldots,n_i\}\;\;\text{such that}\\
    &\text{supp}f_q\cap\text{supp}x_p\neq\emptyset\}
  \end{split}\]
  Clearly we have that $A\cup B=\emptyset$ and $A\cap B=\emptyset$. It is also easy to see that $|B|\leq d-1<n_i$ and $|f(x_p)|\leq \frac{1}{m_i}$, for all $p\in A$. Hence
  \[\begin{split}
    \Big|f\Big(\frac{1}{n_{j_0}}\sum_{p=1}^{n_{j_0}}x_p\Big)\Big|&\leq
    \frac{1}{n_{j_0}}\sum_{p=1}^{n_{j_0}}|f(x_p)|
    =\frac{1}{n_{j_0}}\sum_{p\in A}|f(x_p)|+\frac{1}{n_{j_0}}\sum_{p\in B}|f(x_p)|\\
    &\leq \frac{1}{m_i}+\frac{n_i}{n_{j_0}}\leq \frac{1}{m_i}+\frac{n_i}{n_{i+1}}\leq\frac{2}{m_i}=\frac{2}{w(f)}
  \end{split}\]
\end{proof}
\begin{lem}\label{small means}
  Let $d_0<j_0$ in $\nn$, and $(x_q)_{q=1}^{n_{j_0}}$ be a block sequence in $B_X$. We set $E=\{n\in\nn:n>d_0\}$ and $w_q=(\|x_q\|_j)_j$, for all $1\leq q\leq n_{j_0}$. Assume that for some $0<\ee<1$ there exists a disjointly supported finite sequence $(w'_q)_{q=1}^{n_{j_0}}$ in $\ell^2$ such that $\|E(w_q-w'_q)\|_{\ell^2}<\ee$, for all $1\leq q\leq n_{j_0}$. Then \[\Big\|\frac{x_1+\ldots+x_{n_{j_0}}}{n_{j_0}}\Big\|<0.2+\ee+2n_{j_0}^{-\frac12}\]
\end{lem}
\begin{proof}
   By Lemma
  \ref{small_estimation_on_means}, we have that
  \[\Bigg\|\Big(\Big\|\frac{\sum_{q=1}^{n_{j_0}}x_q}{n_{j_0}}\Big\|_j \Big)_{j=1}^{d_0}\Bigg\|_{\ell^2}
  \leq\sum_{j=1}^{d_0}\Big\|\frac{\sum_{q=1}^{n_{j_0}}x_q}{n_{j_0}}\Big\|_j\leq\sum_{j=1}^{d_0}\frac{2}{m_j}<0,2\]
   Using the above and the observation that $\|E(w'_q)\|_{\ell^2}\leq2$, for all $1\leq q\leq n_{j_0}$, we get the following.
  \[\begin{split}
    \Bigg\|\Big(\Big\|\frac{1}{n_{j_0}}\sum_{q=1}^{n_{j_0}}x_q \Big\|_j\Big)_j\Bigg\|_{\ell_2}&
    \leq0,2+\Bigg\|\Big(\Big\|\frac{1}{n_{j_0}}\sum_{q=1}^{n_{j_0}}x_q\Big\|_j\Big)_{j>d_0}\Bigg\|_{\ell^2}\\
    &\leq0,2+\Bigg\|\frac{1}{n_{j_0}}\sum_{q=1}^{n_{j_0}}\big(w_q(j)\big)_{j>d_0}\Bigg\|_{\ell^2}
    \leq0,2+\Big\|\sum_{q=1}^{n_{j_0}}\frac{E(w'_q)}{n_{j_0}}\Big\|_{\ell^2}+\ee\\
    &\leq0,2+\Big(\sum_{q=1}^{n_{j_0}}\Big(\frac{2}{n_{j_0}}\Big)^2\Big)^\frac{1}{2}+\ee= 0,2+\ee+2n_{j_0}^{-\frac12}
  \end{split}\]
  Moreover  $\|\frac{1}{n_{j_0}}\sum_{q=1}^{n_{j_0}}x_q\|_\infty\leq\frac{1}{n_{j_0}}<\frac{1}{m_1}<0,1$. Hence by (\ref{eq17}) the proof is completed.
\end{proof}

\begin{prop}\label{non containing l^1 block spreading
model}
  The space $X$ does not admit $\ell^1$ as plegma block generated
  spreading model of any order.
\end{prop}
\begin{proof}
  Assume on the contrary that there exist $\ff$ regular thin,
  $M\in[\nn]^\infty$ and $(x_s)_{s\in\ff\upharpoonright M}$ which
  plegma block generates $\ell^1$ as an $\ff$-spreading model. We may
  also assume that $x_s\in B_X$ for all $s\in\ff\upharpoonright M$
  and $(x_s)_{s\in\ff\upharpoonright M}$ plegma block generates $\ell^1$
  as an $\ff$-spreading model with constant $1-\ee$,  where
  $\ee=0,1$.

  For every $s\in\ff\upharpoonright M$ we define
  $w_s=(\|x_s\|_j)_{j\in\nn}$ which clearly belongs to
  $B_{\ell^2}$. By the reflexivity of $\ell^2$, we have that the $\ff$-sequence $(w_s)_{s\in\ff}$ is weakly relatively compact. By Proposition \ref{cor for subordinating}, there exists $M'\in[M]^\infty$ such that the $\ff$-subsequence $(w_s)_{s\in\ff\upharpoonright M'}$ is subordinated with respect the weak topology on $\ell^2$ and let $\widehat{\varphi}:\widehat{\ff}\upharpoonright M'\to\ell^2$ be the continuous map witnessing this. Let $(\ee_n)_{n\in\nn}$ be a decreasing sequence of positive numbers such that $\ee_n<\frac{\ee}{2}$, for all $n\in\nn$. By Theorem \ref{canonical tree} there exist $L\in[M]^\infty$ and an
  $\ff$-subsequence $(\widetilde{w}_s)_{s\in\ff\upharpoonright L}$
  in $X$ satisfying the following.
  \begin{enumerate}
    \item[(i)] For every $s\in\ff\upharpoonright L$, $\|w_s-\widetilde{w}_s\|<\ee_n$, where $\min
  s=L(n)$.
\item[(ii)] The $\ff$-subsequence
$(\widetilde{w}_s)_{s\in\ff\upharpoonright
    L}$ is subordinated with respect to the weak topology of $\ell^2$. Moreover, if
    $\widetilde{\varphi}:\widehat{\ff}\upharpoonright L\to (\ell^2,w)$ is the continuous map witnessing
    this, then
    $\widetilde{\varphi}(\emptyset)=\widehat{\varphi}(\emptyset)$.
    \item[(iii)] The $\ff$-subsequence $(\widetilde{w}_s)_{s\in\ff\upharpoonright
    L}$ admits a canonical tree decomposition.
\end{enumerate}
Let $d_0\in\nn$ such that
$\|E(\widehat{\varphi}(\emptyset))\|_{\ell^2}<\frac{\ee}{2}$, where
$E=\{d_0+1,\ldots\}$. For every $s\in\ff\upharpoonright
    L$ we set
$w'_s=\widetilde{w}_s-\widehat{\varphi}(\emptyset)$.
It is easy to see that the $\ff$-subsequence $(w'_s)_{s\in\ff\upharpoonright
    L}$ is plegma disjointly
supported. Moreover, notice that $\|E(w_s-w'_s)\|_{\ell^2}<\ee$, for
all $s\in\ff\upharpoonright
    L$. We pick $j_0>d_0$ such that
$2n_{j_0}^{-\frac12}<\ee$. Since $(x_s)_{s\in\ff\upharpoonright
    L}$ generates
$\ell^1$ as an $\ff$-spreading model of constant $0,9$, we may choose
$(s_q)_{q=1}^{n_{j_0}}\in\textit{Plm}_{n_{j_0}}(\ff\upharpoonright
    L)$ such
that
\begin{equation}\label{eq18}
\Big\| \frac{1}{n_{j_0}}\sum_{q=1}^{n_{j_0}}x_{s_q} \Big\|\geq 0,8
\end{equation}
Observe that $d_0,j_0,\ee$, $(x_{s_q})_{q=1}^{n_{j_0}}$ and
$(w'_{s_q})_{q=1}^{n_{j_0}}$ satisfy the assumptions of Lemma
\ref{small means}. Hence \[\Big\|
\frac{1}{n_{j_0}}\sum_{q=1}^{n_{j_0}}x_{s_q}
\Big\|<0,2+\ee+2n_{j_0}^{-\frac12}<0,4\]
  which contradicts (\ref{eq18}) and the proof is complete.
\end{proof}


Using
$\frac{n_j}{m_j}\stackrel{j\to\infty}{\longrightarrow}\infty$ it
is easy to see that the space $X$ satisfies the property
$\mathcal{P}$ (see Definition \ref{notation of P property}). This
easily implies that the space $X$ does not contain any isomorphic
copy of $c_0$. Proposition \ref{non containing l^1 block spreading
model} implies that the space $X$ does not contain any isomorphic
copy of $\ell^1$. Since the basis of $X$ is unconditional we
conclude the following.
\begin{cor} The space  $X$ is reflexive.
\end{cor}
Moreover we have the following.
\begin{cor}\label{corell1}
  The space $X$ does not admit any $\ell^1$ spreading model of
  any order.
\end{cor}
\begin{proof}
  Suppose on the contrary that there exist a regular thin family
  $\ff$, $M\in[\nn]^\infty$ and a bounded $\ff$-sequence
  $(x_s)_{s\in\ff}$ such that $(x_s)_{s\in\ff\upharpoonright M}$
  generates $\ell^1$ as an $\ff$-spreading model. By the reflexivity of $X$
  and the boundness of $(x_s)_{s\in\ff}$, we have that the
  $\ff$-sequence $(x_s)_{s\in\ff}$ is weakly relatively compact.
  Since $X$ satisfies the property $\mathcal{P}$, by Theorem \ref{getting block generated ell^1 spreading
  model} the space $X$ admits $\ell^1$ as a plegma block generated
  $\ff$-spreading model, which contradicts Proposition \ref{non containing l^1 block spreading
  model}.
\end{proof}
\subsection{The space $X$ does not admit $\ell^p$, for $1<p$, or $c_0$ as a spreading model}
We proceed to show that $X$ does not admit any $\ell^p$, $p>1$ or
$c_0$ as a spreading model. First we state some preliminary
lemmas.
\begin{lem}\label{breaking upper l^p}
  For every $p>1$ and for every $\delta,c>0$ there exists
  $k_0\in\nn$ such that for every $k\geq k_0$ and $(x_j)_{j=1}^{n_k}$
  block sequence in $X$ with $\|x_j\|>\delta$ for all $1\leq j\leq
  n_k$, we have that
  \[\Big\|\sum_{j=1}^{n_k}x_j\Big\|>cn_k^\frac{1}{p}\]
\end{lem}
\begin{proof}
  Since
  $\frac{n_k^{1-\frac{1}{p}}}{m_k}\stackrel{k\to\infty}{\longrightarrow}\infty$,
  there exists $k_0\in\nn$ such that for every $k\geq k_0$ we have
  that $\frac{n_k^{1-\frac{1}{p}}}{m_k}>\frac{c}{\delta}$. Let $(x_j)_{j=1}^{n_k}$
  be a block sequence in $X$ such that $\|x_j\|>\delta$ for all $1\leq j\leq
  n_k$. Let $f_1,\ldots,f_{n_k}\in W$ such that $f_j(x_j)>\delta$
  and $\text{supp}(f_j)\subseteq\text{supp}(x_j)$ for all $1\leq j\leq
  n_k$. Then the functional $f=\frac{1}{m_k}\sum_{j=1}^k f_j$
  belongs to $W$. Hence
  \[\Big\|\sum_{j=1}^{n_k}x_j\Big\|\geq f\Big(\sum_{j=1}^{n_k}x_j\Big)=
  \frac{1}{m_k}\sum_{j=1}^{n_k}f_j(x_j)>\frac{n_k}{m_k}\delta>cn_k^\frac{1}{p}\]
\end{proof}
\begin{lem}\label{breaking l^p means}
  Let $p>1$ and $\xi<\omega_1$. For every regular thin family $\ff$ of order $\xi$, $M\in[\nn]^\infty$, $c,\delta>0$ and  $\ff$-sequence
  $(\widetilde{x}_s)_{s\in\ff}$,
  such that $\|\widetilde{x}_s\|>\delta$ for all
  $s\in\ff\upharpoonright M$ and $(\widetilde{x}_s)_{s\in\ff\upharpoonright
  M}$ admits a disjoint canonical tree decomposition, there exist
  $L\in[M]^\infty$ and $k_0\in\nn$ such that for every
  $k\geq k_0$
  \[\Big\|\sum_{j=1}^{n_k}\widetilde{x}_{s_j}\Big\|>cn_k^\frac{1}{p}\]
  for every
  plegma $n_k$-tuple $(s_j)_{j=1}^{n_k}$ in $\ff\upharpoonright L$.
\end{lem}
\begin{proof}
  We will prove the lemma using induction on $\xi$.
  For $\xi=1$ we have that the sequence $(\widetilde{x}_{\{m\}})_{m\in
  M}$ forms a block sequence in $X$. Hence by Lemma \ref{breaking upper
  l^p} the result follows.

  Let $\xi<\omega_1$ and assume  that for every $\zeta<\xi$ the
  lemma is true. We will show that it also  holds  for
  $\xi$. Indeed, let $c,\delta>0$, $\ff$ be a regular thin family of order $\xi$,
  $M\in[\nn]^\infty$ and $(\widetilde{x}_s)_{s\in\ff}$ an $\ff$-sequence
  such that the $\ff$-subsequence
  $(\widetilde{x}_s)_{s\in\ff\upharpoonright M}$ admits a disjoint canonical tree decomposition $(\widetilde{y}_t)_{t\in\widehat{\ff}\upharpoonright
  M}$ and $\|\widetilde{x}_s\|>\delta$, for all $s\in\ff\upharpoonright
  M$.
  Let $L_1\in[M]^\infty$ such that $\ff$ is very large in
  $L_1$. By Theorem \ref{ramseyforplegma} there exists
  $L_2\in[L_1]^\infty$ such that one of the following holds
  \begin{enumerate}
    \item[(i)] $\|\sum_{t\sqsubseteq s_2/_{s_1}}\widetilde{y}_t \|\leq\frac{\delta}{2}$ for every plegma pair $(s_1,s_2)$ in $\ff\upharpoonright
    L_2$,
    \item[(ii)] $\|\sum_{t\sqsubseteq s_2/_{s_1}}\widetilde{y}_t \|>\frac{\delta}{2}$ for every plegma pair $(s_1,s_2)$ in $\ff\upharpoonright
    L_2$.
  \end{enumerate}

  Suppose that (i) occurs. Then by Lemma \ref{breaking upper l^p}
  there exist $k_0$ such that  for every $k\geq k_0$ and $(x_j)_{j=1}^{n_k}$
  block sequence in $X$ with $\|x_j\|>\frac{\delta}{2}$ for all $1\leq j\leq
  n_k$, we have that
  \[\Big\|\sum_{j=1}^{n_k}x_j\Big\|>cn_j^\frac{1}{p}\]
  Let $k\geq k_0$ and $(s_j)_{j=1}^{n_k}$ be a plegma $n_k$-tuple
  in $\ff\upharpoonright L_2(2\nn)$.
  Let $L_2=\{l^2_1,l^2_2,\ldots\}$ and for every
  $1\leq j\leq n_k$, let
  $s_j=\{l^2_{2\rho_1},\ldots,l^2_{2\rho_{|s_j|}}\}$. For all $1\leq j\leq n_k$ we set
  $s_j^*$ to be the unique element in $\ff\upharpoonright L_2$
  with
  $s_j^*\sqsubseteq\{l^2_{2\rho_1-1},\ldots,l^2_{2\rho_{|s_j|}-1}\}$
  and
  \[z_{s_j}=\widetilde{x}_{s_j}-\sum_{t\sqsubseteq s_j/_{s_j^*}}\widetilde{y}_t \]
  Notice that the sequence $(z_{s_j})_{j=1}^{n_k}$
  forms a block sequence in $X$ such that
  $\|z_{s_j}\|>\frac{\delta}{2}$ for all $1\leq j\leq n_k$.
  Since $k\geq k_0$
   and
   the basis $(e_n)_{n\in\nn}$ is 1-unconditional, we get that
  \[\Big\| \sum_{j=1}^{n_k}\widetilde{x}_{s_j} \Big\|\geq\Big\| \sum_{j=1}^{n_k}z_{s_j} \Big\|>cn_j^\frac{1}{p}\]

  Suppose that (ii) occurs. Let $\g=\ff/_{L_2}$. For every
  $t\in\g$, we set
  $z_t=\widetilde{x}_s^{(1,\g)}$ (see Definition \ref{defn g
  splitting}) where
  $s\in\ff\upharpoonright L_2(2\nn)$ and $t\sqsubset s$. Since
  $o(\g)<\xi$, by the inductive hypothesis (for $\zeta=o(\g)$, $c,\frac{\delta}{2}$, $(z_t)_{t\in\g}$ and
  $L_2(2\nn)$),
  there exist $L\in[L_2(2\nn)]^\infty$ and  $k_0\in\nn$ such
  that for every $k\geq k_0$,
  \[\Big\| \sum_{j=1}^{n_k}z_{s_j} \Big\|>cn_j^\frac{1}{p},\]
  for every plegma $n_k$-tuple $(t_j)_{j=1}^{n_k}$ in $\g\upharpoonright
  L$.

  Let $(s_j)_{j=1}^{n_k}$ plegma $n_k$-tuple in $\ff\upharpoonright
  L$. For all $1\leq j\leq n_k$, let $t_j$ in $\g\upharpoonright L$
  such that $t_j\sqsubset s_j$. Then $(t_j)_{j=1}^{n_k}$ is an
  plegma $n_k$-tuple in $\g\upharpoonright L$ and therefore
  \[\Big\| \sum_{j=1}^{n_k}\widetilde{x}_{s_j} \Big\|\geq\Big\| \sum_{j=1}^{n_k}z_{s_j} \Big\|>cn_j^\frac{1}{p}\]
\end{proof}
The following is immediate by the above.
\begin{cor}
  Let $\ff$ be a regular thin family, $M\in[\nn]^\infty$ and
  $(x_s)_{s\in\ff}$ an $\ff$-sequence in $X$. If the
  $\ff$-subsequence $(x_s)_{s\in\ff\upharpoonright M}$ admits a
  disjoint canonical tree decomposition then $(x_s)_{s\in\ff\upharpoonright
  M}$ does not admits any $\ell^p$, for $1<p<\infty$, or $c_0$ as
  an $\ff$-spreading model.
\end{cor}
Moreover since $X$ is reflexive by Corollary \ref{thm disjointly
generic decomposition for wekly relatively compact ff-sequences}
we get the following.
\begin{cor}\label{adfs}
  The space $X$ does not admit any $\ell^p$, for $1<p<\infty$, or $c_0$ as spreading model of
  any order.
\end{cor}
\subsection{The main result and some consequences}
We are now ready to state our main results.
\begin{thm}\label{Odel slumpr theorem}
  Every spreading model, of any order, of the space $X$ does not
  contain any isomorphic copy of $\ell^p$, for every $p\in[1,\infty)$, or $c_0$.
\end{thm}
\begin{proof}
  Assume that for some $\xi<\omega_1$ the space $X$ admits a
  spreading model of order $\xi$ containing an isomorphic copy of
  $\ell^p$ (or $c_0$). Since $X$ is reflexive by Proposition \ref{getting l^p spr mod by
  containing} we have that $X$ admits $\ell^p$ (resp. $c_0$) as a
  spreading model of order $\xi+1$, which contradicts Corollaries
  \ref{adfs} and \ref{corell1}.
\end{proof}
Moreover by Corollaries
  \ref{adfs}, \ref{corell1} and \ref{reflexive spaces have c_0 or l^1 or rflxv spr mod} we
obtain the following.
\begin{cor}\label{odel slump space every nontrivial sprmod is reflexive}
  For every nontrivial spreading model $(e_n)_{n\in\nn}$ admitted
  by $X$, we have that the space $E$ generated by
  $(e_n)_{n\in\nn}$ is reflexive.
\end{cor}
\begin{lem}
  For every $k\in\nn$ and every Banach space $E$ such that
  $X\substack{k\\\longrightarrow\\\;}E$, we have that
  $X\substack{k\\\longrightarrow\\\text{bl}}E$.
\end{lem}
\begin{proof}
   For $k=1$ the
  result is easily verified by the reflexivity of the space $X$
  and the standard sliding hump argument. Suppose that for some
  $k\in\nn$ the lemma is true. Let $E$ be a Banach space such that
  $X\substack{k+1\\\longrightarrow\\\;}E$. Then there exists a
  Banach space $E'$ with a Schauder basis $(e'_n)_{n\in\nn}$ such
  that $X\substack{k\\\longrightarrow\\\;}E'$ and $E'\to E$. By the
  inductive hypothesis we have that
  $X\substack{k\\\longrightarrow\\\text{bl}}E'$. By Corollary
  \ref{Proposition correlating the two kinds of bloch spreading model
  notions} we have that  $(e'_n)_{n\in\nn}\in \mathcal{SM}_k(X)$. Therefore
   by Corollary \ref{odel slump space every nontrivial sprmod is
  reflexive} we have that $E'$ is reflexive. Hence by the
  standard sliding hump argument,
  $E'\substack{\;\\\longrightarrow\\\text{bl}}E$. Thus
  $X\substack{k+1\\\longrightarrow\\\text{bl}}E$. By induction on
  $k\in\nn$ the proof is complete.
\end{proof}
By the above lemma, Corollary \ref{Proposition correlating the two
kinds of bloch spreading model
  notions} and Theorem \ref{Odel slumpr theorem} we have the
  following which answers the related question of \cite{O-S}.
  \begin{cor}
    For every $k\in\nn$ and Banach space $E$ such that
  $X\substack{k\\\longrightarrow\\\;}E$, we have that $E$ does not
  contain any isomorphic copy of $\ell^p$, for all $1\leq
  p\leq\infty$, or $c_0$.
  \end{cor}
  \begin{rem}
    It is immediate by the reflexivity of the space $X$, Corollary \ref{corell1} and
    Theorem \ref{duality c_0 l^1 thm} that the dual space
    $X^*$ of $X$ does not admit any $c_0$ spreading model.
  \end{rem}
It is worth pointing out that the answer to the Problem
\ref{problem2} in Chapter \ref{Chapter 2} is unknown for the space
$X$. Namely we do not know if there exists an ordinal $\xi$ such
that every spreading model of $X$ is equivalent to a
$\xi$-spreading model of $X$. It is also unknown if the spreading
models of $X$ include reflexive Banach spaces which are totally
incomparable to spaces with a saturated norm.

\end{document}